\documentclass[reqno,11pt]{amsart}

\pdfoutput1

\usepackage[utf8]{inputenc}
\usepackage{pigpen}
\usepackage{multicol}

\DeclareFontEncoding{LS1}{}{}
\DeclareFontSubstitution{LS1}{stix2}{m}{n}
\DeclareSymbolFont{symbols4}      {LS1}{stix2bb}   {m}{it}
\DeclareMathSymbol{\smblkdiamond}{\mathord}{symbols4}{"E4}

\usepackage[bbgreekl]{mathbbol} 
\DeclareSymbolFontAlphabet{\mathbb}{AMSb} 

\usepackage{enumerate}
\usepackage{amsmath}
\usepackage{amssymb}
\usepackage{mathtools}

\usepackage[hidelinks]{hyperref}
\makeatletter
\renewcommand{\eqref}{%
             \@ifstar
                  \eqrefStar%
                  \eqrefNoStar%
}
\newcommand{\eqrefStar}[1]{\textup{\tagform@{\ref*{#1}}}}
\newcommand{\eqrefNoStar}[1]{\textup{\tagform@{\ref{#1}}}}
\makeatother

\usepackage{mathrsfs}

\usepackage{thmtools}
\usepackage{thm-restate}

\usepackage{tikz}
\usetikzlibrary{decorations.markings}

\usepackage{tikz-cd}

\newcommand\mnote[1]{} 

\newcommand{\acomment}[1]{} 

\newcounter{thecounter}
\numberwithin{thecounter}{section}
\newtheorem{lemma}[thecounter]{Lemma}
\newtheorem{prop}[thecounter]{Proposition}
\newtheorem{thrm}[thecounter]{Theorem}
\newtheorem{cor}[thecounter]{Corollary}

\theoremstyle{definition}

\newtheorem{example}[thecounter]{Example}

\newtheorem{rem}[thecounter]{Remark}

\newtheorem{notation}[thecounter]{Notation}
\newtheorem{defn}[thecounter]{Definition}

\numberwithin{equation}{section}

\newcommand{\Map}{\operatorname{Map}}

\DeclareMathOperator{\ho}{Ho}

\DeclareMathOperator*{\colim}{colim}

\newcommand{\Hom}{\operatorname{Hom}}
\DeclareMathOperator{\im}{Im}
\DeclareMathOperator*{\limOne}{lim^1}
\let\ker\undefined
\DeclareMathOperator{\ker}{Ker}

\newcommand{\id}{\mathrm{id}}

\newcommand{\holim}{\operatorname{holim}}

\newcommand{\lcom}{\hat{{}_\ell}}

\newcommand{\Q}{{\mathbb {Q}}}
\newcommand{\R}{{\mathbb {R}}}
\newcommand{\F}{{\mathbb {F}}}
\newcommand{\Z}{{\mathbb {Z}}}

\newcommand{\bbk}{\Bbbk}

\newcommand{\beq}{\begin{eqnarray*}}
\newcommand{\eeq}{\end{eqnarray*}}

\newcommand{\tuborg}{\left\{\begin{array}{ll}}
\newcommand{\sluttuborg}{\end{array}\right.}

\newfont{\bm}{msbm10}

\usepackage[all,cmtip]{xy}

\newcommand{\calc}{{\mathcal{C}}}

\newcommand{\mynote}[1]{}

\let\phi\varphi

\newcommand{\Fun}{\operatorname{Fun}}

\def\co{\colon\thinspace}

\newcommand{\sSet}{\mathbf{sSet}} 
\newcommand{\Cat}{\mathbf{Cat}}
\newcommand{\smCat}{\mathbf{smCat}}
\newcommand{\Fib}{\mathbf{Fib}}

\newcommand{\smFib}{\mathbf{smFib}}
\newcommand{\Spaces}{\mathbf{Spaces}}
\newcommand{\PointedSpaces}{\mathbf{PointedSpaces}}
\newcommand{\Kan}{\mathbf{Kan}}
\newcommand{\QCat}{\mathbf{QCat}}
\newcommand{\Spectra}{\mathbf{Sp}}
\newcommand{\hpSpectra}{\mathbf{hpSp}}
\newcommand{\hpSpaces}{\mathbf{hpSpaces}}
\newcommand{\hpC}{\hp\calC}
\newcommand{\hpD}{\hp\calD}
\newcommand{\hpT}{hp\calT} 
\newcommand{\pT}{p\calT} 
\newcommand{\pC}{\mathbf{p}\calC}
\newcommand{\calTrel}{\calT^{2}}
\newcommand{\pCrel}{\mathbf{p}^{2}\calC}
\newcommand{\hpCrel}{\hp^{2}\calC}

\newcommand{\hpMod}{\mathbf{hpMod}}
\newcommand{\Mod}{\mathbf{Mod}}
\newcommand{\Ab}{\mathbf{Ab}}
\newcommand{\grMod}{\mathbf{grMod}}

\newcommand{\InftyCats}{\mathbf{\Cat_\infty}}

\newcommand{\Ex}{\ensuremath{\mathbb{E}\mathbf{x}}}
\newcommand{\Fr}{\ensuremath{\mathbb{F}\mathbf{r}}}
\newcommand{\frbMod}{\ensuremath{\mathbb{M}\mathbf{od}}}

\newcommand{\grothconstr}{{\textstyle\int}}

\newcommand{\fib}{\mathrm{fib}}

\newcommand{\cofib}{\mathrm{cofib}}
\newcommand{\pt}{\mathrm{pt}}

\newcommand{\rightcone}{\vartriangleright}
\newcommand{\leftcone}{\vartriangleleft}

\newcommand{\cupprod}{\smallsmile} 
\newcommand{\capprod}{\smallfrown} 

\newcommand{\incl}{\hookrightarrow}

\newcommand{\tensor}{\otimes}
\newcommand{\smashprod}{\wedge}
\newcommand{\extsmashprod}{\mathbin{\bar{\wedge}}}
\newcommand{\exttensor}{\mathbin{\bar{\tensor}}}

\newcommand{\isom}{\cong}
\newcommand{\homot}{\simeq}

\newcommand{\homeom}{\approx}

\newcommand{\suspension}{\Sigma}
\newcommand{\loops}{\Omega}

\newcommand{\Sing}{\mathrm{Sing}}

\newcommand{\oc}{\mathrm{oc}}
\newcommand{\op}{\mathrm{op}}
\newcommand{\fw}{\mathrm{fw}}
\newcommand{\fop}{\mathrm{fop}}
\newcommand{\dfop}{\mathrm{u}}
\newcommand{\cart}{\mathrm{cart}}
\newcommand{\opcart}{\mathrm{opcart}}
\newcommand{\internal}{\mathrm{int}}
\newcommand{\bdry}{\partial}
\newcommand{\hp}{\mathbf{hp}}
\newcommand{\pr}{\mathrm{pr}}

\DeclareSymbolFont{bbold}{U}{bbold}{m}{n}
\DeclareSymbolFontAlphabet{\mathbbold}{bbold}

\makeatletter
\def\slashedarrowfill@#1#2#3#4#5{%
  $\m@th\thickmuskip0mu\medmuskip\thickmuskip\thinmuskip\thickmuskip
   \relax#5#1\mkern-7mu%
   \cleaders\hbox{$#5\mkern-2mu#2\mkern-2mu$}\hfill
   \mathclap{#3}\mathclap{#2}%
   \cleaders\hbox{$#5\mkern-2mu#2\mkern-2mu$}\hfill
   \mkern-7mu#4$%
}
\def\rightslashedarrowfill@{%
  \slashedarrowfill@\relbar\relbar\mapstochar\rightarrow}
\newcommand\xhto[2][]{%
  \ext@arrow 1579{\rightslashedarrowfill@}{#1}{#2}}
\makeatother

\newcommand{\hto}{\xhto{}}

\makeatletter
\def\circarrowfill@#1#2#3#4#5{%
  $\m@th\thickmuskip0mu\medmuskip\thickmuskip\thinmuskip\thickmuskip
   \relax#5#1\mkern-7mu%
   \cleaders\hbox{$#5\mkern-2mu#2\mkern-2mu$}\hfill
   \mathclap{#3}\mathclap{#2}%
   \cleaders\hbox{$#5\mkern-2mu#2\mkern-2mu$}\hfill
   \mkern-7mu#4$%
}
\def\rightcircarrowfill@{%
  \circarrowfill@\relbar\relbar{\mkern1.8mu\circ}\rightarrow}
\newcommand\xoto[2][]{%
  \ext@arrow 1579{\rightcircarrowfill@}{#1}{#2}}
\makeatother

\newcommand{\oto}{\xoto{}}
\newcommand{\longoto}{\xoto{\quad}}

\def\circsym{%
  \mathchoice%
	{\raisebox{-2.75pt}[0pt][0pt]{$\displaystyle{\circ}$}}
	{\raisebox{-2.75pt}[0pt][0pt]{$\textstyle{\circ}$}}
    {\raisebox{-2.0pt}[0pt][0pt]{$\scriptstyle{\circ}$}}
    {\raisebox{-1.5pt}[0pt][0pt]{$\scriptscriptstyle{\circ}$}}
}
\def\circdec{\mathclap{\circsym}}

\newcommand{\xto}{\xrightarrow}
\newcommand{\xot}{\xleftarrow}

\newcommand{\longto}{\xto{\quad}}
\newcommand{\longot}{\xot{\quad}}
\newcommand{\longhto}{\xhto{\quad}}

\newcommand{\longrightleftarrows}{\mathrel{\substack{\longrightarrow \\[-.7ex] \longleftarrow}}}

\newcommand{\longincl}{\xhookrightarrow{\quad}}

\newcommand{\bbD}{\mathbb{D}}

\newcommand{\calA}{\mathcal{A}}
\newcommand{\calB}{\mathcal{B}}
\newcommand{\calC}{\mathcal{C}}
\newcommand{\calD}{\mathcal{D}}
\newcommand{\calE}{\mathcal{E}}

\newcommand{\calH}{\mathcal{H}}

\newcommand{\calK}{\mathcal{K}}
\newcommand{\calL}{\mathcal{L}}
\newcommand{\calM}{\mathcal{M}}

\newcommand{\calS}{\mathcal{S}}
\newcommand{\calT}{\mathcal{T}}
\newcommand{\calU}{\mathcal{U}}

\DeclareMathSymbol{\mathinvertedexclamationmark}{\mathord}{operators}{'074}
\DeclareMathSymbol{\mathexclamationmark}{\mathord}{operators}{'041}
\makeatletter
\newcommand{\raisedmathinvertedexclamationmark}{%
  \mathord{\mathpalette\raised@mathinvertedexclamationmark\relax}%
}
\newcommand{\raised@mathinvertedexclamationmark}[2]{%
  \raisebox{\depth}{$\m@th#1\mathinvertedexclamationmark$}%
}
\makeatother
\newcommand{\invshriek}{{\raisedmathinvertedexclamationmark}}

\newcommand{\pb}{\ar@{}[dr]|(0.33){\text{\pigpenfont J}}}
\newcommand{\pbop}{\ar@{}[dl]|(0.33){\text{\pigpenfont L}}}

\allowdisplaybreaks

\newcounter{saveenumi}
\newcommand{\savecounteri}{\setcounter{saveenumi}{\value{enumi}}}
\newcommand{\restorecounteri}{\setcounter{enumi}{\value{saveenumi}}}

\usepackage{xr}

\begin{document}
\title{A theory of twisted umkehr maps}

\date{\today}

\author{Anssi Lahtinen}

\thanks{%
    \begin{tabular*}{0.97\textwidth}{@{}c@{}@{\extracolsep{\fill }}p{0.90\textwidth}@{}}
    \raisebox{-0.66\height}{%
        \begin{tikzpicture}
        	[
            	y=0.80pt, 
    			x=0.8pt, 
    			yscale=-1, 
    			inner sep=0pt, 
    			outer sep=0pt, 
    			scale=0.12
			]
            \definecolor{c003399}{RGB}{0,51,153}
            \definecolor{cffcc00}{RGB}{255,204,0}
            \begin{scope}[shift={(0,-872.36218)}]
            	\path[shift={(0,872.36218)},fill=c003399,nonzero rule] 
					(0.0000,0.0000) rectangle (270.0000,180.0000);
            	\foreach \myshift in 
                	{
    					(0,812.36218), 
						(0,932.36218), 
                		(60.0,872.36218), 
						(-60.0,872.36218), 
                		(30.0,820.36218), 
						(-30.0,820.36218),
                		(30.0,924.36218), 
						(-30.0,924.36218),
                		(-52.0,842.36218), 
						(52.0,842.36218), 
                		(52.0,902.36218), 
						(-52.0,902.36218)
					}
                    \path[shift=\myshift,fill=cffcc00,nonzero rule] 
                    	(135.0000,80.0000) -- 
    					(137.2453,86.9096) -- 
    					(144.5106,86.9098) -- 
    					(138.6330,91.1804) -- 
    					(140.8778,98.0902) -- 
    					(135.0000,93.8200) -- 
    					(129.1222,98.0902) -- 
    					(131.3670,91.1804) -- 
    					(125.4894,86.9098) -- 
    					(132.7547,86.9096) -- 
						cycle;
            \end{scope}
        \end{tikzpicture}%
    }
    &
    Supported by the Danish National Research Foundation grants DNRF92 and DNRF151
    and by the 
    European Union's Horizon 2020 research and 
    innovation programme under grant agreements 
    No 800616 and 682922.   
    \end{tabular*}\nopunct%
}

\address[A. Lahtinen]{Copenhagen, Denmark}
\email{anssi@anssilahtinen.net}

\subjclass[2020]{%
55N20 
(Primary)
55R70,
55M05
(Secondary)}

\begin{abstract} 
We develop a theory of umkehr maps for 
twisted generalized homology theories. In this theory, 
interesting 
umkehr maps, including generalizations of 
important classical ones, 
are induced by cartesian morphisms
of a certain category opfibred over the category 
of spaces and continuous maps, making it possible 
to access them through
universal properties. 
\end{abstract}

\maketitle

\setcounter{tocdepth}{2}
\tableofcontents

\section{Introduction}
Since Hopf's construction of what he called an Umkehrungshomomorphismus \cite{Hopf30},
umkehr maps have played a crucial role in topology. They are, in various guises, also called Gysin maps, wrong-way maps, shriek maps, integration along fibre maps, pushforward maps and dimension-shifting transfer maps. 
They feature in formulations of topological 
Grothendieck--Riemann--Roch theorems
and the Atiyah-Singer index theorem for families,
in the construction of the 
generalized Miller--Morita--Mumford characteristic classes 
of bundles of smooth manifolds,
and provide the key ingredient in the construction 
of string topology operations. 
We refer the reader to \cite{BGHistoryOfDuality}
for a historical perspective on umkehr maps.

In this paper, we develop a theory of umkehr maps
which is applicable to twisted generalized homology theories
and 
which recovers and generalizes important classical examples of 
umkehr maps.
A key feature of the theory is that in it interesting 
umkehr maps, including the classical ones, arise as 
maps induced by cartesian morphisms of a certain category opfibred
over the category of spaces and continuous maps. This allows the 
umkehr maps to be accessed through clean categorically defined universal properties
divorced from the often messy details of particular constructions.
The paper is motivated
by the needs of the author's work with Grodal
on string topology of finite groups of Lie type \cite{stringtop},
but should be useful for other purposes as well.

What do we mean by an umkehr map? In common informal parlance, an umkehr map 
is simply an induced map which goes in the opposite direction from the usual
induced maps in the context. 
For example, if $\bbk$ is a field and $f\colon M^m\to N^n$
is a continuous map between  $\bbk$--oriented smooth closed manifolds,
composing the usual induced maps $f^\ast \colon H^\ast(N;\bbk) \to H^\ast(M,\bbk)$
and $f_\ast \colon H_\ast(M;\bbk) \to H_\ast(N;\bbk)$
with the Poincaré duality isomorphisms for $M$ and $N$
yields umkehr maps
\begin{equation}
\label{eq:pdumkehrmaps}
 	f^! \colon H_{\ast}(N;\bbk) \longto H_{\ast+n-m}(M;\bbk)
	\qquad\text{and}\qquad
	f_! \colon H^\ast(M;\bbk) \longto H^{\ast+m-n}(N;\bbk)
\end{equation}
in ordinary homology and cohomology;
and if $p\colon E\to B$ is a fibre bundle whose fibre 
is a closed $\bbk$--oriented $d$--manifold and whose base space is simply connected,
``integration along the fibre'' yields umkehr maps
\begin{equation}
\label{eq:integrationalgonfibreumkehrmaps}
	p_! \colon H^\ast(E;\bbk) \longto H^{\ast-d}(B;\bbk)
	\qquad\text{and}\qquad
	p^! \colon H_\ast(B;\bbk) \longto H_{\ast+d}(E;\bbk).
\end{equation}
The Becker--Gottlieb transfer maps
$p_\natural \colon H^\ast(E;\bbk) \to H^\ast(B;\bbk)$
and
$p^\natural \colon H_\ast(B;\bbk) \to H_\ast(E;\bbk)$
for $p$ \cite{BeckerGottliebTransfer}
provide further examples, although in this paper our main interest
will lie with umkehr maps like $f^!$, $f_!$, $p_!$, and $p^!$ 
which may shift degrees.
A variety of other methods applicable under different circumstances
exist for constructing such umkehr maps. For example, Boardman 
\cite{BoardmanNotesCh5} considers no fewer than eight constructions 
for umkehr maps, all producing the same result when more than one 
are available for a given continuous map. 
While axiomatizations of certain kinds of umkehr maps have been 
considered -- see e.g.\ \cite{ck2009umkehr} and \cite{Lewis_1983} --
we will not adopt or pursue one here. Instead, we point out that 
our umkehr maps fit the informal notion outlined above, 
and recover and generalize classical examples of umkehr maps,
including those of \eqref{eq:pdumkehrmaps} and 
\eqref{eq:integrationalgonfibreumkehrmaps}, justifying our use of the term.

Umkehr maps on the level of homology and cohomology groups
frequently arise from more fundamental umkehr maps defined
on the level of  underlying homotopy categories. For example, 
the maps $f^!$ and $f_!$ above arise 
from a map
\begin{equation}
\label{eq:nmtntommtm}
	N^{-\tau_N} \longto M^{-\tau_M} 
\end{equation}
of spectra
given by the Pontryagin--Thom construction,
and $p_!$ and $p^!$ from a map
\begin{equation}
\label{eq:btoemtp}
	\suspension^\infty_+ B \longto E^{-\tau_p}
\end{equation}
of spectra
where $\tau_p$ is the vertical tangent bundle of $p$
(see \cite[Lemma V.6.22]{BoardmanNotesCh5}).
Here the Thom spectra  $M^{-\tau_M}$, $N^{-\tau_N}$
and $E^{-\tau_p}$ should be thought of as 
desuspensions of  $M$, $N$ and $E$
twisted by the virtual bundles $-\tau_M$, $-\tau_N$ and $-\tau_p$, respectively,
and the maps $f^!$, $f_!$ and $p_!$, $p^!$ are
obtained by composing the maps induced by 
\eqref{eq:nmtntommtm} and \eqref{eq:btoemtp}
with the Thom isomorphisms arising from
the orientations of  $-\tau_M$, $-\tau_N$,
and $-\tau_p$ determined by the assumed orientations of $M$, $N$ and 
the fibre of $p$. Informally, we may think of the Thom isomorphisms
as a combination of undoing the twisting and applying a suspension isomorphism,
with the orientation data specifying how the twisting is to be undone.
In this paper, we will work on the level 
of homotopy categories. Moreover, as our aim is a theory applicable to 
twisted homology theories, we choose not to undo the twistings, obviating 
the need to choose and keep track of orientations.
Even in cases where the ultimate aim is to work in an untwisted setting, 
this approach has the benefit of enabling a clean separation of orientation 
questions from other aspects of the work.

We now outline the categorical setting for our theory in more detail. 
Let $\calC$ be a presentable symmetric monoidal $\infty$--category
with monoidal product $\tensor$. 
Examples of interest for us include the $\infty$--categories $\Spectra$
of spectra, $\Mod^R$ of $R$--modules for a commutative ring spectrum $R$,
and $\Spectra^\ell$ of $\ell$--complete (that is, $H\F_\ell$--local) spectra
for a prime $\ell$, all with $\tensor$ given by the corresponding version 
of smash product. The twisted homology groups of a space $B$ 
in these examples of interest arise from parametrized $\calC$--objects 
over $B$, and as we wish to let the space $B$ vary, it will be convenient for us 
assemble the homotopy categories $\ho(\calC_{/B})$
of parametrized $\calC$--objects over varying base spaces $B$ into a single category 
$\hpC$ fibred and opfibred over the category $\calT$ of topological spaces.
(The `h' in $\hp\calC$ stands for homotopy and the `p' for parametrized.)
More precisely, the category $\hpC$ is obtained by applying the 
Grothendieck construction (see Definition~\ref{def:grothendieckconstr})
to the pseudofunctor
\[
	\calT^\op \longto \Cat,
	\qquad
	B \longmapsto \ho(\calC_{/B})
	\qquad
	(f \colon A \to B) \longmapsto (f^\ast \colon \ho(\calC_{/B}) \to \ho(\calC_{/A}))
\]
where $\Cat$ denotes the $2$--category of categories and 
$f^\ast$ is the pullback functor for parametrized $\calC$--objects
induced by $f$.
The fibre of the projection $\hpC \to \calT$ over a space $B$ is precisely
the category $\ho(\calC_{/B})$.
Homology in this setting is then given by a functor 
\begin{equation}
\label{eq:hbullethpc}
	H_\bullet \colon \hp\calC \longto \ho(\calC)
\end{equation}
which for $\calC = \Mod^R$ sends a parametrized $R$--module $X$ over $B$
to the generalized Thom spectrum of $X$ in the sense of 
Ando, Blumberg and Gepner \cite[Def.~3.14]{ABGparam}. 
Thinking of a parametrized $\calC$--object $X$ over $B$ as analogous to a local coefficient 
system over $B$, we write $H_\bullet(B;X)$ for the value of $H_\bullet$ on $X$.
The usual $X$--homology groups of a space $B$ for an unparametrized $R$--module $X$ 
are recovered as $X_\ast(B) \isom \pi_\ast H_\bullet(B;\underline{X})$
where $\underline{X}$ is the trivial parametrized $R$--module over 
$B$ defined by $X$. See Example~\ref{ex:hbulletfortrivialcoeffs2}.
Moreover, for  $\calC = \Mod^{H\Z}$, 
Corollary~\ref{cor:homologywithlocalcoefficientscomp}
shows that the ordinary homology groups with local coefficients 
over $B$ can be recovered from the objects $H_\bullet(B;X)$.

The values of the functor $H_\bullet$ on morphisms of $\hpC$ give the ordinary 
induced maps of our theory: applied to a morphism $\phi \colon X \to Y$
in $\hpC$ covering a continuous map $f\colon A \to B$, the functor
$H_\bullet$ gives us a morphism
\[
	(f,\phi)_\bullet \colon H_\bullet(A;X) \longto H_\bullet(B;Y).
\]
The umkehr maps in our theory are induced by morphisms
of another category $\hpC^\dfop$ 
opfibred (but not fibred) over $\calT$.
This category has the same objects as $\hpC$, and is obtained from 
$\hpC$, roughly speaking, by keeping all opcartesian morphisms of 
$\hpC$ while replacing all fibres of the projection $\hpC \to \calT$
by their opposite categories. 
More precisely, the morphisms of $\hpC^\dfop$ covering a continuous
map $f\colon A \to B$ are given by equivalence classes of zigzags
\[
	X \xto{\ \alpha\ } X' \xot{\ \beta\ } Y
\]
where $\alpha$ is an opcartesian morphism of $\hpC$ covering $f$
and $\beta$ is a morphism of $\hpC$ covering the identity map of $B$.
See Definition~\ref{def:dfopdef}.
From $\hpC^\dfop$, we have a
contravariant functor (also called $H_\bullet$)
\[
	H_\bullet \colon \hpC^\dfop \longto \ho(\Spectra)
\]
agreeing with \eqref{eq:hbullethpc} on objects and taking a morphism 
$\theta\colon X \oto Y$ of $\hpC^\dfop$ covering a continuous map 
$f\colon A\to B$ to an umkehr map
\[
	(f,\theta)^\leftarrow \colon H_\bullet(B;Y) \longto H_\bullet(A;X).
\]
The small circle in the notation $\theta\colon X\oto Y$ indicates
that $\theta$ is a morphism of $\hpC^\dfop$ rather than $\hpC$.

While in principle any map of $\hpC^\dfop$ 
induces an umkehr map, it is to be expected 
that interesting umkehr maps will arise from 
maps of $\hpC^\dfop$ 
with special properties. 
We will pay particular attention to cartesian morphisms of $\hpC^\dfop$.
Notice that as the category $\hpC^\dfop$
is opfibred rather than fibred over $\calT$,
not all morphisms of $\calT$ have a plentiful supply 
of cartesian morphisms covering them. 
The following key existence result for cartesian morphisms
connects cartesian morphisms in $\hpC^\dfop$ to Costenoble--Waner duality,
discovered by Costenoble and Waner \cite{CWv1} and named and studied 
further by May and Sigurdsson \cite{MaySigurdsson}, and to Pontryagin--Thom collapse maps.
In the statement, we make use of a pairing $\oslash$ which for 
a cartesian morphism $\phi \colon X \to Y$ of $\hpC$ and any
morphism $\theta\colon Z \oto W$ of $\hpC^\dfop$ covering the 
same continuous map $f\colon A \to B$ yields a morphism 
\[
	\phi \oslash \theta 
	\colon
	X \tensor_{A} Z
	\longoto
	Y \tensor_{B} W
\]
of $\hpC^\dfop$ covering $f$. Here $\tensor_A$ refers to 
the symmetric monoidal product in $\ho(\calC_{/A})$ induced by 
$\tensor$, and similarly for $\tensor_B$.
The monoidal unit for $\tensor_B$ is denoted by $S_B$ or, if
clarity requires, by $S_B^\calC$.
The theorem will be proven 
in Section~\ref{subsubsec:proofsofresultsfromintro}.

\begin{thrm}
\label{thm:thetapc}
Let $\calC$ be a presentable symmetric monoidal $\infty$--category, and 
suppose $p\colon E \to B$ is a continuous map of spaces
which is small-fibred with respect to $\calC$
in the sense that all its homotopy fibres
are Costenoble--Waner dualizable with respect to $\calC$.
Then

\begin{enumerate}[(i)]
\item
\label{it:thetapc}
	There exists a cartesian morphism $\theta_p \colon \omega_p \oto S_B$ 
	of $\hp\calC^\dfop$ covering $p$.
\item
\label{it:phioslashthetap}
	Given a cartesian morphism morphism $\phi \colon Y \to X$ of $\hp\calC$
	covering $p$, the product
	\[
		\phi\oslash \theta_p 
		\colon 
		Y \tensor_{E} \omega_p 
		\longoto 
		X \tensor_{B} S_B 
		\homot 
		X
	\]
	is a cartesian morphism of $\hp\calC^\dfop$ covering $p$. In particular, 
	for any $X \in \ho(\calC_{/B})$, there exists a cartesian morphism of 
	$\hp\calC^\dfop$ covering $p$ and having codomain $X$.
\item
\label{it:thetapformanifoldbundles}
	Suppose $\calC = \Spectra$.
	If $p$ is a bundle of smooth closed manifolds, 
	then $p$ is small-fibred with respect to $\Spectra$ and as 
	$\theta_p$ we may choose a map
	\[
		\theta_p = \theta_p^{\Spectra}\colon S^{-\tau_{p}} \longoto S_B
	\]
	given by a fibrewise Pontryagin--Thom collapse map. Here $\tau_p$
	denotes the vertical tangent bundle of $p$.
\item
\label{it:thetapformanifoldbundlesc}
	Suppose $\calC$ admits a symmetric monoidal $\infty$--functor
	$L\colon \Spectra \to \calC$ which is a left adjoint.
	If $p$ is a bundle of smooth closed manifolds with vertical tangent bundle $\tau_p$,
	then $p$ is small-fibred with respect to $\calC$, and as $\theta_p$
	we may choose the map
	\[
		\theta_p 
		\colon 
		L_\fw (S^{-\tau_{p}})
		\xoto{\ L_\fw(\theta_p^{\Spectra})\ }
		L_\fw(S_B^\Spectra) 
		\homot
		S_B^\calC
	\]
	where 
	\[
		L_\fw \colon \hpSpectra^\dfop \longto \hp\calC^\dfop
	\]
	is the functor induced by $L$ and $\theta_p^\Spectra$ is the map of 
	part (\ref{it:thetapformanifoldbundles}).
\end{enumerate}
\end{thrm}
Here  $S^{-\tau_p}$ denotes the monoidal inverse of $S^{\tau_p}$ in $\ho(\Spectra_{/E})$
with respect to $\smashprod_{E}$ 
where  $S^{\tau_p}$ in turn denotes the fibrewise spectrum over $E$ 
defined by the fibrewise one-point-compactification of 
the vector bundle $\tau_p$ over $E$.
See Definition~\ref{def:sxi}.
The reader not yet acquainted with Costenoble--Waner duality may find some of 
the examples and remarks at the beginning of 
Section~\ref{subsec:cwdualityandhypercartesianmorphisms}
illuminating. For $\calC$ such as $\Mod^R$ and $\Spectra^\ell$
admitting a symmetric monoidal $\infty$--functor $\Spectra \to \calC$ 
which is a left adjoint, the theorem in particular applies 
to maps $p$ whose homotopy fibres are weakly equivalent to finite CW complexes.

In combination with Theorem~\ref{thm:thetapc}, the following theorem 
is the key result enabling the applications of our theory to string topology
of finite groups of Lie type in \cite{stringtop}. It is an immediate consequence of 
the slightly more informative Theorem~\ref{thm:ellcptgrpscwdualizable}.

\begin{thrm}
\label{thm:cwdualizibilityofbg}
Let $\ell$ be a prime and let $BG$ be a semisimple $\ell$--compact group. 
Then the space $G = \loops BG$
is Costenoble--Waner dualizable with respect to
the $\infty$--category $\Spectra^\ell$ of $\ell$--complete spectra.
\qed
\end{thrm}

Our next goal is to illustrate how Theorem~\ref{thm:thetapc} enables the 
construction of various twisted umkehr maps,
including those of equations~\eqref{eq:nmtntommtm} and \eqref{eq:btoemtp}
and therefore those of equations~\eqref{eq:pdumkehrmaps}
and \eqref{eq:integrationalgonfibreumkehrmaps}.
Let $\calC$ be a presentable symmetric monoidal $\infty$--category, and 
suppose 
\begin{equation}
\label{eq:fp1p2new}
\vcenter{\xymatrix@R+2ex{
	E_1
	\ar[rr]^{f}
	\ar[dr]_{p_1}
	&&
	E_2
	\ar[dl]^{p_2}
	\\
	&
	B
}}
\end{equation}
is a commutative triangle of spaces where $p_1$ and $p_2$ are small-fibred
with respect to $\calC$
in the sense of Theorem~\ref{thm:thetapc}.
Choosing cartesian morphisms
$\theta_{p_1} \colon \omega_{p_1} \oto S_B$
and 
$\theta_{p_2} \colon \omega_{p_2} \oto S_B$
as in Theorem~\ref{thm:thetapc}(\ref{it:thetapc}),
basic properties of cartesian morphisms yield a unique cartesian morphism
$\kappa \colon \omega_{p_1} \oto \omega_{p_2}$ making
\begin{equation}
\label{diag:kappadef} 
\vcenter{\xymatrix@R+2ex{
	\omega_{p_1}
	\ar@{-->}|{\circdec}[rr]^{\kappa\vphantom{f}}_\cart
	\ar[dr]|{\circdec}_{\theta_{p_1}}^\cart
	&&
	\omega_{p_2}
	\ar[dl]|{\circdec}^{\theta_{p_2}}_\cart
	\\
	&
	S_B
}}
\end{equation}
a commutative triangle in $\hpC^\dfop$ covering \eqref{eq:fp1p2new}.
Applying $H_\bullet$, we obtain an umkehr map
\begin{equation}
\label{eq:fkappaumk}
	(f,\kappa)^{\leftarrow}
	\colon 
	H_\bullet(E_2;\omega_{p_2}) 
	\longto
	H_\bullet(E_1;\omega_{p_1}).
\end{equation}
Moreover, given an object $X \in \ho(\calC_{/E_2})$, taking the $\oslash$--product of 
$\kappa$ and the canonical cartesian morphism $\phi\colon f^\ast X \to X$
covering $f$ gives a morphism 
\begin{equation}
\label{eq:kappax}
	\kappa_X = \phi \oslash \kappa
	\colon
	f^\ast X \tensor_{E_1} \omega_{p_1} 
	\longoto
	X \tensor_{E_2} \omega_{p_2}
\end{equation}
yielding the following $X$--twisted version of \eqref{eq:fkappaumk}:
\begin{equation}
\label{eq:fkappaxumk}
	(f,\kappa_X)^{\leftarrow} 
	\colon 
	H_\bullet(E_2; X \tensor_{E_2} \omega_{p_2}) 
	\longto
	H_\bullet(E_1; f^\ast X \tensor_{E_1} \omega_{p_1}).
\end{equation}
Finally, in case $\omega_{p_2}$ is invertible as an object of $\ho(\calC_{/E_2})$, substituting 
$X  \tensor_{E_2} \omega_{p_2}^{-1}$ for $X$ in \eqref{eq:kappax}
yields a morphism
\begin{equation}
\label{eq:tildekappax}
    \tilde{\kappa}_X
    \colon
    f^\ast X \tensor_{E_1} f^\ast \omega_{p_2}^{-1} \tensor_{E_1} \omega_{p_1}
    \longoto
    X
\end{equation}
and thereby an umkehr map
\begin{equation}
\label{eq:fkappatildexumk}
	(f,\tilde{\kappa}_X)^{\leftarrow} 
	\colon 
	H_\bullet(E_2; X) 
	\longto
	H_\bullet(
		E_1; f^\ast X \tensor_{E_1} f^\ast \omega_{p_2}^{-1} \tensor_{E_1} \omega_{p_1}
	)
\end{equation}
where the twistings by $\omega_{p_1}$ and $\omega_{p_2}$ only appear on the 
codomain.
The maps $p_2$ for which $\omega_{p_2}$ is invertible can be
characterized as those continuous maps whose homotopy fibres  $F$
are not only Costenoble--Waner dualizable but in fact ``$\calC$--Poincaré 
duality spaces'' in the sense that the Costenoble--Waner
dual of $F$ is invertible as an object of $\ho(\calC_{/F})$.
See Proposition~\ref{prop:omegapinvertibility}.

Specializing to the 
setting of Theorem~\ref{thm:thetapc}(\ref{it:thetapformanifoldbundles})
where $\calC = \Spectra$
and $p_1$ and $p_2$ are bundles of smooth closed manifolds,
the maps $\kappa$, $\kappa_X$ and $\tilde{\kappa}_X$ 
of diagram \eqref{diag:kappadef} and equations \eqref{eq:kappax} and \eqref{eq:tildekappax} take the form
\[
	\kappa \colon S^{-\tau_{p_1}} \longoto S^{-\tau_{p_2}},
\]
\[
	\kappa_X = \phi\oslash \kappa
	\colon
	f^\ast X \smashprod_{E_1} S^{-\tau_{p_1}}
	\longoto
	X \smashprod_{E_2} S^{-\tau_{p_2}}
\]
and
\[
	\tilde{\kappa}_X
	\colon 
	f^\ast X \smashprod_{E_1} S^{f^\ast \tau_{p_2}-\tau_{p_1}} \longoto X,
\]
yielding umkehr maps
\begin{equation}
\label{eq:fkappaumkmfld}
	E_2^{-\tau_{p_2}}
	=
	H_\bullet(E_2;S^{-\tau_{p_2}})
	\xto{\ (f,\kappa)^{\leftarrow}\ }
	H_\bullet(E_1;S^{-\tau_{p_1}})
	=
	E_1^{-\tau_{p_1}},
\end{equation}
\begin{equation}
\label{eq:fkappaxumkmfld}
	H_\bullet(E_2;X \smashprod_{E_2} S^{-\tau_{p_2}})
	\xto{\ (f,\kappa_X)^{\leftarrow}\ }
	H_\bullet(E_1;f^\ast X \smashprod_{E_1} S^{-\tau_{p_1}})
\end{equation}
and
\[
	H_\bullet(E_2; X)
	\xto{\ \ (f, \tilde{\kappa}_X)^\leftarrow\ }
	H_\bullet(E_1; f^\ast X \smashprod_{E_1} S^{f^\ast \tau_{p_2} - \tau_{p_1}}).
\]
We emphasize that for the construction of the umkehr map 
\eqref{eq:fkappaxumkmfld}, it is essential to work in the parametrized setting
and use the map $\kappa$; while one may (correctly) view 
\eqref{eq:fkappaxumkmfld} as an $X$--twisted variant of 
the unparametrized umkehr map \eqref{eq:fkappaumkmfld}, 
it is not easily constructible from \eqref{eq:fkappaumkmfld}.

If $\bbk$ is a field and the vector bundles $\tau_{p_1}$ and $\tau_{p_2}$
are $\bbk$--oriented, 
composing the maps induced by \eqref{eq:fkappaumkmfld} 
on $\bbk$--homology and cohomology 
with Thom isomorphisms yields umkehr maps
\begin{equation}
\label{eq:homologyumkfromkappa}
\xymatrix@!0@C=4.6em{
	H_{\ast+m_2}(E_2;\bbk) 
	\ar[rr]^-{\mathrm{Thom}^{-1}}_-\isom
	&&
	H_\ast (E_2^{-\tau_{p_2}};\bbk)
	\ar[rrr]^{H_\ast((f,\kappa)^\leftarrow;\,\bbk)}
	&&&
	H_\ast (E_1^{-\tau_{p_1}};\bbk)
	\ar[rr]^-{\mathrm{Thom}}_-\isom
	&&
	H_{\ast+m_1} (E_1;\bbk)
}
\end{equation}
and
\begin{equation}
\label{eq:cohomologyumkfromkappa}
\xymatrix@!0@C=4.6em{
	H^{\ast+m_1}(E_1;\bbk) 
	\ar[rr]^-{\mathrm{Thom}}_-\isom
	&&
	H^\ast (E_1^{-\tau_{p_1}};\bbk)
	\ar[rrr]^{H^\ast((f,\kappa)^\leftarrow;\,\bbk)}
	&&&
	H^\ast (E_2^{-\tau_{p_2}};\bbk)
	\ar[rr]^-{\mathrm{Thom}^{-1}}_-\isom
	&&
	H^{\ast+m_2} (E_2;\bbk)
}
\end{equation}
where $m_1$ and $m_2$ are the dimensions of $\tau_{p_1}$ and $\tau_{p_2}$,
respectively.
Part (\ref{it:pdumkcomparison}) of the following result
is a special case of the more general Theorem~\ref{thm:pdumkcomp}
proven in Section~\ref{subsec:pdcomp}
while part (\ref{it:iafumkcomparison})
is proven as Theorem~\ref{thm:integrationalongfibremfldbdlecomp}
in Section~\ref{subsec:iafcomp}.

\begin{thrm}
\label{thm:classicalumkehrmapscomparison}
Suppose $\bbk$ is a field and 
$p_1$ and $p_2$ in \eqref{eq:fp1p2new} are bundles of smooth 
closed manifolds of dimensions $m_1$ and $m_2$, respectively, 
such that the vertical tangent bundles $\tau_{p_1}$ and $\tau_{p_2}$
are $\bbk$--oriented.
Then the umkehr maps \eqref{eq:homologyumkfromkappa} and 
\eqref{eq:cohomologyumkfromkappa}
recover the classical umkehr maps arising from Poincaré duality 
and integration along fibre as follows:
\begin{enumerate}[(i)]
\item\label{it:pdumkcomparison}
	When $B = \pt$, the maps \eqref{eq:homologyumkfromkappa} and 
	\eqref{eq:cohomologyumkfromkappa}
	agree with the umkehr maps \eqref{eq:pdumkehrmaps}
	arising from Poincaré duality.
\item\label{it:iafumkcomparison}
	When $E_2 = B$ and $p_2 = \id_B$, 
	the maps \eqref{eq:homologyumkfromkappa} and 
	\eqref{eq:cohomologyumkfromkappa}
	agree with the umkehr maps \eqref{eq:integrationalgonfibreumkehrmaps}
	given by integration along fibre.
\end{enumerate}
\end{thrm}

In addition to Theorem~\ref{thm:classicalumkehrmapscomparison},
we will prove the following comparison result explaining how 
to recover the twisted umkehr maps of \cite{ABGparam}
from our theory. It will be proven in Section~\ref{subsec:abgcomparison}.
\begin{thrm}
\label{thm:abgumkehrcomparison}
Let $p\colon E \to B$ be a bundle of smooth closed manifolds,
let $R$ be a commutative ring spectrum,
and let $X \in \ho(\Mod^R_{/B})$ be a parametrized $R$--module over $B$.
Let 
\begin{equation}
\label{eq:abgumkehrcomparisontheta}
	\theta \colon p^\ast X \smashprod_E S^{-\tau_p} \longoto X
\end{equation}
be the cartesian morphism of $(\hpMod^R)^\dfop$ covering $p$ afforded
by Theorem~\ref{thm:thetapc}. Then the umkehr map 
\[
	(p,\theta)^\leftarrow
	\colon
	H_\bullet(B;X) \longto H_\bullet(E;p^\ast X \smashprod_E S^{-\tau_p})
\]
agrees with the twisted Pontryagin--Thom transfer map of 
\cite[Prop.~4.18]{ABGparam}. Consequently, the map
\[
	((p,\theta)^\leftarrow)^\ast
	\colon	
	R^\ast H_\bullet(E;p^\ast X \smashprod_E S^{-\tau_p}) \longto R^\ast H_\bullet(B;X)
\]
agrees with the twisted umkehr map of 
\cite[Prop.~4.18]{ABGparam}.
\end{thrm}

Recall that the map $\kappa$ of diagram \eqref{diag:kappadef}
is cartesian in $\hpC^\dfop$.
The following theorem shows that the map $\kappa_X$ 
of equation \eqref{eq:kappax}
frequently also is. 
We will prove
the result in Section~\ref{subsubsec:proofsofresultsfromintro}.
\begin{thrm}
\label{thm:kappaxcartcritc}
The map $\kappa_X$ of equation \eqref{eq:kappax} is cartesian 
if either of the following conditions holds:
\begin{enumerate}
\item \label{it:xdualizable}
	$X$ is dualizable as an object of the symmetric monoidal category 
	$(\ho(\calC_{/E_2}),\tensor_{E_2})$.
\item \label{it:xpulledbackfromb}
	$X$ is pulled back along $p_2$ from an object of $\ho(\calC_{/B})$,
	that is, there exists a cartesian morphism $X \to Y$ of $\hpC$
	covering $p_2$.
\end{enumerate}
\end{thrm}

In the case where $\omega_{p_2}$ is invertible,
we have the following complement to Theorem~\ref{thm:kappaxcartcritc}.
It too will be proven in
Section~\ref{subsubsec:proofsofresultsfromintro}.
\begin{thrm}
\label{thm:tildekappaxcart}
Suppose $\omega_{p_2}$ is invertible. Then for any $X \in \ho(\calC_{/E_2})$,
the map $\tilde{\kappa}_X$ of equation \eqref{eq:tildekappax}
is cartesian if and only if the map $\kappa_X$ of equation \eqref{eq:kappax} is.
\end{thrm}

The paper is structured as follows. In Section~\ref{sec:parametrized-homotopy-theory},
we recall basic material from parametrized homotopy theory and interpret it 
in terms of fibred category theory. The main goal of the section 
is to set up the categories $\hpC$ fibred and opfibred over the
category of topological spaces.  In Section~\ref{sec:hbullet},
we introduce the homological functor $H_\bullet \colon \hpC \to \ho(\calC)$.
The technical heart of the paper is Section~\ref{sec:umkehrmaps},
where we define the category $\hpC^\dfop$ and the functor
$H_\bullet \colon \hpC^\dfop \to \ho(\calC)$; study cartesian morphisms in 
$\hpC^\dfop$, singling out for special scrutiny two classes of 
cartesian morphisms we term supercartesian and hypercartesian; 
and relate the existence of hypercartesian morphisms to Costenoble--Waner duality.
In Section~\ref{sec:cwdualizabilityofsemisimplegrps}, we prove
the Costenoble--Waner dualizability of the based loop spaces of semisimple 
$\ell$--compact groups with respect to 
the $\infty$--category of $\ell$--complete spectra.
In Section~\ref{sec:serress}, we construct the Serre spectral sequence
for parametrized $R$--modules where $R$ is a commutative ring spectrum.
Finally, in Section~\ref{sec:comparisonofumkehrmaps}, we relate the umkehr maps 
constructed in this paper to various  umkehr maps appearing 
in the existing literature. The paper concludes with two appendices.
In Appendix~\ref{app:exbc}, we construct the framed bicategories
$\Ex_B(\calC)$ which in Section~\ref{sec:umkehrmaps} play a key role in 
relating the existence of hypercartesian morphisms to Costenoble--Waner duality,
while Appendix~\ref{app:dualizability} contains a short discussion of 
duality in bicategories.

\subsection*{Notations and conventions}
We denote by $\calT$ the category of compactly generated weak Hausdorff spaces.
Whenever working with topological spaces, we work within $\calT$.

\subsection*{Acknowledgements}
The author would like to thank 
his coauthor on \cite{stringtop}, Jesper Grodal, 
for his patience and helpful discussions; 
John Lind for fruitful conversations 
on the bicategory $\Ex(\Spectra)$;
and Dustin Clausen for sharing insights 
which helped resolve the question of whether 
$(S^1)\lcom \homot \loops (BS^1)\lcom$
is Costenoble--Waner dualizable in $\Spectra^\ell$.
See Remark~\ref{rk:s1lcomnotcwdualizable}.

\section{Parametrized homotopy theory and fibred category theory}
\label{sec:parametrized-homotopy-theory}

Our aim in this section is to recall basic parametrized homotopy theory
needed in the paper, and to reinterpret the material in the language of
fibred category theory. In particular, we will construct the 
categories $\hpC$ fibred and opfibred over the category of topological spaces
serving as the domains for the functors $H_\bullet \colon \hpC \to \ho(\calC)$.
Throughout the section, $\calC$ will denote 
a presentable symmetric monoidal $\infty$--category
with tensor product $\tensor$.
Examples of interest for us include:
\begin{enumerate}
\item the $\infty$--category $\Spaces$ of spaces equipped with 
	the cartesian product $\times$ of spaces;
\item the $\infty$--category $\PointedSpaces$ of pointed spaces
	equipped with the smash product $\smashprod$ of such; 	
\item the $\infty$--category $\Spectra$ of spectra equipped with 
	the smash product $\smashprod$ of spectra;
\item for $\ell$ a prime, the $\infty$--category $\Spectra^\ell$ of 
	$\ell$--complete spectra and the smash product $\smashprod^\ell$
	of such; and
\item for a commutative ring spectrum $R$, the $\infty$--category
	$\Mod^R$ of $R$--modules equipped with the smash product $\smashprod^R$ of such.
\end{enumerate}
For our discussion, we will rely on the $\infty$--categorical 
foundations for parametrized homotopy theory developed by
Ando, Blumberg, Gepner, Hopkins and Rezk \cite{ABGHR-infinity,ABGparam}.
For more information about parametrized homotopy theory, 
see also May and Sigurdsson's monograph \cite{MaySigurdsson}.
In this paper, by an $\infty$--category we will mean a quasicategory,
so that $\infty$--categories are particular kinds of simplicial sets.
We refer the reader to \cite{HTT}, \cite{kerodon} and \cite{HA}
for further information on this view on $\infty$--categories.

The section is structured as follows. In Section~\ref{subsec:paramhtpybasics2},
we recall basic parametrized homotopy theory phrased in terms of 
individual homotopy categories $\ho(\calC_{/B})$ of parametrized $\calC$--objects
over a space $B$ and base change functors between them.
In Section~\ref{subsec:fibredcats}, we continue by recalling fibred category 
theory, including cartesian and opcartesian morphisms,
the Grothendieck construction, internal and external
tensor products, and the fibrewise opposite of a fibred category.
In Section~\ref{subsec:hpcffw}, we construct the fibred categories
$\hpC$ and use these categories to interpret the parametrized homotopy of 
Section~\ref{subsec:paramhtpybasics2} in terms of 
fibred category theory.
In Section~\ref{subsec:pointsetmodelsforhpspaces}, we 
relate the abstractly defined category $\hpSpaces$ to more concrete
categories defined in terms of point-set level data
and show how parametrized spaces in the sense of
continuous maps give rise to objects in the categories $\hpC$.
Although for most of our purposes it suffices to work with the 
ordinary categories $\hpC$, that is not quite true for the construction 
of the Serre spectral sequences in Section~\ref{sec:serress}.
Therefore in Section~\ref{subsec:pc}, we refine the fibred and 
opfibred category $\hpC \to \calT$ into a cartesian and cocartesian fibration
of $\infty$--categories $\pC \to N\calT$. 
Finally, in Section~\ref{subsec:cosheafpc} we prove a result,
Theorem~\ref{thm:cosheafpc},
showing that objects of $\pC$ covering a space $B$
can be recovered as the $\infty$--categorical 
colimits of their restrictions to the members of
an open cover of $B$.
This result, through Theorem~\ref{thm:cosheaf}, 
plays a role in establishing the excision property for the 
functor~$H_\bullet$.

\subsection{Parametrized homotopy theory basics}
\label{subsec:paramhtpybasics2}

In this subsection, we will introduce the homotopy categories
$\ho(\calC_{/B})$
of parametrized $\calC$--objects over $B$,
the base change functors $f_!$, $f^\ast$ and $f_\ast$ connecting them,
and functors $T_B \colon \ho(\calC_{/B}) \to  \ho(\calD_{/B})$
induced by an $\infty$--functor $T \colon \calC\to \calD$.

\subsubsection{Parametrized \texorpdfstring{$\calC$}{C}--objects and 
the categories \texorpdfstring{$\ho(\calC_{/B})$}{Ho(C\_/B)}}
\label{subsubsec:hocoverb}

Given a space $B\in\calT$, a \emph{parametrized $\calC$--object over $B$} 
is an $\infty$--functor 
\[
	X\co (\Pi_\infty B)^\op \longto \calC
\]
where $\Pi_\infty B$ is the $\infty$--groupoid determined by $B$;
explicitly, $\Pi_\infty B = \Sing_\bullet B$,
the singular simplicial set of $B$.
The value  $X_b$  on the object of $(\Pi_\infty B)^\op$
given by a point $b\in B$ is called the \emph{fibre} 
of $X$ over $b$. We write 
\[
	\calC_{/B} = \Fun((\Pi_\infty B)^\op,\calC)
\]
for the $\infty$--category of such $\infty$--functors, 
and note that we can identify $\calC_{/\pt}$ with $\calC$.
Given an object $X\in \calC$, we sometimes write $\underline{X}$
for the  parametrized $\calC$--object over a space $B$ given by the constant
functor $\Pi_\infty(B)^\op \to \calC$ defined by $X$. 
Which space $B$ the object $\underline{X}$ lives over is
determined by context.

As for most of our purposes it is sufficient
to work on the level of homotopy categories, 
for simplicity we will
focus our discussion on the homotopy categories
$\ho(\calC_{/B})$, which are ordinary rather 
than $\infty$--categories, although most of the structure
discussed lifts to the level $\infty$--categories.
The symmetric monoidal structure on $\calC$ induces 
a closed symmetric monoidal structure on 
$\ho(\calC_{/B})$, and we write $\tensor_B$ for the 
tensor product, $S_B$ for the identity object, 
and $F_B(-,-)$ for the internal hom in $\ho(\calC_{/B})$,
where $\tensor$ is the symbol used for the symmetric monoidal
product in $\calC$.
Thus we have, in particular, tensor products
$\times_B$ on $\ho(\Spaces_{/B})$,
$\smashprod_B$ on $\ho(\Spectra_{/B})$,
$\smashprod^\ell_B$ on $\ho(\Spectra^\ell_{/B})$,
and
$\smashprod^R_B$ on $\ho(\Mod^R_{/B})$
when $R$ is a commutative ring spectrum.
In some cases, we will often employ a more specific notation:
\begin{enumerate}
\item When $\calC = \Spectra^\ell$, we write 
	$S_{B,\ell}$ for $S_B$.
\item When $\calC = \Mod^R$ for a commutative ring spectrum $R$,
	we write  $R_B$ for $S_B$ and $F^R_B(-,-)$ for $F_B(-,-)$.
\item When $\calC = \Spaces$, we will usually write $B$ or $(B,\id)$
	for the object $S_B$.
\end{enumerate}
We note that $S_B = \underline{S}$ for $S$ the identity object of $\calC$.

\subsubsection{The base change functors \texorpdfstring{$f_!$}{f\_!}, \texorpdfstring{$f^\ast$}{f\textasciicircum *}, and \texorpdfstring{$f_\ast$}{f\_*}}
\label{subsubsec:basechangefunctors}

A key feature of the parametrized theory are the various
base change functors relating the categories $\ho(\calC_{/B})$
for varying base spaces $B$. 
Given a map of spaces $f\co A\to B$,
precomposition with $(\Pi_\infty f)^\op$ induces a 
functor 
\[
	f^\ast \co \ho(\calC_{/A}) \longto \ho(\calC_{/B})
\]
which has both a left adjoint $f_!$ and a right adjoint $f_\ast$,
and which is an equivalence of categories 
if $f$ is a weak equivalence. In the case
where $\calC$ is symmetric monoidal, the functor $f^\ast$
is a closed symmetric monoidal functor, giving together with
its adjoints rise to 
a ``Wirthmüller context'' \cite{FauskHuMay}.
A number of natural equivalences then relate the base change
functors and the closed symmetric monoidal structures on 
$\ho(\calC_{/A})$ and $\ho(\calC_{/B})$: for 
$X\in \ho(\calC_{/A})$ and $Y,Z\in \ho(\calC_{/B})$, we have
\begin{align}
    \label{eq:bc-unit}
    f^\ast S_B &\homot S_A\\
    \label{eq:bc-product}
    f^\ast(Y\tensor_B Z) &\homot f^\ast Y\tensor_A f^\ast Z\\
    \label{eq:bc-staradj}
    F_B(Y,f_\ast X) &\homot f_\ast F_A(f^\ast Y,X) \\
    \label{eq:bc-homs}
    f^\ast F_B(Y,Z) &\homot F_A(f^\ast Y, f^\ast Z)\\
    \label{eq:bc-projformula}
    f_! (f^\ast Y\tensor_A X) &\homot Y \tensor_B f_! X\\    
    \label{eq:bc-shriekadj}
    F_B(f_! X, Y) &\homot f_\ast F_A(X,f^\ast Y).
\end{align}
where we have written $S_B$ for the 
unit object and $F_B(-,-)$ for the internal hom 
in $\ho(\calC_{/B})$.
See \cite[Prop.~6.8]{ABGparam}.
Equivalence \eqref{eq:bc-projformula}
is called the \emph{projection formula}. 
Equivalences 
\eqref{eq:bc-unit}
and
\eqref{eq:bc-product}
are part of the data that make $f^\ast$ into a symmetric monoidal functor,
and \eqref{eq:bc-staradj} is obtained from \eqref{eq:bc-product}
as explained in \cite[\S3]{FauskHuMay}. Equivalence~\eqref{eq:bc-homs}
is given by the adjoint of the composite
$f^\ast F_B(Y,Z)\tensor_B f^\ast Y 
\xto{\ \homot\ }
f^\ast (F_B(Y,Z)\tensor_A Y)
\to
f^\ast Z$,
where the first map is an instance  of \eqref{eq:bc-product}
and the second is induced by the evaluation map;
that this adjoint is an equivalence is part of the assertion that $f^\ast$
is a closed symmetric monoidal functor. 
\eqref{eq:bc-projformula} and \eqref{eq:bc-shriekadj}
follow from \eqref{eq:bc-homs} in the manner explained in \cite{FauskHuMay}.
In particular, one can work out that the projection formula
\eqref{eq:bc-projformula}
is given by the composite
\begin{equation}
\label{eq:projformulaexpanded}
	f_! (f^\ast Y \tensor_A X)
	\longto
	f_! (f^\ast Y \tensor_A f^\ast f_! X)
	\xto{\ \homot\ }
	f_! f^\ast (Y \tensor_B f_! X)
	\longto
	Y \tensor_B f_! X
\end{equation}
where the first and last maps are
induced by the unit and the counit of the  of the $(f_!,f^\ast)$
adjunction, respectively,
and where the middle map is an instance of 
\eqref{eq:bc-product}; cf. \cite[Prop.~4.5]{FauskHuMay}.

The base change functors also satisfy certain commutation relations.
Call a commutative square in $\calT$ \emph{homotopy cartesian}
if it is a homotopy fibre square in the sense of 
\cite[Def.~13.3.12]{Hirschhorn} with respect to 
the Quillen model structure on $\calT$.
Given a homotopy cartesian square
\newcommand{\hcsquarecontent}{\xymatrix{
	C
	\ar[r]^{g}
	\ar[d]_{h}
	&
	D
	\ar[d]^{k}
	\\
	A
	\ar[r]^{f}
	&
	B}}
\begin{equation}
\label{eq:hcsquare}
	\vcenter{\hcsquarecontent}
\end{equation}
in $\calT$, the composites
\begin{equation}
\label{eq:shriekastcomm}
	h_! g^\ast
	\xto{\ 	\eta\ }
	h_! g^\ast k^\ast k_!
	\xto{\ \homot\ }
	h_! h^\ast f^\ast k_!
	\xto{\ \varepsilon\  }	
	f^\ast k_!	 
\end{equation}
and
\begin{equation}
\label{eq:astastcomm}
	k^\ast f_\ast
	\xto{\ \eta\ }
	k^\ast f_\ast h_\ast h^\ast	
	\xto{\ \homot\ }	
	k^\ast k_\ast g_\ast h^\ast	
	\xto{\ \varepsilon\ }
	g_\ast h^\ast
\end{equation}
induced by units and counits of adjunctions and the 
equivalences 
$g^\ast k^\ast \homot  (kg)^\ast = (fh)^\ast \homot h^\ast f^\ast$
and $f_\ast h_\ast \homot (fh)_\ast = (kg)_\ast \homot  k_\ast g_\ast$
give  natural equivalences
\begin{align}
	\label{eq:commrelshriek}
	h_! g^\ast
	& \xto{\ \homot\ }
	f^\ast k_!	
\intertext{and}
	\label{eq:commrelstar}
	k^\ast f_\ast
	& \xto{\ \homot\ }
	g_\ast h^\ast
\end{align}
as follows from \cite[Cor.~5.12]{ABGparam}.
Equivalences between parametrized objects are detected on fibres:
a map $\phi \co X \to Y$ in $\ho(\calC_{/B})$ is an equivalence
if and only if the map $\phi_b \co X_b \to Y_b$
is for every $b\in B$, where we have written 
$\phi_b = i_b^\ast (\varphi)$, $X_b = i_b^\ast X$ and $Y_b = i_b^\ast Y$
for  $i_b \co \{b\} \incl B$ the inclusion.

\subsubsection{The functors \texorpdfstring{$T_B$}{T\_B}}
\label{subsubsec:tb}

Constructions for unparametrized objects generalize to parametrized
objects by performing them fibrewise. Formally, 
an $\infty$--functor $T\co \calC\to \calD$ induces, 
via postcomposition, functors
\[
	T_B \co \ho(\calC_{/B}) \longto \ho(\calD_{/B})
\]
commuting with the pullback functors $f^\ast$.
In particular, we have
\begin{equation}
\label{eq:tbonfibres}
	T_B(X)_b = T(X_b)
\end{equation}
for all $X \in \ho(\calC_{/B})$ and $b\in B$.
If $T$ is symmetric monoidal,
so are the functors $T_B$. Moreover,
an adjoint pair of $\infty$--functors
$T\co \calC \longrightleftarrows \calD \co V$
induces an adjoint pair of functors
\[
	T_B \co \ho(\calC_{/B}) \longrightleftarrows \ho(\calD_{/B}) \co V_B.
\]

\begin{example}
From the adjunction
\[
	\suspension_+^\infty \co \Spaces 
	\longrightleftarrows 
	\Spectra \co \loops^\infty
\]
we obtain an adjunction
\[
	\suspension_{+B}^\infty \co \ho(\Spaces_{/B})
	\longrightleftarrows 
	\ho(\Spectra_{/B}) \co \loops_B^\infty
\]
with $\suspension_{+B}^\infty$ symmetric monoidal.
\end{example}

\subsection{Fibred category theory}
\label{subsec:fibredcats}

It will be convenient to assemble the 
categories $\ho(\calC_{/B})$ for varying $B \in \calT$
into a single category $\hpC$ fibred over $\calT$.
In this subsection, we will review the fibred category theory
we need, referring the reader to
\cite[Exposé VI]{SGA1}, \cite[Chapter 8]{Borceux2}
and \cite{Shulman}
for more detailed discussions.

\subsubsection{Basic definitions}

\begin{defn}[Fibred categories]
\label{def:fibredcats}
Consider a functor $\Phi \colon \calE \to \calB$.
A morphism $\alpha\colon Y\to X$ is said to \emph{cover} a morphism
$f$ of $\calB$ if $\Phi(\alpha) = f$. 
A morphism covering an identity map is called \emph{vertical}.
The \emph{fibre} of $\Phi$ over an object $B \in \calB$
is the subcategory $\calE_B$ of $\calE$ 
consisting of the objects mapping to $B$ under $\Phi$ and
the morphisms of $\calE$ covering the identity map of $B$.

A morphism $\alpha\colon Y\to X$ in $\calE$ 
covering a morphism $f$ in $\calB$ is called \emph{cartesian}
if for all morphisms $\beta \colon Z \to X$ in $\calE$
and $g \colon \Phi(Z) \to \Phi(Y)$ in $\calB$
such that 
$\Phi(\beta) = f \circ g$, there exists a unique morphism
$\gamma \colon Z \to Y$ such that $\beta = \alpha\circ \gamma$ and 
$\Phi(\gamma) = g$:
\[\xymatrix@C-1em@R-1em{
	Z
	\ar@{|->}[dd]
	\ar@{-->}[dr]^{\exists! \gamma}
	\ar@/^1em/[drrr]^\beta
	\\
	&
	Y
	\ar@{|->}[dd]
	\ar[rr]^\alpha
	&&
	X
	\ar@{|->}[dd]
	\\
	\Phi(Z)
	\ar[dr]^{g}
	\ar@/^1em/[drrr]^{f\circ g }|!{[ur];[dr]}{\hole}
	\\
	&
	\Phi(Y)
	\ar[rr]^{f}
	&&
	\Phi(X)
}\]
The functor $\Phi$ is called a \emph{fibration over} $\calB$
and the category $\calE$ a \emph{category fibred over} $\calB$
if for every 
morphism $f \colon C \to B$ of $\calE$ and every 
object $X$ of $\calE$ such that $\Phi(X) = B$, there exists
a cartesian morphism $\alpha \colon Y \to X$ covering $f$.
A \emph{cleavage} on $\Phi$ is the choice for every morphism
$f \colon C\to B$ of $\calE$ and object $X$ in $\calE_B$
of a cartesian morphism $Y\to X$ covering $f$.

Dually, a morphism $\alpha$ in $\calE$ is called \emph{opcartesian}
if it is a cartesian morphism for  the functor
$\Phi \colon \calE^\op \to \calB^\op$,
and the functor $\Phi$ is called an \emph{opfibration} if 
$\Phi \colon \calE^\op \to \calB^\op$ is a fibration.
An \emph{opcleavage} on $\Phi\colon \calE \to \calB$
is a cleavage on $\Phi\colon \calE^\op \to \calB^\op$.
A functor $\Phi \colon \calE \to \calB$ which is 
both a fibration and an opfibration is called a \emph{bifibration}.
\end{defn}

We list useful basic facts about cartesian morphisms.

\begin{prop}[Properties of cartesian morphisms]
\label{prop:cartmorprops}
Suppose $\Phi \colon \calE \to \calB$ is a functor.
\begin{enumerate}[(i)]
\item\label{it:compiscart}
	The composite of cartesian morphisms of $\calE$ is cartesian.
\item\label{it:isosarecart}
	All isomorphisms in $\calE$ are cartesian.
\item\label{it:cartfactor}
	If $\phi$ and $\psi$ are composable morphisms of $\calE$
	such that $\phi$ and $\phi\circ \psi$ are cartesian, then
	$\psi$ is cartesian.
\item\label{it:cartoveriso}
	A cartesian morphism covering an isomorphism is an isomorphism.
\item\label{it:cartsourceiso}
	If $\phi_1\colon Y_1\to X$ and $\phi_2\colon Y_2 \to X$
	are two cartesian morphisms of $\calE$ covering the same 
	morphism of $\calB$, there exists a unique vertical
	isomorphism $\theta\colon Y_1 \to Y_2$
	such that $\phi_1 = \phi_2\circ \theta$.
	\qed
\end{enumerate}
\end{prop}

For a functor $\Phi \colon \calE \to \calB$, we write
$\calE^{\cart}$ (resp.\ $\calE^{\opcart}$) 
for the subcategory of $\calE$
consisting of all the objects of $\calE$ and 
all cartesian (resp.\ opcartesian) 
morphisms of $\calE$.

\begin{rem}[Fibrations and base change]
\label{rk:fibsandbasechange}
Fibrations provide a useful context for studying base change phenomena.
Suppose $\Phi \colon \calE \to \calB$ is a functor.
We think of a cartesian morphism $Y\to X$ in $\calE$ 
covering a morphism $f$ in $\calB$ as
exhibiting $Y$ as a base change of $X$ along $f$.
Notice that if $Y'\to X$ is another cartesian morphism in $\calE$
covering $f$, Proposition~\ref{prop:cartmorprops}(\ref{it:cartsourceiso})
ensures that $Y$ and $Y'$ are canonically isomorphic.
When $\Phi$ is a fibration, a choice of a cleavage on $\Phi$
yields for every morphism $f\colon A \to B$ in $\calB$
 a ``base change functor''
\begin{equation}
\label{eq:fibbasechange}
	f^\ast \colon \calE_B \longto \calE_A
\end{equation}
which on objects is defined by sending 
an object $X$ in $\calE_B$ to the source of the
chosen cartesian morphism covering $f$ with target $X$,
and which on morphisms is defined 
by the universal property of cartesian arrows.
The functor $f^\ast$ so constructed depends on the 
cleavage chosen, but any two choices of cleavage give functors which 
a canonically isomorphic. Also, we note that
while it is in general \emph{not} true that
$(gf)^\ast = f^\ast g^\ast$
for composable
maps $f$ and $g$, the two functors are nevertheless
always canonically isomorphic.

Dually, we think of an opcartesian morphism 
$X\to Y$ in $\calE$ covering $f$ as
exhibiting $Y$ as a cobase change of $X$ along $f$,
and, if $\Phi$ is an opfibration,
choosing  an opcleavage on $\Phi$ 
yields a ``cobase change functor''
\begin{equation}
\label{eq:fibcobasechange}
	f_! \colon \calE_A \longto \calE_B
\end{equation}
for every morphism $f\colon A\to B$ in $\calB$.
If $\Phi$ is a bifibration, the functor $f_!$ so constructed
is a left adjoint to $f^\ast$, as $\calE_B(f_! Y, X)$
and $\calE_C(Y,f^\ast X)$ are both in bijection
with the set of morphisms from $Y$ to $X$ covering $f$.
Conversely, it is 
easy to see that a fibration is a bifibration 
if all base change functors $f^\ast$ have a left adjoint.
\end{rem}

Later, in Section~\ref{subsec:hpcffw},
we will construct a bifibration $\hpC \to \calT$
whose fibre over $B \in \calT$ is $\ho(\calC_{/B})$
and for which the functors \eqref{eq:fibbasechange},
\eqref{eq:fibcobasechange} of the above discussion
recover the base change functors
of the same names discussed in 
Section~\ref{subsubsec:basechangefunctors}.

\subsubsection{The Grothendieck construction}

Let $\Cat$ be the $2$--category of categories.
A fibration $\Phi\colon \calE \to \calB$ 
turns out to be essentially the same
data as the pseudofunctor $\calB^\op \to \Cat$
sending each object $B$ to the fibre $\calE_B$
and each morphism $f\colon A \to B$
to the functor $f^\ast \colon \calE_B \to \calE_A$
obtained from a choice of cleavage on $\Phi$.
Our goal now is to explain the inverse 
construction producing a fibration $\calE \to \calB$ 
from a pseudofunctor $\calB^\op \to \Cat$.

\begin{defn}[{The $2$--category $\Fib_\calB$ of fibrations}]
\label{def:mortransfib}
Suppose $\Phi_1\colon \calE_1 \to \calB$ and
$\Phi_2\colon \calE_2 \to \calB$ 
are fibrations. A \emph{morphism of fibrations} over $\calB$ 
from $\Phi_1$ to $\Phi_2$
is a functor $F\colon \calE_1 \to \calE_2$
preserving cartesian morphisms and making the triangle
\[\xymatrix{
	\calE_1
	\ar[dr]_{\Phi_1}
	\ar[rr]^F
	&&
	\calE_2
	\ar[dl]^{\Phi_2}
	\\
	&
	\calB
}\]
commute. If $F,G \colon \Phi_1 \to \Phi_2$
are morphisms of fibrations over $\calB$,
a \emph{transformation of fibrations} over $\calB$
from $F$ to $G$ is 
natural transformation $\eta\colon F \to G$
such that $\Phi_2(\eta_X) = \id_{\Phi_1(X)}$ for all objects 
$X$ of $\calE_1$.
We write $\Fib_\calB$ for the  $2$--category
of fibrations over $\calB$ and morphisms and transformations of such.
\end{defn}

\begin{defn}[{The $2$--category $[\calB,\calS]$ of pseudofunctors}]
\label{def:pseudofunctoretc}
Let $\calB$ be an ordinary category, and let $\calS$ be a $2$--category,
such as the $2$--category $\Cat$ of categories, functors and natural
transformations, or the $2$--category $\smCat$ of symmetric 
monoidal categories, symmetric monoidal functors, and symmetric
monoidal transformations. 
Recall that a \emph{pseudo\-functor} $\calB^\op \to \calS$
is then a contravariant ``functor'' from $\calB$ to $\calS$
such that the identities $(g\circ f) ^\ast= f^\ast \circ g^\ast$
and $\id_B^\ast = \id_{F(B)}$ 
for induced maps are only required to hold up to 
coherent $2$--isomorphisms.
See e.g.\ \cite[\S 7.5]{Borceux1}.
Also recall that a \emph{pseudo natural transformation}
$\eta\co F \to G$ 
between pseudofunctors $F,G\co \calB^\op \to \calS$
is a ``natural transformation''
from $F$ to $G$ such that the identity 
$G(f)\circ \eta_C = \eta_B \circ F(f)$
for morphisms $f\co B \to C$ in $\calB$
is only required to hold up to a coherent $2$--isomorphism.
See e.g.\ \cite[\S 7.5]{Borceux1}. Finally,
recall that a \emph{modification} $\xi\co \eta\to \theta$
between pseudo natural transformations $\eta,\theta\co F\to G$
consists of a $2$--morphism $\xi_A \co \eta_A \to \theta_A$
for each object $A$ of $\calB$ satisfying a compatibility 
condition with the coherence isomorphisms of $\eta$ and $\theta$.
See e.g.\ \cite[Def.~A.8]{SchommerPriesThesis}.
We write $[\calB^\op,\calS]$ for the $2$--category of 
pseudofunctors, pseudo natural transformations and modifications
from $\calB^\op$ to $\calS$.
\end{defn}

\begin{defn}
\label{def:grothendieckconstr}
The \emph{Grothendieck construction} 
(see e.g.\ \cite[Thm.~8.3.1]{Borceux2})
is the $2$--functor 
\[
	\grothconstr\co [\calB^\op, \Cat] \longto  \Fib_\calB %
\]
constructed in the following way.
For a pseudofunctor $F\co \calB^\op \to \Cat$, 
the category $\grothconstr F$ over $\calB$ is constructed as follows:
\begin{itemize}
\item The objects of $\grothconstr F$
	are pairs $(B,X)$ where $B$ is an object of 
	$\calB$ and $X$ is an object of $F(B)$.
\item A morphism from $(B,X)$ to $(C,Y)$ in $\grothconstr F$ is a pair 
	$(f,\phi)$ where $f\co B \to C$ is a morphism in $\calB$ and
    $\phi\co X \to f^\ast(Y)$ is a morphism in $F(B)$.
\item The composite of two such morphisms
    $(f,\phi)\co (B,X) \to (C,Y)$ and
	$(g,\psi)\co (C,Y) \to (D,Z)$ is 
	the pair $(g\circ f,\psi * \phi)$ 
	where $\psi * \phi \co X \to F(g\circ f)(Z)$
    is the composite
    \[
    	X
    	\xto{\ \phi\ }
    	f^\ast(Y) 
    	\xto{\ f^\ast(\psi)\ }
    	f^\ast g^\ast(Z) 
    	\xto{\ \isom\ }
    	(g\circ f)^\ast(Z).
    \]
    Here the last arrow is the coherence isomorphism that is 
    part of the data of a pseudofunctor.
\item The functor $\grothconstr F \to \calB$ is projection onto first coordinate.
\end{itemize}
Notice that a morphism $(f,\phi)$ in $\grothconstr F$
is cartesian if and only if $\phi$ is an isomorphism.
In particular, the category $\grothconstr F$ is fibred over $\calB$,
as desired.

For a pseudo natural transformation
$\eta \co F \to G$ between pseudofunctors
$\calB^\op \to \Cat$,
the functor
\[
	\grothconstr \eta \co \grothconstr F \longto \grothconstr G
\]
is defined as follows:
\begin{itemize}
\item On objects, $(\grothconstr \eta)(B,X) = (B,\eta_B(X))$.
\item On morphisms, $\grothconstr \eta$ sends a morphism
	$(f,\phi)\co (B,X) \to (C,Y)$
	to the morphism $(f,\zeta)$ where $\zeta$ is the composite
    \[
    	\eta_B(X) 
    	\xto{\ \eta_B(\phi)\ }
    	\eta_B F(f)(Y)
    	\xto{\ \isom\ }
    	G(f) \eta_C (Y).
    \] 
	Here the last arrow is the coherence isomorphism that is part of 
	the data of a pseudo natural transformation.
\end{itemize}
Finally, for a modification $\xi\co \eta \to \theta$
between pseudo natural transformations $\eta,\theta\co F \to G$,
the natural transformation
\[
	\grothconstr \xi \co \grothconstr \eta \longto \grothconstr \theta
\]
is defined by letting $(\grothconstr \xi)_{(B,X)} = (\id_B,\xi_B')$
with $\xi_B'$ is the composite
\[
	\eta_B(X)
	\xto{\ 	\xi_B\ }
	\theta_B(X)
	\xto{\ \isom\ }
	\id_B^\ast \theta_B(X)
\]
where the last morphism is the coherence isomorphism that is 
part of the data of a pseudofunctor.
\end{defn}

The following result is well known; 
see e.g.\ \cite[Thm.~8.3.1]{Borceux2}.
\begin{thrm}
\label{thm:grothconstr}
The Grothendieck construction gives an equivalence of $2$--categories
\[
	\pushQED{\qed} 
	\grothconstr\co [\calB^\op, \Cat] \xto{\ \homot\ } \Fib_\calB. 
    \qedhere
    \popQED
\]
\end{thrm}

\begin{rem}
\label{rk:gcfibres}
Notice that the fibre of the functor $\grothconstr F \to \calB$ over 
an object $B$ of $\calB$ is canonically isomorphic to the category $F(B)$.
In this way, we may think of the category $\grothconstr F$ as 
obtained by assembling the categories $F(B)$ for varying $B$
together.
Moreover, notice that morphisms of the form $(f,\id)$ give
$\grothconstr F$ a cleavage such that the resulting base change functor
\[
	f^\ast \colon (\grothconstr F)_B \longto (\grothconstr F)_A
\]
agrees with the functor $F(f) \colon F(B) \to F(A)$
for every morphism $f\colon A \to B$ in $\calB$.
\end{rem}

\subsubsection{Symmetric monoidal fibrations and the external tensor product \texorpdfstring{$\exttensor$}{}}

The Grothendieck construction also lifts to 
the symmetric monoidal context, at least when 
the category $\calB$ is 
cartesian symmetric monoidal.

\begin{defn}[cf.~{\cite[\S12]{Shulman}}]
\label{def:smfibs}
Let $\calB$ be a  symmetric monoidal category.
A \emph{symmetric monoidal fibration} over $\calB$ is a 
fibration $\Phi\colon \calE \to \calB$ such that 
$\calE$ is a symmetric monoidal category, the tensor
product on $\calE$ preserves cartesian morphisms, and $\Phi$ is a 
strict symmetric monoidal functor. 
Given symmetric monoidal fibrations 
$\Phi_1 \colon \calE_1 \to \calB$
and
$\Phi_2 \colon \calE_2 \to \calB$,
a \emph{symmetric monoidal morphism of fibrations} 
$F\colon \Phi_1 \to \Phi_2$
is a morphism of the underlying fibrations
such that the functor $F\colon \calE_1 \to \calE_2$ 
is symmetric monoidal
and $\Phi_1 = \Phi_2 \circ F$ as symmetric monoidal functors.
Finally, a \emph{symmetric monoidal transformation of fibrations}
is a transformation of fibrations which is a symmetric monoidal
natural transformation.
We write $\smFib_\calB$ for the $2$--category of 
symmetric monoidal fibrations over $\calB$
and morphisms and transformations of such.
\end{defn}

\begin{thrm}[{\cite[Thm.~12.7]{Shulman}}]
\label{thm:smgrothendieckconstr}
Let $\calB$ be a cartesian monoidal category. 
Then the Grothendieck construction 
of Theorem~\ref{thm:grothconstr}
lifts to an equivalence of $2$--categories
\[
	\pushQED{\qed} 
	\grothconstr\colon [\calB^\op,\smCat] \xto{\ \homot\ } \smFib_\calB.
    \qedhere
    \popQED
\]
\end{thrm}
Explicitly, 
given a pseudofunctor $F\colon \calB^\op \to \smCat$
with $\calB$ cartesian monoidal,
the tensor product on $\grothconstr F$ is given by the formula
\begin{equation}
\label{eq:exttensordef}
	(A,X)\exttensor (C,Y) 
	= 
	(\pi^{AC}_A)^\ast X \tensor_{AC} (\pi^{AC}_C)^\ast Y
\end{equation}
where the maps $\pi$ are projections with the indicated 
sources and targets, and we have omitted all $\times$--signs.
Here $\tensor_{AC}$ refers to the 
tensor product in $F(A\times C)$.
On morphisms, $\exttensor$ is given as follows:
if
$\phi = (f,\bar{\phi}) \colon (A,X) \to (B,X')$
and
$\psi = (g,\bar{\psi}) \colon (C,Y) \to (D,Y')$
are morphisms
in $\grothconstr F$, the tensor product of $\phi$ and $\psi$ is given by
$\phi \exttensor \psi = (f\times g, \theta)$
where $\theta$ is the composite
\begin{equation}
\label{eq:exttensoronmors}
\vcenter{\xymatrix@!0@C=4.5em{
	X\exttensor Y 
	= 
	(\pi^{AC}_A)^\ast X \tensor_{AC} (\pi^{AC}_C)^\ast Y
	\\
	\ar[r]^{\bar{\phi} \tensor_{AC} \bar{\psi}}
	&
	*!L{\;
		(\pi^{AC}_A)^\ast f^\ast X' 
		\tensor_{AC} 
		(\pi^{AC}_C)^\ast g^\ast Y'
	}
	\\
	\ar[r]^{\homot}
	&
	*!L{\;
    	(f\times g)^\ast (\pi^{BD}_B)^\ast X' 
    	\tensor_{AC} 	
    	(f\times g)^\ast (\pi^{BD}_D)^\ast Y'
	}
	\\
	\ar[r]^{(f\times g)^\ast_\tensor}_\homot
	&
	*!L{\;
    	(f\times g)^\ast \big(
    		(\pi^{BD}_B)^\ast X'\tensor_{BD} (\pi^{BD}_D)^\ast Y'
    	\big)
    	=
    	(f\times g)^\ast (X'\exttensor Y').
	}
}}
\end{equation}
Here $(f\times g)^\ast_\tensor$ denotes the 
monoidality constraint of $(f\times g)^\ast$.
The identity object 
for $\exttensor$ is the pair $(\pt,I_\pt)$ where $\pt$
is the terminal object of $\calB$ and $I_\pt$ is the 
identity object of $F(\pt)$.
Notice that the symmetric monoidal structure on 
$F(B)$ can be recovered from that on $\grothconstr F$: for $X,Y\in F(B)$,
we have a natural isomorphism
\begin{equation}
\label{eq:internaltensorformula}
X\tensor_B Y \isom \Delta^\ast (X\exttensor Y)
\end{equation}
where $\Delta \co B \to B \times B$ is the diagonal map.
We refer to the tensor product $\exttensor$ on $\grothconstr F$
as the \emph{external tensor product} on $\grothconstr F$.

\subsubsection{The internal tensor product \texorpdfstring{$\tensor_\internal$}{}}

Given a pseudofunctor $F \colon \calB^\op \to \smCat$,
in addition to the external tensor product 
$\exttensor \colon \grothconstr F  \times \grothconstr F  \to  \grothconstr F$,
we may also assemble the symmetric monoidal structures on 
the categories $F(B)$ into a structure on $\grothconstr F$
that is internal to $\Fib_\calB$. Recall that a 
\emph{symmetric pseudomonoid}
in a $2$--category $\mathbf{C}$ with finite products 
is the notion obtained by substituting objects, $1$--morphisms,
and $2$--morphisms in $\mathbf{C}$ for categories,
functors, and natural transformations in the definition
of a symmetric monoidal category.
Thus a symmetric pseudomonoid in 
$\Cat$ is simply a symmetric monoidal category.
With $F$ as above, we define functors
\[
	\tensor_\internal
	\colon
	(\grothconstr F) \times_\calB (\grothconstr F)
	\longto
	\grothconstr F
	\qquad\text{and}\qquad
	I_\internal
	\colon 
	\calB \longto \grothconstr F	
\]
as follows. On objects, we set
\[
	(B,X) \tensor_\internal (B,Y) = (B, X\tensor_B Y)
	\qquad\text{and}\qquad
	I_\internal (B) = I_B,
\]
where $I_B$ denotes the identity object of $F(B)$.
On morphisms, we define
\[
	(f,\bar{\phi}) \tensor_\internal (f,\bar{\psi})
	=
	(
		f, 
		X\tensor_A Y 
		\xto{\ \bar{\phi}\tensor_A \bar{\psi}\ } 
		f^\ast X' \tensor_A f^\ast Y'
		\xto[\isom]{\ f^\ast_\tensor\ }
		f^\ast (X' \tensor_B Y')
	)
\]
for morphisms
$(f,\bar{\phi}) \colon (A,X) \to (B,X')$
and 
$(f,\bar{\psi}) \colon (A,Y) \to (B,Y')$ 
in $\grothconstr F$
and 
\[
	I_\internal (f) = f^{\ast}_I \colon I_A \xto{\ \isom\ } f^\ast I_B
\]
for a morphism $f\colon A \to B$ in $\calB$.
Here $f^\ast_\tensor$ and $f^\ast_I$ denote the monoidality
and unitality constraints for 
the symmetric monoidal functor $f^\ast \colon F(B) \to F(A)$.
The functors $\tensor_\internal$ and $I_\internal$,
together with the transformations constructed in the evident way
from the associativity and unit constraints of 
the symmetric monoidal categories $F(B)$, now make
$\grothconstr F$ into a symmetric pseudomonoid 
in the $2$--category $\Fib_B$.
We call $\tensor_\internal$ the \emph{internal tensor product} on $\grothconstr F$.

\subsubsection{The fibrewise opposite of a fibration}
\label{subsubsec:fop}

Given a pseudofunctor $F\colon \calB^\op \to \Cat$,
the \emph{fibrewise opposite} of the fibration 
$\grothconstr F \to \calB$ can be described quickly as the fibration
$\grothconstr F' \to \calB$ where $F'$ is the composite
\[
	F' \colon \calB^\op \xto{\ F\ } \Cat \xto{\ \op\ } \Cat
\]
of $F$ and the functor sending a category to its opposite.
Our aim now is to give a description of the fibrewise opposite 
which does not involve the Grothendieck construction.

\begin{defn}[Fibrewise opposite of a fibration]
\label{def:fop}
Suppose $\Phi \colon \calE \to \calB$ is a fibration.
Then the \emph{fibrewise opposite} 
$\Phi^\fop \colon \calE^\fop \to \calB$ of $\Phi$ is
constructed as follows. The objects of $\calE^\fop$ 
are the objects of $\calE$, while a morphism 
$X\to Y$ in $\calE^\fop$ is an equivalence class of zigzags
\[
	X \xot{\ \alpha\ } Y' \xto{\ \beta\ } Y
\]
of morphisms of $\calE$ 
where $\alpha$ is vertical and $\beta$ is cartesian.
Two such zigzags
\[
	X \xot{\ \alpha\ } Y' \xto{\ \beta\ } Y
	\qquad\text{and}\qquad
	X \xot{\ \alpha'\ } Y'' \xto{\ \beta'\ } Y
\]
are equivalent if there exists a vertical isomorphism
$\theta\colon Y' \to Y''$ such that 
$\alpha = \alpha' \circ \theta$ and $\beta = \beta' \circ \theta$.
The composite of the morphisms 
$[X \xot{\ \alpha\ } Y' \xto{\ \beta\ } Y]$
and 
$[Y \xot{\ \gamma\ } Z' \xto{\ \delta\ } Z]$
is the morphism represented by the 
composites along the two sides of the 
commutative diagram
\[\xymatrix@-1em{
    &&
    Z''
    \ar[dl]_{\gamma'}
    \ar[dr]^{\beta'}
    \\
    &
    Y'
    \ar[dl]_{\alpha}
    \ar[dr]^{\beta}
    &&
    Z'
    \ar[dl]_\gamma
    \ar[dr]^\delta
    \\
    X
    &&
    Y
    &&	
    Z
}\]
where $\beta'$ is a cartesian morphism covering
$\Phi(\beta)$ and $\gamma'$ is obtained by 
using the universal property of the cartesian morphism $\beta$.
We define $\Phi^\fop \colon \calE^\fop \to \calB$ by setting
$\Phi^\fop(X) = \Phi(X)$ on objects and 
\[
	\Phi^\fop [X \longot Y' \xto{\ \beta\ } Y]
	=
	\Phi(\beta) 
\]
on morphisms.
\end{defn}

It is easy to show that a morphism 
$[X \xot{\ \alpha\ } Y' \xto{\ \beta\ } Y]$
in $\calE^\fop$ is
cartesian precisely when  $\alpha$ is an isomorphism.
Thus $\Phi^\fop$ is a fibration.
Notice that for an object $B$ of $\calB$,
the functor sending a morphism $\alpha \colon Y \to X$
to $[X\xot{\alpha} Y \xto{\id} Y]$
provides an isomorphism 
$(\calE_B)^\op \xto{\isom} (\calE^\fop)_B$ 
between fibres over $B$
while the functor sending
a morphism $\beta\colon X\to Y$ to 
$[X\xot{\id} X \xto{\beta} Y]$
provides an isomorphism 
$\calE^\cart \xto{\ \isom\ } (\calE^\fop)^\cart$
between subcategories of cartesian morphisms.

When the fibration $\Phi\colon \calE \to \calB$ is symmetric monoidal,
so is $\Phi^{\fop}\colon \calE^{\fop} \to \calB$: the identity
object in $\calE^\fop$ is given by the identity object in $\calE^\fop$,
and the tensor product on $\calE^{\fop}$ agrees with that on $\calE$
on objects and is given on morphism by the formula
\[
	[X_1 \xot{\ \alpha_1\ } Y'_1 \xto{\ \beta_1\ } Y_1]
	\tensor
	[X_2 \xot{\ \alpha_2\ } Y'_2 \xto{\ \beta_2\ } Y_2]
	=
	[
		X_1 \tensor X_2 
		\xot{\ \alpha_1 \tensor \alpha_2\ } 
		Y'_1 \tensor Y'_2 
		\xto{\ \beta_1 \tensor \beta_2\ } 
		Y_1 \tensor Y_2
	].
\]
As in the non-monoidal case, for a pseudofunctor 
$F\colon \calB^\op \to \smCat$, the symmetric monoidal
fibration $(\grothconstr F)^\fop \to \calB$ is isomorphic to 
$\grothconstr F' \to \calB$ where $F'$ is the composite
\[
	F' \colon \calB^\op \xto{\ F\ } \smCat \xto{\ \op\ } \smCat
\]
of $F$ and the functor sending a symmetric monoidal category to its opposite.

It is straightforward to extend the definition of $(-)^{\fop}$
to morphisms and transformations of fibrations and symmetric monoidal 
fibrations, so that altogether we have a $2$--functors 
\[
	(-)^\fop \colon \Fib_\calB \longto \Fib_\calB
	\qquad\text{and}\qquad
	(-)^\fop \colon \smFib_\calB \longto \smFib_\calB
\]
reversing the direction of $2$--morphisms.
See \cite[\S 5.3]{HL15} for some more detail 
in the symmetric monoidal case.

\subsection{Parametrized homotopy theory in the language of fibred category theory}
\label{subsec:hpcffw}

In this subsection, we will reinterpret the parametrized homotopy theory 
discussed in Section~\ref{subsec:paramhtpybasics2}
in terms of fibred category theory.
Using the Grothendieck construction,
we will assemble the  
categories $\ho(\calC_{/B})$ 
and functors
$F_B\colon  \ho(\calC_{/B}) \to \ho(\calD_{/B})$
for varying 
$B\in \calT$ into a fibration $\hpC \to \calT$
and a morphism
$F_\fw \colon \hpC \to \hpD$
of fibrations over $\calT$.
In addition, we will present a number of results concerning
preservation of cartesian and opcartesian morphisms
in the fibrations $\hpC\to \calT$,
in particular reinterpreting
the commutation relations
\eqref{eq:commrelshriek}
and
\eqref{eq:commrelstar}
and the relations~\eqref{eq:bc-unit} 
through \eqref{eq:bc-shriekadj}
for a Wirthmüller context
in terms of cartesian and opcartesian morphisms.

\begin{defn}[The categories $\hpC$]
\label{def:hpconstr}
We let $\hpC \to \calT$ be the symmetric monoidal fibration
obtained by applying the 
Grothendieck construction of Theorem~\ref{thm:smgrothendieckconstr}
to the functor $\calT^\op \to \smCat$
given by $B \mapsto \ho(\calC_{/B})$, $f \mapsto f^\ast$.
We think of $\hpC$ as the category of parametrized 
$\calC$--objects over varying base spaces,
and will often write just $X$ for an object $(B,X)$
of $\hpC$, leaving the base space $B$ as implicitly 
understood. 
The letters h and p stand for `homotopy' and `parametrized,'
respectively.
\end{defn}

Notice that given $f\colon A\to B$,
Remark~\ref{rk:gcfibres}
ensures that 
under the identifications $\hpC_A = \ho(\calC_{/A})$ and
$\hpC_B = \ho(\calC_{/B})$,
the base change functor $f^\ast$
of Remark~\ref{rk:fibsandbasechange}
arising from the fact that $\hpC \to \calT$ is a fibration
agrees with the base change functor $f^\ast$
of Section~\ref{subsubsec:basechangefunctors}.
As the latter functor $f^\ast$ admits a left adjoint,
Remark~\ref{rk:fibsandbasechange}
implies that $\hpC\to \calT$ is a bifibration.
Moreover, the uniqueness of left adjoints ensures that
the functors $f_!$ of
Remark~\ref{rk:fibsandbasechange}
and 
Section~\ref{subsubsec:basechangefunctors}
agree up to canonical natural equivalence.

\begin{defn}[The functors $F_\fw$]
\label{def:fwfun}
For a (not necessarily symmetric monoidal)
$\infty$--functor $F\co \calC \to \calD$
between presentable symmetric monoidal $\infty$--categories, we write 
\[
	F_\fw \co \hpC \longto \hpD
\]
for the morphism of fibrations over $\calT$
 obtained by the Grothendieck construction 
from the pseudo natural transformation with components 
$F_B \co \ho(\calC_{/B}) \to \ho(\calD_{/B})$.
(Thus on fibres over $B$, 
the functor $F_\fw$ restricts to the functor $F_B$.)
The subscript $\fw$ stands for `fibrewise': intuitively,
$F_\fw$ is simply the functor obtained by applying 
$F$ to a parametrized object fibrewise.
\end{defn}

Notice that an adjoint pair
$T\co \calC \longrightleftarrows \calD \co V$
of $\infty$--functors 
gives rise to an adjoint pair 
\[
	T_\fw \co \hpC \longrightleftarrows \hpD \co V_\fw.
\]
of morphisms in the $2$--category $\Fib_\calT$ of fibrations over $\calT$,
and that Theorem~\ref{thm:smgrothendieckconstr} ensures that 
$F_\fw\co \hpC \to \hpD$
is symmetric monoidal 
if $F\co \calC \to \calD$ is.
If $\tensor$ denotes the 
symmetric monoidal product in $\calC$, 
we write $\exttensor \colon \hpC \times \hpC \to \hpC$
for the symmetric monoidal product in $\hpC$
and 
$\tensor_\internal \colon \hpC \times_\calT \hpC \to \hpC$ 
for the symmetric pseudomonoid multiplication 
on $\hpC$ in $\Fib_\calT$.
In particular, we obtain products
$\extsmashprod$ and $\smashprod_\internal$ on $\hpSpectra$,
$\extsmashprod^\ell$ and $\smashprod^\ell_\internal$ on $\hpSpectra$,
and
$\extsmashprod^R$ and $\smashprod^R_\internal$ on $\hpMod^R$
when $R$ is a commutative ring spectrum.

By the definition of a morphism of fibrations over $\calT$, 
the functor $F_\fw$ 
of Definition~\ref{def:fwfun}
preserves cartesian morphism.
It is formal to verify that the left adjoint in an adjunction 
in the $2$--category $\Fib_\calB$ 
preserves opcartesian morphisms,
so we also have the following result.

\begin{prop}
\label{prop:opcartpreservation}
Suppose $\calC$ and $\calD$
are presentable symmetric monoidal $\infty$--categories,
and suppose $F\co \calC \to \calD$
is a (not necessarily symmetric monoidal) $\infty$--functor
admitting a right adjoint.
Then the functor $F_\fw \colon \hpC\to \hpD$
preserves opcartesian morphism. \qed
\end{prop}

\begin{rem}
\label{rk:opcartpreservation2}
Phrased in terms of base change functors, 
Proposition~\ref{prop:opcartpreservation}
implies that when $F$ admits a right adjoint,
there exists a canonical natural equivalence
\begin{equation}
\label{eq:fshriektcomm}
	f_! F_A \homot F_B f_!.
\end{equation}
for every continuous map $f\colon A \to B$.
\end{rem}

The following proposition is part of Lemma~\ref{lm:hpbcprops} below,
but we record it here separately for clarity and emphasis. 
The proposition implies that $\hpC \to \calT$
is a symmetric monoidal bifibration in the sense of 
\cite[Def.~12.1]{Shulman}.
We remind the reader that $\exttensor$ preserves
cartesian morphisms by the definition
of a symmetric monoidal fibration.
\begin{prop}
\label{prop:exttensoropcarts}
Suppose $\phi$ and $\psi$ are opcartesian morphisms in $\hpC$.
Then their external tensor product $\phi \exttensor \psi$ 
is opcartesian.
\qed 
\end{prop} 
 
\begin{rem}
Phrased in terms of base change functors, 
Proposition~\ref{prop:exttensoropcarts}
implies that given continuous maps 
$f\colon A\to C$ and $g\colon B\to D$, there
exists a canonical equivalence
\begin{equation}
	(f\times g)_! (X \exttensor  Y) \homot f_! X \exttensor g_! Y
\end{equation}
for all $X\in \ho(\calC_{/A})$ and $Y \in \ho(\calC_{/B})$.
\end{rem}
 
We proceed to reformulate the relations~\eqref{eq:bc-unit}
through \eqref{eq:bc-shriekadj}
and the commutation relations
\eqref{eq:commrelshriek}
and
\eqref{eq:commrelstar}
in the language of fibred category theory.
For this end, notice that
in addition to the base change functors $f_!$ and $f^\ast$,
the base change functor $f_\ast$ can also be understood
in terms of a fibration:
it is straightforward to verify that given a map
$f\colon A \to B$ in $\calT$ and an object $X \in \ho(\calC_{/A})$,
the zigzag 
\[\xymatrix@1{
	X 
	&
	\ar[l]_-{\varepsilon}
	f^\ast f_\ast X
	\ar[r]^-{\cart}
	&
	f_\ast X
}\]
given by the counit of the $(f^\ast,f_\ast)$ adjunction and 
the canonical cartesian morphism 
$f^\ast f_\ast X \to f_\ast X$
defines an opcartesian morphism in $\hpC^\fop$ covering $f$,
so $\hpC^\fop \to \calT$ is a bifibration whose 
opcartesian morphisms encode the functor $f_\ast$.
We now have the following reformulations of the natural
equivalences \eqref{eq:bc-unit} through \eqref{eq:bc-shriekadj}.
The reformulation of the projection formula \eqref{eq:bc-projformula}
as number (\ref{it:bc-projformula-ref}) below 
will be especially useful for us.

\begin{prop}
\label{prop:reformulations}
Let 
\[
	F_\internal 
	\colon 
	\hpC^\fop 
	\times_\calT
	\hpC
	\longto
	\hpC
\]
be the functor sending a pair of objects $(X,Y)$ over $B\in\calT$ 
to $F_B(X,Y)$, and let
\[
	F_\internal^\fop 
	\colon 
	\hpC
	\times_\calT
	\hpC^\fop 
	\longto
	\hpC^\fop 
\]
be the fibrewise opposite of $F_\internal$.
Then the functors
$F_\internal$, $F_\internal^\fop$,
the internal tensor product
\[
	\tensor_\internal 
	\colon 
	\hpC\times_\calT \hpC 
	\longto 
	\hpC
\]
and the unit
\[
	I_\internal\colon \calT \longto \hpC
\]
have the following properties:
\begin{enumerate}[(i)]
\item \label{it:bc-unit-ref}
	$I_\internal(f)$ is cartesian for all morphisms $f$.
\item \label{it:bc-product-ref}
	If $\phi$ and $\psi$ are cartesian, 
	$\phi \tensor_\internal \psi$ is cartesian.
\item \label{it:bc-projformula-ref}
	If $\phi$ is cartesian and $\psi$ is opcartesian, 
	$\phi \tensor_\internal \psi$ is opcartesian.
\item \label{it:bc-homs-ref1}
	If $\phi$ and $\psi$ are cartesian,
	$F_\internal(\phi,\psi)$ is cartesian.
\item \label{it:bc-homs-ref2}
	If $\phi$ and $\psi$ are cartesian,
	$F_\internal^\fop(\phi,\psi)$ is cartesian.
\item \label{it:bc-staradj-ref}
	If $\phi$ is cartesian and $\psi$ is opcartesian,
	$F_\internal^\fop(\phi,\psi)$ is opcartesian.
\item \label{it:bc-shriekadj-ref}
	If $\phi$ is opcartesian and $\psi$ is cartesian,
	$F_\internal^\fop(\phi,\psi)$ is opcartesian.
\end{enumerate}
\end{prop}
\begin{proof}
(\ref{it:bc-unit-ref}) and (\ref{it:bc-product-ref})
are reformulations of 
\eqref{eq:bc-unit} and \eqref{eq:bc-product} and
are encoded in the assertion that 
$I_\internal$ and $\tensor_\internal$
are morphisms in $\Fib_\calT$. See 
Definition~\ref{def:mortransfib}. 
On the other hand,
(\ref{it:bc-projformula-ref}),
(\ref{it:bc-staradj-ref}),
and
(\ref{it:bc-shriekadj-ref})
are reformulations of 
\eqref{eq:bc-projformula},
\eqref{eq:bc-staradj},
and 
\eqref{eq:bc-shriekadj},
respectively, while
(\ref{it:bc-homs-ref1}) and
(\ref{it:bc-homs-ref2})
are both reformulations of 
\eqref{eq:bc-homs}.
\end{proof}

The following proposition provides a ``non-algebraic'' way to phrase the 
commutation relation~\eqref{eq:commrelshriek}. 
See e.g.\ \cite[Lemma~16.1]{Shulman} for the proof.
The same statement with $\hpC$ replaced by $\hpC^\fop$
provides a reformulation of the commutation relation
\eqref{eq:commrelstar}.

\begin{prop}
\label{prop:commrelshriekinterpretation2}
Let
	\[\xymatrix{
		Z
		\ar[r]^{\beta}
		\ar[d]_{\mu}
		& 
		W
		\ar[d]^{\nu}
		\\
		X
		\ar[r]^{\alpha}
		& 
		Y	
	}\]
be a commutative square in $\hpC$ covering a 
homotopy cartesian square in $\calT$,
and assume that $\beta$ is cartesian and $\nu$ is opcartesian.
Then $\alpha$ is cartesian 
if and only if $\mu$ is opcartesian.
\qed
\end{prop}

We conclude the subsection with results
concerning the behaviour of morphisms
covering weak equivalences.
The following result is a reformulation of the fact 
that the adjunction $(f_!,f^\ast)$ is an adjoint 
equivalence of categories when $f$ is a weak equivalence. 
\begin{prop}
\label{prop:wecartiffopcart}
Suppose $f$ is a weak equivalence in $\calT$. 
Then a morphism of $\hpC$
covering $f$ is cartesian 
if and only if it is opcartesian. \qed
\end{prop}

Finally, we have the following propositions.

\begin{prop}
\label{prop:wecancel}
Suppose $\phi \colon X\to Y$ and $\psi \colon Y\to Z$ 
are morphisms
of $\hpC$ such that $\phi$ is an opcartesian 
morphism covering a weak equivalence. Then
$\psi$ is cartesian if and only if 
$\psi\circ \phi$ is.
\end{prop}
\begin{proof}
If $\psi$ is cartesian, then so is 
$\psi\circ \phi$ by Proposition~\ref{prop:wecartiffopcart}
and Proposition~\ref{prop:cartmorprops}(\ref{it:compiscart}).
Suppose conversely that $\psi\circ \phi$ is cartesian.
Choose a cartesian morphism $\psi'\colon Y'\to Z$ covering the 
same morphism as $\psi$ does. Using the universal property
of $\psi'$, we may factor $\psi\circ\phi$ as a composite
$\psi\circ\phi = \psi'\circ\phi'$ where $\phi'\colon X\to Y'$ covers
the same morphism as $\phi$. By 
Proposition~\ref{prop:cartmorprops}(\ref{it:cartfactor}),
the morphism $\phi'$ is cartesian, so by 
Proposition~\ref{prop:wecartiffopcart} $\phi'$ is also 
opcartesian. By the analogue of 
Proposition~\ref{prop:cartmorprops}(\ref{it:cartsourceiso})
for opcartesian morphism, there exists a unique
vertical isomorphism
$\theta\colon Y \to Y'$ such that $\phi' = \theta \circ \phi$.
From the uniqueness part of the universal property of
the opcartesian morphism $\phi$ it now follows that 
$\psi =  \psi' \circ \theta$. 
Thus $\psi$ is cartesian by parts
(\ref{it:isosarecart})
and (\ref{it:compiscart}) of 
Proposition~\ref{prop:cartmorprops}.
\end{proof}

\begin{prop}
\label{prop:wecancel2}
Suppose $\phi \colon X\to Y$ and $\psi \colon Y\to Z$ 
are morphisms
of $\hpC$ such that $\psi$ is an opcartesian 
morphism covering a weak equivalence. Then
$\phi$ is cartesian if and only if 
$\psi\circ \phi$ is.
\end{prop}

\begin{proof}
The claim is immediate from 
Proposition~\ref{prop:wecartiffopcart}
and parts 
(\ref{it:compiscart})
and
(\ref{it:cartfactor})
of Proposition~\ref{prop:cartmorprops}.
\end{proof}

\subsection{Point set models for \texorpdfstring{$\hpSpaces$}{hpSpaces}}
\label{subsec:pointsetmodelsforhpspaces}

In this subsection, we will relate the 
rather abstractly defined category $\hpSpaces$
to much more concrete ones built out of point set level data. 
Moreover, in Proposition~\ref{prop:tcalc2} we show 
how a parametrized space over $B$ in the sense of a 
continuous map $\alpha \colon X \to B$ 
defines, for any $\calC$, an object of $\hpC$
over $B$. 
Recall that $\calT$ denotes the category of topological
spaces (assumed to be compactly generated and weak Hausdorff)
and continuous maps.

\begin{defn}
We write $\calT/B$ for the usual category of spaces
over a space $B$, so that an object in $\calT/B$
is a continuous map $X \to B$,
and a map in $\calT/B$ from $X\to B$ to $Y\to B$
is a continuous map $\phi\co X \to Y$ making the triangle
\begin{equation}
\label{diag:caltoverbmor}
\vcenter{\xymatrix{
	X 
	\ar[rr]^\phi
	\ar[dr]%
	&&
	Y
	\ar[dl]%
	\\
	&
	B
}}
\end{equation}
commutative. 
\end{defn}

We equip the categories $\calT/B$ with  
the standard model structure induced by the 
Quillen model structure on $\calT$, so that
a map in $\calT/B$ as in \eqref{diag:caltoverbmor} 
is a fibration, cofibration, or a weak equivalence
if and  only if the underlying map $\phi\colon X\to Y$ in $\calT$ is so
in the Quillen model structure on $\calT$.
(This model structure on $\calT/B$ is Quillen equivalent to 
May and Sigurdsson's \cite{MaySigurdsson} preferred model structure on 
$\calT/B$ and suffices for our purposes.)
We have a Quillen adjunction
\[
	f_! \colon \calT/A \longrightleftarrows \calT/B \colon f^\ast
\]
where the right adjoint is given by pullback along $f$
and the left adjoint by composition with $f$,
and therefore an induced adjunction
\[
	f_! \colon \ho(\calT/A) \longrightleftarrows \ho(\calT/B) \colon f^\ast
\]
on the level of homotopy categories.
We equip $\ho(\calT/B)$ with the cartesian symmetric monoidal structure,
and write $\times_B$ for the tensor product on $\ho(\calT/B)$.

\begin{defn}
Define 
\[
	\hpT \longto \calT
\] 
to be the symmetric monoidal fibration 
obtained by applying the Grothendieck construction of 
Theorem~\ref{thm:smgrothendieckconstr}
to the pseudofunctor
\begin{equation}
\label{eq:hptpf}
	\calT^\op \longto \smCat,
	\quad
	B \longmapsto \ho(\calT/B),
	\quad
	f \longmapsto f^\ast;
\end{equation}
cf.\ Definition~\ref{def:hpconstr}. 
\end{defn}

From the comparison given in \cite[App.~B]{ABGparam}, 
we have for every space $B$ an
equivalence of symmetric monoidal categories
\begin{equation}
\label{eq:hocaltoverbvshospacesoverb} 
	\ho(\calT/B) \homot \ho(\Spaces_{/B}),
\end{equation}
and these equivalences furthermore commute with the base change
functors $f_!$ and $f^\ast$.
Applying the Grothendieck construction to these 
equivalences, we now obtain
\begin{prop}
\label{prop:hpthomothpspaces}
There is an equivalence 
\begin{equation}
\label{eq:hpthomothpspaces}
	\hpT \homot \hpSpaces 
\end{equation}
in $\smFib_{\calT}$ 
preserving both cartesian and opcartesian morphisms.\qed
\end{prop}

For working with $\hpT$, it is useful to have 
explicit descriptions of the cartesian monoidal 
product $\times_B$ and the base change functors 
$f_!$ and $f^\ast$ on the categories $\ho(\calT/B)$.

\begin{lemma}
\label{lm:prodandbcfundesc}
Choose a functorial factorization of maps $X\to B$ in $\calT$ 
as composites $X \xto{\homot} X' \to B$ where the first map is a 
weak equivalence and the second one is a Serre fibration. Then:
\begin{enumerate}[(i)]
\item\label{it:hotbprod}
	There is a canonical natural equivalence
    \[
    	(X\xto{\,\alpha\,} B)\times_B (Y \xto{\,\beta\,} B) \homot (X'\times_B Y' \to B)
    \]
    where $\times_B$ on the left hand side is the
    cartesian product in $\ho(\calT/B)$ and 
    $\times_B$ on the right hand side refers to
    the cartesian product in $\calT/B$.
\item\label{it:hotbfast}
	Given $f\colon A \to B$, 
	there is a canonical natural equivalence
	\[
    	f^\ast(X\to B) \homot (X'\times_B A  \to A)
    \]
    where  $f^\ast$ is the functor
    $f^\ast \colon \ho(\calT/B) \to \ho(\calT/A)$.
\item\label{it:hotbfshriek}
	Given $f\colon A \to B$,  
	there is a canonical natural equivalence
	\[
		f_! (X\xto{\ \alpha\ } A) \homot (X \xto{\ f\circ \alpha\ } B).
	\]
	where $f_!$ is the functor $f_! \colon \ho(\calT/A) \to \ho(\calT/B)$.	
\end{enumerate}
\end{lemma}
\begin{proof}
An object $X\to B$ in $\calT/B$ is fibrant in our model structure
if and only if it is a Serre fibration. Thus
parts (\ref{it:hotbprod}) and (\ref{it:hotbfast}) follow.
Finally, part~(\ref{it:hotbfshriek})
follows from the observation that 
composition with $f$ preserves all weak equivalences, so
it is not necessary to pass to a cofibrant approximation when 
computing $f_!$.
\end{proof}

Our next goal is to exhibit a fibred category over $\calT$
of which $\hpT$ (and therefore $\hpSpaces$) can reasonably be regarded
as the homotopy category.

\begin{defn}
Define a category $\pT$ as follows. The objects of $\pT$ are
continuous maps $\alpha\colon X \to B$ in $\calT$, and a morphism 
in $\pT$ from $\alpha \colon X\to B$ to $\beta \colon Y\to C$ 
is a pair $(f,\bar{f})$ of continuous maps making the square 
\begin{equation}
\label{eq:pcaltmor}
\vcenter{\xymatrix{
	X
	\ar[r]^{\bar{f}}
	\ar[d]_\alpha
	&
	Y
	\ar[d]^{\beta}
	\\
	B
	\ar[r]^f
	&
	C
}}
\end{equation}
commutative. We equip $\pT$ with the symmetric monoidal structure
given by the direct product
\begin{equation}
\label{eq:pcalttensorprod} 
	(X\xto{\ \alpha\ } B) \times (Y \xto{\ \beta\ } C) 
	= 
	(X\times Y \xto{\ \alpha\times \beta\ } B\times C).
\end{equation}
\end{defn}

\begin{rem}[Cartesian and opcartesian morphisms in $\pT$]
It is readily verified that 
the functor $\pT \to \calT$ sending $\alpha \colon X\to B$ to $B$
is a symmetric monoidal fibration where a morphism $(f,\bar{f})$ in $\pT$
is cartesian if and only if the corresponding square \eqref{eq:pcaltmor}
is a pullback square in $\calT$. Moreover, it is not difficult to see 
that a morphism $(f,\bar{f})$ in $\pT$ is opcartesian if and only if
the map $\bar{f}$ is a homeomorphism.
\end{rem}

\begin{notation}
\label{not:xf}
Given a map $f\colon X \to B$ in $\calT$, we  
write $(X,f)$ for the object defined by $f$
in any of the categories
$\calT/B$, $\pT$,  $\ho(\calT/B)$, 
$\hpT$, $\ho(\Spaces_{/B})$, and $\hpSpaces$.
\end{notation}

\begin{prop}
\label{prop:tcalc2}
Suppose $\calC$ is a presentable symmetric monoidal $\infty$--category.
Then, up to canonical natural equivalence, there exists 
a unique functor
\[
	t_\calC \colon \pT \longto \hpC
\]
over $\calT$ which
\begin{enumerate}[\quad\;(a)]
\item \label{it:tcalc2oc}
	preserves opcartesian morphisms,
\item \label{it:tcalc2terminalobject}
	sends the terminal object $(\pt,\id_\pt)$ of $\pT$ to $S_\pt$, and
\item \label{it:tcalc2cweak}
	for every space $B \in \calT$ sends the unique map  
	$(B,\id_B)\to (\pt,\id_\pt)$ in $\pT$  
	to a cartesian morphism in $\hpC$.
\savecounteri
\end{enumerate}
Moreover, this functor 
\begin{enumerate}[\quad\;(a)]
\restorecounteri
\item \label{it:tcalc2sm}
	is symmetric monoidal 
	with respect to the product $\times$ on $\pT$ and 
	the external tensor product $\exttensor$ on $\hpC$,
	and
\item \label{it:tcalc2cstrong}
	sends a morphism $(f,\bar{f})$ of $\pT$ to a cartesian morphism in $\hpC$
    whenever the square \eqref{eq:pcaltmor} corresponding to $(f,\bar{f})$
    is homotopy cartesian.
\end{enumerate}
Explicitly,  $t_\calC$ is given on objects by
\[
	t_\calC (X,\alpha) = \alpha_! S_X
\]
where $S_X$ is the unit object in $\ho(\calC_{/X})$ 
and on morphisms by
\[
	t_{\calC}(f,\bar{f}) = \phi %
\]
where $\phi$ is the unique morphism 
making the square on the left below a commutative square
covering the square on the right.
\[
    \xymatrix{
    	S_X
    	\ar[r]^{\cart}
    	\ar[d]_{\opcart}
    	&
    	S_{Y}
    	\ar[d]^{\opcart}
    	\\
    	\alpha_! S_X
    	\ar@{-->}[r]^\phi
    	&
    	\beta_! S_{Y}
    }
    \qquad\qquad\qquad
    \xymatrix{
    	X
    	\ar[r]^{\bar{f}}
    	\ar[d]_{\alpha}
    	&
    	Y
    	\ar[d]^{\beta}
    	\\
    	B
    	\ar[r]^{f}
    	&
    	C
    }
\]
\end{prop}
\begin{proof}
It is readily verified that the definitions given 
yield a functor $t_\calC\colon \pT \to \hpC$
over $\calT$.
From the construction, we see that the functor $t_\calC$
so defined satisfies (\ref{it:tcalc2terminalobject}),
and our description of the opcartesian morphisms in $\pT$ and 
the analogue of Proposition~\ref{prop:cartmorprops}(\ref{it:cartfactor})
for opcartesian morphisms imply that it satisfies (\ref{it:tcalc2oc}).
Proposition~\ref{prop:commrelshriekinterpretation2} implies that 
$t_\calC$ satisfies (\ref{it:tcalc2cstrong}) and therefore also
(\ref{it:tcalc2cweak}).
Finally, (\ref{it:tcalc2sm}) follows from 
Proposition~\ref{prop:exttensoropcarts}.

To see the uniqueness of $t_\calC$,
notice that for every object $(X,\alpha) \in \pT$
there is a natural zigzag of morphisms in $\pT$
\begin{equation}
\label{eq:ptzigzag}
\xymatrix@C+1em{
	(\pt,\id_\pt) 
	&
	\ar[l]^-{}
	(X,\id_X)
	\ar[r]^-{(\alpha,\id_X)}_-\opcart
	&
	(X,\alpha)
} 
\end{equation}
where the first map is the unique one to the terminal object of $\pT$,
and that for every morphism $(f,\bar{f}) \colon (X,\alpha) \to (Y,\beta)$
in $\pT$ there is a commutative diagram
\begin{equation}
\label{eq:ptdiag}
\vcenter{\xymatrix@R-3ex{
	&
	(X,\id_X) 
	\ar[rr]^-{(\alpha,\id_X)}_-\opcart
	\ar[dd]^{(\bar{f},\bar{f})}
	\ar[dl]_{}
	&&
	(X,\alpha)
	\ar[dd]^{(f,\bar{f})}
	\\
	(\pt,\id_\pt)
	\\
	&
	(Y,\id_Y)
	\ar[rr]^-{(\beta,\id_Y)}_-\opcart
	\ar[ul]^{}
	&&
	(Y,\beta)
}}
\end{equation}
Given functors $t_\calC, t'_\calC \colon \pT \to \hpC$ 
over $\calT$ satisfying (\ref{it:tcalc2oc}) through (\ref{it:tcalc2cweak}),
the universal properties of cartesian and opcartesian morphisms
yield a canonical isomorphism between the images of 
diagram \eqref{eq:ptzigzag} under $t_\calC$ and $t'_\calC$, 
and similarly for diagram \eqref{eq:ptdiag}. 
The isomorphism between the images of \eqref{eq:ptzigzag}
in particular gives an isomorphism 
$t_\calC(X,\alpha) \isom t'_\calC(X,\alpha)$,
and the isomorphism between the images of \eqref{eq:ptdiag}
shows that this isomorphism is natural.
\end{proof}

Our next goal is to assemble the projections $\calT/B \to \ho(\calT/B)$
into a functor $\pT \to \hpT$. We wish to use the Grothendieck construction
to do so, but since the pullback functors $f^\ast\colon \calT/B \to \calT/A$
and $f^\ast\colon \ho(\calT/B) \to \ho(\calT/A)$ fail to commute with the
aforementioned projections, it is more convenient to make use of the 
pushforward functors $f_!$ in the construction. As 
the functors $\pT \to \calT$ and $\hpT \to \calT$ 
are opfibrations, their opposites 
$\pT^\op \to \calT^\op$ and $\hpT^\op \to \calT^\op$ 
are fibrations, so we may obtain them by 
applying the Grothendieck construction to the 
pseudofunctors $(\calT^\op)^\op  = \calT \to \Cat$
given by
\[
	B\longmapsto \calT/B, \ f \longmapsto f_!
	\qquad\text{and}\qquad
	B\longmapsto \ho(\calT/B), \ f \longmapsto f_!
\]
and passing to opposite categories.
In view of Lemma~\ref{lm:prodandbcfundesc}(\ref{it:hotbfshriek}),
the projections $\calT/B \to \ho(\calT/B)$
give a pseudo natural transformation between 
these pseudofunctors. 
We define the projection 
\begin{equation}
\label{eq:pttohpt}
	\pT \longto \hpT 
\end{equation}
 to be the functor
obtained by applying the Grothendieck construction 
to this pseudo natural transformation and passing to 
to opposite categories.

\begin{prop}
The projection 
\[
	\pT \longto \hpT
\]
satisfies conditions (\ref{it:tcalc2oc})---(\ref{it:tcalc2cweak})
of Proposition~\ref{prop:tcalc2}.
Consequently, 
it
satisfies conditions
(\ref{it:tcalc2oc})---(\ref{it:tcalc2cstrong})
of Proposition~\ref{prop:tcalc2}.
In particular, it is symmetric monoidal,
preserves opcartesian morphisms, and
sends a morphism $(f,\bar{f})$ of $\pT$
to a cartesian morphism in $\hpT$ 
whenever the commutative square in $\calT$
corresponding to $(f,\bar{f})$ is 
homotopy cartesian.
\end{prop}

\begin{proof}
By construction, $\pT \to \hpT$ is the opposite of a
morphism $\pT^\op \to \hpT^\op$ of fibrations,
and hence preserves opcartesian morphisms.
Thus it satisfies condition (\ref{it:tcalc2oc})
of Proposition~\ref{prop:tcalc2}.
Condition (\ref{it:tcalc2terminalobject})
is also immediate from the construction of 
the projection $\pT \to \hpT$.

To prove condition (\ref{it:tcalc2cweak}),
notice first that maps 
$(B,\id_B) \to (\pt,\id_\pt)$
in $\hpT$ 
are in bijection with 
maps 
$r^B_!  (B,\id_B) = (B,r_B)\to (\pt,\id_\pt)$
in the fibre $\ho(\calT/\pt)$
of $\hpT \to \calT$ over $\pt$.
Here $r^B = r_B \colon B \to \pt$ is the unique continuous map.
Since  $(\pt,\id_\pt)$ is terminal
in $\ho(\calT/\pt)$,
it follows that there is a unique 
morphism
$(B,\id_B)\to (\pt,\id_\pt)$
in $\hpT$. 
This morphism must be therefore be the image
of the unique morphism 
$(B,\id_B)\to (\pt,\id_\pt)$
in $\pT$
under the projection $\pT \to \hpT$.
Using Lemma~\ref{lm:prodandbcfundesc}(\ref{it:hotbfast}),
we see that the pullback $r_B^\ast (\pt,\id_\pt)$
in $\hpT$
is equivalent in the fibre $\ho(\calT/B)$
of $\hpT\to \calT$ over $B$ 
to $(B,\id_B)$,
so the morphism
$(B,\id_B)\to (\pt,\id_\pt)$
in $\hpT$ is cartesian.
Thus condition (\ref{it:tcalc2cweak})
holds.

In view of the equivalence 
$\hpT\homot\hpSpaces$ 
of Proposition~\ref{prop:hpthomothpspaces},
the rest of the claim now follows
from Proposition~\ref{prop:tcalc2}.
\end{proof}

\begin{rem}
\label{rk:tcalcrecognition}
Suppose $\calC$ is a presentable symmetric monoidal $\infty$--category.
Given a symmetric monoidal $\infty$--functor
$F \colon \Spaces \to \calC$
admitting a right adjoint,
the uniqueness part of Proposition~\ref{prop:tcalc2}
implies that the composite 
\[
	\pT \longto \hpT \homot \hpSpaces \xto{\ F_\fw\ } \hpC
\]
of $F_\fw$ with 
\eqref{eq:hpthomothpspaces} and  \eqref{eq:pttohpt}
is naturally equivalent to $t_\calC$. This illustrates the behaviour
of the functor $F_\fw$, and applies for example to the functor
$\suspension^\infty_+ \colon \Spaces \to \Spectra$, 
the composite 
\[
	\Spaces \xto{\ \suspension^\infty_+\ } \Spectra \xto{\ R\smashprod\ } \Mod^R
\]
for a ring spectrum $R$,
and the composite
\[
	\Spaces \xto{\ \suspension^\infty_+\ } \Spectra \xto{\ L\ } \Spectra^\ell
\]
where $L$ is the localization functor.
\end{rem}

We conclude the section by pointing out 
that the categories $\ho(\PointedSpaces_{/B})$
are also equivalent to ones
built from point set level data.

\begin{defn}
\label{def:exspace}
By an \emph{ex-space} over a space $B \in \calT$,
we mean a space $X\in\calT$ 
equipped with a map $s_X \colon B \to X$, called the \emph{section},
and a map $p_X \colon X \to B$, called the \emph{projection}, whose composite
is the identity map of $B$. A morphism $X\to Y$ of ex-spaces over $B$ 
is a continuous map $X\to Y$ commuting with the sections and the projections.
We write $\calT_B$ for the category of ex-spaces over $B$.
\end{defn}

\begin{rem}
Ex-spaces over $B$ are sometimes also called \emph{retractive spaces} over $B$
(since the section and the projection of an ex-space $X$ over $B$
exhibit $B$ as a retract of $X$)
and \emph{parametrized pointed spaces} over $B$ (since the section of an ex-space
can be thought of as a continuous choice of a basepoint for each fibre of the projection).
\end{rem}

By the comparison given in \cite[App.~B]{ABGparam},
there exist equivalences
\begin{equation}
\label{eq:fwpointedspacesandexspaces}
	\ho(\PointedSpaces_{/B})
	\homot
	\ho(\calT_B)
\end{equation}
where $\calT_B$ is the category of ex-spaces over $B$,
and these equivalences preserve the smash products $\smashprod_B$
and the base change functors $f_!$ and $f^\ast$, where $\smashprod_B$,
$f_!$, and $f^\ast$ on the right hand side are as in 
May and Sigurdsson's work \cite{MaySigurdsson}.
A very useful alternative reference for the aforementioned
structure on $\ho(\calT_B)$ is \cite{Malkiewich20}.

\subsection{The \texorpdfstring{$\infty$--}{infinity }category \texorpdfstring{$\pC$}{pC}}
\label{subsec:pc}
Our aim in this subsection is to refine the
fibred and opfibred category
$\hpC \to \calT$
into a cartesian and cocartesian fibration
of $\infty$--categories
$\pC \to N\calT$.
This refinement will be needed 
for two main purposes: first,
for formulating and proving
Theorem~\ref{thm:cosheaf},
which asserts that the 
objects $H_\bullet(B;X)$
studied in Section~\ref{sec:hbullet}
can be computed by assembling together
the objects $H_\bullet(U_i;X|U_i)$, $i\in I$,
when the sets $U_i$, $i\in I$ form a suitable open cover of $B$; 
and second, for constructing
relative versions $H_\bullet(B,B_0;X)$
of the objects $H_\bullet(B;X)$.
These relative objects will be 
needed in the construction of the 
Serre spectral sequence in Section~\ref{sec:serress}.
The reader not interested in 
Theorem~\ref{thm:cosheaf}
or the relative $H_\bullet$ may safely
skip over to Section~\ref{sec:hbullet}.

\begin{defn}
\label{def:pc}
We define
\[
	V\colon \pC \longto N\calT
\]
to be the cartesian fibration of $\infty$--categories
obtained by applying the weighted nerve construction
of \cite[\href{https://kerodon.net/tag/025X}{Def.~025X}]{kerodon}
to the functor
\[
	\calT^\op \longto \sSet, 
	\qquad
	B \longmapsto (\calC_{/B})^\op,
	\quad
	f \longmapsto (f^\ast)^\op
\]
and passing to opposite $\infty$--categories. 
Here $\sSet$ denotes the 
category of simplicial sets.
That $V$ is a cartesian fibration follows from
\cite[\href{https://kerodon.net/tag/046Y}{Cor.~046Y}]{kerodon}.
Explicitly, $n$--simplices in $\pC$ are pairs $(\vec{B},\vec{\sigma})$
where $\vec{B}$ is a functor $[n] \to \calT$, that is, a sequence
\begin{equation}
\label{eq:bseq}
	\vec{B} 
	= 
	\big(	
    	B_0 
    	\xto{\ f_1\ }
    	B_1
    	\xto{\ f_2\ }
    	\cdots
    	\xto{\ f_n\ }
    	B_n 
	\big)
\end{equation}
of spaces and continuous maps, and $\vec{\sigma}$ is a collection of 
simplices  
\[
	\vec{\sigma} 
	= 
	\{
		\sigma_j\colon \Delta^{\{j,\ldots,n\}} \to \calC_{/B_j} 
	\}_{0 \leq j \leq n}
\]
fitting into a commutative diagram
\[\xymatrix{
	\Delta^{\{0,\ldots,n\}}
	\ar[d]_{\sigma_0}
	&
	\ar@{_{(}->}[l] %
	\Delta^{\{1,\ldots,n\}}
	\ar[d]_{\sigma_1}
	&
	\ar@{_{(}->}[l] %
	\quad
	\cdots
	\quad
	&
	\ar@{_{(}->}[l] %
	\Delta^{\{n\}}
	\ar[d]^{\sigma_n}
	\\
	\calC_{/B_0}
	&
	\ar[l]_-{f_1^\ast}
	\calC_{/B_1}
	&
	\ar[l]_-{f_2^\ast}
	\quad\cdots\quad
	&
	\ar[l]_-{f_n^\ast}
	\calC_{/B_n}	
}\]
Here $\Delta^S$ for $S\subset \{0,\ldots,n\}$ denotes the face of $\Delta^n$
spanned by the vertices belonging to $S$.
The map $\alpha^\ast \colon \pC_n \to \pC_m$ associated 
to an order-preserving map $\alpha \colon [m] \to [n]$
is given by
\[
	\alpha^\ast(\vec{B}, \vec{\sigma})
	=
	(\vec{B} \circ \alpha, \vec{\sigma}')
\]
where the $i$--th simplex in $\vec{\sigma}'$ is the composite
\[
	\sigma'_i 
	\colon 
	\Delta^{\{i,\ldots,m\}}
	\xto{\ \alpha\ } 
	\Delta^{\{\alpha(i),\ldots,n\}}
	\xto{\ \sigma_{\alpha(i)}\ }
	\calC_{/B_{\alpha(i)}}.
\]
In particular, an object in $\pC$ is a pair $(B,X)$
where $B$ is a space and $X$ is an object in $\calC_{/B}$,
and a morphism in $\pC$ from $(B,X)$ to $(C,Y)$ is a pair
$(f,\phi)$ where $f\colon B\to C$ is a continuous map and
$\phi$ is a morphism $X \to f^\ast Y$ in $\calC_{/B}$.
Compare with the Grothendieck construction of Definition~\ref{def:grothendieckconstr}.
\end{defn}

We note that the fibre of $V \colon \pC\to N\calT$
over a vertex $B$ of $N\calT$ is canonically isomorphic to $\calC_{/B}$.
Moreover, using the explicit description of $\pC$ 
given in Definition~\ref{def:pc}, it is straightforward to 
verify that applying the homotopy category functor
to the cartesian fibration $V \colon \pC\to N\calT$ recovers
the fibration $\hpC \to \calT$.

\begin{lemma}
The  cartesian fibration 
$V \colon \pC\to N\calT$
is also a cocartesian fibration.
\end{lemma}
\begin{proof}
Given an edge $f\colon A \to B$ in $N\calT$, the induced functor
between the fibres of $V$ is $f^\ast \colon \calC_{/B} \to \calC_{/A}$.
Since $f^\ast$ admits a left adjoint, the claim follows
from \cite[Cor.~5.2.2.5]{HTT}.
\end{proof}

We wish to relate $V$--cartesian and $V$--cocartesian morphisms
in $\pC$ to cartesian and opcartesian morphisms in $\hpC$.

\begin{defn}
Let $F\colon \calA \to \calB$ be a functor between ordinary categories.
A morphism $f \colon Y \to X$ in $\calA$ is called \emph{precartesian}
with respect to $F$
if for every $g \colon Z \to X$ satisfying $F(g) = F(f)$
there exists a unique morphism $h \colon Z \to Y$ such that 
$g = f h$
and
$F(h) = \id_{F(Y)}$.
\end{defn}

Using the definitions of $\pi$--cartesian 
and locally $\pi$--cartesian morphisms
(\cite[\href{https://kerodon.net/tag/01T5}{Def.~01T5}
and 
\href{https://kerodon.net/tag/01TX}{Def.~01TX}]{kerodon})
and the explicit description of the homotopy category of an 
$\infty$--category
(\cite[\href{https://kerodon.net/tag/0049}{Subsec.~0049}]{kerodon}),
it is straightforward to verify the following lemma. 

\begin{lemma}
\label{lm:locpicartimpliesprecart}
Let $\calE$ be an $\infty$--category and let $\pi\colon \calE \to S$ be a morphism
of simplicial sets. Suppose $e$ is a locally $\pi$--cartesian
edge of $\calE$. Then the image of $e$ in the homotopy category $\ho(\calE)$ of $\calE$
is precartesian with respect to the functor 
$\ho(\pi) \colon \ho(\calE)\to \ho(S)$ induced by $\pi$
on the level of homotopy categories.
\qed
\end{lemma}

Using Lemma~\ref{lm:locpicartimpliesprecart}, we can verify

\begin{prop}
\label{prop:hcartisfib}
Let $\calE$ and $\calB$ be $\infty$--categories, and suppose $\pi \colon \calE \to \calB$
is a cartesian fibration. Then the functor $\ho(\pi) \colon \ho(\calE) \to \ho(\calB)$ 
induced by $\pi$ on the level of homotopy categories is a fibration.
Moreover, a morphism in $\calE$ is $\pi$--cartesian if and only if 
its image in $\ho(\calE)$ is cartesian with respect to $\ho(\pi)$. 
\end{prop}

\begin{proof}
Since $\pi$--cartesian morphisms are locally $\pi$--cartesian,
it follows from Lemma~\ref{lm:locpicartimpliesprecart}
that passing from $\calE$ to $\ho(\calE)$ sends $\pi$--cartesian morphisms
in $\calE$ to precartesian morphisms in $\ho(\calE)$.
Consequently,
for every morphism $f\colon A \to B$ in $\ho(\calB)$ and object $\bar{B} \in \ho(\calE)$
satisfying $(\ho(\pi))(\bar{B}) = B$, there exists a precartesian morphism
$\phi\colon \bar{A}\to \bar{B}$ in covering $f$,
namely the image in $\ho(\calE)$ of a $\pi$--cartesian morphism 
$\tilde{\phi}\colon \bar{A}\to \bar{B}$ in $\calE$ 
such that the morphism $\pi(\phi)$ in $\calB$ maps to $f$ in $\ho(\calB)$.
Moreover, every precartesian morphism in $\ho(\calE)$ is the image of some 
$\pi$--cartesian morphism in $\calE$.
To see this, suppose
$\phi'\colon \bar{A}'\to \bar{B}$ is another precartesian morphism in $\ho(\calE)$
covering $f$. There then exists an isomorphism 
$\psi\colon \bar{A}' \to \bar{A}$ in $\calE$
such that $\phi' = \phi \psi$ and $(\ho(\pi))(\psi) =\id_{A}$.
Pick a morphism $\tilde{\psi}$ in $\calE$ mapping to $\psi$ in $\ho(\calE)$,
and let $\tilde{\phi}'$ be a composite of $\tilde{\psi}$ and $\tilde{\phi}$.
Then the morphism $\tilde{\phi}'$
in $\calE$ maps to $\phi'$ in $\ho(\calE)$. Moreover, 
the morphism
$\tilde{\psi}'$ is an equivalence in $\calE$, so
\cite[\href{https://kerodon.net/tag/01TT}{Cor.~01TT}]{kerodon}
implies that $\tilde{\phi}'$ is $\pi$--cartesian, showing that $\phi'$
is the image of a $\pi$--cartesian morphism, as desired. We have shown that
precartesian morphisms
in $\ho(\calE)$ are precisely the images of $\pi$--cartesian morphism in $\calE$.
It follows that the class of precartesian morphisms in $\ho(\calE)$ is closed
under composition. Now \cite[Prop.~8.1.7]{Borceux2} implies that 
$\ho(\pi)$ is a fibration. Since for a fibration a morphism is cartesian 
if and only it is precartesian, the claim follows.
\end{proof}

Applying Proposition~\ref{prop:hcartisfib} to 
the cartesian fibrations $V\colon \pC \to N\calT$ and $V^\op \colon \pC^\op \to N\calT^\op$,
we obtain
\begin{cor}
\label{cor:pccartcocartmor}
A morphism in $\pC$ is  $V$--cartesian (resp.\ $V$--cocartesian)
if and only if its image in $\hpC$ is cartesian (resp.\ opcartesian)
with respect to the functor $\hpC\to \calT$. \qed
\end{cor}

\subsection{The cosheaf property of \texorpdfstring{$\pC$}{pC}}
\label{subsec:cosheafpc}

Our aim in this subsection is to prove 
Theorem~\ref{thm:cosheafpc} below showing that 
an object of $\pC$ covering a space $B$
can be recovered as the $\infty$--categorical 
colimit of its restrictions to the members of
an open cover of $B$.
This result will be the key ingredient
in the proof of Theorem~\ref{thm:cosheaf}
below. The reader not interested in the proof of 
Theorem~\ref{thm:cosheaf} may safely omit the subsection.

\begin{defn}
\label{def:admissibleindexedopencover}
Suppose $B$ is a space,
and write $\calU(B)$ for the poset of open subsets of $B$.
Let $I$ be a small category.
By an \emph{admissible $I$--indexed open cover} of $B$, we mean 
a functor $U = U_{(-)} \colon I \to \calU(B)$
such that for every point $x\in B$, 
the nerve of the full subcategory of $I$ spanned by 
those $i\in I$ satisfying $x \in U_i$ is weakly contractible.
\end{defn}

\begin{defn}
Given a category $I$, define the \emph{right cone} of $I$ 
and the \emph{left cone} of $I$ 
to be the joins
$I^\rightcone = I \star [0]$
and
$I^\leftcone = [0] \star I$,
respectively,
where $[0]$ is the terminal category.
Thus $I^\rightcone$ is obtained from $I$
by adjoining a final object $\infty$,
and $I^\leftcone$ is obtained by 
adjoining an initial object $-\infty$.
See \cite[Sec.~1.2.8]{HTT}.
Notice that nerve $N(I^\rightcone)$ then agrees with the 
right cone $(NI)^\rightcone$ of the simplicial set
as defined in \cite[Notation~1.2.8.4]{HTT},
and similarly for left cones.
We call the objects $\infty$ of $I^\rightcone$ 
and $-\infty$ of $I^\leftcone$
the \emph{cone points}. 
\end{defn}

\begin{defn}
\label{def:liftxf}
Given a functor $F \colon I^{\rightcone} \to \calT$  
and an object $X$ in $\calC_{/F(\infty)}$,
we write $X_F$ for the evident lift
\[\xymatrix{
	&
	\pC
	\ar[d]^V
	\\
	NI^{\rightcone}
	\ar[r]^{NF}
	\ar[ur]^{X_F}
	&
	N\calT
}\]
defined on vertices by setting $X_F(i) = F(\alpha_i)^\ast X$ 
where $\alpha_i \colon i \to \infty$
is the unique map from $i$ to $\infty$ in $I^{\rightcone}$.
Explicitly, $X_F$ sends the simplex in $NI^{\rightcone}$
given by the sequence of maps
\[
	i_0 \xto{\ f_1\ } i_1 \xto{\ f_2\ } \cdots \xto{\ f_n\ } i_n
\]
to the simplex in $\pC$ given by the sequence
\[
	F(i_0) \xto{\ F(f_1)\ } F(i_1) \xto{\ F(f_2)\ } \cdots \xto{\ F(f_n)\ } F(i_n)
\]
and the simplices 
$
	\vec{\sigma}
	=
	\{
		\sigma_j\colon \Delta^{\{j,\ldots,n\}} \to \calC_{/F(i_j)} 
	\}_{0 \leq j \leq n}
$
where $\sigma_j$ is the composite 
\[
	 \sigma_j 
	 \colon 
	 \Delta^{\{j,\ldots,n\}} 
	 \longto \Delta^0 
	 \xto{\ X\ }
	 \calC_{/F(\infty)}
	 \xto{\ \alpha_i^\ast\ }
	 \calC_{/F(i)}.
\]
\end{defn}

\begin{thrm}
\label{thm:cosheafpc}
Suppose $B$ is a space and $X$ is an object of $\calC_{/B}$.
Let $I$ be a small category, and let $U = U_{(-)}\colon I \to \calU(B)$
be an admissible $I$--indexed open cover of $B$.
Write $\bar{U} \colon I^\rightcone \to \calU(B)$ for the extension of $U$
sending the cone point to $B$, and let $J\colon \calU(B) \to \calT$ be the 
inclusion. Then the map 
\[
	X_{J\bar{U}} \colon NI^\rightcone \longto \pC
\]
is a colimit diagram, so that
\[
	X \homot \colim_{i\in I}\, X | U_i
\]
in $\pC$.
\end{thrm}

We will deduce Theorem~\ref{thm:cosheafpc}
from the following general result.

\begin{thrm}
\label{thm:cartfibcolim}
Suppose $\calE$ and $\calB$ are $\infty$--categories and $V\colon \calE \to \calB$
is a cartesian fibration with essentially small fibres. 
Let $\bar{p}\colon K^\rightcone \to \calE$
be a diagram such that 
\begin{enumerate}
\item \label{it:barpvcart}
    $\bar{p}$ sends every morphism in 
    $K^\rightcone$ to a $V$--cartesian morphism in $\calE$;
\item \label{it:vbarpcolimdiag}
	the composite $V \bar{p} \colon K^\rightcone \to \calB$ 
	is a colimit diagram in $\calB$; and 
\item \label{it:vbartrlimdiag}
	the composite 
    \begin{equation}
    \label{eq:kopleftconetocatinfty}
    	(K^\op)^\leftcone 
    	=
    	(K^\rightcone)^\op
    	\xto{\ (V \bar{p})^\op\ }
    	\calB^\op
    	\xto{\ \mathrm{Tr}_{\calE/\calB}\ }	
    	\Cat_\infty	 
    \end{equation}
    is a limit diagram in the $\infty$--category $\Cat_\infty$ of 
    small $\infty$--categories.
\end{enumerate}
Here $\mathrm{Tr}_{\calE/\calB}$ denotes the transport 
representation of $V$. 
Then $\bar{p}$ is a colimit diagram in $\calE$.
\end{thrm}
See \cite[\href{https://kerodon.net/tag/028V}{Def.~028V}]{kerodon}
for the definition of a transport representation of a cocartesian fibration;
the notion of a transport representation of a cartesian fibration is 
obtained by dualizing.

To prove Theorem~\ref{thm:cartfibcolim},
we first need to establish an auxiliary result.
Given a simplicial set $S$,
write 
$T_{n,i}(S) \subset S^\rightcone \times \Delta^n$
for the join
\[
	T_{n,i}(S) 
	= 
	(S \times \Delta^{\{0,\ldots,i\}}) 
		\star 
	(\{\infty\} \times \Delta^{\{i,\ldots,n\}})
	\subset 
	S^\rightcone \times \Delta^n.
\]
Here $\infty$ denotes the cone point in $S^\rightcone$.
Then
\[
	S^\rightcone \times \Delta^n = \bigcup_{i=0}^n T_{n,i}(S).
\]
Write 
\[
	T'_{n,0}(S)
	=
	(S \times \Delta^{\{0\}})
		\star 
	(\{\infty\} \times \bdry\Delta^{\{0,\ldots,n\}})
\]
and let
\[
	P_n(S) 
	=
	T'_{n,0}(S) \cup \bigcup_{i=1}^n T_{n,i}(S)
	\subset
	S^\rightcone \times \Delta^n.
\]
Moreover, let
\[
	Q_n(S) 
	= 
	(S^\rightcone \times \bdry \Delta^n) \cup (S \times \Delta^n)
	\subset 
	S^\rightcone \times \Delta^n
\]
and observe that in fact $Q_n(S) \subset P_n(S)$.
Given a simplicial subset $A \subset S^\rightcone \times \Delta^n$,
we will promote $A$ to a marked simplicial set (see \cite[\S 3.1]{HTT})
by defining an edge of $A$ to be marked
if and only if its image in
$(S^\rightcone)^\sharp \times (\Delta^n)^\flat$
is.
\begin{lemma}
\label{lm:markedanodynemors}
Suppose $S$ is a simplicial set. Then both
\begin{enumerate}[(i)]
\item \label{it:ttos}
	the inclusion $T'_{n,0}(S) \incl P_n(S)$  and 
\item \label{it:qtos}
	the inclusion $Q_n(S) \incl P_n(S)$ 
\end{enumerate}
are marked anodyne for all $n\geq 1$.
\end{lemma}
\begin{proof}
(\ref{it:ttos}):
Working simplex by simplex in $S$, we see that it is enough to show that 
the inclusion 
\[
	T'_{n,0} (\Delta^k)
	\cup
	P_n(\bdry\Delta^k)
	\longincl
	P_n(\Delta^k)
\]
is marked anodyne for all $k \geq 0$.
The nondegenerate $(n+k+1)$--simplices of 
$(\Delta^k)^\rightcone \times \Delta^n = \Delta^{k+1} \times \Delta^n$
are given by sequences
\[
	(\alpha(0),\beta(0)),\,\ldots,\, (\alpha(n+k+1),\beta(n+k+1))
\]
where $\alpha \colon \{0,\ldots,n+k+1\} \to \{0,\ldots,k+1\}$
and $\beta \colon \{0,\ldots,n+k+1\} \to \{0,\ldots,n\}$
are order-preserving surjections and $(\alpha(i+1),\beta(i+1))$
is either $(\alpha(i) + 1, \beta(i))$ or $(\alpha(i),\beta(i)+1)$
for all $i=0,\ldots,n+k$.
Setting $\gamma(i) = 0$ in the former case and $\gamma(i)=1$
in the latter case, we see that we may alternatively 
parametrize the set of these simplices by the elements of the set
\[
	\Gamma = \bigl\{ 
		\gamma
		\in \{0,1\}^{\{0,\ldots,n+k\}}
		\mid
		|\gamma^{-1}(0)| = k+1
		\text{ and }
		|\gamma^{-1}(1)| = n		
	\bigr\}.
\]
Given $\gamma\in \Gamma$, write $\sigma(\gamma)$ for the corresponding simplex.
Then $\Delta^{k+1} \times \Delta^n = \bigcup_{\gamma \in \Gamma} \sigma(\gamma)$.
Let $N = |\Gamma|$, and let
\[
	\gamma^1 \leq \gamma^2 \leq \cdots \leq \gamma^N
\]
be the elements of $A$ in lexicographical order.
For $1 \leq j \leq N$, write 
\[
	P'(j) 
	= 
	T'_{n,0}(\Delta^k)
		\cup 
	P_n(\bdry\Delta^k) 
		\cup 
	\bigcup_{i=2}^j \sigma(\gamma^j).
\]
Then 
$P'(1) = T'_{n,0}(\Delta^k)\cup P_n(\bdry\Delta^k)$
and
$P'(N) = P_n(\Delta^k)$,
so the claim follows by noticing that the 
inclusion $P'(j) \incl P'(j+1)$ is marked anodyne for all $j=1,\ldots,N-1$.

(\ref{it:qtos}):
As in part (\ref{it:ttos}), it suffices to show that  the inclusion 
\[
	Q_n(\Delta^k) \cup P_n(\bdry \Delta^k) \longincl P_n(\Delta^k)
\]
is marked anodyne for all $k\geq 0$.
Writing
\[
	P''(j) 
	=
	Q_n(\Delta^k) 
		\cup 
	P_n(\bdry \Delta^k) 
		\cup
	\,\bigcup_{\mathclap{i=N-j+2}}^N\, \sigma(\gamma^j),	
	\qquad
	1 \leq j \leq N,
\]
the claim follows by observing that 
$P''(1) = Q_n(\Delta^k)\cup P_n(\bdry \Delta^k)$ and
$P''(N) = P_n(\Delta^k)$
and by noticing that the inclusion
$P''(j) \incl P''(j+1)$ is marked anodyne for all $j=1,\ldots,N-1$.
\end{proof}

\begin{proof}[Proof of Theorem~\ref{thm:cartfibcolim}]
In view of \cite[Prop.~4.3.1.5(2) and Example~4.3.1.3]{HTT}, the assumption 
that $Vp$ is a colimit diagram in $\calB$ implies that it is 
enough to show that $\bar{p}$ is a $V$--colimit diagram.
Factorize $\bar{p} \colon K^\rightcone \to \calE$ as a composite
$
	K^\rightcone 
	\xto{\,\ \mathclap{i^\rightcone}\,\ } 
	\calK^\rightcone 
	\xto{\,\ \mathclap{\tilde{p}}\,\ } 
	\calE
$
where $\calK$ and hence $\calK^\rightcone$ are $\infty$--categories
and $i$ is a map $K\to \calK$. To obtain such a factorization,
we might for example factorize the adjoint $K \to \calE_{/\bar{p}(\infty)}$
of $\bar{p}$ as a composite $K \xto{\,i\,} \calK \to \calE_{/\bar{p}(\infty)}$
of a trivial cofibration followed by a fibration in the Joyal model structure
on simplicial sets and pass to adjoints again.
Here $\infty$ is the cone point in $K^\rightcone$.
Let
\[\xymatrix@C+1em{
	\calE'
	\ar[r]
	\ar[d]_{V'}
	&
	\calE
	\ar[d]^V
	\\
	\calK^\rightcone
	\ar[r]^{V\tilde{p}}
	& 
	\calB
}\]
be a pullback diagram, and let $\bar{p}' \colon K^\rightcone \to \calE'$
be the map induced by $\bar{p}$ and $i^\rightcone$.
By \cite[Cor.~4.3.1.15]{HTT}, it is enough to show that 
$\bar{p}'$ is a $V'$--colimit diagram. 
By the definition of $V'$--colimit diagram \cite[Def.~4.3.1.1]{HTT},  
our task is therefore to show that the map
\[
	\calE'_{\bar{p}'/}
	\longto
	\calE'_{p'/}
	\times_{\calK^\rightcone_{V'p'/}}
	\calK^\rightcone_{V'\bar{p}'/}
\]
is a trivial fibration, where $p' = \bar{p}' | \colon \calK \to \calE'$.
That is, we must show that every commutative square
\[\xymatrix{
	\bdry\Delta^n_{\vphantom{g}}
	\ar[r]
	\ar@{_{(}->}[d]%
	&
	\calE'_{\bar{p}'/}
	\ar[d]
	\\
	\Delta^n
	\ar[r]
	\ar@{-->}[ur]
	&
	\calE'_{p'/}
	\times_{\calK^\rightcone_{V'p'/}}
	\calK^\rightcone_{V'\bar{p}'/}
}\]
($n\geq 0$) can be completed with a dashed arrow, or, equivalently,
that all diagrams of the form
\[\xymatrix@!0@C=4em@R=3.3ex{
	&
	K \star \Delta^{\{0\}}
	\ar[ddr]^{\bar{p}'}
	\ar@{}[ddl]
	\ar@{_{(}->}[]!/3.5ex/;[ddl] %
	\\
	\\
	K\star\bdry\Delta^n_{\vphantom{g}}
	\ar@{_{(}->}[ddd] %
	\ar[rr]^f
	&&
	\calE'
	\ar[ddd]^{V'}
	\\
	\\
	\\
	K \star \Delta^n
	\ar[rr]^g
	\ar@{-->}[uuurr]
	&&
	\calK^\rightcone
}\]
($n\geq 1$) can. %

Suppose we have been given such a diagram.
Then the map $g$ necessarily sends $\Delta^n \subset K \star \Delta^n$
to the cone point of $\calK^{\rightcone}$,
so the map $g$ factors as a composite
\[
	K\star \Delta^n
	\xto{\ \id \star r\ }
	K \star \Delta^0
	=
	K^\rightcone
	\xto{\ g'\ }
	\calK^\rightcone
\]
where $r\colon \Delta^n \to \Delta^0$ is the unique map.
Letting $V'' \colon \calE''\to K^\rightcone$ be the pullback of 
$V'\colon \calE'\to \calK^\rightcone$ along $g'$,
we are reduced to showing that the diagram
\begin{equation}
\label{diag:vdoubleprimesquare}
\vcenter{\xymatrix@!0@C=4em@R=3.3ex{
	K\star\bdry\Delta^n_{\vphantom{g}}
	\ar@{_{(}->}[ddd] %
	\ar[rr]^{f'}
	&&
	\calE''
	\ar[ddd]^{V''}
	\\
	\\
	\\
	K \star \Delta^n
	\ar[rr]^{\id\star r}
	\ar@{-->}[uuurr]
	&&
	K^\rightcone
}}
\end{equation}
can be completed, where $f'$ is the map induced by $f$. 

Embed $K\star \bdry\Delta^n$ and $K \star \Delta^n$ into $K^\rightcone\times \Delta^n$
as the simplicial subsets $T'_{n,0}(K)$ and $T_{n,0}(K)$, respectively, and notice that 
then the map $\id\star r \colon K\star \Delta^n \to K^\rightcone$ 
corresponds to the restriction of the projection map $K^\rightcone\times \Delta^n \to K^\rightcone$ to $T_{n,0}(K)$.
We have the following diagram of simplicial sets
\[\xymatrix{
	\Map^\sharp_{K^\rightcone}(T'_{n,0}(K), (\calE'')^\natural)
	&
	\ar[l]_{\rho_1}
	\Map^\sharp_{K^\rightcone}(P_n(K), (\calE'')^\natural)
	\ar[r]^{\rho_2}
	&
	\Map^\sharp_{K^\rightcone}(Q_n(K) , (\calE'')^\natural)
	\\
	&&
	\ar[u]_{\rho_3}
	\ar[ul]^{\rho_4}
	\Map^\sharp_{K^\rightcone}(K^\rightcone \times \Delta^n, (\calE'')^\natural)
	\ar[d]^{\rho_5}
	\\
	&&
	\Map^\sharp_{K^\rightcone}(T_{n,0}(K), (\calE'')^\natural)
}\]
where all of the morphisms are restriction maps.
By \cite[Rk.~3.1.3.4]{HTT}, Lemma~\ref{lm:markedanodynemors}
implies that the restriction map 
\[
	\Map^\flat_{K^\rightcone}(P_n(K),(\calE'')^\natural)
	\longto
	\Map^\flat_{K^\rightcone}(T'_{n,0}(K),(\calE'')^\natural)
\]
is a trivial fibration, and passing to cores, 
\cite[\href{https://kerodon.net/tag/01EZ}{Prop.~01EZ}]{kerodon}
implies that $\rho_1$ is also a trivial fibration.
Similarly, the map $\rho_2$ is a trivial fibration.
By 
\cite[\href{https://kerodon.net/tag/02T8}{Thm.~02T8}
and \href{https://kerodon.net/tag/02T9}{Rk.~02T9}]{kerodon},
the assumption that 
\eqref{eq:kopleftconetocatinfty}
is a limit diagram
implies that  the restriction map 
\[
	\Map^\flat_{K^\rightcone}((K^\rightcone)^\sharp,(\calE'')^\natural)
	\longto
	\Map^\flat_{K}(K^\sharp,(\calE''')^\natural)
\]
is a trivial fibration where
$\calE'''$ denotes the restriction of $\calE''$ to $K \subset K^{\rightcone}$.
It follows easily that the restriction map
\[
	\Map^\flat_{K^\rightcone}(K^\rightcone\times \Delta^n,(\calE'')^\natural)
	\longto
	\Map^\flat_{K^\rightcone}(Q_n(K),(\calE'')^\natural)	
\]
satisfies the extension and lifting property
characterizing trivial Kan fibrations, 
and passing to cores, 
\cite[\href{https://kerodon.net/tag/01EZ}{Prop.~01EZ}]{kerodon}
implies that the map $\rho_3$ is a trivial fibration.
By \cite[Lemma~3.1.3.6]{HTT} and 
the 2-out-of-3 property of weak equivalences, 
the map $\rho_4$  is also a trivial fibration.
In particular, the composite $\rho_1 \rho_4$ is surjective on 
vertices.
The assumption that $\bar{p}$ sends every morphism in $K^\rightcone$
to a $V$--cartesian morphism in $\calE$ implies that $f'$
is a vertex in 
$\Map^\sharp_{K^\rightcone}(T'_{n,0}(K), (\calE'')^\natural)$,
so we may find a vertex $f'_1$ in 
$\Map^\sharp_{K^\rightcone}(K^\rightcone \times \Delta^n, (\calE'')^\natural)$
such that $\rho_1 \rho_4 (f'_1) = f'$.  Now $\rho_5(f'_1)$
provides the desired completion for diagram
\eqref{diag:vdoubleprimesquare}. 
\end{proof}

\begin{proof}[Proof of Theorem~\ref{thm:cosheafpc}]
The claim follows by applying Theorem~\ref{thm:cartfibcolim}
to the diagram $X_{J\bar{U}}$
once we have verified that $X_{J\bar{U}}$
satisfies the assumptions of that theorem.
It is easily seen from the construction of $X_{J\bar{U}}$ that 
it sends all morphisms in $NI^\rightcone$
to $V$--cartesian morphisms in $\pC$,
so that $X_{J\bar{U}}$ satisfies assumption 
(\ref{it:barpvcart}) of  Theorem~\ref{thm:cartfibcolim}.
Moreover,
it is easy to check that 
$B = \colim_{i\in I} U_i$ in $\calT$,
so that $VX_{J\bar{U}} \colon NI^\rightcone \to N\calT$
is a colimit diagram in $N\calT$
and $X_{J\bar{U}}$ satisfies assumption 
(\ref{it:vbarpcolimdiag}).

To show that $X_{J\bar{U}}$ satisfies assumption 
(\ref{it:vbartrlimdiag}) 
of Theorem~\ref{thm:cartfibcolim}, 
we will first construct a transport representation
for $\pC \to N\calT$ as a cartesian fibration.
Recall that the $\infty$--category $\Spaces$ of spaces can be constructed as
the homotopy coherent nerve
\[
	\Spaces = N^{\mathrm{hc}} \Kan
\]
where $\Kan$ is the full simplicial subcategory of $\sSet$ 
spanned by Kan complexes,
and that the $\infty$--category $\InftyCats$ of small $\infty$--categories
can be constructed as
\[
	\InftyCats = N^{\mathrm{hc}} \QCat
\]
where $\QCat$ is the simplicial subcategory of $\sSet$
whose objects are quasicategories and where the simplicial set of 
morphisms from a quasicategory $\calD$ to a quasicategory $\calE$
is given by the core $\Fun(\calD,\calE)^\homot$.
We equip $\calT$ with the trivial simplicial enrichment 
where the simplicial set of maps from $A$ to $B$ 
is the constant simplicial set given by the set of 
continuous maps from $A$ to $B$.
We then have simplicial functors
\[
	\Pi_\infty^\op \colon \calT \longto \Kan,
	\quad
	B \longmapsto \Pi_\infty(B)^\op
\]
and
\begin{equation}
\label{eq:funkanopqcat}
	\Fun(-,\calC)\colon \Kan^\op \longto \QCat,
	\quad
	X \longmapsto \Fun(X,\calC),
\end{equation}
and 
\cite[\href{https://kerodon.net/tag/027J}{Prop.~027J}]{kerodon}
together with 
\cite[%
	\href{https://kerodon.net/tag/01EA}{Example~01EA}
	and
	\href{https://kerodon.net/tag/0285}{Prop.~0285}%
]{kerodon}
show that the composite
\begin{equation}
\label{eq:ncaltoptoinftycats} 
	N\calT^\op
	\xto{\,N^{\mathrm{hc}}(\Pi_\infty^\op)^\op\,}
	\Spaces^\op
	\xto{\,N^{\mathrm{hc}}\Fun(-,\calC)\,}
	\InftyCats
\end{equation}
is a transport representation
for $\pC \to N\calT$ as a cartesian fibration.
See \cite[\href{https://kerodon.net/tag/028V}{Def.~028V}]{kerodon}.

Recall that $\Pi_\infty$ is another name for the singular simplicial set 
functor $\Sing_\bullet \colon \calT \to \sSet$.
The composite
\[
	(NI^\op)^\leftcone
	\isom
	(NI^\rightcone)^\op
	\xto{\,N\bar{U}^\op\,}
	N\calU(B)^\op
	\xto{\,NJ^\op\,}
	N\calT^\op
	\xto{\,N^{\mathrm{hc}}(\Pi_\infty^\op)^\op\,}
	\Spaces^\op
\]
is now a limit diagram in $\Spaces^\op$ by
\cite[Thm.~A.3.1]{HA} and 
\cite[Thm.~4.2.4.1]{HTT}.
The simplicial functor $\Fun(-,\calC)$ 
of equation \eqref{eq:funkanopqcat}
has a simplicial left adjoint
\[
	\Fun(-,\calC)^\homot \colon \QCat \longto \Kan^\op,
	\quad
	\calD \longmapsto \Fun(\calD,\calC)^\homot.
\]
By \cite[Cor.~5.2.4.5]{HTT},
it follows that the functor $N^{\mathrm{hc}}\Fun(-,\calC)$
of $\infty$--categories in \eqref{eq:ncaltoptoinftycats}
is a right adjoint, and therefore preserves limit diagrams.
We conclude that the composite of 
$(V X_{J\bar{U}})^\op$ and \eqref{eq:ncaltoptoinftycats}
is a limit diagram, so that $X_{J\bar{U}}$
also satisfies assumption 
(\ref{it:vbartrlimdiag}) 
of Theorem~\ref{thm:cartfibcolim}.
\end{proof}

\section{The functor \texorpdfstring{$H_\bullet$}{H.}}
\label{sec:hbullet}

In this section, we will
define and study the functor $H_\bullet$
associating to each parametrized 
$\calC$--object $X$ over a space $B$ 
an object $H_\bullet(B;X)$ of $\calC$, 
the ``homology of $B$ with coefficients in $X$.''
For most of our purposes, it suffices to work on the level 
of homotopy categories, so we start by 
constructing $H_\bullet$ as a functor $\hpC \to \ho(\calC)$
in Section~\ref{subsec:hbullethc}.
This functor on the level of homotopy categories is in particular
sufficient for our needs in Sections~\ref{sec:umkehrmaps}
and~\ref{sec:cwdualizabilityofsemisimplegrps}.
In Section~\ref{subsec:hbulletinftycats},
we will nevertheless refine this functor into an $\infty$--functor
$\pC\to \calC$. This refinement will be needed for 
formulating and proving Theorem~\ref{thm:cosheaf}
as well as for constructing relative versions
$H_\bullet(B,B_0;X)$ of the $H_\bullet$--objects,
which we will do in Section~\ref{subsec:relativehbullet}.
These relative objects will make a brief but crucial appearance
in the construction of the Serre spectral sequences
in Section~\ref{sec:serress}.
Throughout the section, $\calC$ will be a symmetric monoidal
presentable $\infty$--category which in Section~\ref{subsec:relativehbullet}
will in addition be assumed to be pointed.

\subsection{The functor \texorpdfstring{$H_\bullet$}{H.} on the level of homotopy categories}
\label{subsec:hbullethc}

\begin{defn}
\label{def:hbullethc}
Define a functor 
\[
	H_\bullet = H_\bullet^\calC \colon \hpC\longto \ho(\calC)
\]
as follows. 
Identifying, as usual, $\ho(\calC)$ with $\ho(\calC_{/\pt})$,
on objects we set
\[
	H_\bullet (B,X) = H_\bullet(B;X) = r_!^B X
\]
where $r^B \colon B \to \pt$ is the unique continuous map.
For a morphism $\phi \colon X \to Y$ in $\hpC$
covering a map $f\colon A \to B$ in $\calT$,
we define $H_\bullet (\phi)$ to be the unique morphism
\[
	(f,\phi)_\bullet \colon H_\bullet(A;X) \longto H_\bullet(B;Y)
\]
making the square on the left below, with vertical arrows
given by the canonical opcartesian morphisms,
into a commutative square covering the square on the right.
\begin{equation}
\label{diag:hbulletmordef} 
	\vcenter{\xymatrix@!0@C=5.5em@R=9ex{
    	X 
    	\ar[r]^{\phi\vphantom{f}}
    	\ar[d]
    	&
    	Y
    	\ar[d]
    	\\
    	r^A_! X
    	\ar@{-->}[r]^{(f,\phi)_\bullet}
    	&
    	r^B_! Y
    }}
	\qquad\qquad\qquad
	\vcenter{\xymatrix@!0@C=5.5em@R=9ex{
    	A 
		\ar[r]^f
		\ar[d]
		&
		B
		\ar[d]
		\\
		\pt
		\ar@{=}[r]
		&
		\pt_{\vphantom{!}}
    }}
\end{equation}
\end{defn}

The notation $H_\bullet(B;X)$ is intended to suggest
``homology of $B$ with (twisted) coefficients in $X$.''
When $X \in \ho(\calC)$ and 
$\calC$ is $\Spectra$ or more generally $\Mod^{R}$
for a commutative ring spectrum $R$,
the object
$H_\bullet(B; \underline{X}) \in \ho(\calC)$ 
does recover ordinary $X$--homology of $B$.
(We remind that the reader that $\underline{X}$
denotes the constant parametrized $\calC$--object
defined by $X$.)

\begin{example}
\label{ex:hbulletfortrivialcoeffs2}
Suppose $X \in \ho(\calC)$.
Then for all spaces $B$ we have
equivalences (natural in $B$ and $X$)
\begin{equation}
\label{eq:hbullettrivcoeffs2}
	H_\bullet(B; \underline{X})
	=
	r^B_! r_B^\ast X
	\homot
	r^B_! (r_B^\ast X \tensor_B S_B)
	\homot 
	X \tensor r^B_! S_B
	=
	X \tensor t_\calC(B,r_B)
\end{equation}
where the second equivalence follows from the projection formula
and $t_\calC$ is the functor of Proposition~\ref{prop:tcalc2}.
From Remark~\ref{rk:tcalcrecognition}, we see that for $\calC=\Spectra$,
we have
\[
	t_{\Spectra}(B,r_B) \homot \suspension^\infty_+ B,
\]
so in this case we obtain a natural isomorphism
\[
	\pi_\ast H_\bullet(B;\underline{X})
	\isom 
	\pi_\ast (X \smashprod \suspension^\infty_+ B)
	=
	X_\ast(B).
\]
Similarly, when $\calC = \Mod^R$ for a commutative ring spectrum $R$,
we have  $t_{\Mod^{R}}(B,r_B) \homot R \smashprod \suspension^\infty_+ B$,
yielding a natural isomorphism
\[
	\pi_\ast H_\bullet(B;\underline{X})
	\isom
	\pi_\ast (X \smashprod^R (R\smashprod \suspension^\infty_+ B))
	\isom
	\pi_\ast (X \smashprod \suspension^\infty_+ B)
	= 
	X_\ast(B).
\]
A similar remark also applies for $\calC = \Spectra^\ell$,
for example.
\end{example}

\begin{rem}
\label{rk:thomspacefromhbullet}
The functor $H_\bullet$ can be thought of as a generalization 
of the Thom space construction. To understand how, 
recall that given a vector bundle $\alpha$ over a space,
by definition the Thom space $B^\alpha$ 
of $\alpha$ is the pointed space obtained
by first forming the fibrewise one-point compactification
$S^\alpha$ of $\alpha$,
an ex-space over $B$ with section given by the points at infinity,
and by then collapsing all the points at infinity into a single point. 
Translating to the language of point-set level
base change functors, this is to say that 
$B^\alpha$ is the image of $S^\alpha \in \calT_B$ under the functor 
$r_! \colon \calT_B \to \calT_\pt$ induced by the map $r\colon B \to \pt$.
At least when $B$ is a CW complex, $S^\alpha \in \calT_B$ is cofibrant 
in the Quillen model structure on
$\calT_B$, 
so in such cases the Thom space $B^\alpha$ is the image 
of $S^\alpha \in \ho(\calT_B)$ also under the derived functor 
$r_! \colon \ho(\calT_B) \to \ho(\calT_\pt)$.
But under the equivalences \eqref{eq:fwpointedspacesandexspaces},
this functor agrees with the functor 
$H_\bullet \colon \ho(\PointedSpaces_{/B}) \to \ho(\PointedSpaces)$.
Continuing to write $S^\alpha$ and $B^\alpha$ for the 
corresponding objects of $\ho(\PointedSpaces_{/B})$ and $\ho(\PointedSpaces)$
under the equivalences \eqref{eq:fwpointedspacesandexspaces},
we therefore have
\[
	H_\bullet(B; S^\alpha) \homot B^\alpha \in \ho(\PointedSpaces)
\] 
when $B$ is a CW complex.
\end{rem}

Motivated by Remark~\ref{rk:thomspacefromhbullet}, 
we digress for a moment to use 
$H_\bullet$ to define Thom spectra of a virtual bundles.

\begin{defn}[The parametrized spectrum $S^\xi$]
\label{def:sxi}
Given a vector bundle $\alpha$ over a space $B$, 
we continue to write $S^\alpha$ for the fibrewise spectrum
defined by the 
the ex-space $S^\alpha$ of Remark~\ref{rk:thomspacefromhbullet},
that is, for the image of $S^\alpha \in \ho(\PointedSpaces_{/B})$
under the functor 
$\suspension^\infty_B \colon \ho(\PointedSpaces_{/B}) \to \ho(\Spectra_{/B})$.
As the fibres of $S^\alpha \in \ho(\Spectra_{/B})$ are spheres
and therefore invertible in $\ho(\Spectra)$, 
it follows from Corollary~\ref{cor:fwdinvcrit} that $S^\alpha$
is invertible in $(\ho(\Spectra_{/B}), \smashprod_B)$.
We write $S^{-\alpha}$ for an inverse of $S^\alpha$.
Finally, given a virtual bundle $\xi = \alpha-\beta$ over $B$, we let
\[
	S^\xi = S^\alpha \smashprod_B  S^{-\beta} \in \ho(\Spectra_{/B}).
\]
\end{defn}

It is readily verified that $S^\xi$ is functorial with respect to
stable isomorphisms of virtual bundles. 
See e.g.\ \cite[Def.~II.8.2]{CrabbJames} for the definition
of such isomorphisms.

\begin{defn}[The Thom spectrum $B^\xi$]
\label{def:thomspectrumofvirtualbundle}
Given a virtual bundle $\xi$ over $B$, we define
the \emph{Thom spectrum of $\xi$} to be the spectrum
\[
	B^\xi = H_\bullet(B;S^\xi) \in \ho(\Spectra).
\]
\end{defn}

Having defined the Thom spectrum $B^\xi$, we continue 
with our discussion of the properties of the functor $H_\bullet$.

\begin{prop}
\label{prop:opcarttoequiv}
The functor
$H_\bullet \colon \hpC \to \ho(\calC)$
sends opcartesian morphisms to equivalences.
\end{prop}
\begin{proof}
If $\phi \colon X \to Y$ is opcartesian, the
map $(f,\phi)_\bullet$ in \eqref{diag:hbulletmordef}
is an equivalence by the analogues
of 
Proposition~\ref{prop:cartmorprops}(\ref{it:compiscart}), 
(\ref{it:cartfactor}),
and (\ref{it:cartoveriso}) 
for opcartesian morphisms.
\end{proof}

\begin{cor}
\label{cor:hbulletandshriek}
Given $X \in \ho(\calC_{/A})$ and a continuous map
$f\colon A \to B$, the canonical opcartesian morphism
$X \to f_!X$ covering $f$ induces an equivalence
\[
	\pushQED{\qed} 
	H_\bullet(A;X) \homot H_\bullet(B;f_!X). 
    \qedhere
    \popQED
\]
\end{cor}

\begin{prop}
\label{prop:hbulletsm}
The functor
$H_\bullet \colon (\hp\calC,\exttensor) \to (\ho(\calC),\tensor)$
is symmetric monoidal, so that in particular we have an equivalence
\begin{equation}
\label{eq:hbulletmonconstr}
	\times 
	\colon
	H_\bullet(A;X) \tensor H_\bullet(B;Y) 
	\xto{\ \homot\ } 
	H_\bullet(A\times B; X\exttensor Y)
\end{equation}
natural in $X$ and $Y$.
\end{prop}
\begin{proof}
Suppose $(A,X)$ and $(B,Y)$ are objects in $\hp\calC$.
By Proposition~\ref{prop:exttensoropcarts},
the external tensor product $\exttensor$
on $\hp\calC$ preserves opcartesian morphism.
Factoring the canonical opcartesian  morphism
$X\exttensor Y \to r^{A\times B}_! (X\exttensor Y)$
through the $\exttensor$--product of the canonical
opcartesian morphisms $X \to r^A_! X$ and $ Y \to r^B_! Y$
yields an opcartesian morphism
\begin{equation}
\label{eq:mu}
	r^A_! X \exttensor r^B_! Y \to  r^{A\times B}_! (X\exttensor Y)
\end{equation}
which is an equivalence as it covers the homeomorphism 
$\pt \times \pt \to \pt$.
See Proposition~\ref{prop:cartmorprops}(\ref{it:cartfactor})
and~(\ref{it:cartoveriso}).
Composing with the morphism
\begin{equation}
\label{eq:delta}
r^A_! X \tensor_{\pt} r^B_! Y
\to
r^A_! X \exttensor r^B_! Y 
\end{equation}
covering $\Delta\colon \pt \to \pt\times\pt$
afforded by \eqref{eq:internaltensorformula}---also an equivalence---yields 
the monoidality constraint
$
	r^A_! X \tensor r^B_! Y \homot r^{A\times B}_! (X \exttensor Y).
$
The unit constraint is given by the opcartesian 
morphism $S_\pt \to r^\pt_! S_\pt$, which is an equivalence as it
covers the identity map of $\pt$.
\end{proof}

\begin{example}
When $\calC= \Spectra$  and $X,Y \in \ho(\Spectra)$,
the composite
\begin{multline*}
	\pi_\ast H_\bullet(A, \underline{X})
	\tensor
	\pi_\ast H_\bullet(B,\underline{Y})
	\xto{\ \times\ }
	\pi_\ast(H_\bullet(A;\underline{X}) \smashprod H_\bullet(B;\underline{Y}))
	\xto{\ \isom\ } 
	\pi_\ast H_\bullet(A\times B, \underline{X} \extsmashprod \underline{Y})
	\\
	\xto{\ \isom\ }
	\pi_\ast H_\bullet(A\times B, \underline{X \smashprod Y})	 
\end{multline*}
of the cross product on homotopy groups of spectra,
the map induced by \eqref{eq:hbulletmonconstr},
and the map induced by the equivalence 
$\underline{X} \extsmashprod \underline{Y} \homot \underline{X \smashprod Y}$
recovers the usual cross product
\[
	X_\ast(A) \tensor Y_\ast(B) 
	\xto{\ \times\ } 
	(X\smashprod Y)_\ast(A\times B).
\]
A similar remark applies to $\calC = \Spectra^\ell$ 
for a prime $\ell$ and $\calC = \Mod^{R}$
for a commutative ring spectrum~$R$.
\end{example}

Suppose $\calD$ is another symmetric monoidal presentable 
$\infty$--category, and let $F\colon \calC\to \calD$ be 
a (not necessarily symmetric monoidal) $\infty$--functor
which admits a right adjoint.
Then as a special case of Proposition~\ref{prop:opcartpreservation},
we have an equivalence
\begin{equation}
\label{eq:hbulletandfcommute}
\theta \colon H_\bullet(B;F_\fw X) \xto{\ \homot\ } F H_\bullet(B;X)
\end{equation}
natural in the object $(B,X) \in \hpC$.

\begin{prop}
\label{prop:fhsm}
Let $\calD$ be another symmetric monoidal presentable 
$\infty$--category, and let $F\colon \calC\to \calD$ be 
a symmetric monoidal $\infty$--functor admitting a right adjoint.
Then the natural equivalence $\theta$ of equation \eqref{eq:hbulletandfcommute}
is symmetric monoidal.
\end{prop}
\begin{proof}
The claim follows from the commutativity of the diagram below:
\[\xymatrix@!0@C=6.5em@R=9ex{
	H_\bullet F_\fw X \tensor H_\bullet F_\fw Y
	\ar[rrrr]^{\theta\tensor\theta}_\homot
	\ar[dr]^\delta_\homot
	\ar[ddd]_{(H_\bullet)_\tensor}^\homot
	&&&&
	F H_\bullet X \tensor F H_\bullet Y
	\ar[dl]_\delta^\homot
	\ar[dd]^{F_\tensor}_\homot
	\\
	&
	H_\bullet F_\fw X \exttensor H_\bullet F_\fw Y
	\ar[ddl]_{\mu}^\homot %
	\ar[rr]^{\theta\exttensor\theta}_{\homot}
	&&
	F H_\bullet X \exttensor F H_\bullet Y
	\ar[dd]_(0.6){(F_\fw)_\tensor}^(0.6){\homot}
	\\
	&&
	F_\fw X \exttensor F_\fw Y
	\ar[dll]_{\oc}
	\ar[ul]|{\oc\exttensor\oc_{\vphantom{\gamma}}}
	\ar[ur]|{F_\fw(\oc)^{\vphantom{h}}_{\vphantom{\gamma}} \exttensor F_\fw(\oc)}
	\ar[dd]_{(F_\fw)_\tensor}^\homot
	&&
	F (H_\bullet X \tensor H_\bullet Y)
	\ar[dl]^{F_\fw(\delta)}_\homot
	\ar[ddd]^{F (H_\bullet)_\tensor}_\homot
	\\
	H_\bullet (F_\fw X \exttensor F_\fw Y)
	\ar[dd]_{H_\bullet (F_\fw)_\tensor}^\homot
	&&&
	F_\fw(H_\bullet X \exttensor H_\bullet Y)
	\ar[ddr]^{F_\fw(\mu)}_\homot %
	\\
	&&
	F_\fw(X\exttensor Y)
	\ar[ur]|{F_\fw(\oc\exttensor^{\vphantom{h}} \oc)} %
	\ar[dll]_{\oc}
	\ar[drr]^{F_\fw(\oc)}
	\\
	H_\bullet F_\fw (X\exttensor Y)
	\ar[rrrr]^\theta_\homot
	&&&&
	F H_\bullet (X\exttensor Y)
}\]
Here $X$ and $Y$ are objects in $\hp\calC$,
and we have, for brevity, omitted base spaces 
from notation for the functor $H_\bullet$;
$\delta$'s refer to maps covering 
the diagonal map $\Delta\colon \pt \to \pt\times\pt$
afforded by~\eqref{eq:internaltensorformula};
$\mu$'s refer to maps given by~\eqref{eq:mu};
$G_\tensor$ for a symmetric monoidal 
functor $G$ denotes the 
monoidality constraint for $G$;
and the various maps $\oc$ 
are  opcartesian morphisms 
evident from the context. 
\end{proof}

Although we will focus on the ``homological'' functor $H_\bullet$, 
we note that there also exists an analogous
``cohomological'' functor $H^\bullet$. The natural domain of
definition for $H^\bullet$ is not $\hp\calC$, however, but 
$\hp\calC^\fop$.

\begin{defn}
Define a contravariant functor 
\[
	H^\bullet = H^\bullet_\calC \colon \hp\calC^\fop \longto \ho(\calC)
\]
as follows. 
On objects we set
\[
	H^\bullet (B,X) = H^\bullet(B;X) = r_\ast^B X
\]
where $r^B \colon B \to \pt$ is the unique continuous map.
To define $H^\bullet$ on morphisms, recall that 
the fibre of
$\hp\calC^\fop$ over the one--point space $\pt$
identifies with the \emph{opposite} category 
$\ho(\calC)^\op$.
For a morphism $\phi \colon X \to Y$ in $\hp\calC^\fop$
covering a map $f\colon A \to B$ in $\calT$,
we define $H^\bullet (\phi)$ to be the morphism
\[
	(f,\phi)^\bullet \colon H^\bullet(B;Y) \longto  H^\bullet(A;X)
\]
in $\ho(\calC)$ corresponding to the unique 
dashed morphism 
making the square on the left below, with vertical arrows
given by the  canonical opcartesian morphisms $\hp\calC^\fop$,
into a commutative square in $\hp\calC^\fop$
covering the square on the right.
\begin{equation}
\label{diag:hupperbulletmordef} 
	\vcenter{\xymatrix@!0@C=5.5em@R=9ex{
    	X 
    	\ar[r]^\phi
    	\ar[d]
    	&
    	Y
    	\ar[d]
    	\\
    	r^A_\ast X
    	\ar@{-->}[r]^{(f,\phi)^\bullet}
    	&
    	r^B_\ast Y
    }}
	\qquad\qquad\qquad
	\vcenter{\xymatrix@!0@C=5.5em@R=9ex{
    	A 
		\ar[r]^f
		\ar[d]
		&
		B
		\ar[d]
		\\
		\pt
		\ar@{=}[r]
		&
		\pt
    }}
\end{equation}
\end{defn}

\subsection{The functor \texorpdfstring{$H_\bullet$}{H.} on the level of \texorpdfstring{$\infty$--}{infinity }categories}
\label{subsec:hbulletinftycats}

Our aim in this subsection is to 
refine the functor $H_\bullet \colon \hpC \to \ho(\calC)$
into an $\infty$--functor $H_\bullet \colon \pC \to \calC$.
To do so, we will construct an adjoint pair of $\infty$--functors
\begin{equation}
\label{eq:caprshriekcaprastadj}
	R_! \colon \pC \longrightleftarrows N\calT \times \calC \colon R^\ast.
\end{equation}
The $\infty$--functor $H_\bullet$ will then be defined in terms of 
the left adjoint $R_!$.

Let 
\[
	c \colon N\calT \times \Delta^1 \isom N(\calT \times [1]) \longto N\calT
\]
be the map induced by the unique natural transformation
from the identity functor of $\calT$ to the constant functor
sending every space to the one-point space $\pt$,
and let $V'\colon \calM \to N\calT \times \Delta^1$ be the pullback
of $V\colon \pC \to N\calT$ along c:
\[\xymatrix{
	\calM 
	\ar[r]^{\tilde{c}}
	\ar[d]_{V'}
	&
	\pC
	\ar[d]^V
	\\
	N\calT \times \Delta^1
	\ar[r]^-c
	&
	N\calT	
}\]
Then $V'$ is a cartesian and cocartesian fibration.
Notice that the restriction of $V'$ to $N\calT \times \{0\} \subset N\calT \times \Delta^1$
identifies with the fibration $\pC\to N\calT$, 
while the restriction of $V'$ to $N\calT \times \{1\} \subset N\calT \times \Delta^1$
identifies with the projection  $\pr \colon N\calT \times \calC \to N\calT$.
Let 
\[
	s \colon N\calT \times \calC \times \Delta^1 \longto \calM
\]
be the map defined by the projection 
$N\calT \times \calC \times \Delta^1 \to N\calT \times \Delta^1$
and the map 
$
	N\calT \times \calC \times \Delta^1 \to \pC
$
which sends the $n$--simplex given by a sequence
\[
	B_0 
	\xto{\ f_1\ }
	B_1
	\xto{\ f_2\ }
	\cdots
	\xto{\ f_n\ }
	B_n,
\]
a simplex $\sigma \colon \Delta^n \to \calC$
and a simplex
\[
	t_k \in \Delta^1_n = \Delta([n],[1]),
	\qquad
	t_k(i) 
	=
    \begin{cases}
    0 & \text{if $i<k$} 
    \\
    1 & \text{otherwise}
    \end{cases}
\]
to the $n$--simplex of $\pC$ given by the sequence
\[
	B_0 
	\xto{\ f_1\ }
	B_1
	\xto{\ f_2\ }
	\cdots
	\xto{\ f_{k-1}\ }
	B_{k-1}
	\xto{\ r_{B_{k-1}}\ }
	\pt
	\xto{\ \id\ }
	\cdots
	\xto{\ \id\ }
	\pt
\]
and the simplices
\[
	\vec{\tau} 
	= 
	\{
		\tau_j \colon \Delta^{\{j,\ldots,n\}} \to \calC_{/B_j} 
		\mid
		\tau_j = r^\ast_{B_j}\sigma | \Delta^{\{j,\ldots,n\}}
		\text{ if $j < k$ and }
		\tau_j = \sigma | \Delta^{\{j,\ldots,n\}} \text{ otherwise}	
	\}_{0 \leq j \leq n}.
\]
We let 
\[
	R^\ast \colon N\calT \times \calC \longto \pC
\]
be the functor defined by the restriction of $s$ to 
$N\calT \times \calC \times \{0\} \subset N\calT \times \calC \times \Delta^1$.

Consider now the commutative diagram of 
marked simplicial sets (\cite[Def.~3.1.0.1]{HTT}) 
given by the 
solid arrows below
\[\xymatrix@R+1ex@C+1.5em{
	(\pC^\op)^\flat_{\vphantom{f}} \times \{0\}
	\ar[r]
	\ar@{_{(}->}[d] %
	&
	(\calM^\op)^\natural
	\ar[d]^{(V')^\op}
	\\
	(\pC^\op)^\flat \times ((\Delta^1)^\op)^\sharp
	\ar[r]_-{V^\op \times \id}
	\ar@{-->}[ur]^{s_1}
	&
	(N\calT^\op)^\sharp \times ((\Delta^1)^\op)^\sharp
}\]
The inclusion $\{0\} \incl ((\Delta^1)^\op)^\sharp$ is isomorphic 
to the inclusion $\{1\} \incl (\Delta^1)^\sharp$, and 
hence is marked anodyne by \cite[Def.~3.1.1.1(2)]{HTT}.
Consequently, the vertical map  on the left is marked anodyne by
\cite[Prop.~3.1.2.3]{HTT}, and \cite[Rk.~3.1.1.10]{HTT}
implies the existence of the dashed morphism $s_1$ completing the diagram.
We let 
\[
	R_! \colon \pC \longto N\calT \times \calC
\]
be the functor defined by the restriction of 
$s_1^\op \colon \pC\times \Delta^1 \to \calM$
to $\pC\times \{1\}$.

It remains to show that $\infty$--functors $R_!$ and $R^\ast$ are adjoint.
The composite 
\[
	V'' \colon \calM \xto{\ V'\ } N\calT \times \Delta^1 \xto{\ \pr\ } \Delta^1
\]
is a cartesian and cocartesian fibration with
the fibre $\calM_{\{0\}}$ of $V''$ over the vertex
$0$ of $\Delta^1$ canonically isomorphic to $\pC$
and the fibre $\calM_{\{1\}}$ over the vertex $1$ canonically isomorphic to 
$N\calT \times \calC$. Using 
\cite[%
	\href{https://kerodon.net/tag/01UF}{Rk.~01UF}, 
	\href{https://kerodon.net/tag/01UL}{Prop.~01UL}
	and
	\href{https://kerodon.net/tag/01T8}{Ex.~01T8}%
]{kerodon},
we see that a morphism in $\calM$ is $V''$--cartesian 
(resp.\ $V''$--cocartesian)
if and only if its image under $\tilde{c}$
is $V$--cartesian (resp.\ $V$--cocartesian)
and the $N\calT$--coordinate of its image under $V'$ is an equivalence.
It follows that the maps $s$ and $s_1$ witness $R^\ast$ and $R_!$ as the functors 
associated to $V''\colon \calM\to \Delta^1$ in the sense of 
\cite[Def.~5.2.1.1]{HTT}.
Thus $R_!$ and $R^\ast$ form an adjoint pair of functors
as claimed.
See \cite[Def.~5.2.2.1]{HTT}.
Notice that by construction, the functors $R_!$ and $R^\ast$ commute with the projections
$V \colon \pC \to N\calT$ and $\pr \colon N\calT\times \calC \to N\calT$:
we have $\pr R_!  = V$ and $V R^\ast = \pr$.

\begin{defn}
We define the $\infty$--functor
$H_\bullet \colon \pC \to \calC$ to be the composite
\begin{equation}
\label{eq:hbulletrefined}
	H_\bullet \colon \pC \xto{\ R_!\ } N\calT \times \calC \xto{\ \pr\ } \calC.
\end{equation}
Following the notation we used for the functor $H_\bullet$ defined on 
the level of homotopy categories, we write $H_\bullet(B;X)$
for the value of $H_\bullet$ on an object $(B,X)$ of $\pC$ and 
$(f,\phi)_\bullet$ for its value on a morphism $(f,\phi)$ of $\pC$.
\end{defn}

\begin{rem}
\label{rk:hbulletvaluescharcocart}
The object $H_\bullet(B;X)$ for an object $(B,X)$ of $\pC$ can be 
characterized up to contractible choice as the target of 
a $V$--cocartesian morphism in $\pC$ having source $(B,X)$ and covering 
the unique map $B \to \pt$. To see this, notice that 
$\tilde{c}s_1^\op | \{(B,X)\}\times \Delta^1$ is a $V$--cocartesian morphism 
$(B,X) \to H_\bullet(B;X)$ covering $B\to \pt$,
and recall that cocartesian morphisms
covering a given map and having a given source  
are parametrized by a contractible Kan complex. Cf.~\cite[Rk.~2.4.1.9]{HTT}.
\end{rem}

\begin{rem}
\label{rk:hbulletvaluescharcolim}
Suppose $B$ is a space. Restricting \eqref{eq:caprshriekcaprastadj}
to fibres over $\{B\}\in N\calT$, we obtain functors
\[
	\calC_{/B} \longrightleftarrows  \calC
\]
which the restriction of $V'\colon \calM \to N\calT \times \Delta^1$ 
to $\{B\} \times \Delta^1\subset N\calT \times \Delta^1$ 
witnesses as adjoint to each other.
By the construction of the functor $R^\ast$, 
the right adjoint in this adjunction is precisely
$r_B^\ast \colon \calC \to \calC_{/B}$, so the left adjoint must be
$r^B_!\colon \calC_{/B} \to \calC$.
Thus we also have
\[
	H_\bullet(B;X) \homot r^B_! X \homot \colim_{\Pi_\infty(B)^\op} X
\]
where the latter equivalence follows by noting that 
the left adjoint of $r_B^\ast$  is the $\infty$--categorical colimit functor
$\calC_{/B} = \Fun(\Pi_\infty(B),\calC) \to \calC$.
See 
\cite[\href{https://kerodon.net/tag/02JL}{Prop.~02JL}]{kerodon}.
\end{rem}

Corollary~\ref{cor:pccartcocartmor} 
and
Remark~\ref{rk:hbulletvaluescharcocart}
imply
\begin{prop}
The functor induced by $H_\bullet \colon \pC \to \calC$
on homotopy categories is canonically naturally equivalent
to the previously defined functor $H_\bullet \colon \hpC \to \ho(\calC)$.
\qed
\end{prop}

As a left adjoint, the functor $R_!$ preserves colimits.
See \cite[\href{https://kerodon.net/tag/02KE}{Cor.~02KE}]{kerodon}.
Thus we have

\begin{prop}
\label{prop:hbulletandcolimits}
The functor $H_\bullet \colon \pC \to \calC$ preserves all colimits that
exist in $\pC$. \qed
\end{prop}

As an immediate consequence of 
Proposition~\ref{prop:hbulletandcolimits}
and 
Theorem~\ref{thm:cosheafpc},
we have the following result.

\begin{thrm}
\label{thm:cosheaf}
Suppose $B$ is a space and $X$ is an object of $\calC_{/B}$.
Let $I$ be a small category, and let $U = U_{(-)}\colon I \to \calU(B)$
be an admissible $I$--indexed open cover of $B$
in the sense of Definition~\ref{def:admissibleindexedopencover}.
Write $\bar{U} \colon I^\rightcone \to \calU(B)$ for the extension of $U$
sending the cone point to $B$, and let $J\colon \calU(B) \to \calT$ be the 
inclusion.
 Then the composite
\[
	NI^\rightcone \xto{\ X_{J\bar{U}}\ } \pC \xto{\ H_\bullet\ } \calC
\]
of $H_\bullet$ and the map 
$X_{J\bar{U}}$ constructed in Definition~\ref{def:liftxf}
is a colimit diagram, so that 
\[
	H_\bullet (B;X) \homot \colim_{i\in I } H_\bullet (U_i; X | U_i)
\]
in $\calC$.
\qed
\end{thrm}

\subsection{Relative \texorpdfstring{$H_\bullet$}{H.} and the Eilenberg--Steenrod axioms for \texorpdfstring{$H_\bullet$}{H.}}
\label{subsec:relativehbullet}

Having succeeded in refining the functor $H_\bullet \colon \hpC \to \ho(\calC)$
to an $\infty$--functor $H_\bullet \colon \pC \to \calC$,
we proceed to construct the relative objects $H_\bullet(B,B_0;X)$.
The basic idea is to define
$H_\bullet(B,B_0;X)$ as the cofibre of the map
$H_\bullet(B_0; X | B_0) \to H_\bullet(B; X)$
induced by the inclusion $B_0 \incl B$.
To make this idea precise, it is necessary
for us to work on the level of $\infty$--categories 
instead of their homotopy categories.
Moreover, for the cofibres to make sense,
we need to assume that $\calC$ is pointed, so
\emph{in this subsection, we assume that our 
presentable symmetric monoidal $\infty$--category
$\calC$ is \emph{pointed} in the sense
that $\calC$ contains an object $0$ which is both initial and terminal}.

\begin{defn}
We let $\calTrel$ be the category of pairs of space, so that 
objects in $\calTrel$ are pairs $(B,B_0)$ consisting of a space $B$ and
a subspace $B_0 \subset B$, and a morphism in $\calTrel$ 
from $(A,A_0)$ to $(B,B_0)$ is a continuous map $f\colon A \to B$ 
satisfying $f A_0 \subset B_0$.  
\end{defn}

\begin{defn}
We define the $\infty$--category $\pCrel$ to be the pullback
\[\xymatrix{
	\pCrel
	\ar[r]
	\ar[d]
	&
	\pC
	\ar[d]^V
	\\
	N\calTrel
	\ar[r]^{NF}
	&
	N\calT
}\]
where $F\colon \calTrel \to \calT$ is the forgetful functor sending 
a pair $(B,B_0)$ of spaces to the space $B$.
Explicitly, the $n$--simplices in $\pCrel$ are as in $\pC$, except
that the sequence \eqref{eq:bseq} of spaces and continuous maps
is replaced by a sequence
\begin{equation}
\label{eq:b2seq}
	(\vec{B}, \vec{B}')
	=
	\big(
    	(B_0, B'_0)
    	\xto{\ f_1\ }
    	(B_1, B'_1)
    	\xto{\ f_2\ }
    	\cdots
    	\xto{\ f_n\ }
    	(B_n,B'_n)
	\big)
\end{equation}
of pairs of spaces and maps between them. In particular,
an object in $\pCrel$ is a pair $((B,B_0), X)$
where $(B,B_0)$ is a pair of spaces and $X$ is an object
in $\calC_{/B}$, and a morphism in $\pCrel$ 
from $((B,B_0),X)$ to $((C,C_0),Y)$
is a pair $(f,\phi)$ where $f \colon (B,B_0) \to  (C,C_0)$
is a map of pairs of spaces and $\phi \colon X \to f^\ast Y$
is a morphism in $\calC_{/B}$.
\end{defn}

Let
\[
	T\colon \pCrel \times \Delta^1 \longto \pC
\]
be the map given on $n$--simplices by sending the simplex
of $\pCrel \times \Delta^1$
defined by a sequence 
$(\vec{B}, \vec{B}')$ as in \eqref{eq:b2seq},
simplices 
\[
	\vec{\sigma}
	=
	\{
		\sigma_j\colon \Delta^{\{j,\ldots,n\}} \to \calC_{/B_j} 
	\}_{0 \leq j \leq n},
\]
and a simplex
\[
	t_k \in \Delta^1_n = \Delta([n],[1]),
	\qquad
	t_k(i) 
	=
    \begin{cases}
    0 & \text{if $i<k$} 
    \\
    1 & \text{otherwise}
    \end{cases}
\]
to the simplex of $\pC$ given by the sequence
\[
	B'_0
	\xto{\quad\mathclap{f'_1}\quad}
	\cdots
	\xto{\quad\mathclap{f'_{k-1}}\quad}
	B'_{k-1}
	\xto{\;\;\quad\mathclap{f_k i_{k-1}}\quad\;\;}
	B_k
	\xto{\quad\mathclap{f_{k+1}}\quad}
	\cdots
	\xto{\quad\mathclap{f_n}\quad}
	B_n
\]
and the simplices
\[
	\vec{\tau} 
	= 
	\{
		\tau_j \colon \Delta^{\{j,\ldots,n\}} \to \calC_{/B_j} 
		\mid
		\tau_j = i_j^\ast \sigma_j
		\text{ if $j < k$ and }
		\tau_j = \sigma_j \text{ otherwise}	
	\}_{0 \leq j \leq n}.
\]
Here $i_j \colon B'_j \incl B_j$ is the inclusion 
and $f'_i \colon B'_{i-1} \to B'_i$ is the map defined by $f_i$.

\begin{defn}[Relative $H_\bullet$  on the level of $\infty$--categories]
We define the functor
\[
	H^{\mathrm{rel}}_\bullet \colon \pCrel \longto \calC
\]
to be the composite 
\[
	\pCrel \longto \Fun(\Delta^1,\calC) \xto{\ \cofib\ } \calC
\]
where the first arrow is adjoint to the composite
\[
	\pCrel \times \Delta^1 
	\xto{\ T\ }
	\pC
	\xto{\ H_\bullet\ }
	\calC
\]
and $\cofib$ is the functor sending a morphism in $\calC$ to its cofibre. 
See \cite[Rk.~1.1.1.7]{HA}.
We will write  $H_\bullet(B,B_0;X)$
for the value of $H^{\mathrm{rel}}_\bullet$ at an object $((B,B_0),X)$
and $(f,\phi)_\bullet$
for its value at a morphism $(f,\phi)$.
Moreover, we will usually write $H_\bullet$ for $H^{\mathrm{rel}}_\bullet$.
\end{defn}

\begin{rem}
Notice that the restriction of $T$ to $\pCrel \times \{0\}$
sends an object $((B,B_0),X)$ to the object $(B_0, X | B_0)$;
that the restriction of $T$ to $\pCrel \times \{1\}$
sends  $((B,B_0),X)$ to the object $(B,X)$;
and that the restriction of $T$ to $\{((B,B_0),X)\} \times \Delta^1$
is the morphism $(B_0, X | B_0) \to (B,X)$ in $\pC$
given by the inclusion $B_0 \incl B$ and the identity map of $X | B_0$.
Thus we have a cofibre sequence 
\begin{equation}
\label{eq:hbulletcofibreseq}
	H_\bullet(B_0; X|B_0) \longto H_\bullet(B;X) \longto H_\bullet(B,B_0;X)
\end{equation}
in $\calC$ where the first morphism is the one induced by the inclusion
$B_0\incl B$.
\end{rem}

\begin{rem}
Observe that there is an evident embedding
$J\colon \pC \incl \pCrel$ sending an object $(B,X)$ to the object $((B,\emptyset),X)$.
Since $H_\bullet(\emptyset;\emptyset) \homot 0$, the composite
\[
	\pC
	\xto{\quad\mathclap{J}\quad}
	\pCrel
	\xto{\quad\mathclap{H_\bullet}\quad}
	\calC
\]
is naturally equivalent to the functor $H_\bullet \colon \pC \to \calC$.
In other words, we have a natural equivalence
\begin{equation}
\label{eq:absfromrel}
	H_\bullet (B;X) \homot H_\bullet(B,\emptyset;X) 
\end{equation}
for all objects $(B,X)$ of $\pC$. Moreover, under equivalence
\eqref{eq:absfromrel}, the map 
$H_\bullet(B;X) \to H_\bullet(B,B_0;X)$ of \eqref{eq:hbulletcofibreseq}
corresponds to the one induced by the morphism of $\pCrel$ 
defined by the inclusion $(B,\emptyset)\incl (B,B_0)$
and the identity map of $X$.
\end{rem}

\begin{defn}[Relative $H_\bullet$ on the level of homotopy categories]
We let $\hpCrel \to \calTrel$ be the fibration obtained from
$\pCrel \to N\calTrel$ by passing to homotopy categories.
Explicitly, the objects in the homotopy category $\hpCrel$ of $\pCrel$
are pairs $((B,B_0),X)$ where $(B,B_0)$ is a pair of spaces
and $X$ is an object in $\ho(\calC_{/B})$,
and a morphism in $\hpCrel$ from $((B,B_0),X)$ to 
$((C,C_0),Y)$ is a pair $(f,\phi)$ where $f$ is a map $(B,B_0) \to (C,C_0)$
and $\phi$ is a morphism in $\ho(\calC_{/B})$ from $X$ to $f^\ast Y$.
We continue to write $H_\bullet$ and $H^{\mathrm{rel}}_\bullet$
for the functor $\hpCrel \to \ho(\calC)$ induced by 
$H_\bullet = H^{\mathrm{rel}}_\bullet \colon \pCrel \to \calC$,
and continue to use the notations $H_\bullet(B,B_0;X)$ and $(f,\phi)_\bullet$
for the values of the functor defined on homotopy categories.
\end{defn}

We note that Example~\ref{ex:hbulletfortrivialcoeffs2}
generalizes to the relative situation.

\begin{example}
\label{ex:relativehbulletfortrivialcoeffs}
Suppose  $X \in \ho(\Spectra)$. In view of 
Example~\ref{ex:hbulletfortrivialcoeffs2},
$H_\bullet(B,B_0;\underline{X})$ agrees with
the cofiber of the map
\[
	X \smashprod \suspension^\infty_+ B_0 
	\longto
	X \smashprod \suspension^\infty_+ B 
\]
induced by the inclusion $i_B \colon B_0\incl B$. Thus
\[
	H_\bullet(B,B_0;\underline{X})
	\homot
	X \smashprod \suspension^\infty C(i_B) 
\]
where $C(i_B)$ denotes the homotopy cofiber of the map 
$i_B \colon (B_0)_+ \to B_+$ in the category of pointed spaces. 
In particular,
\[
	\pi_\ast H_\bullet(B, B_0; \underline{X}) \isom X_\ast(B,B_0).
\]
Again, similar remarks also apply to $X \in \ho(\Spectra^\ell)$
and $X\in \ho(\Mod^{R})$ for $R$ a commutative ring spectrum.
\end{example}

Write $\bigvee_{\alpha\in I} X_\alpha$ for the $\infty$--categorical coproduct of a family 
$\{X_\alpha\}_{\alpha \in I}$ of objects in $\calC$.
\begin{thrm}[Eilenberg--Steenrod axioms for $H_\bullet$]
\label{thm:eilenbergsteenrodaxioms}
The objects $H_\bullet(B;X)$ and $H_\bullet(B,B_0;X)$ satisfy the following
analogues of the Eilenberg--Steenrod axioms:
\begin{enumerate}
\item\label{it:homotopyinvariance}
	(Homotopy invariance)
	Suppose $(f,\phi) \colon (A,X) \to (B,Y)$ is a cartesian morphism in $\pC$
	with $f\colon A \to B$ a weak equivalence. Then the induced map
	\[
		(f,\phi)_\bullet \colon H_\bullet (A;X) \longto H_\bullet(B;Y)
	\]
	is an equivalence.
\item\label{it:additivity}
	(Additivity)
	Given spaces $B_\alpha$, $\alpha \in I$, and an object 
	$X\in \calC_{/\bigsqcup_\alpha B_\alpha}$, the map
	\[
		\bigvee_{\alpha\in I} 
		H_\bullet(B_\alpha;X | B_\alpha)
		\longto %
		H_\bullet \Bigl(\bigsqcup_{\alpha \in I} B_\alpha; X\Bigr)     
	\]
	induced by the inclusions 
	$B_\alpha \incl \bigsqcup_{\alpha \in I} B_\alpha$
	is an equivalence. 
\item\label{it:exactness}
	(Exactness)
	Given a pair $(B,B_0)$ of spaces and an object $X \in \calC_{/B}$,
	there is a cofiber sequence
	\[
		H_\bullet(B_0;X|B_0) \longto H_\bullet(B;X) \longto H_\bullet(B,B_0;X)
	\]
	where the maps are induced by the inclusions $B_0\incl B$ and 
	$B = (B,\emptyset) \incl (B,B_0)$.
\item\label{it:excision}
	(Excision)
	Suppose $B$ is a space and 
	$B_0, B_1 \subset B$ are open subspaces
	such that $B = B_0\cup B_1$. Write $B_{01} = B_0\cap B_1$.
	Then for every object $X \in \calC_{/B}$,
	the map
	\[
		H_\bullet(B_1, B_{01}; X | B_1) \longto H_\bullet(B,B_0;X)
	\]
	induced by the inclusion $(B_1, B_{01}) \incl (B,B_0)$
	is an equivalence.
\end{enumerate}
\end{thrm}

\begin{rem}
\label{rk:additivitystrengthening}
The additivity property (\ref{it:additivity})
is a special case of the stronger result 
Proposition~\ref{prop:hbulletandcolimits}
(or Theorem~\ref{thm:cosheaf}).
\end{rem}

\begin{proof}[Proof of Theorem~\ref{thm:eilenbergsteenrodaxioms}]
(\ref{it:homotopyinvariance}):
By Corollary~\ref{cor:pccartcocartmor}, the image of $(f,\phi)$ in $\hpC$
is cartesian. By Proposition~\ref{prop:wecartiffopcart},
it is also opcartesian, so the claim follows from Proposition~\ref{prop:opcarttoequiv}.

(\ref{it:additivity}): 
As pointed out in Remark~\ref{rk:additivitystrengthening},
the claim is a special case of Proposition~\ref{prop:hbulletandcolimits}.

(\ref{it:exactness}):
The claim follows from the construction of $H_\bullet(B,B_0;X)$.

(\ref{it:excision}):
Let $I$ be the opposite of the poset of non-empty subsets of $\{0,1\}$.
Then the functor $U\colon I \to \calU(B)$, $i\mapsto B_i$, is an admissible 
$I$--indexed open cover of $B$ 
in the sense of Definition~\ref{def:admissibleindexedopencover}.
We have a commutative cube 
(or more precisely, a map $\Delta^1 \times \Delta^1 \times \Delta^1 \to \calC$)
\[\xymatrix@!0@C=8em@R=9ex{
	&
	H_\bullet(B_{01}; X|B_{01})
	\ar@{-->}[dddl]
	\ar[rr]
	\ar[dl]
	\ar[dd]|!{[dl];[dr]}\hole
	&&
	H_\bullet(B_{1}; X|B_{1})
	\ar@{-->}[dddl]
	\ar[dd]
	\ar[dl]
	\\
	H_\bullet(B_{0}; X|B_{0})
	\ar[dd]
	\ar[rr]
	&&
	H_\bullet(B; X)
	\ar[dd]
	\\
	&
	0
	\ar[rr]|!{[ur];[dr]}\hole
	\ar[dl]_\homot
	&&
	H_\bullet(B_{1},B_{01}; X|B_{1})
	\ar[dl]
	\\
	0
	\ar[rr]
	&&
	H_\bullet(B,B_{0}; X)
}\]
constructed as follows:
the top square is the pushout square 
obtained by applying  Theorem~\ref{thm:cosheaf}
with the admissible $I$--indexed open cover $U$,
and the rest of the square is obtained by 
applying the cofibre functor 
\[
	\cofib \colon \Fun(\Delta^1,\calC) \longto \Fun(\Delta^1\times \Delta^1,\calC)
\]
to the top square.
In particular, the front and back squares are pushout squares,
and the slanted map at the bottom right is the one we must show
to be an equivalence.
By \cite[Lemma~4.4.2.1]{HTT}, the square indicated with the dashed
arrows is a pushout square, and 
applying \cite[Lemma~4.4.2.1]{HTT} 
again shows that the bottom square is a pushout square.
As the slanted map on the left in the bottom square 
is an equivalence, so is the map on the right.
\end{proof}

We note that the absolute versions of the homotopy invariance
and additivity axioms stated in Theorem~\ref{thm:eilenbergsteenrodaxioms}
readily imply the following relative variants
\begin{cor}[Homotopy invariance, relative version]
\label{cor:relhomotopyinvariance}
	Let $(f,\phi) \colon ((A,A_0),X) \to ((B,B_0),Y)$ 
	be a cartesian morphism in $\pCrel$
	such that $f\colon A \to B$ and $f| \colon A_0 \to B_0$ 
	are weak equivalences. Then the induced map
	\[
		(f,\phi)_\bullet \colon H_\bullet (A,A_0;X) \longto H_\bullet(B,B_0;Y)
	\]
	is an equivalence.
\qed
\end{cor}
\begin{cor}[Additivity, relative version]
\label{cor:reladditivity}
	Given pairs of spaces $(B_\alpha,B'_\alpha)$, $\alpha \in I$, and an object 
	$X\in \calC_{/\bigsqcup_\alpha B_\alpha}$, the map
	\[
		\bigvee_{\alpha\in I} 
		H_\bullet(B_\alpha,B'_\alpha; X | B_\alpha)
		\longto %
		H_\bullet \Bigl(
			\bigsqcup_{\alpha \in I} B_\alpha, \bigsqcup_{\alpha \in I} B'_\alpha; X
		\Bigr)     
	\]
	induced by the inclusions 
	$(B_\alpha,B'_\alpha) 
	\incl 
	(\bigsqcup_{\alpha \in I} B_\alpha, \bigsqcup_{\alpha \in I} B'_\alpha)$
	is an equivalence. 	
\qed
\end{cor}

It is possible to strengthen the excision axiom.
\begin{defn}
Call a commutative square 
\begin{equation}
\label{eq:bsquare}
\vcenter{\xymatrix{
	B_{01}
	\ar[r]^{f_{01,1}}
	\ar[d]_{f_{01,0}}
	&
	B_1
	\ar[d]^{f_{1,\emptyset}}
	\\
	B_0
	\ar[r]^{f_{0,\emptyset}}
	&
	B
}}
\end{equation}
of spaces and continuous maps, viewed as a functor $F\colon I^\rightcone\to \calT$
where $I$ is the opposite of the poset of nonempty subsets of $\{0,1\}$,
\emph{assembling} if for all objects $X \in \calC_{/B}$, the composite 
\[
	NI^\rightcone \xto{\ X_F\ } \pC \xto{\ H_\bullet\ } \calC
\]
of the lift of Definition~\ref{def:liftxf}
and $H_\bullet$
is a colimit diagram,
so that we have a pushout square
\[\xymatrix{
	H_\bullet(B_{01};\, f_{01,\emptyset}^\ast X)
	\ar[r]
	\ar[d]
	&
	H_\bullet(B_1;\, f_{1,\emptyset}^\ast X)
	\ar[d]
	\\
	H_\bullet(B_0;\, f_{0,\emptyset}^\ast X)
	\ar[r]
	&
	H_\bullet(B;\, X)
}\]
where $f_{01,\emptyset} = f_{01,1} f_{1,\emptyset}$.
\end{defn}

\begin{example}
\label{ex:opencoverassembling}
Suppose in \eqref{eq:bsquare} the spaces $B_0$ and $B_1$ are open subsets
of $B$, $B_{01} = B_0 \cap B_1$,  $B = B_0 \cup B_1$, and all 
morphisms are inclusions. Then \eqref{eq:bsquare} is assembling by 
Theorem~\ref{thm:cosheaf}.
\end{example}

The proof of 
Theorem~\ref{thm:eilenbergsteenrodaxioms}(\ref{it:excision})
generalizes to show
\begin{prop}
Suppose \eqref{eq:bsquare} is assembling with $f_{01,1}$ and $f_{0,\emptyset}$ 
inclusions of subspaces. Then the map 
\[
	H_\bullet(B_1, B_{01}; f_{1,\emptyset}^\ast X) \longto H_\bullet(B,B_0;X)
\]
induced by $f_{1,\emptyset}$ is an equivalence. \qed
\end{prop}

Our next goal is to show that homotopy cocartesian squares in $\calT$ 
are assembling.

\begin{lemma}
\label{lm:assemblingweinvariance}
Suppose 
\[\xymatrix@C-1.1em@R-2.5ex{
	B'_{01}
	\ar[rr]
	\ar[dd]
	\ar[dr]^\sim
	&
	&
	B'_1
	\ar[dd]|!{[dl];[dr]}\hole
	\ar[dr]^\sim
	\\
	&
	B_{01}
	\ar[dd]
	\ar[rr]
	&&
	B_1
	\ar[dd]
	\\
	B'_0
	\ar[rr]|!{[ur];[dr]}\hole
	\ar[dr]_\sim
	&
	&
	B'
	\ar[dr]_\sim
	\\
	&
	B_0
	\ar[rr]
	&&
	B
}\]
is a commutative cube of spaces and continuous maps where all the slanted
arrows are weak equivalences. Then the front square is assembling
if and only if the back square is.
\end{lemma}
\begin{proof}
The claim follows from the homotopy invariance axiom (Theorem~\ref{thm:eilenbergsteenrodaxioms}(\ref{it:homotopyinvariance}))
together with the observation that the pullback functor $\calC_{/B} \to \calC_{/B'}$
induced by the weak equivalence $B' \xto{\,\sim\,} B$ is an equivalence
of $\infty$--categories.
\end{proof}

\begin{prop}
Any homotopy cocartesian square in $\calT$ (with respect to the Strøm model structure)
is assembling.
\end{prop}
\begin{proof}
Suppose \eqref{eq:bsquare} is homotopy cocartesian.
Let $M_{f_1}$ be the mapping cylinder of $f_1$, and let 
\begin{equation}
\label{eq:bmusq}
\vcenter{\xymatrix{
	B_{01}
	\ar[r]
	\ar[d]
	&
	M_{f_1}
	\ar[d]
	\\
	B_0
	\ar[r]
	&
	U
}}
\end{equation}
be a pushout square.
Since the map $B_{01} \incl M_{f_1}$  is a cofibration, \eqref{eq:bmusq}
is homotopy cocartesian. It follows that the evident map from $U$ to $B$
is a homotopy equivalence; see \cite[Prop.~13.5.10]{Hirschhorn}.
Notice that the space $U$ is homeomorphic to 
\[
 	B_0 \cup_{B_{01}\times\{0\}} (B_{01}\times I) \cup_{B_{01}\times\{1\}} B_1.
\]
Let $U_0, U_1 \subset U$ be the subspaces corresponding to
the subspaces
\[
	B_0 \cup_{B_{01}\times \{0\}} \bigl(B_{01}\times [0,3/4)\bigr)
	\qquad\text{and}\qquad
	\bigl(B_{01}\times (1/4,1]\bigr)  \cup_{B_{01}\times \{1\}} B_1,  %
\]
respectively, and let $U_{01} = U_0 \cap U_1$. Then the square
\begin{equation}
\label{eq:usquare}
\vcenter{\xymatrix{
	U_{01}
	\ar[r]
	\ar[d]
	&
	U_1
	\ar[d]
	\\
	U_0
	\ar[r]
	&
	U
}} 
\end{equation}
is assembling by Example~\ref{ex:opencoverassembling},
so the claim follows from Lemma~\ref{lm:assemblingweinvariance}
by noticing that \eqref{eq:usquare} admits a map to \eqref{eq:bsquare}
where all the morphisms are homotopy equivalences.
\end{proof}

Let $(A,A_0)$ and $(B,B_0)$ be pairs of spaces,
and let $X \in \ho(\calC_{/A})$ and $Y \in \ho(\calC_{/B})$.
We conclude the section by sketching
a generalization of the cross product
\begin{equation}
\label{eq:abshbulletcross}
	\times 
	\colon
	H_\bullet(A;X) \tensor H_\bullet(B;Y) 
	\xto{\ \homot\ } 
	H_\bullet(A\times B; X\exttensor Y)
\end{equation}
in $\ho(\calC)$ to a cross product
\begin{equation}
\label{eq:relhbulletcross}
	\times
	\colon
	H_\bullet(A,A_0;X) \tensor H_\bullet(B,B_0;Y) 
	\longto
	H_\bullet((A,A_0)\times (B,B_0); X\exttensor Y)
\end{equation}
in $\ho(\calC)$ where 
\[
	(A,A_0)\times (B,B_0) = (A\times B, A \times B_0 \cup A_0 \times B).
\]
From the symmetric monoidal structure on $\calC$, we obtain an $\infty$--functor
$\tensor \colon \calC\times \calC \to \calC$ inducing 
the symmetric monoidal product on $\ho(\calC)$, and using this $\infty$--functor,
it is straightforward to construct an $\infty$--functor
$\exttensor \colon \pC \times \pC \to \pC$ inducing the 
external tensor product $\exttensor \colon \hpC \times \hpC \to \hpC$
upon passage to homotopy categories.
Since the product $\exttensor$ on $\hpC$ preserves opcartesian morphisms,
so does by Corollary~\ref{cor:pccartcocartmor}
the $\infty$--functor $\exttensor$ on $\pC$, resulting 
in a refinement of \eqref{eq:abshbulletcross} into 
a natural equivalence between $\infty$--functors $\pC\times \pC \to\calC$.
We have a commutative cube in $\calC$ 
(or more precisely, 
a map $\Delta^1\times \Delta^1 \times \Delta^1 \to \calC$) 
\begin{equation}
\label{eq:ptimesdiag}
\vcenter{\xymatrix@!0@C=3em@R=6ex{
	&
	H_\bullet(A_0;X) \tensor H_\bullet(B_0;Y)
	\ar[ddl]
	\ar[dddrr]
	\ar[rrrrrr]^\times_\homot
	&&&&&&
	H_\bullet(A_0\times B_0; X\exttensor Y)
	\ar[ddl]
	\ar[dddrr]
	\\ \\
	H_\bullet(A;X)\tensor H_\bullet(B_0;Y)
	\ar[dddrr]
	\ar[rrrrrr]^\times_\homot|!{[uur];[drrr]}\hole
	&&&&&&
	H_\bullet(A\times B_0; X\exttensor Y)
	\ar[dddrr]|!{[lld];[drrr]}\hole
	\\
	&&&
	H_\bullet(A_0;X) \tensor H_\bullet(B;Y)
	\ar[ddl]
	\ar[rrrrrr]^\times_\homot
	&&&&&&
	H_\bullet(A_0\times B; X\exttensor Y)
	\ar[ddl]
	\\ \\
	&&
	P
	\ar@{-->}[rrrrrr]^\times
	&&&&&&
	H_\bullet(A\times B_0 \cup A_0 \times B; X\exttensor Y)	
}}
\end{equation}
where we have continued to write $X$, $Y$, and $X\exttensor Y$
for the restrictions of these parametrized objects to subspaces;
where the left face is a pushout square;
where the right face is induced by inclusion maps;
where the two faces meeting along the topmost horizontal edge
arise from the naturality of $\times$; 
and where the rest of the cube is obtained by composing 
the simplices already indicated and by applying the 
universal property of pushouts. In particular, the dashed
morphism arises from the universal property of $P$.
We have an induced map $P \to H_\bullet(A;X)\tensor H_\bullet(B;Y)$,
and from the diagram 
\[\xymatrix@C-1em{
	H_\bullet(A_0;X) \tensor H_\bullet(B_0;Y)
	\ar[r]
	\ar[d]
	&
	H_\bullet(A_0;X) \tensor H_\bullet(B;Y)
	\ar[d]
	\\
	H_\bullet(A;X) \tensor H_\bullet(B_0;Y)
	\ar[r]
	\ar[d]
	&
	P
	\ar[r]
	\ar[d]
	&
	H_\bullet(A;X) \tensor H_\bullet(B;Y)
	\ar[d]
	\\
	0
	\ar[r]
	&
	H_\bullet(A_0;X) \tensor H_\bullet(B,B_0;Y)
	\ar[r]
	\ar[d]
	&
	H_\bullet(A;X) \tensor H_\bullet(B,B_0;Y)
	\ar[d]
	\\
	&
	0
	\ar[r]
	&
	H_\bullet(A,A_0;X) \tensor H_\bullet(B,B_0;Y)
}\]
where all squares are pushouts
we conclude that the cofibre of this map is 
$H_\bullet(A,A_0;X) \tensor H_\bullet(B,B_0;Y)$.
The desired map \eqref{eq:relhbulletcross}
is now obtained by passing to cofibres
\[\xymatrix{
	P
	\ar[r]^(0.47)\times
	\ar[d]
	&
	H_\bullet(A\times B_0 \cup A_0 \times B; X\exttensor Y)	
	\ar[d]
	\\
	H_\bullet(A;X) \tensor H_\bullet(B;Y)
	\ar[r]^(0.47)\times_(0.47)\homot
	\ar[d]
	&
	H_\bullet(A\times B; X\exttensor Y)
	\ar[d]
	\\
	H_\bullet(A,A_0;X) \tensor H_\bullet(B,B_0;Y)
	\ar@{-->}[r]^(0.47)\times
	&
	H_\bullet((A,A_0)\times (B,B_0); X\exttensor Y)
}\]
where the top horizontal map is the dashed morphism from
\eqref{eq:ptimesdiag}.

\begin{prop}
Suppose the square 
\[\xymatrix{
	A_0\times B_0
	\ar[r]
	\ar[d]
	&
	A\times B_0
	\ar[d]
	\\
	A_0\times B
	\ar[r]
	&
	A\times B_0 \cup A_0\times B
}\]
where all the maps are inclusions is assembling.
Then the relative cross product
\[
	\times
	\colon
	H_\bullet(A,A_0;X) \tensor H_\bullet(B,B_0;Y) 
	\longto
	H_\bullet((A,A_0)\times (B,B_0); X\exttensor Y)
\]
is an equivalence.
\end{prop}
\begin{proof}
The assumption implies that the dashed morphism in 
\eqref{eq:ptimesdiag} is an equivalence,
from which the claim follows.
\end{proof}

\section{Umkehr maps for \texorpdfstring{$H_\bullet$}{H.}}
\label{sec:umkehrmaps}

In this section, we will present our theory of umkehr maps
between the $H_\bullet$--objects.
The theory will be developed in terms of 
the covariant functor
\[
	H_\bullet \colon \hpC \longto \ho(\calC),
	\quad
	(B,X) \longmapsto H_\bullet(B;X),
	\quad
	(f,\phi) \longmapsto (f,\phi)_\bullet
\]
of Section~\ref{subsec:hbullethc}
and a \emph{contravariant} functor
\[
	H_\bullet\colon \hpC^\dfop \longto \ho(\calC),
	\quad
	(B,X) \longmapsto H_\bullet(B;X),
	\quad
	(f,\theta) \longmapsto (f,\theta)^{\leftarrow}
\]
which should be thought of as sending a morphism 
$(f,\theta)$ in $\hpC^\dfop$ to an umkehr map
associated to~$f$.
Here $\hpC^\dfop \to \calT$ is a certain opfibration
with the same objects as $\hpC$. 

The section is structured as follows.
In Section~\ref{subsec:dfopumkehrs},
we will introduce the opfibration
$\hpC^\dfop \to \calT$ and 
the contravariant functor
$H_\bullet\colon \hpC^\dfop \to \calT$.
While \emph{any} morphism in $\hpC^\dfop$
can be thought to give rise to an umkehr
map upon the application of $H_\bullet$,
of particular interest are the umkehr maps arising from 
morphisms of $\hpC^\dfop$ with distinguished properties. 
Therefore 
in Section~\ref{subsec:dfopcartmors},
we will study cartesian morphism in 
$\hpC^\dfop$, focusing in particular on 
two types of cartesian morphism 
called supercartesian and hypercartesian morphisms
(Definitions~\ref{def:supercart} and \ref{def:hypercart}). 
In addition to proving numerous basic properties
of supercartesian and hypercartesian morphisms
(Proposition~\ref{prop:superandhypercartmorprops}),
we relate the existence of 
cartesian, supercartesian, and hypercartesian
morphisms in $\hpC^\dfop$ covering a map $f \colon A\to B$
in $\calT$ to the existence of a left adjoint
$f^\invshriek \colon \ho(\calC_{/B})\to \ho(\calC_{/A})$
for the base change functor $f_!\colon  \ho(\calC_{/A})\to \ho(\calC_{/B})$
(Proposition~\ref{prop:dfopcartinterpretation}, 
Remark~\ref{rk:dfopcartinterpretation},
Definition~\ref{def:dualizingobject},
and
Proposition~\ref{prop:dualizingobjectsandcartesianmorphisms}).
Finally, in Section~\ref{subsec:cwdualityandhypercartesianmorphisms},
we relate hypercartesian morphisms in $\hpC^\dfop$
to dualizability in the framed bicategory $\Ex_B(\calC)$ 
of Appendix~\ref{app:exbc},
and provide a criterion 
for the existence of hypercartesian morphisms
(Theorems~\ref{thm:hypercartdata} and~\ref{thm:hypercartexistence}).
Moreover, we expand on the implications 
of the existence of supercartesian and hypercartesian morphisms
on the existence and relationships between base change functors
(Section~\ref{subsubsec:superhypercartandbasechangefunctors})
and interpret the 
parametrized Pontryagin--Thom transfer map
of Ando, Blumberg and Gepner \cite[Def.~4.14]{ABGparam}
as a hypercartesian morphism in $\hpSpectra^\dfop$ (Theorem~\ref{thm:abgpthc2}).

Throughout the section, $\calC$ will be 
a symmetric monoidal presentable $\infty$--category
with monoidal product $\tensor$ and monoidal unit $S$.

\subsection{The category \texorpdfstring{$\hpC^\dfop$}{hpC\textasciicircum dfop} and umkehr maps for \texorpdfstring{$H_\bullet$}{H.}}
\label{subsec:dfopumkehrs}

In this subsection, we will introduce the opfibration
$\hpC^\dfop \to \calT$ and define the contravariant functor
$H_\bullet \colon \hpC^\dfop \to \ho(\calC)$.

\begin{defn}
\label{def:dfopdef}
Given an opfibration $\Phi\colon\calE \to \calB$
of categories, we define an opfibration 
$
	\Phi^{\dfop}  
	\colon
	\calE^\dfop
	\to 
	\calB
$
by applying the fibrewise opposite construction 
of Definition~\ref{def:fop} to the fibration
$\Phi^\op \colon \calE^\op \to \calB^\op$ and 
passing to opposite categories:
\[
	\calE^\dfop = ((\calE^\op)^\fop)^\op
	\qquad\text{and}\qquad
	\Phi^{\dfop}  
	= ((\Phi^\op)^\fop)^\op
	\colon
	\calE^\dfop
	\longto 
	\calB.
\]
\end{defn}
Explicitly, a morphism $X \to Y$ in $\calE^\dfop$ covering 
a morphism $f \colon A \to B$ in $\calB$ is an equivalence
class of zigzags
\[
	X \xto{\ \alpha\ } X' \xot{\ \beta\ } Y
\]
of morphisms of $\calE$ 
where $\alpha$ is an opcartesian 
morphism covering $f$ and $\beta$ is a vertical morphism
covering $B$,
and the composite of 
$[X \xto{\ \alpha\ } X' \xot{\ \beta\ } Y]$
and
$[Y \xto{\ \alpha\ } Y' \xot{\ \beta\ } Z]$
is given by composites along the two sides in the diagram
\[\xymatrix@-1em{
    &&
    X''
    \ar@{<-}[dl]_{\gamma'}
    \ar@{<-}[dr]^{\beta'}
    \\
    &
    X'
    \ar@{<-}[dl]_{\alpha}
    \ar@{<-}[dr]^{\beta}
    &&
    Y'
    \ar@{<-}[dl]_\gamma
    \ar@{<-}[dr]^\delta
    \\
    X
    &&
    Y
    &&	
    Z
}\]
where $\gamma'$ is an opcartesian morphism covering $\Phi(\gamma)$
and $\beta'$ is the unique vertical map making the
diamond in the middle commute.
Moreover, the assignment 
$\alpha \mapsto [\bullet \xto{\alpha}\bullet \xot{\id} \bullet]$
defines an isomorphism 
$\calE^\opcart \xto{\ \isom\ }(\calE^\dfop)^\opcart$
of subcategories of opcartesian morphisms.

\begin{defn}[The opfibration $\hpC^\dfop \to \calT$]
\label{def:dfopmordata}
The opfibration $\hpC^\dfop \to \calT$
is obtained by applying
the above construction
to  the bifibration $\hpC \to \calT$.
To distinguish morphisms in $\hpC^\dfop$ 
from those in $\hpC$ notationally,
we write $X\oto Y$ for a morphism from $X$ to $Y$ in $\hpC^\dfop$.
Notice that every morphism $\theta\colon X\oto Y$ in $\hpC^\dfop$
covering a map $f\colon A \to B$ in $\calT$
is represented by a unique
zigzag of the form
\[
	X\longto f_! X \xot{\ \beta\ } Y
\]
where the first arrow is the canonical opcartesian morphism
covering $f$, so that morphisms $X\oto Y$ covering 
$f\colon A\to B$ are in bijection with morphisms 
$Y\to f_! X$ over $B$. We call $\beta$ \emph{the vertical morphism}
$Y\to f_! X$ \emph{determined by} $\theta$ and $\theta$
\emph{the morphism} $X\oto Y$ \emph{determined by} $\beta$.
\end{defn}

The $\exttensor$--product on $\hpC$ induces 
on $\hpC^\dfop$ a symmetric monoidal structure
with tensor product $\exttensor'$ which 
on objects is given by $\exttensor$ and 
which on morphism is given by
\[
	[X_1 \xto{\ \alpha_1\ } X'_1 \xot{\ \beta_1\ } Y_1]
	\exttensor'
	[X_2 \xto{\ \alpha_2\ } X'_2 \xot{\ \beta_2\ } Y_2]
	=
	[X_1 \exttensor X_2 
	\xto{\ \alpha_1\exttensor \alpha_2\ }
	X'_1 \exttensor X'_2
	\xot{\ \beta_1\exttensor \beta_2\ }
	Y_1 \exttensor Y_2];
\] 
notice that the map
$\alpha_1\exttensor \alpha_2$
is opcartesian by Proposition~\ref{prop:exttensoropcarts}.

\begin{defn}[Umkehr maps]
\label{def:hprime}
Our symmetric monoidal functor 
$H_\bullet \colon (\hpC,\exttensor) \to (\ho(\calC),\tensor)$
defines a symmetric monoidal \emph{contravariant} functor
\begin{equation}
\label{def:hprimonobjects}
	H'_\bullet \colon (\hpC^\dfop,\exttensor') 
	\longto 
	(\ho(\calC),\tensor).
\end{equation}
On objects, we set
\[
	H'_\bullet(B,X) = H_\bullet(B; X),
\]
and for a morphism 
$\theta = [X \xto{\ \alpha\ } X' \xot{\ \beta\ } Y] \colon X \oto Y$
covering $f\colon A \to B$, we define $H'_\bullet(\theta)$
to be the composite
\[
	(f,\theta)^{\leftarrow} 
	\colon 
	H_\bullet(B;Y) 
	\xto{\ (\id,\beta)_\bullet\ }	
	H_\bullet(B;X')
	\xto[\homot]{\ (f,\alpha)_\bullet^{-1}\ }
	H_\bullet(A;X).
\]
Notice that the morphism $(f,\alpha)_\bullet$
is invertible by Proposition~\ref{prop:opcarttoequiv}. 
We call $(f,\theta)^{\leftarrow}$ the 
\emph{umkehr map} associated to $\theta$.
\end{defn}

The following proposition is immediate from the definition.
\begin{prop}
\label{prop:opcarttoequivdfop}
The functor
$H'_\bullet \colon (\hpC^\dfop)^\op \to \ho(\calC)$
sends opcartesian morphisms in $\hpC^\dfop$
to equivalences. \qed
\end{prop}

The following proposition follows easily from 
Proposition~\ref{prop:opcartpreservation}.
\begin{prop}
\label{prop:fprimefw}
Suppose $\calD$ is another 
symmetric monoidal presentable $\infty$--category,
and suppose $F\colon \calC \to \calD$ is a 
symmetric monoidal $\infty$--functor which 
admits a right adjoint.
Then $F$ induces a symmetric monoidal functor
\[
	F'_\fw 
	\colon 
	(\hpC^\dfop,\exttensor') 
	\longto 
	(\hpD^\dfop,\exttensor').
\]
defined on objects by $F'_\fw (X) = F_\fw(X)$
and on morphisms by
\[
	\pushQED{\qed} 
	F'_\fw([X \xto{\ \alpha\ } X' \xot{\ \beta\ } Y])
	=
	[ 
		F_\fw(X) 
		\xto{\  F_\fw(\alpha)\ }  
		F_\fw(X') 
		\xot{\  F_\fw(\beta)\ }  
		F_\fw(Y)
	].	
    \qedhere
    \popQED
\]
\end{prop}

\begin{rem}
\label{rk:opcartpresfprimefw}
It is immediate from Proposition~\ref{prop:opcartpreservation}
that $F'_\fw$ preserves opcartesian morphisms.
See Proposition~\ref{prop:hypercartmorpreservation} 
for a result concerning 
preservation of cartesian morphisms under $F'_\fw$.
\end{rem}

\begin{notation}
We will usually simply write $H_\bullet$ for $H'_\bullet$
as there seems to be little need to distinguish the two 
functors in notation. Similarly, we will usually write 
$\exttensor$ for $\exttensor'$ and $F_\fw$ for $F'_\fw$.
\end{notation}

\begin{defn}[Mixed commutativity]
\label{def:mixedcomm}
Consider the square on the left 
\begin{equation}
\label{diag:mixedcomm} 
    \vcenter{\xymatrix{
    	X
    	\ar[r]^\phi
    	\ar[d]|{\circdec}_{\theta}
    	&
    	Y	
    	\ar[d]|{\circdec}^{\kappa}
    	\\
    	Z
    	\ar[r]^\psi
    	&
    	W
    }}
    \qquad\qquad\qquad
    \vcenter{\xymatrix@R-2ex{
    	X
    	\ar[r]^\phi
    	\ar[d]_{\alpha}
    	&
    	Y	
    	\ar[d]^{\gamma}
    	\\
    	X'
    	\ar[r]^{\phi'}
    	&
    	Y'	
    	\\
    	Z
    	\ar[u]^\beta
    	\ar[r]^\psi
    	&
    	W
    	\ar[u]_\delta
    }}
\end{equation}
where 
$\theta = [X \xto{\ \alpha\ } X' \xot{\ \beta\ } Z]$,
$\kappa = [Y \xto{\ \gamma\ } Y' \xot{\ \delta\ } W]$,
and $\phi$ and $\psi$ are morphisms in $\hpC$.
We say that the square \emph{commutes} if the 
diagram on the right commutes for some $\phi'$ covering
the same map as $\psi$ does, in which case we also say
that $\phi'$ \emph{witnesses} the commutativity of the square.
\end{defn}

\begin{rem}
It is easy to see that Definition~\ref{def:mixedcomm}
does not depend on the representatives for $\theta$ and $\kappa$
chosen. Moreover,
the universal property of the opcartesian morphism $\alpha$
implies that there is always a unique morphism $\phi'$ which
covers the same map as $\psi$ does and makes the top square
in the diagram on the right commute.
Thus the square on the left commutes precisely when
this map also makes the bottom
square of the diagram on the right commute.
\end{rem}

\begin{rem}
\label{rk:mixedcommunmix}
Of course, commutative squares in the sense of 
Definition~\ref{def:mixedcomm} are usually not commutative
in the ordinary sense of category theory since
the vertical and horizontal morphisms in the square on the left
in \eqref{diag:mixedcomm} belong to different categories 
and hence cannot be composed. However, in view of 
the canonical isomorphism 
$\hpC^\opcart \xto{\ \isom\ }(\hpC^\dfop)^\opcart$,
there are two special cases in which we may interpret
the square on the left
in \eqref{diag:mixedcomm} as an ordinary diagram in some category:
first, when $\theta$ and $\kappa$ 
are opcartesian morphisms in $\hpC^\dfop$,
we may regard it as a square in $\hpC$; and second,
when $\phi$ and $\psi$ are opcartesian morphisms in $\hpC$,
we may regard it as a square in $\hpC^\dfop$.
It is easily verified that in both cases, the square commutes
in the sense of Definition~\ref{def:mixedcomm}
precisely when it is a commutative square in the usual sense 
of category theory.
\end{rem}

The following proposition is straightforward.
\begin{prop}
\label{prop:umkehrprops}
The umkehr maps of 
Definition~\ref{def:hprime}
have the following properties.
\begin{enumerate}[(i)]
\item
Suppose the square on the left is a commutative square
in the sense of Definition~\ref{def:mixedcomm}
covering the square in the middle.
Then the square on the right commutes.
\begin{equation}
\label{diags:unkmehrmixedcomm}
    \vcenter{\xymatrix{
    	X
    	\ar[r]^\phi
    	\ar[d]|{\circdec}_{\theta}
    	&
    	Y	
    	\ar[d]|{\circdec}^{\kappa}
    	\\
    	Z
    	\ar[r]^\psi
    	&
    	W
    }}
    \qquad\qquad\qquad
    \vcenter{\xymatrix{
    	A
    	\ar[r]^f
    	\ar[d]_h
    	& 
    	B
    	\ar[d]^k
    	\\
    	C
    	\ar[r]^g
    	&
    	D
    }}
    \qquad\qquad\qquad
    \vcenter{\xymatrix{
    	H_\bullet(A;X)
    	\ar[r]^{(f,\phi)_\bullet}
    	&
    	H_\bullet(B;Y)
    	\\
    	H_\bullet(C;Z)
    	\ar[r]^{(g,\psi)_\bullet}
    	\ar[u]^{(h,\theta)^{\leftarrow}}
    	&
    	H_\bullet(D;W)
    	\ar[u]_{(k,\kappa)^{\leftarrow}}
    }}
\end{equation}
\item 
Given maps $\theta \colon X' \oto  X$ and 
$\kappa \colon Y' \oto Y$ in $\hpC^\dfop$ covering
$f\colon A' \to A$ and $g\colon B'\to B$, respectively,
the following diagram commutes by monoidality of $H'_\bullet$.
\[\xymatrix@C+3em{
	H_\bullet(A;X) \tensor H_\bullet(B;Y)
	\ar[r]^{(f,\theta)^\leftarrow \tensor (g,\kappa)^\leftarrow}
	\ar[d]_\times^\homot
	&
	H_\bullet(A';X') \tensor H_\bullet(B';Y')
	\ar[d]^\times_\homot
	\\
	H_\bullet(A\times B;X\exttensor Y)
	\ar[r]^{(f\times g, \theta\exttensor \kappa)^\leftarrow}
	&
	H_\bullet(A'\times B';X'\exttensor Y')	
}\]
Here the vertical morphisms are the monoidality constraints of $H'_\bullet$.
\qed
\end{enumerate}
\end{prop}

A key relationship between $\hpC$ and $\hpC^\dfop$
is that the projection formula allows us to form the 
``tensor product'' of a cartesian morphism in $\hpC$
and an arbitary morphism in $\hpC^\dfop$
as long as they cover the same morphism of $\calT$.

\begin{defn}[The product $\oslash$]
\label{def:oslash}
Given morphisms
$\phi \colon X_1 \to X_2$ in $\hpC^\cart$ and 
$\theta = [Y_1 \xto{\alpha} Y'_1 \xot{\beta} Y_2] \colon Y_1 \oto Y_2$
in $\hpC^\dfop$,
both covering $f\colon A \to B$,
we define
\[
	\phi \oslash \theta 
	= 
	[
		X_1 \tensor_A Y_1 
		\xto{\ \phi\tensor_\internal\alpha\ }
		X_2 \tensor_B Y'_1
		\xot{\ \id \tensor_B \beta\ }
		X_2 \tensor_B Y_2
	]
	\colon 
	X_1 \tensor_A Y_1 
	\longoto 
	X_2 \tensor_B Y_2.
\]
Notice that here the morphism
$\phi\tensor_\internal\alpha$ is again opcartesian 
by Proposition~\ref{prop:reformulations}(\ref{it:bc-projformula-ref}),
and that $\phi \oslash \theta$ again covers $f$.
Defining 
\[
	X \oslash Y = X \tensor_\internal Y
\]
on objects, we obtain a functor
\[
	\oslash
	\colon
	\hpC^\cart \times_\calT \hpC^\dfop
	\longto
	\hpC^\dfop
\]
over $\calT$. 
\end{defn}

The category $\hpC^\cart$ inherits
from $\hpC$ the structure of a pseudomonoid in 
the $2$--category $\Cat/\calT$ 
of categories over $\calT$, and 
together with the evident natural 
associativity and unit isomorphisms
\begin{equation}
\label{eq:oslashassocunit} 
	(X\tensor_\internal Y) \oslash Z 
	\isom 
	X\oslash(Y\oslash Z)
	\qquad\text{and}\qquad
	S_B \oslash W
	\isom
	W,
	\quad
	W \in \ho(\calC_{/B}),
\end{equation}
the functor $\oslash$ makes 
$\hpC^\dfop$ into a module over 
$\hpC^\cart$
in the sense that the axioms of an action of a monoidal 
category on a category (see e.g.\ \cite[p.~62]{monCatActions}),
interpreted in the $2$--category $\Cat/\calT$, are satisfied.

\subsection{Cartesian  morphisms in \texorpdfstring{$\hpC^\dfop$}{hpC\textasciicircum dfop}}
\label{subsec:dfopcartmors}

While the behavior of $H_\bullet$ on opcartesian 
morphisms is simple and does not lead to interesting 
umkehr maps (see Proposition~\ref{prop:opcarttoequivdfop}),
its behavior on cartesian morphisms is considerably
more interesting. In this subsection, we will investigate this behavior.

\subsubsection{Ordinary cartesian morphisms in \texorpdfstring{$\hpC^\dfop$}{hpC\textasciicircum dfop}}
\label{subsubsec:dfopordcartmors}

The opfibration $\hpC^\dfop \to \calT$ is usually 
\emph{not} a fibration. Indeed, 
it is straightforward to verify that we have the 
following interpretation of cartesian arrows in 
$\hpC^\dfop$.

\begin{prop}
\label{prop:dfopcartinterpretation}
A morphism $\theta \colon X\oto Y$ 
in $\hpC^\dfop$ covering $f\colon A\to B$
is cartesian if and only if the vertical morphism
$\eta\colon Y \to f_!X$ determined by $\theta$ is universal 
in the sense that for every object 
$Z \in \ho(\calC_{/A})$ 
and morphism 
$\alpha\colon Y\to f_! Z$ in $\ho(\calC_{/B})$,
there exists a unique morphism $\beta\colon X \to Z$ 
in $\ho(\calC_{/A})$
such that the following triangle commutes.
\[
	\pushQED{\qed} 
	\vcenter{\xymatrix{
    	Y
    	\ar[dr]_\alpha 
    	\ar[rr]^\eta
    	&&
    	f_! X
    	\ar[dl]^{f_!\beta}
    	\\
    	&
    	f_! Z
	}}
    \qedhere
    \popQED
\]
\end{prop}

\begin{rem}
\label{rk:dfopcartinterpretation}
Notice that Proposition~\ref{prop:dfopcartinterpretation}
in particular implies that if 
every $Y \in \ho(\calC_{/B})$ admits some cartesian morphism
$\theta_Y \colon f^\invshriek Y \oto Y$ covering $f\colon A \to B$, then 
the assignment $Y \mapsto f^\invshriek Y$ defines a left adjoint for 
$f_!$ with the unit $Y \to f_! f^\invshriek Y$
of the adjunction given by the vertical map determined by $\theta_Y$;
and that conversely, if $f_!$ has a left adjoint $f^\invshriek$, then 
the unit of the adjunction determines for every $Y \in \ho(\calC_{/B})$ 
a cartesian morphism $f^\invshriek Y \oto Y$. 
\end{rem}

The following characterization of cartesian morphisms in $\hpC^\dfop$
reinterprets and generalizes Proposition~\ref{prop:dfopcartinterpretation}.

\begin{prop}
\label{prop:dfopcartmixedprop}
Suppose $\theta \colon X \oto Y$ is a morphism in $\hpC^\dfop$
covering the map $f\colon A \to B$ in $\calT$. Then $\theta$
is cartesian if and only if for every homotopy cartesian square
as on the left below
and all morphisms $\phi$ in $\hpC$
and $\kappa$ in $\hpC^\dfop$ 
as on the right below
covering $v$ and $g$, respectively,
there exists a unique morphism $\psi$ in $\hpC$
making the diagram on the right a commutative
diagram covering the diagram on the left.
\begin{equation}
\label{eq:dfopcartmixeddiags}
	\vcenter{\xymatrix{
		A
		\ar[r]^{\smash{u}}
		\ar[d]_f
		&
		C
		\ar[d]^g
		\\
		B
		\ar[r]^v
		&
		D
	}}
	\qquad\qquad\qquad\qquad
	\vcenter{\xymatrix{
		X
		\ar@{-->}[r]^{\smash{\psi}}
		\ar[d]|\circdec_\theta
		&
		Z
		\ar[d]|\circdec^\kappa
		\\
		Y
		\ar[r]^\phi
		&
		W
	}}
\end{equation}
\end{prop}
\begin{proof}
Let us first prove the reverse implication.
Applied in the special case where
$g = f$, $u = \id_A$, $v = \id_B$, and $\kappa$ is the 
canonical opcartesian morphism $Z \oto f_! Z$,
the assumption on $\theta$
shows that the vertical morphism $Y\to f_!X$ 
determined by $\theta$ satisfies the 
universal property described in 
Proposition~\ref{prop:dfopcartinterpretation}.
Thus the said proposition shows that $\theta$ 
is cartesian.

Let us now prove the forward implication.
Suppose we have been given a homotopy cartesian square 
as on the left in \eqref{eq:dfopcartmixeddiags}
together with maps $\phi$ and $\kappa$
covering $v$ and $g$ as on the 
right in \eqref{eq:dfopcartmixeddiags}.
Consider the diagram
\begin{equation}
\label{diag:decomp} 
\vcenter{\xymatrix@R-2.5ex{
	\mbox{}
	&&
	\\
	X
	\ar@{-->} `u[rru] `[rr]^\psi [rr]
	\ar@{-->}[r]^{\tilde{\psi}}
	\ar[dd]_{\opcart}
	&
	u^\ast Z
	\ar[r]^-{\cart}
	\ar[dd]^{\opcart}
	&
	Z
	\ar[dd]^{\opcart}
	\\
	\\
	f_! X
	\ar@{-->}[r]^{f_! \tilde{\psi}}
	&
	f_! u^\ast Z
	\ar[r]^-\nu_-{\cart}
	&
	g_! Z
	\\
	\\
	Y
	\ar[uu]^{\eta_{\theta}}
	\ar[rr]^\phi
	\ar[uur]_\alpha
	&&
	W
	\ar[uu]_{\eta_{\kappa}}
}}
\end{equation}
Here the unnamed morphisms are 
the evident cartesian or opcartesian morphisms;
$\eta_\theta$ and $\eta_\kappa$ are the vertical morphisms
determined by $\theta$ and $\kappa$, respectively;
$\nu$ is the unique morphism,
obtained from the universal property of the opcartesian morphism
$u^\ast Z \to f_!u^\ast Z$,
making the square on the top right in the diagram 
a commutative square covering the left hand square in
\eqref{eq:dfopcartmixeddiags};
and 
$\alpha$ is 
the unique morphism over $B$ making the 
trapezoid at the bottom commutative,
obtained by observing that $\nu$
is cartesian by Proposition~\ref{prop:commrelshriekinterpretation2}.
Notice that composition with 
the cartesian morphism $u^\ast Z \to Z$ 
gives a bijection between morphisms $\psi\colon X\to Z$
making the square on the right in 
\eqref{eq:dfopcartmixeddiags}
a commutative square over the square on the left in 
\eqref{eq:dfopcartmixeddiags}
and morphisms $\tilde{\psi} \colon  X \to u^\ast Z$
for which the triangle at the bottom left in 
\eqref{diag:decomp} commutes.
By Proposition~\ref{prop:dfopcartinterpretation},
there exists a unique such $\tilde{\psi}$,
so the claim follows.
\end{proof}

We record the following analogue of 
Proposition~\ref{prop:commrelshriekinterpretation2}
for morphisms in $\hpC^\dfop$.

\begin{prop}
\label{prop:commrelinvshriek}
Let
\begin{equation}
\label{sq:commrelinvshriek} 
\vcenter{\xymatrix{
	Z
	\ar[r]|\circdec^{\beta}
	\ar[d]|\circdec_{\mu}
	& 
	W
	\ar[d]|\circdec^{\nu}
	\\
	X
	\ar[r]|\circdec^{\alpha}
	& 
	Y
}}
\end{equation}
be a commutative square in $\hpC^\dfop$ covering a 
homotopy cartesian square in $\calT$,
and assume that $\beta$ is cartesian and $\nu$ is opcartesian.
Then $\alpha$ is cartesian in $\hpC^\dfop$
if and only if $\mu$ is opcartesian in $\hpC^\dfop$.
\end{prop}

\begin{proof}
Suppose 
\[\xymatrix{
	C
	\ar[r]^g
	\ar[d]_u
	&
	D
	\ar[d]^v
	\\
	A
	\ar[r]^f
	&
	B
}\]
is the homotopy cartesian square covered by~\eqref{sq:commrelinvshriek}.
Let us first assume that $\mu$ is opcartesian and show 
that $\alpha$ is cartesian.
Expanding out~\eqref{sq:commrelinvshriek},
we obtain the two squares in the back of the diagram
\begin{equation}
\label{diag:commrelinvshriekproof}
\vcenter{\xymatrix@!0@C=5em@R=8ex{
	Z 
	\ar[rr]^(0.3)\opcart
	\ar[dd]_(0.7)\mu^(0.7)\opcart
	\ar[dr]_{\mathrm{(2)}}
	&&
	g_! Z
	\ar[dd]|!{[dl];[dr]}\hole_(0.7){\opcart}
	\ar[dr]_{g_!\mathrm{(2)}}
	&&
	W
	\ar[ll]_(0.3){\eta_\beta}
	\ar[dd]^(0.7)\nu_(0.7)\opcart
	\ar[dl]^{\mathrm{(1)}}
	\\
	&
	u^\ast T
	\ar[dd]^(0.7)\cart
	\ar[rr]^(0.3)\opcart
	&&
	g_! u^\ast Y
	\ar[dd]_(0.7)\cart
	\\
	X
	\ar[rr]|!{[ur];[dr]}\hole^(0.3){\opcart}
	\ar[dr]_{\mathrm{(3)}}
	&&
	f_! X
	\ar[dr]_{f_! \mathrm{(3)}}
	&&
	Y
	\ar[dl]^\xi
	\ar[ll]|!{[ul];[dl]}\hole^(0.3){\eta_\alpha}
	\\
	&
	T
	\ar[rr]_(0.3)\opcart
	&&
	f_! T
}}
\end{equation}
Here $\eta_\alpha$ and $\eta_\beta$ are the vertical 
morphisms determined by $\alpha$ and $\beta$, respectively,
the horizontal morphisms labeled $\opcart$ are
the canonical opcartesian morphisms, 
and the middle vertical morphism is the unique morphism making 
the left hand square commutative. The 
right hand square then commutes by the commutativity 
of~\eqref{sq:commrelinvshriek}.
Notice that  the middle vertical morphism
is opcartesian by
the analogues of Proposition~\ref{prop:cartmorprops}(\ref{it:compiscart})
and
(\ref{it:cartsourceiso})
for opcartesian morphisms.

Suppose now $T \in \ho(\calC_{/A})$
and $\xi \colon Y \to f_! T$ in $\ho(\calC_{/B})$.
In view of Proposition~\ref{prop:dfopcartinterpretation},
our task is to show that 
there is a unique morphism (3) over $A$
so that the bottom triangle in~\eqref{diag:commrelinvshriekproof}
commutes. To construct morphism $(3)$, let us first construct
the front square in diagram~\eqref{diag:commrelinvshriekproof}.
The left hand vertical morphism is the canonical cartesian 
morphism, and two horizontal morphisms are
the canonical opcartesian morphisms. The right hand 
vertical morphism is the unique morphism covering $v$
which makes the square commutative.
Notice that this morphism is cartesian by 
Proposition~\ref{prop:commrelshriekinterpretation2}.
Now, by the universal property of cartesian morphisms, there 
exists a unique morphism~(1) over $D$ 
which makes the parallelogram on the right commutative.
In view of Proposition~\ref{prop:dfopcartinterpretation},
the assumption that $\beta$ is cartesian
implies that there exists a unique morphism~(2) over $C$
making the top of diagram~\eqref{diag:commrelinvshriekproof}
commutative.
Taking~(3) to be the unique morphism over $A$
which makes the parallelogram on the left commutative
now makes the bottom of 
diagram~\eqref{diag:commrelinvshriekproof} commutative.
Moreover, the uniqueness of (2) implies the desired uniqueness of~(3).
This concludes the proof that $\alpha$ is cartesian.

It remains to prove that $\mu$ is opcartesian when $\alpha$ is
cartesian. Choose an opcartesian morphism $\mu' \colon Z \oto X'$
covering $u$. There then exists a unique morphism
$\alpha'\colon X' \oto Y$ covering $f$
with the property that $\alpha'\mu' = \nu \beta$.
By the implication already proven,
$\alpha'$ is cartesian, so by 
Proposition~\ref{prop:cartmorprops}(\ref{it:cartsourceiso})
there exists a vertical isomorphism $\theta \colon X' \xoto{\homot} X$
such that $\alpha' = \alpha \theta$.
Now $\alpha \theta \mu' = \alpha' \mu' = \nu \beta = \alpha \mu$,
so the assumption that $\alpha$ is cartesian implies that 
$\theta\mu' = \mu$, and the claim follows from the 
analogues of 
Proposition~\ref{prop:cartmorprops}(\ref{it:isosarecart})
and  (\ref{it:compiscart})
for opcartesian morphisms.
\end{proof}

\begin{prop}
\label{prop:cartoslashproperty}
Suppose $\theta$ is a cartesian morphism in $\hpC^\dfop$
covering $f\colon A\to B$. Then for any cartesian morphism $\phi$
in $\hpC$ covering $f$ whose target is 
dualizable
in  $\ho(\calC_{/B})$, the product 
$\phi \oslash \theta$ is cartesian.
\end{prop}
\begin{proof}
Suppose $\theta \colon X \oto Y$ and $\phi \colon Z \to W$. Write
$W^{\vee}$ for the dual of $W$,
and observe that the duality between $W$ and $W^{\vee}$
in $\ho(\calC_{/B})$
pulls back to a duality between $Z$ and the object $Z^{\vee} = f^\ast(W^{\vee})$
in $\ho(\calC_{/A})$.
For every $K \in \ho(\calC_{/A})$  we have 
natural bijections
\begin{align*}
    \Hom_{\ho(\calC_{/A})}(Z \tensor_A X, K)
    &\isom 
    \Hom_{\ho(\calC_{/A})}(X, Z^{\vee} \tensor_A K)
    \\
	&\isom
	\Hom_{\ho(\calC_{/B})}(Y, f_! (Z^{\vee} \tensor_A K))
    \\
    &=
   	\Hom_{\ho(\calC_{/B})}(Y, f_! (f^\ast(W^{\vee}) \tensor_A K))
	\\
	&\isom
	\Hom_{\ho(\calC_{/B})}(Y, W^{\vee} \tensor_A f_! K)
    \\
	&\isom
	\Hom_{\ho(\calC_{/B})}(W \tensor_B Y,  f_! K)
\end{align*}
where the first and last bijections arise from the dual pairs $(Z,Z^{\vee})$
and $(W,W^{\vee})$;
the second one 
is induced by $\theta$ (see Proposition~\ref{prop:dfopcartinterpretation});
and the penultimate one arises from the projection formula.
Moreover, writing $\tilde{\eta} \colon Y \to f_! X$ for the 
map determined by $\theta$, a tedious inspection shows that 
the composite bijection above sends a morphism $\sigma \colon Z \tensor_A X \to K$
to the composite
\[
	W\tensor_B Y 
	\xto{\ \id \tensor \tilde{\eta}\ }
	W\tensor_B f_! X
	\xto{\ \homot\ }
	f_!( f^\ast W\tensor_A X)
	\xto{\ \homot\ }
	f_!( Z\tensor_A X)
	\xto{\ f_! (\sigma)\ }
	f_! K	
\]
where the first equivalence is given by the projection formula and the second
by the isomorphism $f^\ast(W) \isom Z$
of Proposition~\ref{prop:cartmorprops}(\ref{it:cartoveriso}).
But here the composite of the first three morphisms is the morphism 
$W\tensor_B Y  \to 	f_!( Z\tensor_A X)$
determined by $\phi\oslash\theta\colon Z \tensor_A X \oto W \tensor_B Y$,
so the claim follows from Proposition~\ref{prop:dfopcartinterpretation}.
\end{proof}

\subsubsection{Supercartesian and hypercartesian morphisms 
in \texorpdfstring{$\hpC^\dfop$}{hpC\textasciicircum dfop}
}

We are especially interested in two
strong types of cartesian 
morphisms in $\hpC^\dfop$ we call 
supercartesian and hypercartesian morphisms.
Our next goal is to define these morphisms and 
establish many of their basic properties.
The existence of interesting examples of such morphisms 
is demonstrated later in 
Section~\ref{subsec:cwdualityandhypercartesianmorphisms}.

\begin{defn}[Supercartesian morphisms]
\label{def:supercart}
We call a morphism 
$\theta$ of $\hpC^\dfop$
covering a map $f$ of $\calT$
\emph{supercartesian} if $\phi\oslash \theta$ is 
cartesian for \emph{every} cartesian morphism $\phi$ of $\hpC$
covering $f$; cf.\ Proposition~\ref{prop:cartoslashproperty}.
\end{defn}

To define hypercartesian morphisms, we first need to define
base change of morphisms of $\hpC^\dfop$ along homotopy
cartesian squares.

\begin{prop}
\label{prop:basechange}
Suppose the right hand square below is a homotopy cartesian square
in $\calT$, and suppose in the left hand square
$\kappa$ is a morphism of $\hpC^\dfop$ and 
$\phi$ and $\psi$ are cartesian morphisms of $\hpC$
all covering the corresponding morphisms in the right hand square.
Then there exists a unique morphism $\theta$ of $\hpC^\dfop$
making the left hand square a commutative square over 
the right hand square.
\[
	\vcenter{\xymatrix{
    	X
    	\ar[r]^\phi_\cart
    	\ar@{-->}[d]|{\circdec}_{\theta}
    	&
    	Y	
    	\ar[d]|{\circdec}^{\kappa}
    	\\
    	Z
    	\ar[r]^\psi_\cart
    	&
    	W
    }}
	\qquad\qquad\qquad
	\vcenter{\xymatrix{
		D
		\ar[r]^g
		\ar[d]_q
		&
		E
		\ar[d]^p
		\\
		A
		\ar[r]^f_{\phantom{\cart}}
		&
		B
	}}
\]
\end{prop}

\begin{proof}
Suppose $\kappa = [Y \xto{\ \gamma\ } Y' \xot{\ \delta\ } W]$.
To define $\theta$, choose a cartesian $\phi' \colon X' \to Y'$
covering $f$, and use the universal property of cartesian morphisms
to find maps $\alpha \colon X \to X'$ covering $q$
and $\beta \colon Z \to X'$ over $A$
such that $\phi' \alpha = \gamma \phi$
and $\phi' \beta = \delta \psi$.
By Proposition~\ref{prop:commrelshriekinterpretation2},
the map  $\alpha$ is opcartesian. Now 
$\theta = [X \xto{\ \alpha\ } X' \xot{\ \beta\ }Z]$
completes the square on the left as desired.
On the other hand, if 
$\theta' = [X \xto{\ \alpha'\ } X'' \xot{\ \beta'\ }Z]$
is another completion of the left hand square, 
Proposition~\ref{prop:commrelshriekinterpretation2}
implies that 
the morphism $\phi''\colon X'' \to Y'$
witnessing 
commutativity of the completed square is cartesian.
The equivalence $X'' \homot X'$ afforded by 
Proposition~\ref{prop:cartmorprops}(\ref{it:cartsourceiso})
now shows that $\theta' = \theta$, as desired.
\end{proof}

\begin{defn}[Base change]
\label{def:basechange}
In the situation of Proposition~\ref{prop:basechange},
we say that the morphism $\theta$ has been obtained
from $\kappa$ by \emph{base change} along the homotopy 
cartesian square on the right with respect to the cartesian morphisms
$\phi$ and $\psi$, and call the square on the left
a \emph{base change square}. We call $f$ and $g$ the \emph{horizontal morphisms}
of the homotopy cartesian square on the right.
\end{defn}

\begin{defn}[Hypercartesian morphisms]
\label{def:hypercart}
A morphism of $\hpC^\dfop$ is called 
\emph{hypercartesian}
if all morphisms obtained from 
it by base change along homotopy cartesian squares
are supercartesian.
\end{defn}

\begin{prop}[Properties of supercartesian and hypercartesian morphisms]
\label{prop:superandhypercartmorprops}
\mbox{}
\begin{enumerate}[(i)]
\item\label{it:hypercartimpliessupercart}
	Every hypercartesian morphism is supercartesian.
\item\label{it:supercartimpliescart}
	Every supercartesian morphism is cartesian.
\item\label{it:supercartcomp}
	The composite of supercartesian morphisms is supercartesian.
\item\label{it:hypercartcomp}
	The composite of hypercartesian morphisms is hypercartesian.
\item\label{it:isosarehypercart}
	Every isomorphism of $\hpC^\dfop$ is hypercartesian.
\item\label{it:hypercartiffopcart}
	The following are equivalent conditions on a morphism $\theta$
	of $\hpC^\dfop$ covering a weak equivalence of $\calT$:
	(a) $\theta$ is hypercartesian;
	(b) $\theta$ is supercartesian;
	(c) $\theta$ is cartesian;
	(d) $\theta$ is opcartesian.
	(Cf.\ Proposition~\ref{prop:wecartiffopcart}.)
\item\label{it:wecancel}
	Let $\theta_2$ be an opcartesian morphism of $\hpC^\dfop$
	covering a weak equivalence, and let $\theta_1$ be a morphism
	of $\hpC^\dfop$ such that the composite $\theta_1\circ \theta_2$
	is defined. Then $\theta_1$ is 
	cartesian (resp.\ supercartesian or hypercartesian) if and only if
	$\theta_1 \circ \theta_2$ is. (Cf.\ Proposition~\ref{prop:wecancel}.)
\item\label{it:wecancel2}
	Let $\theta_1$ be an opcartesian morphism of $\hpC^\dfop$
	covering a weak equivalence, and let $\theta_2$ be a morphism
	of $\hpC^\dfop$ such that the composite $\theta_1\circ \theta_2$
	is defined. Then $\theta_2$ is 
	cartesian (resp.\ supercartesian or hypercartesian) if and only if
	$\theta_1 \circ \theta_2$ is. (Cf.\ Proposition~\ref{prop:wecancel2}.)
\item\label{it:supercartbc}
	Any morphism obtained from a supercartesian morphism
	by base change along a homotopy cartesian square whose 
	horizontal morphisms are weak equivalences is supercartesian.
\item\label{it:hypercartbc} 
	Any morphism obtained from 
	a hypercartesian morphism by base change
    along a homotopy cartesian square is hypercartesian.
\item\label{it:hypercartpbres}
	A morphism of $\hpC^\dfop$ is hypercartesian
	if and only if morphisms obtained from it by base change
	along pullback squares whose horizontal morphisms are
	fibrations are supercartesian.
\item\label{it:cartoslashhypercart}
	Given a cartesian morphism $\phi$ of $\hpC$ and 
    a supercartesian (resp.\ hypercartesian) morphism $\theta$ 
    of $\hpC^\dfop$
    covering the same morphism of $\calT$, the product 
    $\phi \oslash \theta$ is supercartesian (resp.\ hypercartesian).
\item\label{it:hcexttensorhc}
	If morphisms $\theta_1$ and $\theta_2$ of $\hpC^\dfop$
	are hypercartesian, then so is their
	external tensor product 
	$\theta_1 \exttensor \theta_2$.
\end{enumerate}
\end{prop}

\begin{rem}
Later in Proposition~\ref{prop:hypercartmorpreservation} 
we will establish a result concerning the 
preservation of hypercartesian morphisms under the functors
$F_\fw \colon \hpC^\dfop \to \hpD^\dfop$  
of Proposition~\ref{prop:fprimefw}.
\end{rem}

\begin{proof}[Proof of Proposition~\ref{prop:superandhypercartmorprops}]
(\ref{it:hypercartimpliessupercart}):
The claim follows by considering base change along the degenerate 
homotopy cartesian square where the horizontal maps are identity maps.

(\ref{it:supercartimpliescart}): 
The claim follows from the observation that 
for a morphism $\theta$ of $\hpC^\dfop$ covering $f\colon A\to B$,
the morphism $\phi\oslash \theta$ is equivalent to $\theta$
when $\phi \colon S_A \to S_B$ is the canonical cartesian morphism covering $f$
by using parts~(\ref{it:isosarecart}) and~(\ref{it:compiscart})
of  Proposition~\ref{prop:cartmorprops}.

(\ref{it:supercartcomp}):
Assume $\theta_1$ and $\theta_2$ are
composable supercartesian morphisms of $\hpC^\dfop$ 
covering $f_1$ and $f_2$, respectively.
Given a cartesian morphism $\phi$ of $\hpC$ covering $f_1\circ f_2$, we may factor
$\phi$ as a composite $\phi = \phi_1 \circ \phi_2$ where $\phi_1$ and $\phi_2$
are cartesian morphisms covering $f_1$ and $f_2$, respectively. Now 
the morphism
\[
	\phi \oslash (\theta_1 \circ \theta_2) 
	=
	(\phi_1\circ \phi_2) \oslash (\theta_1 \circ \theta_2) 
	=
	(\phi_1 \oslash \theta_1) \circ (\phi_2 \oslash \theta_2)
\]
is cartesian by Proposition~\ref{prop:cartmorprops}(\ref{it:compiscart}).

(\ref{it:hypercartiffopcart}): 
The implication 
(a) $\Rightarrow$ (b) follows from (\ref{it:hypercartimpliessupercart}) and
the implication
(b) $\Rightarrow$ (c) from (\ref{it:supercartimpliescart}).
Let us show that
(c) $\Leftrightarrow$ (d). 
Suppose $f\colon A \to B$ is a weak equivalence. 
We then have an adjoint equivalence 
$(f^\ast,f_!)$. Write $\eta$ for the unit of this adjoint equivalence.
In view of Remark~\ref{rk:dfopcartinterpretation}, 
for every $Y\in \ho(\calC_{/B})$ the morphism
$\kappa_Y \colon f^\ast Y \oto Y$ determined by 
$\eta \colon Y \to f_!f^\ast Y$ is cartesian. Moreover,
since $\eta$ is an equivalence,
the morphism $\kappa_Y$ is also opcartesian.
Suppose now $\theta \colon X \oto Y$ is a morphism covering $f$.
If $\theta$ is cartesian, by 
Proposition~\ref{prop:cartmorprops}(\ref{it:cartsourceiso})
$\theta$ is isomorphic to $\kappa_Y$ and hence is also opcartesian.
On the other hand, if $\theta$ is opcartesian, then 
making use of the vertical isomorphism $X \to f^\ast f_! X$ 
afforded by the counit of the adjoint equivalence $(f^\ast,f_!)$
and the analogue of 
Proposition~\ref{prop:cartmorprops}(\ref{it:cartsourceiso})
for opcartesian morphisms, we see that 
$\theta$ is
isomorphic to $\kappa_{f_!X}$, and hence is also cartesian.

In view of the equivalence (c) $\Leftrightarrow$ (d) now proven,
the implication (d) $\Rightarrow$ (b) follows from the 
observation that $\phi \oslash \theta$
is opcartesian when $\phi$ is cartesian and $\theta$ is opcartesian.
It remains to show the implication (d) $\Rightarrow$ (a).
Suppose $p\colon E\to B$ is a weak equivalence, 
and let $\kappa \colon Y\oto W$ be 
an opcartesian morphism covering $p$. 
Suppose the square below on the left 
is a base change square covering 
a homotopy cartesian square displayed on the right.
\begin{equation}
\label{sqs:bchc}
	\vcenter{\xymatrix{
    	X
    	\ar[r]^\phi_\cart
    	\ar[d]|{\circdec}_{\theta}
    	&
    	Y	
    	\ar[d]|{\circdec}^{\kappa}
    	\\
    	Z
    	\ar[r]^\psi_\cart
    	&
    	W
    }}
	\qquad\qquad\qquad
	\vcenter{\xymatrix{
		D
		\ar[r]^g
		\ar[d]_q
		&
		E
		\ar[d]^p
		\\
		A
		\ar[r]^f_{\phantom{\cart}}
		&
		B
	}} 
\end{equation}
From Remark~\ref{rk:mixedcommunmix} and 
Proposition~\ref{prop:commrelshriekinterpretation2},
we conclude that $\theta$ is also opcartesian.
As the square on the right is homotopy cartesian, the map 
$q$ must be a weak equivalence. Therefore the implication
(d) $\Rightarrow$ (b) already established shows that 
$\theta$ is supercartesian. 
Thus $\kappa$ is hypercartesian, as desired.

(\ref{it:isosarehypercart}): 
An isomorphism of $\hpC^\dfop$ covers a homeomorphism of $\calT$
and is therefore cartesian by 
Proposition~\ref{prop:cartmorprops}(\ref{it:isosarecart}).
Thus the claim follows from part~(\ref{it:hypercartiffopcart}) already proven.

(\ref{it:hypercartbc}):
The claim follows from the observation that 
the square obtained by stacking two base change
squares for morphisms of $\hpC^\dfop$ together
horizontally is again a base change square.

(\ref{it:wecancel2}):
By part~(\ref{it:hypercartiffopcart}), the opcartesian morphism $\theta_1$
is also a cartesian, supercartesian and hypercartesian morphism 
of $\hpC^\dfop$.
The forward implication now follows from 
Proposition~\ref{prop:cartmorprops}(\ref{it:compiscart})
and parts~(\ref{it:supercartcomp}) and (\ref{it:hypercartcomp})
of the present proposition. 
The reverse implication in the cartesian case follows from 
Proposition~\ref{prop:cartmorprops}(\ref{it:cartfactor}).
Notice that we may interpret the opcartesian morphism $\theta_1$
as an opcartesian morphism in $\hpC$, and that by 
Proposition~\ref{prop:wecartiffopcart} the morphism
$\theta_1$ is then also 
a cartesian morphism of $\hpC$.
By Remark~\ref{rk:mixedcommunmix},
we therefore have a base change square
\[\xymatrix{
	\bullet
	\ar@{=}[r]
	\ar[d]|(0.47){\circdec}_(0.47){\theta_2}
	&
	\bullet
	\ar[d]|(0.47){\circdec}^(0.47){\theta_1\circ \theta_2}
	\\ 
	\bullet
	\ar[r]^{\theta_1}_\cart
	&
	\bullet
}\]
and the reverse implication in the hypercartesian case 
follows from part~(\ref{it:hypercartbc})
already proven.

It remains to prove the reverse implication in the supercartesian case.
Write $f_1$ and $f_2$ for the morphisms of $\calT$ covered by $\theta_1$ 
and $\theta_2$, respectively. Suppose $\phi_2$ is a cartesian 
morphism in $\hpC$ covering $f_2$.
We may choose an opcartesian morphism $\phi_1$ in $\hpC$
covering $f_1$ such that the composite $\phi_1 \circ \phi_2$ is defined.
By Proposition~\ref{prop:wecartiffopcart},
the morphism $\phi_1$ is also cartesian in $\hpC$.
By Proposition~\ref{prop:cartmorprops}(\ref{it:compiscart}), 
the composite $\phi_1\circ\phi_2$ is cartesian, 
so by the assumption that $\theta_1\circ \theta_2$ is supercartesian,
the composite
\[
	(\phi_1 \oslash \theta_1) \circ (\phi_2 \oslash \theta_2)
	=
	(\phi_1\circ \phi_2) \oslash (\theta_1 \circ \theta_2) 
\]
is cartesian.
Since $\theta_1$ is supercartesian, the morphism
$\phi_1 \oslash \theta_1$ is cartesian, 
and hence $\phi_2 \oslash \theta_2$ is cartesian by
Proposition~\ref{prop:cartmorprops}(\ref{it:cartfactor}).
Thus $\theta_2$ is supercartesian.

(\ref{it:supercartbc}):
Consider squares as in \eqref{sqs:bchc} where 
the square on the right is homotopy cartesian,
the square on the left is a base change square covering the square on the right,
$\kappa$ is supercartesian, and $f$ and $g$ are weak equivalences.
By Proposition~\ref{prop:wecartiffopcart}, the morphisms $\phi$ and $\psi$
are also opcartesian, so by Remark~\ref{rk:mixedcommunmix}
we may interpret the square on the left as a 
commutative square in $\hpC^\dfop$. 
By part~(\ref{it:hypercartiffopcart}),
the map $\phi$, interpreted as a
morphism in  $\hpC^\dfop$,
is supercartesian, so the claim follows from parts~(\ref{it:supercartcomp})
and~(\ref{it:wecancel2}) already proven.

(\ref{it:hypercartpbres}):
Given a homotopy cartesian square as on the right in 
\eqref{sqs:bchc},  starting from a factorization of 
$f$ as a composite
$f = \tilde{f} \circ f_{\sim}$ 
where $\tilde{f}$ is a fibration and $f_{\sim}$
is a homotopy equivalence, we can decompose 
the square as
\[\vcenter{\xymatrix@!0@C=3.8em@R=1.7em{
	\mbox{} 
	&&
	\\
	D
	\ar `u[ur] `[rr]^g [rr]	
	\ar[r]^{g_{\sim}}
	\ar[dd]_q
	&
	\tilde{D}
	\ar[r]^{\tilde{g}}
	\ar[dd]
	&
	E
	\ar[dd]^p
	\\ 
	\\
	A
	\ar `d[dr] `[rr]^f [rr]	
	\ar[r]^{f_{\sim}}
	&
	\tilde{A}
	\ar[r]^{\tilde{f}}
	&
	B
	\\
	\mbox{}
	&&
}}\]
where the square on the right is a pullback square and
$g_{\sim}$ is chosen using the universal property of pullbacks
to make the diagram commutative.
As the outer rectangle  and the small square on the right 
are both homotopy cartesian and $f_{\sim}$ is 
a weak equivalence, it follows that $g_{\sim}$ is also 
a weak equivalence; see \cite[Prop.~13.3.14]{Hirschhorn}.
Moreover, the small square on the left is 
also homotopy cartesian; see \cite[Prop.~13.3.15]{Hirschhorn}.
The claim now follows from part~(\ref{it:supercartbc})
already proven.

(\ref{it:hypercartcomp}):
Suppose $\kappa_1$ and $\kappa_2$ are composable hypercartesian morphisms 
covering $p_1 \colon E' \to B$ and $p_2 \colon E \to E'$, 
respectively. In view of part~(\ref{it:hypercartpbres})
already proven, it suffices to show that the base change
of $\kappa = \kappa_1 \circ \kappa_2$ along any pullback
square whose horizontal morphisms are fibrations is 
supercartesian. Suppose the square on the right 
in \eqref{sqs:bchc} is such a square and that the square on the 
left in \eqref{sqs:bchc} is a base change square covering it.
The two squares then decompose as indicated below
\[
	\vcenter{\xymatrix@!0@C=1.9em@R=2.6em{
	    &
    	X
    	\ar[rr]^\phi_\cart
    	\ar `l[dl] `[dd]|{\circdec}_{\theta} [dd]
		\ar[d]|(0.47){\circdec}_(0.47){\theta_2}
    	&&
    	Y
		\ar[d]|(0.47){\circdec}^(0.47){\kappa_2}
		\ar `r[dr] `[dd]|{\circdec}^{\kappa} [dd]
    	\\
		&
    	X
    	\ar[rr]_{\cart}%
		\ar[d]|(0.47){\circdec}_(0.47){\theta_1}
    	&&
    	Y'
		\ar[d]|(0.47){\circdec}^(0.47){\kappa_1}
		&
		\\
		&
    	Z
    	\ar[rr]^\psi_\cart
    	&&
    	W
	}}
	\qquad\qquad\qquad
	\vcenter{\xymatrix@!0@C=1.9em@R=2.6em{
	    &
    	D
    	\ar[rr]^g
    	\ar `l[dl] `[dd]_q [dd]
		\ar[d]
    	&&
    	E
		\ar[d]^(0.47){p_2}
		\ar `r[dr] `[dd]^p [dd]
    	\\
		&
    	D'
    	\ar[rr]
		\ar[d]
    	&&
    	E'
		\ar[d]^(0.47){p_1}
		&
		\\
		&
    	A
    	\ar[rr]^f_{\phantom{\cart}}
    	&&
    	B
	}}
\]
where the two small squares on the right are pullback squares
whose horizontal morphisms are fibrations and where
the two small squares on the left are base change squares
covering the corresponding squares on the right.
The claim now follows from part (\ref{it:supercartcomp}).

(\ref{it:wecancel}):
The forward implication follows from (\ref{it:hypercartiffopcart})
together with Proposition~\ref{prop:cartmorprops}(\ref{it:compiscart})
and parts~(\ref{it:supercartcomp}) and (\ref{it:hypercartcomp})
of the present proposition.
To prove the reverse implication, write 
$f_1$ %
and 
$f_2$ %
for the maps covered by $\theta_1$
and $\theta_2$, respectively.
The reverse implication follows in the cartesian case 
from the characterization of cartesian morphisms given in 
Proposition~\ref{prop:dfopcartinterpretation} 
by using the fact that the assumption that $f_2$ is a weak equivalence 
ensures that the functor 
$(f_2)_!$ %
is an equivalence of categories.
To prove the reverse implication in the supercartesian case,
assume $\phi_1$ is a cartesian morphism
of $\hpC$ covering $f_1$. Pick a cartesian morphism $\phi_2$
of $\hpC$ covering $f_2$. Since $\theta_1 \circ \theta_2$ 
is supercartesian, the morphism
\[
	(\phi_1 \oslash \theta_1) \circ (\phi_2 \oslash \theta_2)
	=
	(\phi_1 \circ \phi_2) \oslash (\theta_1 \circ \theta_2)
\]
is cartesian. As $\phi_2 \oslash \theta_2$ is opcartesian, it follows that
$\phi_1 \oslash \theta_1$ is cartesian by the cartesian case already
proven, as desired. 
Finally,
in view of part~(\ref{it:hypercartpbres})
already proven,
to prove the 
reverse implication in the hypercartesian case,
it suffices to show that the base change of $\theta_1$
along any pullback square whose horizontal morphism
are fibrations is supercartesian. 
Suppose the solid square on the right below is such a pullback
square, and assume that the solid square on the right
is a base change square covering it.
\[
    \vcenter{\xymatrix@!0@C=3.7em@R=2.6em{
       	\bullet
		\ar@{-->}[r]^{\cart}
		\ar@{-->}[d]|(0.5){\circdec}_(0.5){\theta'_2}
		&
		\bullet
		\ar@{-->}[d]|(0.5){\circdec}^(0.5){\theta_2}
		\\
    	\bullet 
    	\ar[r]^\cart
    	\ar[d]|(0.5){\circdec}_(0.5){\theta'_1}
    	&
    	\bullet
    	\ar[d]|(0.5){\circdec}^(0.5){\theta_1}
    	\\
    	\bullet
    	\ar[r]^\cart
    	&
    	\bullet
    }}
    \qquad\qquad\qquad
    \vcenter{\xymatrix@!0@C=3.7em@R=2.6em{
    	\bullet
		\ar@{-->}[r]^{\phantom{\cart}}
		\ar@{-->}[d]_{f'_2}
		&
		\bullet
		\ar@{-->}[d]^{f_2}
		\\
    	\bullet 
    	\ar[r]
    	\ar[d]_{f'_1}
    	&
    	\bullet
    	\ar[d]^{f_1}
    	\\
    	\bullet
    	\ar[r]
    	&
    	\bullet
    }}
\]
Let the dashed square on the right be a pullback square,
and let the dashed square on the left be a base change
square covering it. Notice that then the map $f'_2$
is a weak equivalence (see \cite[Prop.~13.3.14]{Hirschhorn}),
and the map $\theta'_2$ is opcartesian 
(see Remark~\ref{rk:mixedcommunmix} and 
Proposition~\ref{prop:commrelshriekinterpretation2}).
By the assumption that $\theta_1\circ \theta_2$ is 
hypercartesian, the composite $\theta'_1 \circ \theta'_2$
is supercartesian, so the claim follows
from the supercartesian case already proven.

(\ref{it:cartoslashhypercart}):
In the case of supercartesian morphisms, the claim follows from 
the associativity isomorphism
\eqref{eq:oslashassocunit} together with 
parts~(\ref{it:isosarecart})
and~(\ref{it:compiscart})
if Proposition~\ref{prop:cartmorprops}.
Let us now consider the hypercartesian case. 
Suppose the square on the left hand side below
is a base change square covering 
a homotopy cartesian square $S$.
\begin{equation}
	\left(\vcenter{\xymatrix@!0@C=4em@R=3.6em{
    	\bullet
	   	\ar[r]^\cart
    	\ar[d]|{\circdec}_{\rho}
    	&
    	\bullet
    	\ar[d]|{\circdec}^{\phi \oslash \theta}
    	\\
    	\bullet
		\ar[r]_{\phantom{\cart}}^\cart
    	&
    	\bullet
    }}\right)
    \isom
    \left(\vcenter{\xymatrix@!0@C=4em@R=3.6em{
		\bullet
	   	\ar[r]^\cart
		\ar[d]_\psi^\cart
		&
		\bullet
		\ar[d]^\phi_\cart
		\\
		\bullet
		\ar[r]_{\phantom{\cart}}^\cart
		&
		\bullet
	}}\right)	
	\left(\vcenter{\xymatrix@!0@C=4em@R=3.6em{
    	\bullet
	   	\ar[r]^\cart
    	\ar[d]|{\circdec}_{\nu}
    	&
    	\bullet
    	\ar[d]|{\circdec}^{\theta}
    	\\
    	\bullet
		\ar[r]_{\phantom{\cart}}^\cart
    	&
    	\bullet
    }}\right)
\end{equation}
It is readily verified that, as indicated, 
the square on the left hand side factors,
up to isomorphism, as a product of 
a commutative square of cartesian arrows in $\hpC$
covering $S$ and a base change square covering $S$.
Here the product of the
squares on the right is given by $\tensor_\internal$ on 
the horizontal morphisms and by $\oslash$ on the
vertical morphisms.
By the supercartesian case already proven,
the product $\psi \oslash \nu$ is supercartesian,
so the morphism $\rho$ is supercartesian by 
parts~(\ref{it:isosarehypercart}), 
(\ref{it:hypercartimpliessupercart}),
and~(\ref{it:supercartcomp}).

(\ref{it:hcexttensorhc}): 
The claim follows from part~(\ref{it:hypercartcomp})
by factoring $\theta_1 \exttensor \theta_2$
as $(\theta_1\exttensor\id) \circ (\id\exttensor\theta_2)$
and observing that 
and~(\ref{it:hypercartbc}) 
and~(\ref{it:cartoslashhypercart}) imply 
that the factors 
$(\theta_1\exttensor\id)$ and $(\id\exttensor\theta_2)$
are hypercartesian.
\end{proof}

We also have the following result; 
compare with Propositions~\ref{prop:commrelshriekinterpretation2} 
and \ref{prop:commrelinvshriek}.

\begin{prop}
\label{prop:cartandhypercartrel}
Suppose 
\[\xymatrix{
	X
	\ar[r]^\phi
	\ar[d]|\circdec_\theta
	&
	Y
	\ar[d]|\circdec^\kappa
	\\
	Z
	\ar[r]^\psi
	&
	W
}\]
is a commutative square of morphisms of $\hpC$ and $\hpC^\dfop$
covering a homotopy cartesian square in $\calT$.
Assume moreover that $\psi$ is cartesian and $\kappa$ is hypercartesian.
Then $\phi$ is cartesian if and only if $\theta$ is hypercartesian.
\end{prop}
\begin{proof}
If $\phi$ is cartesian, $\theta$ is hypercartesian by 
Proposition~\ref{prop:superandhypercartmorprops}(\ref{it:hypercartbc}).
Assume now that $\theta$ is hypercartesian.
Pick a cartesian morphism $\phi' \colon X' \to Y$
covering the same map as $\phi$ does,
and let $\theta' \colon X' \oto Z$ be the 
base change of $\kappa$
along our homotopy cartesian square
with respect to $\psi$ and $\phi'$.
Then $\theta'$ is hypercartesian by 
Proposition~\ref{prop:superandhypercartmorprops}(\ref{it:hypercartbc}).
By uniqueness of cartesian morphisms, we 
can find a vertical equivalence $\alpha \colon X' \xoto{\,\homot\,} X$
in $\hpC^\dfop$ such that $\theta' = \theta \circ \alpha$.
The equivalence $\alpha$ amounts to a
vertical equivalence $\hat{\alpha} \colon X \xto{\,\homot\,} X'$
in $\hpC$. It is easily checked that the composite $\phi'\circ\hat{\alpha}$
fits in place of the dashed morphism to make the square 
\[\xymatrix{
	X
	\ar@{-->}[r]
	\ar[d]|\circdec_\theta
	&
	Y
	\ar[d]|\circdec^\kappa
	\\
	Z
	\ar[r]^\psi
	&
	W
}\]
commutative. Since $\phi$ also has this property,
Proposition~\ref{prop:dfopcartmixedprop} implies that
$\phi = \phi'\circ\hat{\alpha}$. Thus $\phi$ is cartesian, as claimed.
\end{proof}

\subsubsection{Dualizing and strongly dualizing objects}
\label{subsubsec:dualizingobjects}

We conclude the subsection by introducing dualizing and 
strongly dualizing objects, which 
allow us to interpret supercartesian and hypercartesian
morphisms covering a map $f\colon A \to B$ and having target $S_B$ 
in terms of the existence and properties of a left adjoint 
$f^\invshriek$ for the base change
functor $f_!$.

\begin{defn}[Dualizing objects]
\label{def:dualizingobject}
A \emph{dualizing object} (with respect to the symmetric monoidal
presentable $\infty$--category $\calC$) for 
a continous map $p\colon E \to B$ is an
object $\omega_p\in \ho(\calC_{/E})$ together with a map 
$\tilde{\eta} \colon S_B \to p_!\omega_p$ such that the functor
$p^\invshriek \colon \ho(\calC_{/B}) \to \ho(\calC_{/E})$
given by
\begin{equation}
\label{eq:pinvshriekformula}
	p^\invshriek X = \omega_p \tensor_E p^\ast X
\end{equation}
is a left adjoint to 
$p_! \colon \ho(\calC_{/E}) \to \ho(\calC_{/B})$, 
with the unit of the adjunction given by the composite
\begin{equation}
\label{eq:invshriekunit}
	X 
	\xto{\ \homot\ }
	S_B\tensor_B X 
	\xto{\ \tilde{\eta} \tensor_B \id\ }
	p_!\omega_p \tensor X
	\xto{\ \homot\ }
	p_!(\omega_p \tensor p^\ast X)
	= 
	p_!p^\invshriek X 
\end{equation}
where the first equivalence is given 
by the monoidal unit constraint  and the second one by
the projection formula.
A dualizing object $(\omega_p, \tilde{\eta})$ for $p$
is \emph{strong} if for every homotopy cartesian square
\begin{equation}
\label{sq:hcsq}
\vcenter{\xymatrix{
	D
	\ar[r]^g
	\ar[d]_q
	& 
	E
	\ar[d]^p
	\\
	A
	\ar[r]^f
	&
	B
}}
\end{equation}
the composite
\[
	\tilde{\eta}_q
	\colon
	S_A 
	\xto{\ \homot\ } 
	f^\ast S_B
	\xto{\ f^\ast \tilde{\eta}\ }
	f^\ast p_!\omega_p
	\xto{\ \homot\ }
	q_! g^\ast \omega_p
\]
exhibits $g^\ast \omega_p$ as a dualizing object for $q$.
Here the first equivalence is again the monoidal unit constraint
while the second equivalence is an instance of 
\eqref{eq:commrelshriek}.
We say that the pair $(g^\ast\omega_p,\tilde{\eta}_q)$
constructed above is obtained by \emph{base change} 
from $(\omega_p,\tilde{\eta})$
along the homotopy cartesian square \eqref{sq:hcsq}.
\end{defn}

Comparing definitions, we now have
\begin{prop}
\label{prop:dualizingobjectsandcartesianmorphisms}
Let $p\colon E \to B$ be a map in $\calT$.
Then the map sending a morphism 
$\theta\colon \omega_p \oto S_B$ covering $p$
to the pair $(\omega_p,\tilde{\eta})$ where 
$\tilde{\eta} \colon S_B \to p_!\omega_p$ is the 
morphism determined by $\theta$ provides
a bijection between the set of 
supercartesian (res.\ hypercartesian)
morphisms of $\hpC^\dfop$ which cover $p$
and have target $S_B$ and the set of
dualizing objects (resp.\ strong dualizing objects)
for $p$. \qed
\end{prop}

Later, in Remark~\ref{rk:kleindualizingspectrum},
we will point out a connection between strong
dualizing objects in the sense of 
Definition~\ref{def:dualizingobject}
and the dualizing spectra defined by 
Klein~\cite{KleinDualizingSpectrum}.
Moreover, we will elaborate further 
on the implications of the existence of supercartesian 
and hypercartesian morphisms for base change functors
in Section~\ref{subsubsec:superhypercartandbasechangefunctors}.

\begin{prop}
\label{prop:dualizingobjectfwdual}
Suppose $\omega_p \in \ho(\calC_{/E})$ 
is a dualizing object for a map $p\colon E\to B$.
Then $p_! S_E$ and $p_! \omega_p$
are dual objects in $(\ho(\calC_{/B}),\tensor_B)$.
\end{prop}
\begin{proof}
Suppose  $X, Y \in \ho(\calC_{/B})$.
We have natural equivalences
\[
	p_! \omega_p \tensor_B X 
	\homot 
	p_! (\omega_p \tensor_E p^\ast X) 
	= 
	p_! p^\invshriek X
\]
and
\[
	p_! S_E \tensor_B Y
	\homot 
	p_!(S_E \tensor_E p^\ast Y)
	\homot
	p_! p^\ast Y,
\]
so from the $(p^\invshriek,p_!)$ and $(p_!, p^\ast)$ adjunctions
we obtain a natural bijection between morphism
\[
	p_! \omega_p \tensor_B X \longto Y
\]
and morphisms
\[
	X \longto p_! S_E \tensor_B Y
\]
in $\ho(\calC_{/B})$,
giving the claim. 
\end{proof}

See also Proposition~\ref{prop:hypercartandfwduality}.

\begin{defn}
Call a space $B$ \emph{dualizable} in $\calC$ if 
the object $t_\calC(B,r_B) = r^B_! S_B$ is dualizable in $\ho(\calC)$,
in which situation we call the dual of $r^B_! S_B$ the \emph{dual} of $B$.
\end{defn}

\begin{example}
By Remark~\ref{rk:tcalcrecognition},
a space $B$ is dualizable in $\Spectra$ if $\suspension^\infty_+ B$ is dualizable
in $\ho(\Spectra)$.
\end{example}

\begin{rem}
By Proposition~\ref{prop:dualizingobjectfwdual}, if a map $r\colon B \to \pt$
admits a dualizing object $\omega_r \in \ho(\calC_{/B})$, 
then $B$ is dualizable with dual $r_! \omega_r \in \ho(\calC)$.
We may thus think of the dualizing object $\omega_r$ as a kind of 
fibrewise refinement of the dual of $B$.
\end{rem}

\begin{example}
Combining Example~\ref{ex:cwdualityinspectra} with Theorem~\ref{thm:hypercartdata},
we see that for any closed smooth manifold $M$,
the object $S^{-\tau_M} \in \ho(\Spectra_{/M})$
is a strong dualizing object for $r \colon M \to \pt$,
refining the usual Atiyah duality stating that 
$\suspension^\infty_+ M$ is dualizable in $\ho(\Spectra)$
with dual $M^{-\tau_M}$. (See Definitions~\ref{def:sxi} 
and \ref{def:thomspectrumofvirtualbundle}
for the definitions of $S^{-\tau_M}$ and $M^{-\tau_M}$.)
\end{example}

\begin{example}
Let $BG$ be a connected $d$--dimensional $\ell$--compact group 
which is semisimple (see Definition~\ref{def:semisimple}),
and write $G = \loops BG$.
Combining Example~\ref{ex:cwdualityellcompactgroups}
with Theorem~\ref{thm:hypercartdata},
we see that 
$\suspension_G^{-d} S_{G,\ell} \in \ho(\Spectra^\ell_{/G})$
is a strong dualizing object for the map $G \to \pt$,
refining the duality between $\suspension_+^\infty G$
and $\suspension_+^{\infty-d} G$
in $\ho(\Spectra^\ell)$
\cite[Prop.~22 and Cor.~23]{bauer04}.
\end{example}

\begin{cor}
\label{cor:dualizingobjectdualizablefibres}
Suppose $p\colon E \to B$ admits a dualizing object with respect to $\calC$.
Then all homotopy fibres of $p$ are dualizable in $\calC$.
\end{cor}
\begin{proof}
Factor $p$ as a composite $E \xto{\ \homot\ } \tilde{E} \xto{\ \tilde{p}\ } B$
of a homotopy equivalence and a fibration.
Making use of Propositions~\ref{prop:dualizingobjectfwdual} 
and~\ref{prop:superandhypercartmorprops}(\ref{it:wecancel}),
we see that if $p$ admits a dualizing object, then so does $\tilde{p}$.
Thus we may reduce to the case where $p$ is a fibration.
Suppose $p$ admits a dualizing object.
Let $b\in B$, let $i \colon F \to E$ be the inclusion of the fibre
over $b$ into $E$, and continue to write $b$ for the map  $\pt \to B$
defined by $b$.
By Proposition~\ref{prop:dualizingobjectfwdual}, 
the object $p_! S_E$ is dualizable in $\ho(\calC_{/B})$.
Therefore $b^\ast p_! S_E$ is dualizable in $\ho(\calC)$,
and the claim follows by observing that
$b^\ast p_! S_E \homot r^F_! i^\ast S_E \homot r^F_! S_F$
where the first equivalence follows from the commutation relation
\eqref{eq:commrelshriek}.
\end{proof}

\begin{cor}
\label{cor:nodualizingobject}
The inclusion $i\colon \pt \incl S^1$ of a point into  $S^1$ 
does not admit a dualizing object
with respect to $\Spectra$.
\end{cor}
\begin{proof}
The homotopy fibre $\loops S^1 \homot \Z$ of $i$ is 
not dualizable in $\Spectra$.
\end{proof}

\subsection{Costenoble--Waner duality and the existence of 
hypercartesian morphisms}
\label{subsec:cwdualityandhypercartesianmorphisms}

In this subsection, we demonstrate the existence of 
interesting hypercartesian morphisms
and prove Theorems~\ref{thm:thetapc}, \ref{thm:kappaxcartcritc} and
\ref{thm:tildekappaxcart}.
The main results of the subsection,
Theorem~\ref{thm:hypercartdata} and 
Theorem~\ref{thm:hypercartexistence},
relate hypercartesian morphisms in $\hpC^\dfop$
to dualizability in the framed bicategories $\Ex_B(\calC)$
of Appendix~\ref{app:exbc}, and use that connection
to provide a  criterion for the 
existence of hypercartesian morphisms.
In particular,
Theorem~\ref{thm:hypercartexistence}
shows that 
a continuous map $p\colon E\to B$
whose homotopy fibres are small in the sense that
they are Costenoble--Waner dualizable in $\calC$
admits a hypercartesian morphism
$Y\oto X$ covering $p$ for every $X \in \ho(\calC_{/B})$.
The final theorem of the subsection, Theorem~\ref{thm:abgpthc2},
states that the parametrized Pontryagin--Thom 
transfer map of Ando, Blumberg and Gepner \cite[Def.~4.14]{ABGparam}
for a bundle $p\colon E\to B$ of smooth closed manifolds
defines a hypercartesian morphism 
$S^{-\tau_p} \oto S_B$ in $\hpSpectra^\dfop$, giving a geometric interpretation for 
hypercartesian morphisms $\omega_p \oto S_B$ covering such bundles.

The subsection is structured as follows. In Section~\ref{subsubsec:cwdualizability},
we discuss Costenoble--Waner dualizability of spaces, defining the
notion and illustrating it with remarks and examples.
In Section~\ref{subsubsec:existenceofhypercartesianmorphisms},
we state and prove Theorems~\ref{thm:hypercartdata} and~\ref{thm:hypercartexistence}.
In the sections that follow, 
Sections~\ref{subsubsec:applicationstocwdualizability}
and~\ref{subsubsec:applicationstohypercartesianmorphisms},
we present various results on Costenoble--Waner dualizability
and hypercartesian morphisms whose proofs make use of 
Theorems~\ref{thm:hypercartdata} and~\ref{thm:hypercartexistence}.
In Section~\ref{subsubsec:superhypercartandbasechangefunctors},
we elaborate on the implications of the existence of supercartesian and 
hypercartesian morphisms for the existence of base change functors,
a discussion we already started in Section~\ref{subsubsec:dualizingobjects}.
In Section~\ref{subsubsec:hypercartesianmorphismsandfibrewiseduality},
we present detailed results connecting
hypercartesian morphisms $\omega_p \oto S_B$ covering
a fibration $p \colon E \to B$ to fibrewise duality over $B$.
Finally, in Section~\ref{subsubsec:proofsofresultsfromintro},
we present the proofs of Theorems~\ref{thm:thetapc}, \ref{thm:kappaxcartcritc} and
\ref{thm:tildekappaxcart} and state and prove
Theorem~\ref{thm:abgpthc2} mentioned above.

\subsubsection{Costenoble--Waner dualizability}
\label{subsubsec:cwdualizability}

We start with a discussion of Costenoble--Waner dualizability.
Costenoble--Waner duality was first discovered by Costenoble and Waner \cite{CWv1}.
The name was introduced by May and Sigurdsson \cite{MaySigurdsson}.
The duality is defined in terms of a bicategory $\Ex(\calC)$ where $0$--cells are
spaces, the category $1$--cells from a space $A$ to a space $B$ 
is $\ho(\calC_{/A\times B})$, and the composite of $1$--cells 
$M \colon A \hto B$ and $N \colon B \hto C$ is given in terms of base change
functors as
\[
	M\odot N \homot (\pi_B)_! (\pi_C^\ast M \tensor_{A\times B \times C} \pi_A^\ast N)
\]
where 
$\pi_A \colon A \times B \times C \to B \times C$, 
$\pi_B \colon A \times B \times C \to A \times C$, 
and 
$\pi_C \colon A \times B \times C \to A \times B$
are the projections.
Each  $1$--cell $M \colon A \hto B$ in $\Ex(\calc)$
has an opposite $1$--cell $M^\op  \colon B\hto A$
given by $M^\op = \chi^\ast M$
where $\chi \colon B \times A \to A \times B$ is the coordinate 
exchange map. The detailed construction of $\Ex(\calC)$ and indeed of 
a more general version $\Ex_B(\calC)$ parametrized by a base space 
$B$ is given in Appendix~\ref{app:exbc}.

\begin{defn}
\label{def:cwdualizablespace}
A space $B$ is \emph{Costenoble--Waner dualizable} in $\calC$
if the base change $1$--cell $\pt_{r} \colon \pt \hto B$ 
(see Definition~\ref{def:basechangeobjects})
associated to the unique map $r \colon B \to \pt$
is right dualizable in the sense of Definition~\ref{def:dualpair}
in the bicategory $\Ex(\calC)$ of Definition~\ref{def:exboverpt}.
The right dual is called the \emph{Costenoble--Waner dual} of $B$.
\end{defn}

\begin{rem}
By Remark~\ref{rk:exbcandbasechange}, we may identity
$1$--cells 
$B \hto \pt$
with objects of $\ho(\calC_{/B})$. 
In particular, the Costenoble--Waner
dual of a space $B$ amounts to an object of $\ho(\calC_{/B})$. 
\end{rem}

\begin{rem}
Under the aforementioned identifications, the $1$--cell
$\pt_{r} \colon \pt \hto B$  corresponds to $S_B^\op$.
\end{rem}

\begin{example}
\label{ex:ptiscwdualizable}
In view of formula \eqref{eq:modotbnformula},
the $\odot$--product of $1$--cells $\pt\hto\pt$
in $\Ex(\calC)$ amounts to $\tensor$--product 
of the underlying objects of $\ho(\calC)$.
It follows that the one-point space $\pt$ 
is Costenoble--Waner dualizable in $\calC$
with Costenoble--Waner dual $S\in \ho(\calC_{/\pt})$. 
\end{example}

\begin{example}
\label{ex:cwdualityinspectra}
Every compact ENR and so in particular
every finite CW complex and 
every compact manifold (with or without a boundary)
is Costenoble--Waner dualizable in $\Spectra$
\cite[Theorem 18.5.1 and Proposition 18.3.2]{MaySigurdsson}.
When $M$ is a smooth closed manifold, 
the  Costenoble--Waner dual of $M$ is 
$S^{-\tau_M} \homot S^{\nu_M - \varepsilon^L} \in \ho(\Spectra_{/M})$
where $\nu_M$ is the normal bundle of an embedding of $M$ 
into some high-dimensional Euclidean space $\R^L$
and $\varepsilon^L$ is the $L$--dimensional trivial vector bundle.
See \cite[Theorem 18.6.1 and Proposition 18.3.2]{MaySigurdsson}.
The notation $S^{-\tau_M}$ was established in Definition~\ref{def:sxi}.
\end{example}

\begin{example}
\label{ex:cwdualityellcompactgroups}
By Theorem~\ref{thm:ellcptgrpscwdualizable}, the space
$G = \loops BG$ is Costenoble--Waner dualizable in 
$\Spectra^\ell$ 
with Costenoble--Waner dual
$\suspension_G^{-d} S_{G,\ell}$
when $BG$ is a semisimple
connected $\ell$--compact group
in the sense of Definition~\ref{def:semisimple}.
\end{example}

\begin{rem} %
\label{rk:cwdualityasrefinement}
Suppose $B$ is Costenoble--Waner dualizable 
with Costenoble--Waner dual $T$, so that 
$(S_B^\op,T)$ is a dual pair in $\Ex(\calC)$.
Let $r\colon B \to \pt$ be the unique map.
Then, by Corollary~\ref{cor:shriekduals},
$((r_! S_B)^\op, r_! T)$ is also a dual pair
in $\Ex(\calC)$.
As noted in Example~\ref{ex:ptiscwdualizable},
the $\odot$--product of $1$--cells $\pt\hto\pt$
in $\Ex(\calC)$ amounts to the $\tensor$--product 
of the underlying objects of $\ho(\calC)$.
It follows that  $(r_! S_B, r_! T)$ is a dual pair
in $\ho(\calC)$. We can thus view the dual pair 
$(S_B^\op,T)$ in $\Ex(\calC)$ as a fibrewise refinement
of the dual pair $(r_! S_B, r_! T)$ in $\ho(\calC)$.
\end{rem}

\begin{example}
Applied to the Costenoble--Waner 
dual pair $(S_M^\op,S^{-\tau_M})$
of Example~\ref{ex:cwdualityinspectra},
the procedure of Remark~\ref{rk:cwdualityasrefinement}
yields a dual pair $(\suspension^\infty_+ M,M^{-\tau_M})$
in $\ho(\Spectra)$ recovering the usual 
Atiyah duality for smooth closed manifolds.
\end{example}

\begin{example}
\label{ex:kleindualizingspectrum}
Let $G$ be a Lie group or the realization of a simplicial group.
Inspired by Bieri-Eckmann duality in group cohomology,
Klein~\cite{KleinDualizingSpectrum} has defined and studied 
a $\suspension^\infty_+ G$--module spectrum
$D_G$ associated to $G$ called the \emph{dualizing spectrum} of $G$.
Under the equivalence between $\suspension^\infty_+ G$--module 
spectra and parametrized spectra over $BG$ \cite{LindMalkiewichMorita},
the spectrum $D_G$
corresponds simply to the canonical candidate 
$S_{BG}^\op \vartriangleright U_{BG}$ for the Costenoble--Waner dual of 
$BG$. See Remark~\ref{rk:candidaterightdual}.
\end{example}

Given spaces $B_i$, $i \in I$, notice that the inclusions 
$B_i \incl \bigsqcup_{i\in I} B_i$ 
induce an isomorphism
\begin{equation}
\label{eq:hocdu}
	\ho(\calC_{/\bigsqcup_{i\in I} B_i}) \xto{\ \isom\ } \prod_{i\in I} \ho(\calC_{/B_i}).
\end{equation}

\begin{prop}
Let $\calC$ be a stable presentable 
symmetric monoid
presentable $\infty$--category.
Suppose $B_1,\ldots,B_k$ are spaces Costenoble--Waner dualizable 
in $\calC$ with Costenoble--Waner duals $T_1,\ldots,T_k$,
respectively.
Then the disjoint union $\bigsqcup_{i=1}^k B_i$
is Costenoble--Waner dualizable in $\calC$.
Moreover, the Costenoble--Waner dual
of $\bigsqcup_{i=1}^k B_i$ is given by the object $(T_i)_{i=1}^k \in \prod_{i=1}^k \ho(\calC_{/B_i}) \isom \ho(\calC_{/\bigsqcup_{i} B_i})$.
\end{prop}
\begin{proof}
Let $\eta_i \colon U_\pt \to S_{B_i}^\op \odot T_i$
and $\varepsilon_i \colon T_i \odot S_{B_i}^\op \to U_{B_i}$
be the unit and counit of the dual pair $(S_{B_i}^\op, T_i)$
in $\Ex(\calC)$. For $1\leq i,j \leq k$, 
let $U_{ij} = U_{B_i}$ and $\varepsilon_{ij} = \varepsilon_i$
when $i=j$ 
and $U_{ij} = 0$ and $\varepsilon_{ij} = 0$ 
when $i\neq j$.
Identifying $\ho(\calC_{/\bigsqcup_i B_i})$
with $\prod_{i\in I} \ho(\calC_{/B_i})$
and $\ho(\calC_{/\bigsqcup_i B_i \times \bigsqcup_i B_i})$
with $\prod_{i,j \in I} \ho(\calC_{B_i\times B_j})$
via the isomorphisms of equation \eqref{eq:hocdu},
the maps
\[
	\eta 
	\colon 
	U_\pt 
	=
	S 
	\xto{\ (\eta_i)_{i=1}^k\ }
	\prod_{i=1}^k \bigl(S_{B_i}^\op \odot T_i\bigr)
	\homot
	\bigvee_{i=1}^k \bigl(S_{B_i}^\op \odot T_i\bigr)
	\homot
	(S_{B_i}^\op)_{i=1}^k \odot (T_i)_{i=1}^k
	=
	S_{\bigsqcup_i B_i}^\op \odot (T_i)_{i=1}^k
\]
and
\[
	\varepsilon
	\colon 
	(T_i)_{i=1}^k \odot S_{\bigsqcup_i B_i}^\op 
	=
	(T_i)_{i=1}^k \odot (S_{B_i}^\op)_{i=1}^k
	\homot
	(T_i \odot S_{B_j}^\op)_{i,j=1}^k
	\xto{\ \prod_{i,j} \varepsilon_{ij}\ }
	(U_{ij})_{i,j=1}^k
	\homot
	U_{\bigsqcup_i B_i \times \bigsqcup_i B_i}
\]
exhibit 
$(S_{\bigsqcup_i B_i}^\op, (T_i)_{i=1}^k)$
as a dual pair in $\Ex(\calC)$.
Here the first equivalence in the definition of $\eta$ is 
given by the fact that finite products and coproducts
in a stable $\infty$--category agree; see \cite[Rk.~1.1.3.5]{HA}.
\end{proof}

Later, in Proposition~\ref{prop:cwdualityproduct},
we will show that the product of Costenoble--Waner
dualizable spaces is Costenoble--Waner dualizable.
Moreover, in Proposition~\ref{prop:cwdualityweyinvariance}
we will show that spaces weakly equivalent to 
Costenoble--Waner dualizable spaces
are Costenoble--Waner dualizable.

\begin{prop}
\label{prop:cwdualitytransport}
Suppose $B$ is Costenoble--Waner dualizable in $\calC$
with Costenoble--Waner dual $T \in \ho(\calC_{/B})$,
and assume there exists a symmetric monoidal functor
$F\colon \calC \to \calD$ between symmetric monoidal
presentable $\infty$--categories which admits a right adjoint.
Then $B$ is Costenoble--Waner dualizable in $\calD$
with Costenoble--Waner dual $F_B(T) \in \ho(\calD_{/B})$.
\end{prop}
\begin{proof}
Writing $S_{B,\calC}$ and $S_{B,\calD}$ for the 
monoidal units in 
$\ho(\calC_{/B})$
and
$\ho(\calD_{/B})$,
respectively,
the functor $\Ex(F)$ of Proposition~\ref{prop:exbf}
takes the dual pair $(S_{B,\calC}^\op, T)$ in $\Ex(\calC)$
to a dual pair $(S_{B,\calD}^\op, F_B(T))$ in $\Ex(\calD)$.
\end{proof}

\subsubsection{The existence of hypercartesian morphisms}
\label{subsubsec:existenceofhypercartesianmorphisms}

\begin{thrm}
\label{thm:hypercartdata}
Suppose $p \colon E\to B$ is a fibration.
Then the following are equivalent conditions on 
a pair $(\omega_p,\tilde{\eta})$ where $\omega_p$ is an object
of $\ho(\calC_{/E})$ 
and $\tilde{\eta} \colon S_B \to p_!\omega_p$ is a morphism
in $\ho(\calC_{/B})$:
\begin{enumerate}[(i)]
\item\label{it:hypercart}
	The morphism 
	$\theta\colon\omega_p \oto S_B$ of $\hpC^\dfop$
	over $p$
	defined by $\tilde{\eta}$ is hypercartesian.
\item\label{it:supercartx2}
	The morphism $\theta$ defined as above 
	as well as its base change along the homotopy cartesian square
	\[\xymatrix{
		E\times_B E
		\ar[r]^-{\pi_1}
		\ar[d]_{\pi_2}
		&
		E
		\ar[d]^p
		\\
		E
		\ar[r]^p
		&
		B		
	}\]
	are supercartesian in $\hpC^\dfop$.
	Here $\pi_1$ and $\pi_2$ are the projections to the
	first and second coordinates of	$E\times_B E$, respectively.
\item\label{it:exdualpair}
	View $\omega_p$ as a $1$--cell 
	$\omega_p \colon E\hto B$ in $\Ex_B(\calC)$,
	and let $\eta \colon U^B_B \hto B_p \odot_B \omega_p$
	be the globular $2$--cell of $\Ex_B(\calC)$
	defined by $\tilde{\eta}$
	under equivalences~\eqref{eq:uformula}
	and~\eqref{eq:basechangeformula1}.
	Then $(B_p, \omega_p)$ is a dual pair in $\Ex_B(\calC)$,
	with the map $\eta$ giving the unit of the duality.
\item\label{it:strongdualizing}
	$(\omega_p,\tilde{\eta})$ is a strong dualizing object for $p$.
\end{enumerate}
\end{thrm}

\begin{rem}
When $B = \pt$, part (\ref{it:exdualpair}) of Theorem~\ref{thm:hypercartdata}
amounts to the assertion that $\omega_p$ is the Costenoble--Waner
dual of $E$ with the unit of the duality defined by $\tilde{\eta}$.
\end{rem}

\begin{proof}[Proof of Theorem~\ref{thm:hypercartdata}]
The equivalence 
(\ref{it:hypercart}) $\Leftrightarrow$ (\ref{it:strongdualizing})
was observed in Proposition~\ref{prop:dualizingobjectsandcartesianmorphisms}.
Moreover,
(\ref{it:hypercart}) $\Rightarrow$ (\ref{it:supercartx2}) 
by Proposition~\ref{prop:superandhypercartmorprops}(\ref{it:hypercartimpliessupercart}) and the definition of 
hypercartesian morphisms. To complete the proof, it suffices to prove
the implications 
(\ref{it:supercartx2}) $\Rightarrow$ (\ref{it:exdualpair})
and 
(\ref{it:exdualpair}) $\Rightarrow$ (\ref{it:hypercart}).
To this end, let us first observe that 
for a fibration $q\colon T\to B$, 
the following conditions are equivalent:
\begin{enumerate}[(a)]
\item \label{it:a}
	The base change of the morphism $\theta$ of part~(\ref{it:hypercart})
	along the homotopy cartesian square
    \begin{equation}
    \label{sq:hcsq2}
    \vcenter{\xymatrix{
    	E\times_B T
    	\ar[r]^-{\pi_1}
    	\ar[d]_{\pi_2}
    	&
    	E
    	\ar[d]^p
    	\\
    	T
    	\ar[r]^q
    	&
    	B
    }}
    \end{equation}
    is supercartesian.
\item \label{it:b}
	The base change 
	$(\pi_1^\ast \omega_p, \tilde{\eta}_{\pi_2})$ 
	of $(\omega_p,\tilde{\eta})$
	along \eqref{sq:hcsq2}
	in the sense of Definition~\ref{def:dualizingobject}
	is a dualizing object for $\pi_2$.
\item \label{it:c}
	For all $1$--cells $M \colon B\hto T$ and $N\colon E\hto T$ in 
	$\Ex_B(\calC)$,	the map
	\[
		\alpha 
		\longmapsto 
		(
			M 
    		\homot 
    		U_B^B \odot_B M 
    		\xto{\ \eta\odot_B \id\ }
    		B_p\odot_B \omega_p\odot_B M
    		\xto{\ \id\odot_B \alpha\ }
    		B_p\odot_B N
		)
	\]
	where $\eta$ is as in (\ref{it:exdualpair})
	gives a bijection from the set of globular $2$--cells 
	$\omega_p \odot_B M \to N$ onto the set of globular $2$--cells
	$M \to B_p\odot_B N$.
\end{enumerate}
Here (\ref{it:a}) $\Leftrightarrow$ (\ref{it:b}) by
Proposition~\ref{prop:dualizingobjectsandcartesianmorphisms},
and (\ref{it:b}) and (\ref{it:c}) both amount to the assertion 
that the functor
$\pi_2^\invshriek \colon \ho(\calC_{/T}) \to \ho(\calC_{/E\times_B T})$
given by
\[
	\pi_2^\invshriek (M) = \pi_1^\ast \omega_p \tensor_{E\times_B T} \pi_2^\ast M
\]
is a left adjoint to $(\pi_2)_!$ with unit of the adjunction 
derived from $\tilde{\eta}_{\pi_2}$ as in \eqref{eq:invshriekunit}.

To prove the implication
(\ref{it:supercartx2}) $\Rightarrow$ (\ref{it:exdualpair}),
let $\varepsilon \colon \omega_p \odot_B B_p \to U^B_E$ now be 
the map corresponding to the equivalence $B_p \xto{\homot} B_p \odot_B U^B_E$
under the bijection of part~(\ref{it:c}) above
for $(T \xto{q} B) = (E \xto{p} B)$.  Then $\eta$ and $\varepsilon$
satisfy one of the triangle identities required to make $(B_p, \omega_p)$
a dual pair by construction, and the other identity follows by
observing that the composite
\[
	\omega_p\odot_B U^B_B 
	\xto{\ \id\odot_B\eta\ }
	\omega_p\odot_B B_p \odot_B \omega_p
	\xto{\ \varepsilon \odot_B \id\ }
	U^B_E \odot_B \omega_p
	\xto{\ \homot\ }
	\omega_p
\]
and the equivalence $\omega_p\odot_B U^B_B \xto{\homot} \omega_p$
both correspond to $\eta \colon U^B_B \to B_p \odot \omega_p$
under the bijection of part~(\ref{it:c}) for 
$(T \xto{q} B) = (B \xto{\id} B)$.
Finally, in view of 
Proposition~\ref{prop:superandhypercartmorprops}(\ref{it:hypercartpbres})
and 
Remark~\ref{rk:dualityadjunction},
the implication 
(\ref{it:exdualpair}) $\Rightarrow$ (\ref{it:hypercart})
follows from the implication
(\ref{it:c}) $\Rightarrow$ (\ref{it:a}).
\end{proof}

We point out parenthetically that Theorem~\ref{thm:hypercartdata}
and Example~\ref{ex:kleindualizingspectrum} imply
that a strong dualizing object for 
a map $BG \to \pt$ is simply Klein's dualizing spectrum $D_G$
for $G$ seen from another point of view.

\begin{rem}[Connection with Klein's dualizing spectrum]
\label{rk:kleindualizingspectrum}
Let $G$ be a Lie group or the realization of a simplicial group,
and suppose the map $BG\to \pt$ admits a strong dualizing object
$\omega_{BG} \in \ho(\Spectra_{/BG})$. By
Theorem~\ref{thm:hypercartdata}, $BG$ is then Costenoble--Waner
dualizable with Costenoble--Waner dual $\omega_{BG}$.
In view of Example~\ref{ex:kleindualizingspectrum}, 
$\omega_{BG}$ therefore agrees with Klein's dualizing spectrum 
$D_G$ of $G$ under the equivalence between parametrized spectra over $BG$ and 
$\suspension^\infty_+ G$--module spectra \cite{LindMalkiewichMorita}.
\end{rem}

\begin{defn}
\label{def:smallfibred}
A continuous map $p\colon E\to B$ is called \emph{small-fibred}
with respect to $\calC$ if all its homotopy fibres are 
Costenoble--Waner dualizable in $\calC$.
\end{defn}

\begin{thrm}
\label{thm:hypercartexistence}
The following are equivalent conditions on a continuous map
$p\colon E\to B$:
\begin{enumerate}[(i)]
\item\label{it:smallfibred}
	$p$ is small-fibred with respect to $\calC$.
\item\label{it:strongdualizingobject}
	$p$ admits a strong dualizing object with respect to $\calC$.
\item\label{it:hypercartesianmorphismfors}
	There exists a hypercartesian morphism
	$\omega_p \oto S_B$ in $\hpC^\dfop$ covering $p$.
\item\label{it:hypercartesianmorphismforall}
	For every $X \in \ho(\calC_{/B})$ there exists
	a hypercartesian morphism $Y\oto X$ in $\hpC^\dfop$
	covering $p$.
\item\label{it:rightdualinex}
	If $E\xto{\ i\ } \tilde{E} \xto{\ \tilde{p}\ } B$
    is a factorization of $p$ into a homotopy equivalence 
    followed by a fibration, then the $1$--cell 
    $B_{\tilde{p}} \colon E \hto B$
    is right dualizable in $\Ex_B(\calC)$.
\end{enumerate}
\end{thrm}

\begin{proof}
(\ref{it:smallfibred})
$\Leftrightarrow$
(\ref{it:rightdualinex})
by Proposition~\ref{prop:exbdualizabilitycriterion} and 
(\ref{it:strongdualizingobject}) 
$\Leftrightarrow$ 
(\ref{it:hypercartesianmorphismfors})
by Proposition~\ref{prop:dualizingobjectsandcartesianmorphisms}.
Clearly 
(\ref{it:hypercartesianmorphismforall}) 
$\Rightarrow$
(\ref{it:hypercartesianmorphismfors}),
and the converse implication
(\ref{it:hypercartesianmorphismfors})
$\Rightarrow$
(\ref{it:hypercartesianmorphismforall}) 
follows from
Proposition~\ref{prop:superandhypercartmorprops}(\ref{it:cartoslashhypercart}).
To prove the implication 
(\ref{it:hypercartesianmorphismfors})~%
$\Rightarrow$
(\ref{it:rightdualinex}),
assume $\theta\colon \omega_p \oto S_B$ is 
a hypercartesian morphism covering $p$.
Pick an opcartesian morphism
$\theta_{\sim} \colon \omega_p \oto \tilde{\omega}_p$
covering the homotopy equivalence $i\colon E \xto{\homot} \tilde{E}$.
We may then factor $\theta$ as a composite
$\theta = \tilde{\theta} \circ \theta_{\sim}$
where $\tilde{\theta}$ covers $\tilde{p}$.
By 
Proposition~\ref{prop:superandhypercartmorprops}(\ref{it:wecancel}),
the map $\tilde{\theta}$ is hypercartesian,
so (\ref{it:rightdualinex}) follows from 
Theorem~\ref{thm:hypercartdata}.
Finally, to show that
(\ref{it:rightdualinex})
$\Rightarrow$
(\ref{it:hypercartesianmorphismfors}),
assume that $B_{\tilde{p}}$ has a right dual in $\Ex_B(\calC)$.
By Theorem~\ref{thm:hypercartdata}, there then exists
a hypercartesian morphism 
$\tilde{\theta} \colon \tilde{\omega}_p \oto S_B$
covering $\tilde{p}$.
As $i$ is a homotopy equivalence, the 
functor $i_!$ is an equivalence of categories 
and therefore in particular essentially surjective.
It follows that we may find an opcartesian morphism 
$\theta_{\sim} \colon \omega_p \oto \tilde{\omega}_p$
covering $i$. By 
Proposition~\ref{prop:superandhypercartmorprops}(\ref{it:hypercartiffopcart}),
the map $\theta_{\sim}$ is also hypercartesian. 
By Proposition~\ref{prop:superandhypercartmorprops}(\ref{it:hypercartcomp}),
the composite
$\tilde{\theta} \circ \theta_{\sim} \colon \omega_p \oto S_B$ 
covering $p$ is now a hypercartesian morphism, proving 
(\ref{it:hypercartesianmorphismfors}).
\end{proof}

See Theorem~\ref{thm:abgpthc2} for a geometric interpretation
of the hypercartesian morphism $\omega_p \oto S_B$ of
Theorem~\ref{thm:hypercartexistence}(\ref{it:hypercartesianmorphismfors})
and the object $\omega_p$ when $p\colon E \to B$ is a bundle of smooth closed manifolds
and $\calC$ is the $\infty$--category of spectra.

The following corollary illustrates a pattern for constructing 
further cartesian morphisms in $\hpC^\dfop$ starting 
from the hypercartesian morphisms afforded by Theorem~\ref{thm:hypercartexistence}.
As pointed out in 
Example~\ref{eg:cartesianbutnotsupercartesian} below, 
the cartesian morphisms so constructed are in general not hypercartesian
or even supercartesian, however.

\begin{cor}
\label{cor:manifoldscart}
Suppose $M$ and $N$ are smooth closed manifolds and $f\colon M\to N$
is a continuous map. Then there exist cartesian morphisms
\begin{equation}
\label{eq:manifoldscartmaps}
	S^{-\tau_M}
	\longoto 
	S^{-\tau_N}
	\qquad
	\text{and}
	\qquad
	f^\ast  S^{\tau_N} \smashprod_M S^{-\tau_M}
	\longoto 
	S_N	
\end{equation}
in $\hp\Spectra^\dfop$ covering $f$.
\end{cor}

\begin{proof}
In view of Example~\ref{ex:cwdualityinspectra},
Theorems~\ref{thm:hypercartexistence} and~\ref{thm:hypercartdata}
imply that we have hypercartesian morphisms
$\theta_M \colon S^{-\tau_M} \oto S$ and 
$\theta_N \colon S^{-\tau_N} \oto S$ covering the maps
$M\to\pt$ and $N\to \pt$, respectively.
By Proposition~\ref{prop:superandhypercartmorprops}(\ref{it:hypercartimpliessupercart})
and (\ref{it:supercartimpliescart}),
the morphisms $\theta_M$ and $\theta_N$ are cartesian.
We may therefore find a morphism $\theta \colon S^{-\tau_M} \oto S^{-\tau_N}$
covering $f$ such that $\theta_M = \theta_N \circ \theta$,
and Proposition~\ref{prop:cartmorprops}(\ref{it:cartfactor}) 
implies that $\theta$ is cartesian.
This gives the first map in \eqref{eq:manifoldscartmaps}.
The second map can now be constructed by taking the $\oslash$--product
of $\theta$ with a cartesian morphism $f^\ast S^{\tau_N} \to S^{\tau_N}$
covering $f$; see Proposition~\ref{prop:cartoslashproperty}.
\end{proof}

\begin{example}[Cartesian morphism which is not supercartesian]
\label{eg:cartesianbutnotsupercartesian}
By Corollary~\ref{cor:manifoldscart},
there exists a cartesian morphism of the form $X \oto S_{S^1}$ 
in $\hp\Spectra^\dfop$ covering the inclusion $i\colon \pt \incl S^1$.
By Proposition~\ref{prop:dualizingobjectsandcartesianmorphisms}
and 
Corollary~\ref{cor:nodualizingobject}, $\hpSpectra^\dfop$ 
does not admit supercartesian morphisms of this form, however. 
Thus there are cartesian morphisms of $\hpSpectra^\dfop$ which are not supercartesian.
\end{example}

\begin{rem}
We do not know if there exist supercartesian morphisms that are not hypercartesian.
\end{rem}

\subsubsection{Applications of Theorems~\ref{thm:hypercartdata} and~\ref{thm:hypercartexistence} to Costenoble--Waner dualizability}
\label{subsubsec:applicationstocwdualizability}

Next we will use Theorems~\ref{thm:hypercartdata} and~\ref{thm:hypercartexistence} 
to establish further results on Costenoble--Waner dualizability of spaces.
In view of 
Proposition~\ref{prop:superandhypercartmorprops}(\ref{it:hypercartcomp}),
Theorem~\ref{thm:hypercartexistence} implies
\begin{prop}
\label{prop:smallfibredmaps}
The composite of maps which are small-fibred with respect to $\calC$
is small-fibred with respect to~$\calC$.
\qed
\end{prop}

Applying the previous proposition to the maps $E\to B$ and $B\to \pt$,
we obtain

\begin{cor}
\label{cor:totalspacecwdualizable}
Suppose $p\colon E \to B$ is a continuous map 
such that $B$ and all homotopy fibres of $p$ are
Costenoble--Waner dualizable in $\calC$.
Then $E$ is also Costenoble--Waner dualizable in $\calC$. \qed
\end{cor}

The following proposition could also be proved directly
without appealing to Theorems~\ref{thm:hypercartdata}
and~\ref{thm:hypercartexistence}.

\begin{prop}
\label{prop:cwdualityproduct}
Suppose $A$ and $B$ are Costenoble--Waner dualizable in $\calC$
with Costenoble--Waner duals $T_A$ and $T_B$, respectively.
Then $A\times B$ is also Costenoble--Waner dualizable 
in $\calC$ with Costenoble--Waner dual $T_A \exttensor T_B$.
\qed
\end{prop}
\begin{proof}
That $A\times B$ is Costenoble--Waner dualizable in $\calC$
follows by applying Corollary~\ref{cor:totalspacecwdualizable}
to the projection $A\times B \to B$.
By Theorem~\ref{thm:hypercartdata},
there exist hypercartesian morphisms $\theta_A \colon T_A \oto S$
and $\theta_B \colon T_B \oto S$ covering the maps $A\to \pt$ and 
and $B\to \pt$, respectively. Let $\pi_1 \colon A\times B \to A$
and $\pi_2 \colon A\times B \to B$ be the projections. 
By Proposition~\ref{prop:superandhypercartmorprops}(\ref{it:hypercartbc}),
the morphism $\tilde{\theta}_B\colon \pi_2^\ast T_B \oto S_B$ 
covering $\pi_1$ obtained by base change from 
$\theta_B$ is hypercartesian. Choose a cartesian morphism 
$\phi \colon \pi_1^\ast T_A \to T_A$ in $\hpC$
covering $\pi_1$.
By 
Proposition~\ref{prop:superandhypercartmorprops}(\ref{it:cartoslashhypercart}), 
(\ref{it:isosarehypercart})
and 
(\ref{it:hypercartcomp})
the composite
\[
	T_A \exttensor T_B
	=
	\pi_1^\ast T_A \tensor_{A\times B} \pi_2^\ast T_B
	\xto{\ \phi \oslash \tilde{\theta}_B\ }
	T_A \tensor_A S_A
	\xto{\ \homot\ }
	T_A 
	\xto{\ \theta_A\ }
	S
\]
is hypercartesian. It now follows from
Theorem~\ref{thm:hypercartdata}
that $T_A \exttensor T_B$
is the Costenoble--Waner dual of $A\times B$, as desired.
\end{proof}

\begin{prop}
\label{prop:cwdualityweyinvariance}
Any space weakly equivalent to a 
space which is Costenoble--Waner dualizable 
in $\calC$
is Costenoble--Waner  dualizable in $\calC$.
\end{prop}
\begin{proof}
It is enough to show that given a weak equivalence 
$f\colon A\to B$, 
the space $A$ is Costenoble--Waner dualizable in $\calC$ 
if and only if the space $B$ is or, what in view of 
Theorem~\ref{thm:hypercartexistence} is equivalent,
that the map $r_A \colon A \to \pt$
is covered by a hypercartesian morphism
with target $S \in \ho(\calC)$
if and only if the map $r_B \colon B \to \pt$ is.
Suppose first that $\theta_A\colon \omega_A \oto S$ is a hypercartesian
morphism covering $r_A$. Choose an opcartesian morphism
$\kappa \colon \omega_A \oto \omega_B$ covering $f$, and 
factor $\theta_A$ as a composite $\theta_A = \theta_B \circ \kappa$.
Now $\theta_B$ is a hypercartesian morphism covering $r_B$ by 
Proposition~\ref{prop:superandhypercartmorprops}(\ref{it:wecancel}).
Conversely, suppose $\theta_B \colon \omega_B \oto S$ is
a hypercartesian morphism covering $r_B$.
Choose a cartesian morphism 
$\phi\colon \omega_A \to \omega_B$ in $\hpC$
covering $f$. By Proposition~\ref{prop:wecartiffopcart},
the morphism $\phi$ is also opcartesian in $\hpC$,
so we may interpret it as an opcartesian morphism in $\hpC^\dfop$.
Now the composite $\theta_B \circ \phi$ is a hypercartesian 
morphism covering $r_A$ by 
Proposition~\ref{prop:superandhypercartmorprops}(\ref{it:hypercartiffopcart}) 
and~(\ref{it:hypercartcomp}).
\end{proof}

\subsubsection{Applications of Theorems~\ref{thm:hypercartdata} and~\ref{thm:hypercartexistence} to hypercartesian morphisms}
\label{subsubsec:applicationstohypercartesianmorphisms}

We will now use Theorems~\ref{thm:hypercartdata} and~\ref{thm:hypercartexistence} 
to derive results on hypercartesian morphisms.
We have the following preservation result for hypercartesian morphisms.

\begin{prop}
\label{prop:hypercartmorpreservation}
Suppose $\calD$ is another 
symmetric monoidal presentable $\infty$--category,
and suppose $F\colon \calC \to \calD$ is a 
symmetric monoidal $\infty$--functor which 
admits a right adjoint.  
Then the functor 
$F_\fw \colon \hpC^\dfop \longto \hpD^\dfop$
of Proposition~\ref{prop:fprimefw}
preserves hypercartesian morphisms covering 
continuous maps which are small-fibred with respect to $\calC$
\end{prop}

\begin{proof}
Let $p\colon E\to B$ be a small-fibred continuous map
and let $\theta\colon X \oto Y$ be a hypercartesian morphism
in $(\hpC)^\dfop$ covering $p$.
Our claim is that $F_\fw(\theta)$ is  hypercartesian in 
$(\hpD)^\dfop$.
Factor $p$ as a composite $p = \tilde{p} \circ i$
where $\tilde{p}\colon \tilde{E} \to B$ 
is a fibration and $i \colon E\to \tilde{E}$ 
is a homotopy equivalence, 
pick an opcartesian morphism
$\theta_{\sim} \colon X \oto \tilde{X}$
in $(\hpC)^\dfop$
covering $i$, 
and factor $\theta$ as a composite
$\theta = \tilde{\theta} \circ \theta_{\sim}$
where $\tilde{\theta}\colon \tilde{X} \to Y$ covers $\tilde{p}$.
By 
Proposition~\ref{prop:superandhypercartmorprops}(\ref{it:wecancel}),
the map $\tilde{\theta}$ is hypercartesian.
By 
Proposition~\ref{prop:opcartpreservation},
the map $F_\fw(\theta_{\sim})$ in 
$(\hpD)^\dfop$ is opcartesian, and
hence by 
Proposition~\ref{prop:superandhypercartmorprops}(\ref{it:hypercartiffopcart})
also hypercartesian.
In view of 
Proposition~\ref{prop:superandhypercartmorprops}(\ref{it:hypercartcomp})
and the factorization
$F_\fw(\theta) = F_\fw(\tilde{\theta})\circ F_\fw(\theta_{\sim})$,
to prove the claim, it is enough to show 
that the morphism
$F_\fw(\tilde{\theta})$ is hypercartesian.

Since the map $p$  is small-fibred with respect to $\calC$, 
the fibration $\tilde{p}$ also is.
Thus 
Theorem~\ref{thm:hypercartexistence}
implies that there exists a hypercartesian morphism
$\theta_{\tilde{p}}\colon 
\omega_{\tilde{p}} \oto S_B$ covering $\tilde{p}$.
Let $c \colon \tilde{p}^\ast (Y) \to Y$ be the canonical
cartesian morphism in $\hpC$ covering $\tilde{p}$.
By 
Proposition~\ref{prop:superandhypercartmorprops}(\ref{it:cartoslashhypercart})
the $\oslash$--product 
\[
	c \oslash \theta_{\tilde{p}}
	\colon 
	p^\ast(Y) \tensor_{\tilde{E}} \omega_{\tilde{p}}
	\longoto
	Y \tensor_{B} S_B
\]
covering $\tilde{p}$
is hypercartesian. From the canonical equivalence $Y\tensor_B S_B \homot Y$
and the uniqueness of cartesian morphisms we deduce that 
$\tilde{\theta}$ 
and the morphism 
$c \oslash \theta_{\tilde{p}}$
are equivalent.
Parts (\ref{it:isosarehypercart}) and (\ref{it:supercartcomp})
of Proposition~\ref{prop:superandhypercartmorprops}
imply that any morphism equivalent to a hypercartesian morphism
is hypercartesian.
Thus it is enough to show that the morphism 
$F_\fw(c \oslash \theta_{\tilde{p}})$ is hypercartesian
in $\hpD^\dfop$.
As $F_\fw(c \oslash \theta_{\tilde{p}})$ is equivalent
to the morphism
$F_\fw(c) \oslash F_\fw(\theta_{\tilde{p}})$,
it follows from
Proposition~\ref{prop:superandhypercartmorprops}(\ref{it:cartoslashhypercart}) 
that it is enough to show that 
$F_\fw(\theta_{\tilde{p}})$
is hypercartesian.
The claim now follows from Theorem~\ref{thm:hypercartdata}
and the observation that the strong framed
functor
\[
	\Ex_B(F) \colon \Ex_B(\calC) \longto \Ex_B(\calD)
\]
of Proposition~\ref{prop:exbf}
takes the dual pair 
$(B_{\tilde{p}}, \omega_{\tilde{p}}) = (S_B^\op, \omega_{\tilde{p}})$
in $\Ex_B(\calC)$
to a dual pair 
$(B_{\tilde{p}}, F_\fw(\omega_{\tilde{p}}))$
in $\Ex_B(\calD)$.
\end{proof}

We also record the following criterion for detecting hypercartesian morphisms of 
$\hpC^\dfop$ covering small-fibred fibrations.
\begin{prop}
\label{prop:smallfibcartcrit}
Suppose $p\colon E \to B$ is a fibration which is small-fibred with respect to $\calC$.
Then a morphism $\kappa \colon Y \oto X$ of $\hpC^\dfop$ covering $p$
is hypercartesian if and only if for all $b\in B$
its base change along the pullback square
\begin{equation}
\label{eq:epbfibresq} 
\vcenter{\xymatrix{
	p^{-1}(b)
	\ar@{^{(}->}[r]^-{i_b} %
	\ar[d]_{r_{b}}
	&
	E
	\ar[d]^p
	\\
	\pt
	\ar[r]^b
	&
	B
}}
\end{equation}
is.
\end{prop}
\begin{proof}
By Theorem~\ref{thm:hypercartexistence}, there exists a hypercartesian morphism
$\theta \colon Y' \oto X$ in $\hpC^\dfop$ covering $p$. Factor $\kappa$ as
a composite 
\[
	\kappa 
	\colon
	Y 
	\xoto{\quad\mathclap{\tilde{\kappa}}\quad} 
	Y' 
	\xoto{\quad\mathclap{\theta}\quad} 
	X
\]
where $\tilde{\kappa}$ covers the identity map of $E$.
Given $b\in B$, 
let $\kappa_b \colon i_b^\ast Y \oto b^\ast X$
and $\theta_b \colon i_b^\ast Y' \oto b^\ast X$
be the base changes of $\kappa$ and $\theta$, respectively,
along \eqref{eq:epbfibresq} with respect to the 
cartesian morphism 
$i_b^\ast Y \to Y$ and $i_b^\ast Y' \to Y'$
covering $i_b$ and $b^\ast X \to X$ covering $b\colon \pt \to B$.
Then $\kappa_b$ factors as a composite
\[
	\kappa_b 
	\colon
	i_b^\ast Y 
	\xoto{\quad\mathclap{i_b^\ast\tilde{\kappa}}\quad} 
	i_b^\ast Y' 
	\xoto{\quad\mathclap{\theta_b}\quad} 
	b^\ast X
\]
We now have
\[
\begin{aligned}
	\text{$\kappa$ is hypercartesian}
	&
	\Longleftrightarrow
	\text{$\tilde{\kappa}$ is an equivalence in $(\hpC^\dfop)_E = \ho(\calC_{/E})^\op$}
	\\
	&
	\Longleftrightarrow
	\text{
		$i_b^\ast(\tilde{\kappa})$ 
		is an equivalence in 
		$\ho(\calC_{/p^{-1}(b)})^\op$ for all $b\in B$
	}
	\\
	&
	\Longleftrightarrow
	\text{$\kappa_b$ is hypercartesian for all $b\in B$}
\end{aligned}
\]
where the first equivalence follows from the fact that 
the map $\theta$ is hypercartesian;
the second from the fact that equivalences
in $\ho(\calC_{/E})^\op$ are detected on fibres;
and the last from the fact that 
the map $\theta_b$ is hypercartesian for all $b$ by 
Theorem~\ref{prop:superandhypercartmorprops}(\ref{it:hypercartbc}).
\end{proof}

\subsubsection{Supercartesian and hypercartesian morphisms and
base change functors}
\label{subsubsec:superhypercartandbasechangefunctors}

Continuing the work we started in Section~\ref{subsubsec:dualizingobjects},
we will now elaborate on the implications the 
existence of supercartesian or hypercartesian morphisms
has for the existence of base change functors.
Suppose $p\colon E\to B$ is a continuous map admitting a supercartesian
morphism $\theta\colon\omega_p \oto S_B$ in $\hpC^\dfop$ 
covering it. Then, as we saw in Section~\ref{subsubsec:dualizingobjects}, 
the usual adjunctions $p_! \dashv p^\ast \dashv p_\ast$ 
extend to a string of adjunctions
\begin{equation}
\label{eq:bcadjunctions}
	p^\invshriek 
	\dashv
	p_!
	\dashv
	p^\ast
	\dashv
	p_\ast 
\end{equation}
where 
\begin{equation}
\label{eq:pinvshriekformula2}
 	p^\invshriek X = \omega_p \tensor_E p^\ast X.
\end{equation}
When $\theta$ is hypercartesian,
the string of adjunctions~\eqref{eq:bcadjunctions} extends one step
to the right. To see this, notice that without loss of generality
we may assume that $p$ is a fibration. 
Then we have a dual pair $(B_p,\omega_p)$ in $\Ex_B(\calC)$,
and applying $(-)^\op$ we obtain another dual pair 
$(\omega_p^\op, B_p^\op) = (\omega_p^\op, {}_p B)$.
In view of Remark~\ref{rk:dualityadjunction},
together with the dual pair $({}_p B, B_p)$ of 
Proposition~\ref{prop:basechangeobjdualpair},
these dual pairs imply that we have a string of adjunctions
\[
	\omega_p \odot_B (-)
	\,\dashv\,
	B_p \,\odot_B (-)
	\,\dashv\,
	{}_p B \odot_B (-)
	\,\dashv\,
	\omega_p^\op \odot_B (-)
	\,\dashv\,
	\,
	(-) \vartriangleleft_B \omega_p^\op,
\]
implying adjunctions
\[
	p^\invshriek 
	\dashv
	p_!
	\dashv
	p^\ast
	\dashv
	p_\ast
	\dashv
	p^{!}
\]
where the functor
$p^! \colon \ho(\calC_{/B}) \to \ho(\calC_{/E})$ is given by
\begin{equation}
\label{eq:puppershriekformula}
	p^! X = F_E(\omega_p,p^\ast X).
\end{equation}
(While the above formula certainly makes sense when $\theta$
is just supercartesian, it does not seem clear how to prove
that it defines a right adjoint to $p_\ast$ without knowing 
that $\theta$ is hypercartesian.)

The base change functors
$p^\invshriek$, $p_!$, $p^\ast$, $p_\ast$, and $p^{!}$
satisfy various 
compatibility relations with respect to each other and 
the tensor products $\tensor_E$ and $\tensor_B$ and 
the hom functors $F_E$ and $F_B$ beyond the ones
listed in \eqref{eq:bc-unit}--\eqref{eq:bc-shriekadj}.
Instead of attempting an
exhaustive treatment here, we will simply point out 
some examples. First, assuming that $\theta$ us
supercartesian, \eqref{eq:pinvshriekformula2}
readily implies that 
\[
	p^\invshriek (X\tensor_B Y) \homot p^\invshriek X \tensor_E p^\ast Y
\]
for $X,Y\in \ho(\calC_{/B})$, and conjugating this identity yields
the identities
\[
	p_\ast F_E(p^\invshriek X, Z) \homot F_B(X, p_! Z)
	\qquad
	\text{and}
	\qquad
	p_! F_E(p^\ast Y, Z) \homot F_B(Y,p_!Z)
\]
for $X,Y\in \ho(\calC_{/B})$ and $Z\in \ho(\calC_{/E})$.
Moreover, conjugating~\eqref{eq:pinvshriekformula2}
yields the formula
\[
	p_! Z \homot p_\ast F_E(\omega_p,Z)
\]
Finally, assuming that $\theta$ is hypercartesian,
conjugating \eqref{eq:puppershriekformula} yields a
``Wirthmüller isomorphism''
\begin{equation}
\label{eq:wirthmulleriso}
	 p_!(\omega_p \tensor_E Z) \homot p_\ast Z 
\end{equation}
for $Z\in \ho(\calC_{/E})$; cf.\ \cite[(19.5.4)]{MaySigurdsson}.
See \cite{BDS} for related work in a slightly different context;
we warn the reader that their $\omega_p$ corresponds to the dual
of ours.

Further assumptions on the dualizing object $\omega_p$ 
imply the existence of further base change functors.
It is common for the dualizing object $\omega_p$ to be invertible
in $\ho(\calC_{/E})$;
in view of Proposition~\ref{prop:omegapinvertibility} below,
Theorem~\ref{thm:hypercartdata},  
Example~\ref{ex:cwdualityinspectra} and
Theorem~\ref{thm:ellcptgrpscwdualizable},
this happens for example when 
$\calC = \Spectra$ and
$p\colon E\to B$ is a bundle of closed smooth manifolds,
and when $\calC = \Spectra^\ell$ and $p\colon E\to B$
is a fibration with fibre $\loops BG$ where $BG$ is a semisimple
$\ell$--compact group. 
When 
$\omega_p$ is invertible, using formula \eqref{eq:pinvshriekformula2},
it is easy to see that 
\eqref{eq:bcadjunctions} extends to an infinite string of adjunctions
\[\xymatrix@C=0.0em@!R=1.4ex{
	\renewcommand{\labelstyle}{\textstyle}
	\cdots 
	& \dashv & p_{(-1)} 
	& \dashv & p^{(1)} 
	& \dashv & p_{(0)} 
	& \dashv & p^{(0)}
	& \dashv & p_{(1)} 
	& \dashv & p^{(-1)} 
	& \dashv & p_{(2)}
	& \dashv &
	\cdots
	\\
	&&&&
	p^\invshriek \ar@{}[u]|{\rotatebox{-90}{$\homot$}}
	&&
	p_! \ar@{}[u]|{\rotatebox{-90}{$\homot$}}
	&&
	p^\ast \ar@{}[u]|{\rotatebox{-90}{$\homot$}}
	&&
	p_\ast \ar@{}[u]|{\rotatebox{-90}{$\homot$}}
	&&
	p^! \ar@{}[u]|{\rotatebox{-90}{$\homot$}}
}\]
where 
\[
	p^{(n)} X \homot \omega_p^{\tensor n}\tensor_E p^\ast X
	\qquad\text{and}\qquad
	p_{(n)} Z \homot p_!(\omega_p^{\tensor  n}\tensor_E Z)
\]
for all $n\in \Z$. Here $\omega_p^{\tensor 0} = S_E$ and 
$\omega_p^{\tensor (-n)} = (D\omega_p)^{\tensor n}$
for $n > 0$. Cf.\ \cite[Thm.~1.9]{BDS}. In fact, for the left adjoint $p_{(-1)}$
of $p^\invshriek$ to exist, it is enough to assume that
$\omega_p$ is dualizable; in that case, one deduces 
from \eqref{eq:pinvshriekformula2} that $p^\invshriek$
is given by $p^\invshriek X = F_E(D\omega_p, p^\ast X)$,
so that the functor $p_{(-1)}$ defined by the above formula
gives a left adjoint for $p^\invshriek$.
Similarly, when $\theta$ is hypercartesian, 
formula \eqref{eq:puppershriekformula} implies that 
dualizability of $\omega_p$  suffices to ensure that 
the right adjoint  $p_{(2)}$ of $p^!$ exists.

In the above discussion, we made use of the following
criterion for the invertibility of a strong dualizing 
object for a map $p$.
\begin{prop}
\label{prop:omegapinvertibility}
Suppose $p\colon E \to B$ is a continuous map which is small-fibred with respect to $\calC$,
and let $\omega_p \oto S_B$ be the hypercartesian morphism covering $p$ afforded by 
Theorem~\ref{thm:hypercartexistence}. Then $\omega_p$ is invertible as
an object of $\ho(\calC_{/E})$ if and only if for each homotopy fibre $F$ of $p$, 
the Costenoble--Waner dual of $F$ is invertible as an object of $\ho(\calC_{/F})$
\end{prop}

\begin{proof}
Let $\theta\colon \omega_p \oto S_B$ be the hypercartesian morphism covering $p$.
Factor $p$ as a composite $p = \tilde{p} \circ i$ where $p\colon \tilde{E} \to B$
is a fibration and $i\colon E \to \tilde{E}$ is a homotopy equivalence.
Let $\theta_{\sim} \colon \omega_p \oto \tilde{\omega}_p$
be an opcartesian morphism covering $i$, and factor $\theta$ as 
$\theta = \tilde{\theta} \circ \theta_{\sim}$ where $\tilde{\theta}$ covers $\tilde{p}$.
Since $i$ is a homotopy equivalence, the functor 
$i^\ast \colon \ho(\calC_{/\tilde{E}}) \to \ho(\calC_{/E})$ is a symmetric monoidal
equivalence of categories with inverse $i_!$.
Consequently, the object $\omega_p \homot i^\ast i_! \omega_p$
is invertible if and only if the object $\tilde{\omega}_p\homot i_! \omega_p$ is.
As the homotopy fibres of $p$ are homotopy equivalent to the actual fibres of $\tilde{p}$,
we conclude that it suffices to show that $\tilde{\omega}_p$ is invertible 
if and only if the Costenoble--Waner duals of the fibres of $\tilde{p}$ are 
invertible. 

By Proposition~\ref{prop:superandhypercartmorprops}(\ref{it:wecancel}),
the map $\tilde{\theta} \colon \tilde{\omega}_p \oto S_B$ is hypercartesian.
Given a point $b\in B$, let 
\[\xymatrix{
	F_b 
	\ar[r]^{j_b}
	\ar[d]_{r_b}
	&
	\tilde{E}
	\ar[d]^{\tilde{p}}
	\\
	\pt
	\ar[r]^b
	&
	B
}\]
be a pullback diagram. By 
Proposition~\ref{prop:superandhypercartmorprops}(\ref{it:hypercartbc}),
the morphism $j_b^\ast \tilde{\omega}_p \oto S_\pt$ obtained from 
$\tilde{\theta}\colon\tilde{\omega}_p \oto S_B$
by base change along the above pullback square is hypercartesian,
so by the equivalence of parts (\ref{it:hypercart}) and (\ref{it:exdualpair})
of Theorem~\ref{thm:hypercartdata} the object $j_b^\ast \tilde{\omega}_p$
is the Costenoble--Waner dual of $F_b$. 
On the other hand, as every point of $\tilde{E}$ is in the image of
the map $j_b$ for some $b\in B$, Corollary~\ref{cor:fwdinvcrit} implies that 
$\tilde{\omega}_p$ is invertible if and only if $j_b^\ast\tilde{\omega}_p$
is for all $b\in B$. The claim follows.
\end{proof}

\subsubsection{Hypercartesian morphisms and fibrewise duality}
\label{subsubsec:hypercartesianmorphismsandfibrewiseduality}

We will next present two propositions connecting 
hypercartesian morphisms to fibrewise duality.
Generalizing Remark~\ref{rk:cwdualityasrefinement}, we have
the following slightly more informative version of 
Proposition~\ref{prop:dualizingobjectfwdual}.
\begin{prop}
\label{prop:hypercartandfwduality}
Suppose $p\colon E \to B$ is a fibration admitting a 
hypercartesian morphism $\theta \colon \omega_p \oto S_B$
in $\ho\calC^\dfop$ covering it. Then 
$(p_! S_E, p_! \omega_p)$
is a dual pair in $(\ho(\calC_{/B}),\tensor_B)$
with unit of the duality given by the composite
\begin{equation}
\label{eq:psepomegapunit} 
	S_B 
	\xto{\ \tilde{\eta}\ } 
	p_! \omega_p  
	\xto{\ p_!(\eta)\ }
	p_! p^\ast p_! \omega_p
	\homot
	p_! (S_E \tensor_{E} p^\ast p_! \omega_p)
	\homot
	p_! S_E \tensor_{B} p_! \omega_p.
\end{equation}
Here $\tilde{\eta}$ is the morphism determined by $\theta$,
the second map is induced by the unit of the $(p_!,p^\ast)$ adjunction,
the first equivalence is induced by the monoidal unit constraint,
and the second one is given by the projection formula.
\end{prop}
\begin{proof}
By Theorem~\ref{thm:hypercartdata}, 
$(B_p,\omega_p) = (S_E^\op, \omega_p)$ is a dual pair in 
$\Ex_B(\calC)$,
as is $({}_p B,B_p)$ by Proposition~\ref{prop:basechangeobjdualpair}.
It follows that $(B_p \odot_B {}_p B, B_p \odot \omega_p)$
is a dual pair in $\Ex_B(\calC)$ with unit given by the composite
\[
	U_B^B 
	\longto 
	B_p \odot_B \omega_p 
	\longto 
	B_p \odot_B ({}_p B \odot_B B_p) \odot_B \omega_p 
	\homot
	(B_p \odot_B {}_p B) \odot_B (B_p \odot_B \omega_p)
\]
of the units for $(B_p,\omega_p)$ and $({}_p B,B_p)$.
Using Remark~\ref{rk:exbcandbasechange} to translate 
from $\Ex_B(\calC)$ to the language of base change functors
now yields the claim.
\end{proof}

Recall that the Thom diagonal map $A^\xi \to A_+ \smashprod A^\xi$
for a vector bundle $\xi$ over $A$ can be obtained by applying the Thom space
functor to the map $\alpha$ of vector bundles fitting into 
\[\vcenter{\xymatrix@R-2ex@C+2em{
	&&
	\\
   	\ar `u[urr] `_d[urr]^{\id} [rr]
	\xi
	\ar@{-->}[r]^(0.45)\alpha
	\ar[dd]
	&
	\pi_2^\ast \xi
	\ar[dd]
	\ar[r]
	&
	\xi
	\ar[dd]
	\\ \\
	A 
	\ar[r]^(0.45){\Delta}
	&
	A\times A
	\ar[r]^-{\pi_2}
	&
	A
}}\]
where $\Delta$ is the diagonal map, $\pi_2$ is the projection onto the 
second coordinate, and the square on the right is a pullback square.
The following remark allows us to reinterpret the map
\[
	p_! \omega_p  
	\longto
	p_! S_E \tensor_{B} p_! \omega_p
\]
in \eqref{eq:psepomegapunit}
as a ``generalized Thom diagonal map.''

\begin{rem}
\label{rk:pshriekeataltdesc}
We wish to give an alternative description of the morphism $p_! (\eta)$
in \eqref{eq:psepomegapunit}.
Suppose  $p\colon E \to B$ is a fibration and $X \in \ho(\calC_{/E})$.
Let us define morphisms $\alpha$ and $\beta$ by 
considering the diagram in $\hpC$ 
\[
    \vcenter{\xymatrix@R-2ex@C+2em{
    	&
    	&
    	\\
    	X
		\ar[ddr]_{\eta}
    	\ar@{-->}[r]^\alpha_{\cart}
    	\ar `u[urr] `_d[urr]^{\id} [rr]
    	&
    	\pi_2^\ast X
    	\ar[r]^{\cart}
    	\ar@{-->}[dd]_\beta^{\opcart}
    	&
    	X
    	\ar[dd]^{\opcart}
    	\\
    	\\
    	&
    	p^\ast p_! X
    	\ar[r]^{\cart}
    	&
    	p_! X
    }}
    \qquad\text{covering}\qquad
    \vcenter{\xymatrix@R-2ex@C+2em{
    	&
    	&
    	\\
    	E
    	\ar[r]^-\delta
		\ar[ddr]_{\id}
    	\ar `u[urr] `_d[urr]^{\id} [rr]
    	&
    	E \times_B E
    	\ar[r]^{\pi_2}
    	\ar[dd]_{\pi_1}
    	&
    	E
    	\ar[dd]^p
    	\\
    	\\
    	&
    	E
    	\ar[r]^p
    	&
    	B	
    }}
\]
Here $\pi_1$ and $\pi_2$ are projections onto the first and second coordinates,
respectively;
$\delta$ is the diagonal map; the unnamed maps labeled $\cart$ and $\opcart$ are
canonical cartesian and opcartesian morphisms; $\eta$ is the unit 
of the $(p_!,p^\ast)$ adjunction;
and  $\alpha$ and $\beta$ are the unique
morphisms making the triangle on the top and the square commute.
The map $\alpha$ is cartesian by Proposition~\ref{prop:cartmorprops}
parts (\ref{it:isosarecart})
and (\ref{it:cartfactor}), and the map $\beta$ is opcartesian by 
Proposition~\ref{prop:commrelshriekinterpretation2}.
Finally, by the choice of $\eta$, the outer diagram on the left commutes,
so by the universal property of the cartesian morphism $p^\ast p_! X \to p_! X$
we must have $\eta = \beta\alpha$. 

Define now maps $\hat{\alpha}$ and $\hat{\beta}$ 
as the unique morphisms making the parallelograms in diagram
\[
	\vcenter{\xymatrix@!0@R=3.6ex@C=2.5em{
		X
		\ar[rrr]^\alpha_{\cart}
		\ar[ddddd]_(0.60){\opcart}
		\ar@(dr,l)[ddrrrrr]_(0.35){\eta}
		&&&
		\pi_2^\ast X
		\ar[ddrr]^\beta_{\opcart}
		\ar[ddddd]^(0.60){\opcart}
		\\
		\\
		&&&&&
		p^\ast p_! X
		\ar[ddddd]^(0.60){\opcart}
		\\
		\\
		\\
		p_! X
		\ar@{-->}[rrr]^{\hat{\alpha}}
		\ar@(dr,l)[ddrrrrr]_(0.35){p_!(\eta)}
		&&&
		(p\pi_1)_! \pi_2^\ast X
		\ar@{-->}[ddrr]^{\hat{\beta}}_{\homot}
		\\
		\\
		&&&&&
		p_! p^\ast p_! X		
	}}
	\qquad\text{covering}\qquad
	\vcenter{\xymatrix@!0@R=3.6ex@C=2.5em{
		E
		\ar[rrr]^\delta
		\ar[ddddd]_(0.60){p}
		\ar@(dr,l)[ddrrrrr]_(0.35){\id}
		&&&
		E \times_B E
		\ar[ddrr]^{\pi_1}
		\ar[ddddd]^(0.60){p\pi_1}
		\\
		\\
		&&&&&
		E
		\ar[ddddd]^(0.60){p}
		\\
		\\
		\\
		B
		\ar[rrr]^{\id}
		\ar@(dr,l)[ddrrrrr]_(0.35){\id}
		&&&
		B
		\ar[ddrr]^{\id}
		\\
		\\
		&&&&&
		B\vphantom{p}	
	}}
\]
commute. Here the vertical maps labeled $\opcart$ are canonical opcartesian morphisms.
By the uniqueness of opcartesian morphisms, the map $\hat{\beta}$ is an equivalence.
Moreover, by the universal property of the opcartesian morphism $X \to p_! X$,
we have 
\[
	p_!(\eta) = \hat{\beta} \hat{\alpha},
\]
giving us the alternative description of $p_!(\eta)$ we were looking for.
The equivalence $\hat{\beta}$ can be described explicitly as the composite
\[\xymatrix{
	(p\pi_1)_! \pi_2^\ast X 
	\ar[r]^-\homot
	&
	p_! (\pi_1)_! \pi_2^\ast X
	\ar[r]^-\homot
	&
	p_! p^\ast p_! X
}\]
of the natural equivalence $(p\pi_1)_! \homot p_! (\pi_1)_!$ 
and the commutation relation \eqref{eq:commrelshriek}.
\end{rem}

\begin{example}
\label{ex:unitfactorsthroughthomdiagonal}
Suppose $M$ is a smooth closed manifold.
In view of Remark~\ref{rk:pshriekeataltdesc},
in the dual pair $(\suspension^\infty_+ M, M^{-\tau_M})$
obtained by  applying
Proposition~\ref{prop:hypercartandfwduality}
to the map $M \to \pt$ and the hypercartesian morphism $S^{-\tau_M} \oto S$
in $\hpSpectra$, the unit of the duality is given by a composite
\[
	S 
	\xto{\ \tilde{\eta}\ } 
	M^{-\tau_M} 
	\xto{\ \hat{\Delta}_{-\tau_M}\ } 
	\suspension^\infty_+ M \smashprod M^{-\tau_M}
\]
where $\hat{\Delta}_{-\tau_M}$ is the Thom diagonal map for $-\tau_M$.
The map $\tilde{\eta}$ is given by the Pontryagin--Thom collapse map;
see \cite[Prop.~18.3.2, Thm.~18.5.1 and bottom of p.~302]{MaySigurdsson}.
\end{example}

Roughly speaking, the following proposition 
asserts that the dual of the map 
$(p_1)_! S_{E_1} \to (p_2)_! S_{E_2}$
induced by $g$ is the map 
$(p_2)_! \omega_{p_2} \to (p_1)_! \omega_{p_1}$
induced by $\theta$.
It will play a key role later in the proof of Theorem~\ref{thm:pdumkcomp}
and also features in Remark~\ref{rk:othetaaltdesc}.

\begin{prop}
\label{prop:dualmaps}
Suppose the diagram on the right below is a commutative triangle
in $\hpC^\dfop$ covering the triangle on the left, and 
assume that $p_1$ and $p_2$ are fibrations and $\theta_1$
and $\theta_2$ are hypercartesian.
\[
\vcenter{\xymatrix{
	E_1 
	\ar[dr]_{p_1}
	\ar[rr]^g
	&&
	E_2
	\ar[dl]^{p_2}
	\\
	&
	B
}}
\qquad\qquad\qquad
\vcenter{\xymatrix{
	\omega_{p_1}
	\ar[dr]|{\circdec}_{\theta_1}
	\ar[rr]|{\circdec}^\theta
	&&
	\omega_{p_2}
	\ar[dl]|{\circdec}^{\theta_2}
	\\
	&
	S_B
}}
\]
Write $\bar{\theta} \colon \omega_{p_2} \to (p_1)_! \omega_{p_1}$
for the morphism in $\ho(\calC_{/E_2})$ determined by $\theta$.
Then with respect to the dual pairs
$((p_1)_! S_E, (p_1)_! \omega_{p_1})$
and 
$((p_2)_! S_E, (p_2)_! \omega_{p_2})$
of Proposition~\ref{prop:hypercartandfwduality},
the composite
\begin{equation}
\label{eq:inducedbytheta} 
	(p_2)_!\omega_{p_2} 
	\xto{\ (p_2)_!\bar{\theta}\ }
	(p_2)_! g_! \omega_{p_1} 
	\xto{\ \homot\ }
	(p_1)_!\omega_{p_1}
\end{equation}
is dual to the map
\begin{equation}
\label{eq:tcalcg}
	t_\calC(\id_B,g) \colon (p_1)_! S_{E_1}\longto (p_2)_! S_{E_2}
\end{equation}
where $t_\calC$ is the functor of Proposition~\ref{prop:tcalc2}.
\end{prop}
\begin{rem}
Explicitly, 
the map \eqref{eq:tcalcg} is given by the composite
\[
	(p_1)_! S_{E_1}
	\xto{\ \homot\ }
	(p_1)_! g^\ast S_{E_2}
	\xto{\ \homot\ }
	(p_2)_! g_! g^\ast S_{E_2}
	\xto{\ (p_2)_!\varepsilon_g\ }
	(p_2)_! S_{E_2}
\]
where $\varepsilon_g$ denotes the counit of the 
$(g_!,g^\ast)$ adjunction, and the unlabeled
equivalences are the evident ones.
\end{rem}

\begin{proof}[Proof of Proposition~\ref{prop:dualmaps}]
For the proof, let us be slighly more precise about the 
connection between the $\odot_B$--product on $\Ex_B(\calC)$
and base change functors. Given fibrations $q \colon D \to B$
and $p\colon E\to B$, we call a functor 
$F \colon \ho(\calC_{/D}) \to \ho(\calC_{/E})$
$\odot_B$--\emph{representable} if there is an natural equivalence
$F \homot M\,\odot_B$ for some $1$--cell
$M\colon E \hto D$ of $\Ex_B(\calC)$, where we are identifying, as usual,
$\ho(\calC_{/D})$  with the category of globular $1$--cells 
$D \hto B$ in $\Ex_B(\calC)$, and similarly for $\ho(\calC_{/E})$.
See Remark~\ref{rk:exbcandbasechange}.
A choice of $M$ and an equivalence is called 
an $\odot_B$--\emph{representation} of  $F$, and $F$ equipped with 
an $\odot_B$--\emph{representation} is called 
$\odot_B$--\emph{represented}.
A natural transformation between 
$\odot_B$--\emph{represented} functors 
is called $\odot_B$--representable
if it arises from a globular $2$--cell between the representing 
$1$--cells, and an adjunction 
\[
	F 
	\colon 
	\ho(\calC_{/D}) 
	\longrightleftarrows 
	\ho(\calC_{/E}) 
	\colon 
	G
\]
between $\odot_B$--represented functors is called
$\odot_B$--representable if it arises from a duality between 
the representing $1$--cells of $F$ and $G$ in the manner of 
Remark~\ref{rk:dualityadjunction}. 
We note that for a continuous map $f\colon D \to E$ satisfying 
$pf = q$, the base change functors  $f_!$ and $f^\ast$
as well as the adjunction between them are representable:
the latter equivalence in \eqref{eq:basechangeformula1}
provides an $\odot$--representation $f^\ast \homot {}_f E\,\odot_B$
for $f^\ast$,
and then Proposition~\ref{prop:basechangeobjdualpair}
and uniqueness of left adjoints yield an 
$\odot$--representation for $f_!$ so that 
$(f_!,f^\ast)$ adjunction is representable.
Moreover, the assignments $f \mapsto f^\ast$ 
and $f \mapsto {}_f E\,\odot_B$  
are pseudofunctorial,
and the equivalence $f^\ast \homot {}_f E\,\odot_B$ is
the $2$--cell of a globular pseudo natural transformation; 
and similarly for 
$f\mapsto f_!$ and  $E_f\,\odot_B$,
with composition constraints given by mates of those
for $f \mapsto f^\ast$ 
and $f \mapsto {}_f E\,\odot_B$.

With the above remarks in mind, 
translating to the language of $\Ex_B(\calC)$,
and carrying over the names of morphisms
in the categories $\ho(\calC_{/E})$
to the corresponding globular $2$--cells in $\Ex_B(\calC)$,
our task is to show that the $2$--cells
\begin{equation}
\label{eq:2cells}
\vcenter{\xymatrix@C+2em@R-1ex{
    &
    E_2
    \ar|@{|}[dd]^{(E_2)_g}
    \ar|@{|}[dr]^{\vphantom{f}\omega_{p_2}}_{{}}="ur"
    \\
    B
    \ar|@{|}[ur]^{B_{p_2}}_{{}}="ul"
    \ar|@{|}[dr]_{B_{p_1}}^{{}}="dl"
    \ar@{}@<1ex>|{\displaystyle\rotatebox{30}{$\Downarrow$}}^{\,\vphantom{f_A}\homot} "ul";"dl"
    &&
    B
    \\
    &
    E_1
    \ar|@{|}[ur]_{\vphantom{f^A}\omega_{p_1}}^{{}}="dr"
    \ar@{}|{\displaystyle\rotatebox{30}{$\Downarrow$}}^{\,\vphantom{f_A}\bar{\theta}} "ur";"dr"
}}
\qquad\text{and}\qquad
\vcenter{\xymatrix@C+2em@R-1ex{
    &
    E_2
    \ar@{<-}|@{|}[dd]^{{}_g E_2}
    \ar@{<-}|@{|}[dr]^{B_{p_2}}_{{}}="ur"
    \\
    B
    \ar@{<-}|@{|}[ur]^{{}_{p_2}\!\!\: B}_{{}}="ul"
    \ar@{<-}|@{|}[dr]_{{}_{p_1}\!\!\: B\vphantom{f^A}}^{{}}="dl"
    \ar@{<-}@{}@<1ex>|{\displaystyle\rotatebox{30}{$\Uparrow$}}^{\,\homot} "ul";"dl"
    &&
    B
    \\
    &
    E_1
    \ar@{<-}|@{|}[ur]_{\vphantom{f^A}B_{p_1}}^{{}}="dr"
    \ar@{}|{\displaystyle\rotatebox{30}{$\Uparrow$}}^{\,\kappa} "ur";"dr"
}}
\end{equation}
are mates in $\Ex_B(\calC)$,
where $\kappa$ is the $2$--cell
\[
\kappa = 
\left(
	\vcenter{\xymatrix@C+2em@R-1.3ex{
        &
        E_2
        \ar|@{|}[dd]^{(E_2)_g}
        \ar@{<-}|@{|}[dr]^{B_{p_2}}_{{}}="ur"
        \\
        E_2
        \ar@{<-}|@{|}[ur]^{U^B_{E_2}}_{{}}="ul"
        \ar@{<-}|@{|}[dr]_{{}_{g} E_2\vphantom{f^A}}^{{}}="dl"
        \ar@{<-}@{}@<1ex>|{\displaystyle\rotatebox{30}{$\Uparrow$}}^{\,\vphantom{f}\varepsilon_g} "ul";"dl"
        &&
        B
        \\
        &
        E_1
        \ar@{<-}|@{|}[ur]_{\vphantom{f^A}B_{p_1}}^{{}}="dr"
        \ar@{}|{\displaystyle\rotatebox{30}{$\Uparrow$}}^{\,\homot} "ur";"dr"
    }}
\right).
\]
Notice that ${}_{p_1}\!B \homot S_{E_1}$ and 
${}_{p_2} \!\!\: B \homot S_{E_2}$.
Here the unlabeled equivalences are composition constraints for 
the pseudofunctors $f \mapsto E_{f}$ and $f\mapsto {}_f E$.
The two $2$--cells on the left in \eqref{eq:2cells}
are mates, so it is enough to show that $\bar{\theta}$ and $\kappa$
are mates.

Observe now that the choice of $\bar{\theta}$ ensures that the diagram
\[\xymatrix{
	S_B 
	\ar[r]^(0.45){\eta_2}
	\ar[d]_{\eta_1}
	&
	(p_2)_! \omega_{p_2}
	\ar[d]^{\bar{\theta}}
	\\
	(p_1)_! \omega_{p_1}
	\ar[r]^(0.45)\homot
	&
	(p_2)_! g_! \omega_{p_1}	
}\]
commutes, where $\eta_1$ and $\eta_2$
are the morphisms determined by $\theta_1$ and $\theta_2$.
Translating  to the language of $\Ex_B(\calC)$,
it follows that the top square in the following diagram commutes:
\[\xymatrix@C-2em{
    B_{p_1} \odot {}_g E_2
    \ar[rrr]^(0.42){\eta_2 \odot \id \odot \id}
    \ar[d]_{\eta_1 \odot \id \odot \id}
    &&&
    B_{p_2} \odot \omega_{p_2} \odot B_{p_1}\odot {}_g E_2
    \ar[d]^{\id \odot \bar{\theta} \odot \id \odot \id}
    \\
    B_{p_1} \odot \omega_{p_1} \odot B_{p_1} \odot {}_g E_2
    \ar[rrr]^(0.42)\homot
    \ar[d]_{\id \odot \varepsilon_1 \odot \id}
    &&&
    B_{p_2} \odot (E_2)_g \odot \omega_{p_1} \odot B_{p_1} \odot {}_g E_2
    \ar[d]^{\id \odot \id \odot \varepsilon_1 \odot \id}
    \\
    B_{p_1} \odot {}_g E_2
    \ar[rrr]^(0.42)\homot
    &&&
    B_{p_2} \odot (E_2)_g \odot {}_g E_2
    \ar[r]^-{\id \odot \varepsilon_g}
	&
	B_{p_2}
}\]
Here we have written $\odot$ for $\odot_B$
and omitted unit and associativity constraints for brevity;
$\varepsilon_1$ denotes the counit for the dual pair
$(B_{p_1},\omega_{p_1})$; the horizontal equivalences
are induced by the equivalence 
$B_{p_1} \homot  B_{p_2} \odot_B (E_2)_g$; and
the bottom square commutes by naturality.
By definition, the composite of the top, right hand vertical, and 
bottom right morphisms is the mate of $\bar{\theta}$.
Moreover, the composite of the left hand vertical morphisms
is the identity. Thus the claim follows.
\end{proof}

\subsubsection{Proofs of Theorems~\ref{thm:thetapc}, \ref{thm:kappaxcartcritc} and
\ref{thm:tildekappaxcart}}
\label{subsubsec:proofsofresultsfromintro}

We finish the section by proving 
Theorems~\ref{thm:thetapc}, \ref{thm:kappaxcartcritc} and
\ref{thm:tildekappaxcart} from the introduction.
In view of the uniqueness of cartesian morphisms, 
the following result gives geometric interpretations for the 
hypercartesian morphism $\omega_p \oto S_B$ of
Theorem~\ref{thm:hypercartexistence}(\ref{it:hypercartesianmorphismfors})
and the object $\omega_p$ when $p\colon E \to B$ is a bundle of 
smooth closed manifolds and $\calC = \Spectra$.

\begin{thrm}
\label{thm:abgpthc2}
Let $p \colon E \to B$  be a bundle of smooth closed manifolds.
Then the morphism 
\[
	S^{-\tau_p} \longoto S_B 
\]
of $\hpSpectra^\dfop$ covering $p$ defined by the
parametrized Pontryagin--Thom transfer map
\[
	\mathrm{PT}_{/B}(p) \colon S_B \longto p_! S^{-\tau_p}
\]
of \cite[Def.~4.14]{ABGparam}
is hypercartesian.
\end{thrm}
\begin{proof}
By Example~\ref{ex:cwdualityinspectra},
the map $p$ is small-fibred with respect to $\Spectra$, so  
Proposition~\ref{prop:smallfibcartcrit}
allows us to reduce to the case $B=\pt$.
But in this case, 
$E$ is just a smooth closed manifold,
and the parametrized Pontryagin--Thom 
transfer is equivalent to the Pontryagin--Thom map
\[
	S \longto p_! S^{-\tau_E}
\]
for $E$.
Since this latter map is the unit for the dual pair
$(\pt_p,  S^{-\tau_E})$
in $\Ex(\Spectra)$
implicit in Example~\ref{ex:cwdualityinspectra}
(see \cite[Prop.~18.3.2, Thm.~18.5.1 and bottom of p.~302]{MaySigurdsson}),
the claim follows from Theorem~\ref{thm:hypercartdata}.
\end{proof}

\begin{proof}[Proof of Theorem~\ref{thm:thetapc}]
Part~(\ref{it:thetapc}) follows from 
the equivalence of parts~(\ref{it:smallfibred}) and (\ref{it:hypercartesianmorphismfors}) 
of Theorem~\ref{thm:hypercartexistence}.
Indeed, this equivalence shows that the map $\theta_p$ 
can be chosen to be hypercartesian, wherefore by the uniqueness
of cartesian morphism any choice for $\theta_p$ is hypercartesian.
Part~(\ref{it:phioslashthetap}) now follows from 
Proposition~\ref{prop:superandhypercartmorprops}(\ref{it:hypercartimpliessupercart}).
Part~(\ref{it:thetapformanifoldbundles}) follows from Example~\ref{ex:cwdualityinspectra} and Theorem~\ref{thm:abgpthc2}
while
Part~(\ref{it:thetapformanifoldbundlesc})
follows from Proposition~\ref{prop:cwdualitytransport}
and Proposition~\ref{prop:hypercartmorpreservation}.
\end{proof}

\begin{proof}[Proof of Theorem~\ref{thm:kappaxcartcritc}]
The map $\kappa_X$ of equation \eqref{eq:kappax} is cartesian under
condition~(\ref{it:xdualizable}) by Proposition~\ref{prop:cartoslashproperty}.
Suppose now condition~(\ref{it:xpulledbackfromb}) holds, that is, that there 
exists a cartesian morphism $\psi \colon X \to Y$ of $\hpC$ covering $p_2$.
We then have the following commutative triangle covering \eqref{eq:fp1p2new}
\begin{equation}
\label{diag:oslashwithkappadef} 
\vcenter{\xymatrix@R+2ex{
	f^\ast X \tensor_{E_1} \omega_{p_1}
	\ar[rr]|{\circdec}^{\kappa_X = \phi \oslash \kappa \vphantom{f}}
	\ar[dr]|{\circdec}_{\psi \phi \oslash \theta_{p_1}}
	&&
	X \tensor_{E_2} \omega_{p_2}
	\ar[dl]|{\circdec}^{\psi \oslash\theta_{p_2}}
	\\
	&
	Y \tensor_B S_B
}}
\end{equation}
where $\theta_{p_1}$ and $\theta_{p_2}$ are are the cartesian morphisms appearing in 
diagram~\eqref{diag:kappadef}.
By Theorem~\ref{thm:hypercartexistence} and the uniqueness of 
cartesian morphisms, the morphisms $\theta_{p_1}$ and $\theta_{p_2}$
are in fact hypercartesian, so by 
Proposition~\ref{prop:superandhypercartmorprops}(\ref{it:hypercartimpliessupercart})
the morphisms $\psi \phi \oslash \theta_{p_1}$ and $\psi \oslash\theta_{p_2}$
in \eqref{diag:oslashwithkappadef} are cartesian.
But now Proposition~\ref{prop:cartmorprops}(\ref{it:cartfactor})
implies that the morphism $\kappa_X$ is cartesian as  claimed.
\end{proof}

\begin{proof}[Proof of Theorem~\ref{thm:tildekappaxcart}]
Let $\psi_\pm \colon f^\ast \omega_{p_2}^{\pm 1} \to \omega_{p_2}^{\pm 1}$
be the canonical cartesian morphisms covering $f$.
The claim now follows from Proposition~\ref{prop:cartoslashproperty}
by observing that the maps $\tilde{\kappa}_X$ and $\kappa_X$ are isomorphic to the maps 
$\psi_{-} \oslash \kappa_X$ and
$\psi_{+} \oslash \tilde{\kappa}_X$, respectively.
\end{proof}

\section{Costenoble--Waner dualizability of semisimple \texorpdfstring{$\ell$}{l}--compact groups}
\label{sec:cwdualizabilityofsemisimplegrps}

\begin{defn}
\label{def:semisimple}
A connected $\ell$--compact group $BG$ is called
\emph{semisimple}
if the group $\pi_1(G) \isom \pi_2(BG)$ is finite. 
\end{defn}

Our aim in this section is to prove the following result.

\begin{thrm}
\label{thm:ellcptgrpscwdualizable}
Suppose $BG$ is a semisimple
connected $\ell$--compact group
of dimension $d$.
Then the space $G = \loops BG$ 
is Costenoble--Waner dualizable in 
$\Spectra^\ell$
 with Costenoble--Waner dual 
$\suspension^{-d}_G S_{G,\ell}$
where $S_{G,\ell}$ denotes the unit object in 
$\ho(\Spectra^\ell_{/G})$.
\end{thrm}

\begin{proof}
By \cite[Lemma~7]{bauer04}, 
there is an $\F_\ell$--homology equivalence
$\lambda \colon X \to G$ where $X$ has the 
homotopy type of a finite CW complex, 
and factoring $\lambda$ as a homotopy equivalence 
followed by a fibration, we may assume that $\lambda$ 
is a fibration.
By Example~\ref{ex:cwdualityinspectra}
and Propositions~\ref{prop:cwdualityweyinvariance} 
and~\ref{prop:cwdualitytransport}, 
$X$ is Costenoble--Waner
dualizable in $\Spectra^\ell$, so we have a dual pair
$(S_{X,\ell}^\op,T)$ in $\Ex(\Spectra^\ell)$.
By Corollary~\ref{cor:shriekduals},
$((\lambda_! S_{X,\ell})^\op,\lambda_! T)$ 
is a dual pair
in $\Ex(\Spectra^\ell)$.
To prove that $G$ is Costenoble--Waner
dualizable in $\Spectra^\ell$, it therefore suffices to 
show that $\lambda_! S_{X,\ell} \homot S_{G,\ell}$ in 
$\ho(\Spectra^\ell_{/G})$.
We will prove this by showing that 
the counit of the $(\lambda_!,\lambda^\ast)$ adjunction
gives an equivalence 
$\varepsilon\colon \lambda_!\lambda^\ast S_{G,\ell} \xto{\,\homot\,} S_{G,\ell}$.
As $G$ is connected and equivalences between
objects of $\ho(\Spectra^\ell_{/G})$ are detected
on fibres, it suffices to show that $\varepsilon$
restricts to an equivalence on fibres over the basepoint
$e$ of $G$. Writing $F = \lambda^{-1} (e)$ and $r_F = r^F$
for the unique map $F \to \pt$ and $S$ for the 
unit object in $\ho(\Spectra^\ell)$, this restriction
is equivalent to the map
\[
	\varepsilon \colon r^F_! r_F^\ast S \longto S
\]
given by the counit of the $(r^F_!, r_F^\ast)$ adjunction,
which in turn is equivalent to the map
\[
	L\suspension^\infty_+ F \longto L\suspension^\infty_+ \pt
\]
induced by $r_F$. 
Here $L$ denotes the localization functor 
$\Spectra \to \Spectra^\ell$.
As equivalences in $\Spectra^\ell$
are detected on mod $\ell$ homology, we are left with the 
task of showing that the map 
$r^F_\ast \colon H_\ast(F;\,\F_\ell) \to H_\ast(\pt;\,\F_\ell)$
is an isomorphism.

By \cite[Props.~VI.5.4 and VI.4.3]{bk},
the group $\pi_1(G) \isom \pi_2(BG)$
has the structure of a $\Z_\ell$--module,
so the assumption that $BG$ is semisimple
implies that it is a finite abelian $\ell$--group.
Let $C$ be the cokernel of the map 
$\lambda_\ast \colon \pi_1 X \to \pi_1 G$.
Since abelianization and tensoring with $\F_\ell$ are
right exact, by the Hurewicz and Universal Coefficient Theorems
we have an exact sequence
\begin{equation}
\label{eq:lambdahes}
\xymatrix{
	H_1(X;\F_\ell) 
	\ar[r]^{\lambda_\ast}
	&
	H_1(G;\F_\ell) 
	\ar[r]
	&
	C^\mathrm{ab}\tensor \F_\ell
	\ar[r]
	&
	0.
}
\end{equation}
By the choice of $\lambda$, the map $\lambda_\ast$ in 
\eqref{eq:lambdahes} is an isomorphism, so it follows
that $C^\mathrm{ab} \tensor \F_\ell = 0$. But as a quotient of a 
finite abelian $\ell$--group the group
$C$ is also a finite abelian $\ell$--group,
so this implies that $C=0$. Thus 
$\lambda_\ast \colon \pi_1(X) \to \pi_1(G)$ is 
surjective, and 
the long exact homotopy sequence of the fibration $\lambda$
implies that $F$ is connected. 
Moreover, since $\pi_1 G$ is a finite $\ell$--group, 
the action of $\pi_1 G$
on $H_\ast(F;\,\F_\ell)$ is nilpotent. Now the mod-$R$ fibre lemma
\cite[II.5.1]{bk} implies that the sequence
\[
	F\lcom \longto X\lcom \xto{\ \lambda\lcom\ } G\lcom
\]
obtained by applying the Bousfield--Kan $\ell$--completion functor
to the fibre sequence $F\to X \xto{\lambda} G$ is again a fibre sequence.
Since $\lambda$ is an $\F_\ell$--homology equivalence, 
the map $\lambda\lcom$
is a homotopy equivalence \cite[Lemma~I.5.5]{bk}.
Thus $F\lcom$ is contractible. In particular, $F\lcom$ 
and hence $F$ are $\ell$--good \cite[Prop.~I.5.2]{bk}, and 
$H_\ast(F;\,\F_\ell)\isom H_\ast(F\lcom;\,\F_\ell)\isom H_\ast(\pt;\,\F_\ell)$,
as desired.

It remains to show that the Costenoble--Waner
dual of $G$ is $\suspension^{-d}_G S_{G,\ell}$.
From Remark~\ref{rk:candidaterightdual}, 
we know that the Costenoble--Waner dual of $G$ 
is given by the object
\begin{equation}
\label{eq:tgformula}
\begin{aligned}
	T_{G} &
		= 
		S_{{G},\ell}^\op \vartriangleright U_{G}
		\homot (\pi_1)_\ast F_{{G}\times {G}}(\pi_2^\ast S_{{G},\ell}, \Delta_! S_{{G},\ell})
		\homot (\pi_1)_\ast \Delta_! S_{{G},\ell}
		\homot (\pi_1)_\ast \mathrm{sh}_! (er_{G},\id_{G})_! S_{{G},\ell}
		\\
		&
		\homot (\pi_1)_\ast \mathrm{sh}_! (er_{G},\id_{G})_! r^\ast_{G}  S
		\homot (\pi_1)_\ast \mathrm{sh}_! \pi_1^\ast e_! S 
		\homot (\pi_1)_\ast \mathrm{sh}_\ast \pi_1^\ast e_!  S
		\homot \mu_\ast \pi_1^\ast e_!  S
		\homot r_{G}^\ast r^{G}_\ast e_!  S
\end{aligned}
\end{equation}
of $\ho(\Spectra^\ell_{/G})$.
Here $\pi_i \colon {G}\times {G} \to {G}$ denotes projection 
onto the $i$--th factor; $\Delta$ is the diagonal map of ${G}$;
$\mathrm{sh} \colon {G}\times {G} \to {G}$ denotes the shear 
map $(g,h)\mapsto (gh,h)$; $e\colon \pt \to {G}$ is the 
inclusion of the basepoint, given by the constant loop, into ${G}$;
$\mu\colon {G}\times {G} \to {G}$ is the multiplication map of ${G}$;
$r^{G}$ and $r_{G}$ denote the unique map ${G} \to \pt$;
and we have again written $S$ for the unit object in $\ho(\Spectra^\ell)$.
The first equivalence in the top row follows from
formulas \eqref{eq:uformula}
and \eqref{eq:ntrianglerightpformula};
the second from the equivalence
 $\pi_2^\ast S_{{G},\ell} \homot S_{{G}\times {G},\ell}$;
and the third from the observation that $\Delta$ and the 
composite
\[\xymatrix@C+2em{
	{G} 
	\ar[r]^-{(er_{G},\id_{G})}
	&
	{G}\times {G} 
	\ar[r]^-{\mathrm{sh}}_{\homot}
	&
	{G}\times {G}
}\]
are homotopic.
In the second row, the first equivalence follows from 
the equivalence $S_{{G},\ell}\homot r_{G}^\ast S$
and the second from the commutation relation
\eqref{eq:commrelshriek}
arising from the homotopy cartesian square
\[\xymatrix{
	{G}
	\ar[r]^{r_{G}}
	\ar[d]_{\mathllap{(er_{G},\id_{G})}}
	&
	\pt
	\ar[d]^{e}
	\\
	{G}\times {G}
	\ar[r]^-{\pi_1}
	&
	{G}
}\]
The third equivalence in the second row
follows since $\mathrm{sh}_!$ and $\mathrm{sh}_\ast$
are both inverses to the equivalence of categories $\mathrm{sh}^\ast$
induced by the homotopy equivalence $\mathrm{sh}$.
The second last equivalence in the second row now follows
from the observation that $\pi_1 \circ \mathrm{sh} = \mu$,
and the final equivalence from 
the commutation relation \eqref{eq:commrelstar}
arising from the homotopy cartesian square
\[\xymatrix{
	{G}\times {G}
	\ar[r]^-\mu
	\ar[d]_{\pi_1}
	&
	{G}
	\ar[d]^{r_{G}}
	\\
	{G}
	\ar[r]^(0.55){r_{G}}
	&
	\pt	
}\]

To complete the proof, it suffices to 
show that 
$r^{G}_\ast e_! S \homot LS^{-d}$ in $\ho(\Spectra^\ell)$
where $L\colon \Spectra \to \Spectra^\ell$ again is the
localization functor.
By \eqref{eq:tgformula} and 
Remark~\ref{rk:cwdualityasrefinement},
$L\suspension^\infty_+ {G} \homot r^{G}_! S_{G,\ell}$ 
is dualizable in $\ho(\Spectra^\ell)$ 
with dual 
\[
	r^{G}_! T_{G} 
	\homot 
	r^{G}_! r_{G}^\ast r^{G}_\ast e_! S
	\homot
	r^{G}_! (S_{G,\ell} \smashprod^\ell r_{G}^\ast  r^{G}_\ast e_! S)
	\homot
	(r^{G}_! S_{G,\ell}) \smashprod^\ell (r^{G}_\ast e_! S)
\]
where the last equivalence follows from the projection formula.
As the homology of the dual (in $\ho(\Spectra^\ell)$) of 
$L\suspension^\infty_+ {G} \homot r^{G}_! S_{G,\ell}$ 
agrees with $H^{-\ast} ({G};\,\F_\ell)$, %
we conclude that 
\[
	H^{-\ast}({G};\,\F_\ell) 
	\isom	
	H_\ast ({G};\,\F_\ell) \tensor H_\ast( r^{G}_\ast e_! S;\,\F_\ell).
\]
It follows that
\[
	H_n (r^{G}_\ast e_! S;\,\F_\ell) 
	\isom 
    \begin{cases}
        \F_\ell & \text{if $n = -d$}\\
        0 & \text{otherwise}
    \end{cases}
\]
Pick a map $f_0 \co r^{G}_\ast e_! S \to \suspension^{-d} H\F_\ell$ 
representing the 
dual of a generator 
of $H_{-d}(r^{G}_\ast e_! S)$.
The long exact cohomology sequences
associated with the short exact sequences
\begin{equation}
\label{ses:coeffs}
0\longto \Z/\ell \longto \Z/{\ell^{k+1}} \longto \Z/{\ell^k} \longto 0
\end{equation}
of coefficients show that the map
$H^{-d}(r^{G}_\ast e_! S;\,\Z/{\ell^{k+1}}) \to H^{-d}(r^{G}_\ast e_! S;\,\Z/{\ell^{k}})$
is an epimorphism for all $k\geq 1$.
Thus the map $f_0$ can be lifted along the tower
\[
	H\Z_\ell \homot \holim_k H\Z/\ell^k
	\longto
	\cdots
	\longto
	H\Z/{\ell^3}
	\longto
	H\Z/{\ell^2}
	\longto
	H\Z/{\ell}
\]
to a map
\[
	f_1
	\co 
	r^{G}_\ast e_! S 
	\longto 
	\suspension^{-d} H\Z_\ell 
	= 
	\suspension^{-d} H\pi_{-d}(LS^{-d}).
\]
Working up the Postnikov tower of $LS^{-d}$,
we can moreover find a lift 
\[
	f_2 \co r^{G}_\ast e_! S \longto LS^{-d}
\]
of $f_1$:
an induction on $k$ using the long exact 
sequences associated to the short exact sequences
\eqref{ses:coeffs}
proves that 
$H^n (r^{G}_\ast e_! S;\,\Z/\ell^k) = 0$ for all $n\neq -d$ and $k\geq 1$,
showing that all the obstructions for finding a lift vanish.
The map $f_2$ induces a nontrivial map on mod $\ell$ homology 
since the map $f_0$ does. Thus the map $f_2$ is an equivalence.
\end{proof}

\begin{rem}
\label{rk:cwdualofgalt}
We note an alternative description of the Costenoble--Waner
dual of $G$ in the case where $BG = BK\lcom$ for 
a semisimple compact connected Lie group $K$.
Suppose $BG$ is of this form, and let $\lambda \colon K \to G$
be the evident $\F_\ell$--homology equivalence.
By Example~\ref{ex:cwdualityinspectra}, the space $K$ has Costenoble--Waner
dual $S^{-\tau_K}$ with respect to $\Spectra$,
so by Proposition~\ref{prop:cwdualitytransport}
applied to the localization functor $L \colon \Spectra \to \Spectra^\ell$,
it has Costenoble--Waner dual $L_K S^{-\tau_K}$
with respect to $\Spectra^\ell$.
An inspection of the beginning 
of the proof of Theorem~\ref{thm:ellcptgrpscwdualizable}
now reveals that the Costenoble--Waner dual of $G$
with respect to $\Spectra^\ell$
identifies with $\lambda_! L_K S^{-\tau_K}$.
\end{rem}

\begin{rem}
\label{rk:s1lcomnotcwdualizable}
Theorem~\ref{thm:ellcptgrpscwdualizable} does not generalize
to arbitrary connected $\ell$--compact groups. Indeed, 
the space $(S^1)\lcom \homot \loops (BS^1)\lcom$
is not Costenoble--Waner dualizable in $\Spectra^\ell$
or even in $\Mod^{H\F_\ell}$.
\end{rem}

\section{The Serre spectral sequence for parametrized \texorpdfstring{$R$}{R}--modules}
\label{sec:serress}

Our aim in this section is to present a generalization of 
the Serre spectral sequence of a fibration to the context of 
parametrized $R$--modules for $R$ a commutative ring spectrum.
See Theorem~\ref{thm:serresss} below.
Throughout this section, we assume that $R$ is a commutative 
ring spectrum. We write $R_\ast$ for the
graded ring $R_\ast = \pi_\ast(R)$ and let 
$R^\ast = R_{-\ast}$. Moreover, given $R$--modules $M$ and $N$,
we write
\[
	M_\ast(N) = M^R_\ast(N) = \pi_\ast(M\smashprod^R N)
	\qquad\text{and}\qquad
	M^\ast(N) = M_R^\ast(N) = \pi_{-\ast} F^R(N,M)
\]
where $\smashprod^R$ and $F^R(-,-)$ denote the smash product
and the internal hom in $\ho(\Mod^R)$.

\begin{defn}[Local coefficient systems $\calL_\ast(M,X)$ and $\calL^\ast(M,X)$]
\label{def:lmxs}
Recall that objects of $\Mod^{R}_{/B}$ are  $\infty$--functors
$\Pi_\infty (B)^\op \to \Mod^{R}$. Given an object $X \in \Mod^{R}_{/B}$
and an $R$--module $M$,
we let  $\calL_\ast(M,X)$ and $\calL^\ast(M,X)$
be the local coefficient systems
of graded $R_\ast$ and $R^\ast$--modules given by the composites
\[	
	\xymatrix@C+1.3em{	
		\calL_\ast(M,X)
    	\colon 
	    \Pi_1(B)
		\ar[r]^-\isom
		&
    	\Pi_1(B)^\op
    	\ar[r]^-{X}
		&
		\ho(\Mod^{R})
		\ar[r]^-{M_\ast}
		&
		\grMod^{R_\ast}
	}
\]
and 
\[
	\xymatrix@C+1.3em{
		\calL^\ast(M,X)
    	\colon 
    	\Pi_1(B)
    	\ar[r]^-{X^\op}
		&
		\ho(\Mod^{R})^\op
		\ar[r]^-{M^\ast}
		&
		\grMod^{R^\ast}
	}
\]
where we have continued to write $X$ for the 
functor $\Pi_1(B)^\op \to \ho(\Mod^{R})$
induced by $X$ and the 
first functor in the definition of $\calL_\ast(M,X)$
is the isomorphism sending each morphism to its inverse.
\end{defn}

\begin{thrm}[Serre spectral sequences; cf.\ {\cite[Thm.~20.4.1]{MaySigurdsson}}]
\label{thm:serresss}
Let $M$ be an $R$--module
and let $X$ be a parametrized $R$--module over a space $B$.
\begin{enumerate}[(i)] 
\item There exists a strongly convergent spectral sequence
	\begin{equation}
    \label{eq:serresshomology}
        	E^2_{s,t}
			=
			H_s(B;\calL_t(M,X)) 
    		\Longrightarrow 
    		M_{s+t}H_\bullet(B;X).
    \end{equation}
\item \label{it:serresscohomology}
	There exists a spectral sequence
	\begin{equation}
    \label{eq:serresscohomology}
		E_2^{s,t} 
		= 
		H^s(B; \calL^t(M,X))
		\Longrightarrow 
		M^{s+t}H_\bullet(B;X)
    \end{equation}
	which converges strongly when $X$ is 
	bounded from below in the sense that 
	for every $b\in B$
	there exists a $t_0 \in \Z$
	such that $M^t(X_b) = 0$ for all $t < t_0$.
\end{enumerate}
Both spectral sequences are functorial in $B$, $M$ and $X$.
\end{thrm}

\begin{rem}
\label{rk:ordinaryserress}
Given a ring $T$, the usual Serre spectral sequences
of a fibration $p\colon E \to B$ for homology and 
cohomology with coefficients in $T$
can be recovered from 
Theorem~\ref{thm:serresss}
by taking $R=M=HT$ and $X = p_! (HT_E)$.
\end{rem}

\begin{rem}
Fibers of a parametrized $R$--module $X \in \ho(\Mod^{R}_{/B})$ 
over points in the same path component of $B$ are equivalent
to each other. Thus to show that $X$ is bounded from below,
it suffices to show that the condition in 
part (\ref{it:serresscohomology})
is satisfied for \emph{some} point $b$ 
from every path component of $B$.
\end{rem}

Given a CW complex $B$, we write $B^{(n)}$ for the 
$n$--skeleton of $B$, interpreting $B^{(n)}=\emptyset$
for $n < 0$.
If $\calL$ is  a local coefficient system over $B$,
we write $C_\ast(B;\calL)$ (resp.\ $C^\ast(B;\calL)$)
for the cellular chain complex (resp.\ the cellular 
cochain complex) of $B$ with coefficients in $\calL$,
so that
\[
	C_n(B;\,\calL) 
	= 
	H_n(B^{(n)}, B^{(n-1)};\,\calL)
	\qquad\text{and}\qquad
	C^n(B;\,\calL) 
	= 
	H^n(B^{(n)}, B^{(n-1)};\,\calL).
\]
We recall that the differentials in $C_\ast(B;\,\calL)$
and $C^\ast(B;\,\calL)$ are given by the composites
\[
	H_n(B^{(n)}, B^{(n-1)};\,\calL)
	\xto{\ \bdry\ }
	H_{n-1}(B^{(n-1)};\,\calL)
	\longto
	H_{n-1}(B^{(n-1)},B^{(n-2)};\,\calL)
\]
and
\[
	H^n(B^{(n)}, B^{(n-1)};\,\calL)
	\longto
	H^n(B^{(n)};\,\calL)
	\xto{\ \delta\ }	
	H^n(B^{(n+1)},B^{(n)};\,\calL)	
\]
of connecting homomorphisms and maps induced by inclusions.

\begin{proof}[Proof of Theorem~\ref{thm:serresss}]
Recall that a CW complex $B$ is called \emph{regular}
if each closed cell of $B$ is homeomorphic
to a disk $D^n$ for some $n$.
We call a continuous map 
$f\colon A\to B$ 
between regular CW complexes
\emph{strongly cellular}
if for every closed cell $E$ of $A$,
the restriction of $f$ to $E$ defines a homeomorphism
from $E$ onto a closed cell of $B$.
By applying the functorial CW approximation 
given by the composite of the singular semisimplicial set
and the geometric realization functors, 
we see that it is enough to construct the spectral sequences
functorially on the category of regular CW complexes
and strongly cellular maps.

We will first construct the spectral sequence
\eqref{eq:serresshomology}.
Applying $M_\ast$ to the sequence
\begin{equation}
\label{eq:sequence} 
	H_\bullet(B^{(0)};\, X) \longto H_\bullet(B^{(1)};\, X) \longto \cdots
\end{equation}
of $R$--modules obtained from the skeleta of a regular CW complex $B$
yields the unrolled exact couple
\begin{equation}
\label{eq:unrolledechomology}
\vcenter{\xymatrix@!0@C=5.9em@R=10ex{
	\cdots
	\ar@{->}[r]
	&
	M_\ast H_\bullet(B^{(n-1)};\, X) 
	\ar@{->}[rr]
	&&
	M_\ast H_\bullet(B^{(n)};\, X) 
	\ar@{->}[rr]
	\ar@{->}[dl]
	&&
	M_\ast H_\bullet(B^{(n+1)};\, X) 
	\ar@{->}[r]
	\ar@{->}[dl]
	&
	\cdots
	\\
	&
	\cdots\quad
	&
	M_\ast H_\bullet(B^{(n)}, B^{(n-1)};\, X)
	\ar@{->}[ul]
	&&
	M_\ast H_\bullet(B^{(n+1)}, B^{(n)};\, X)
	\ar@{->}[ul]
	&
	\quad\cdots
}}
\end{equation}
(Here for simplicity of notation, we continue to write $X$ for the 
restriction of $X$ to a subspace of $B$.)
By \cite[Thm.~6.1]{BoardmanSS}, the resulting spectral 
sequence converges strongly to the colimit
\[
	\colim_{n\to \infty} M_\ast H_\bullet(B^{(n)};\,X)
	\isom 
	M_\ast H_\bullet(B;\,X).
\]
To construct \eqref{eq:serresshomology},
it remains to identify the $E_2$--page of the spectral sequence.
For $n\geq 0$, let $\{E^n_\alpha\}_{\alpha\in J_n}$
be the set of $n$--cells of $B$, and write $\bdry E^n_\alpha \subset E^n_\alpha$
for the boundary of $E^n_\alpha$. By the assumption that $B$ is regular,
we have $(E^n_\alpha,\bdry E^n_\alpha) \homeom (D^n,S^{n-1})$ for all $\alpha \in J_n$.
We choose arbitrarily a point $x_\alpha \in E^n_\alpha$ for each $\alpha \in J_n$,
and write $X_\alpha$ for the fibre of $X$ over $x_\alpha$. Since $E^n_\alpha$
is contractible, there then exists a unique equivalence
\begin{equation}
\label{eq:xrestriv}
	X | E^n_\alpha \xto{\ \homot\ } \underline{X_\alpha} 
\end{equation}
in $\ho(\Mod^R_{/E^n_\alpha})$
restricting to the identity map on fibres over $x_\alpha$.
Similarly, there exists a unique isomorphism
\begin{equation}
\label{eq:lmnrestriv}
 	\calL_\ast(M,X) | E^n_\alpha \xto{\ \isom\ } M_\ast(X_\alpha)
\end{equation}
of local coefficient systems over $E^n_\alpha$
extending the identity map on fibres over $x_\alpha$. Here the 
the right hand side is to be interpreted as
the constant local coefficient system given by $M_\ast(X_\alpha)$.
We have a sequence of isomorphisms
\begin{align*}
    M_\ast H_\bullet (B^{(n)}&, B^{(n-1)};\,X)
    \xot{\ \isom\ }
    \bigoplus_{\alpha\in J_n}
        M_\ast H_\bullet (E^n_\alpha, \bdry E^n_\alpha;\, X)
    \xto{\ \isom\ }
    \bigoplus_{\alpha\in J_n}
        M_\ast H_\bullet (E^n_\alpha, \bdry E^n_\alpha;\, \underline{X_\alpha})
    \\
    &\xto{\ \isom\ }
    \bigoplus_{\alpha\in J_n}
    	M_\ast (X_\alpha \smashprod E^n_\alpha / \bdry  E^n_\alpha)
    \xot[{\times}]{\ \isom\ }
    \bigoplus_{\alpha\in J_n}
    	M_\ast (X_\alpha) \tensor_{R_\ast} \tilde{R}_\ast(E^n_\alpha / \bdry  E^n_\alpha)
	\\
	&\xto{\ \isom\ }
    \bigoplus_{\alpha\in J_n}
    	M_\ast (X_\alpha) \tensor_{\Z} \tilde{H}_\ast(E^n_\alpha / \bdry  E^n_\alpha;\,\Z)
	\xot{\ \isom\ }
    \bigoplus_{\alpha\in J_n}
    	M_\ast (X_\alpha) \tensor_{\Z} H_\ast(E^n_\alpha, \bdry  E^n_\alpha;\,\Z)
	\\
	&\xto{\ \isom\ }
    \bigoplus_{\alpha\in J_n}
    	H_\ast(E^n_\alpha, \bdry  E^n_\alpha;\, M_\ast (X_\alpha))
	\xot{\ \isom\ }
    \bigoplus_{\alpha\in J_n}
    	H_\ast(E^n_\alpha, \bdry  E^n_\alpha;\, \calL_\ast (M,X))
	\\
	&\xto{\ \isom\ }
    	H_\ast(B^{(n)}, B^{(n-1)};\, \calL_\ast (M,X)) = C_n (B;\,\calL_\ast (M,X))
\end{align*}
where the first isomorphism follows from the excision, homotopy invariance
and additivity properties of $M_\ast$ and $H_\bullet$ 
(Theorem~\ref{thm:eilenbergsteenrodaxioms}(\ref{it:excision}) and
Corollaries~\ref{cor:relhomotopyinvariance} and~\ref{cor:reladditivity});
the second isomorphism is induced by \eqref{eq:xrestriv};
the first isomorphism on the second line follows from 
Example~\ref{ex:relativehbulletfortrivialcoeffs};
the second isomorphism on the second line is given by cross product;
the first isomorphism on the third line is the evident one 
following from the fact that $\tilde{R}(E^n_\alpha/ \bdry E^n_\alpha)$
is a free $R_\ast$--module over a single generator of degree $n$;
the second isomorphism on the third line is induced by the 
quotient map $E^n_\alpha \to E^n_\alpha / \bdry E^n_\alpha$;
the first isomorphism on the fourth line is 
induced by the isomorphism
\[
	M_\ast(X_\alpha) \tensor_\Z S_\ast (E^n_\alpha,\bdry E^n_\alpha)
	\isom
	S_\ast (E^n_\alpha,\bdry E^n_\alpha) \tensor_\Z M_\ast(X_\alpha) 
	=
	S_\ast (E^n_\alpha,\bdry E^n_\alpha; \, M_\ast(X_\alpha))
\]
of singular chain complexes given by the 
symmetry constraint $m\tensor z \mapsto (-1)^{\deg(m)\deg(z)} z\tensor m$;
the second isomorphism on the fourth line is induced by \eqref{eq:lmnrestriv};
and the last isomorphism follows from the excision,
homotopy invariance and additivity properties for homology with local coefficients.
The resulting isomorphism 
\[
	M_\ast H_\bullet (B^{(n)}, B^{(n-1)};\,X)
	\isom	
    C_n (B;\,\calL_\ast (M,X))
\]
is independent of the choice of the points $x_\alpha\in E^n_\alpha$,
and a tedious but straightforward check shows that under these
isomorphisms, the differential on the $E_1$--page 
of the spectral sequence arising from \eqref{eq:unrolledechomology}
corresponds to the differential on the chain complex
$C_\ast (B;\,\calL_\ast (M,X))$.
Moreover, the isomorphism is natural with respect to strongly cellular maps.
This concludes the construction of 
the spectral sequence \eqref{eq:serresshomology}.

The spectral sequence \eqref{eq:serresscohomology}
is constructed in a similar way by applying $M^\ast$
to \eqref{eq:sequence}
to obtain the unrolled exact couple
\begin{equation}
\label{eq:unrolledeccohomology}
\vcenter{\xymatrix@!0@C=5.9em@R=10ex{
	\cdots
	\ar@{<-}[r]
	&
	M^\ast H_\bullet(B^{(n-1)};\, X) 
	\ar@{<-}[rr]
	&&
	M^\ast H_\bullet(B^{(n)};\, X) 
	\ar@{<-}[rr]
	\ar@{<-}[dl]
	&&
	M^\ast H_\bullet(B^{(n+1)};\, X) 
	\ar@{<-}[r]
	\ar@{<-}[dl]
	&
	\cdots
	\\
	&
	\cdots\quad
	&
	M^\ast H_\bullet(B^{(n)}, B^{(n-1)};\, X)
	\ar@{<-}[ul]
	&&
	M^\ast H_\bullet(B^{(n+1)},B^{(n)};\, X)
	\ar@{<-}[ul]
	&
	\quad\cdots
}}
\end{equation}
and identifying the $E_1$--page of the resulting 
spectral sequence as $C^\ast(B;\,\calL^\ast(X))$.
In the identification, of the $E_1$--page, one makes use of 
the isomorphism
\[
	M^\ast(X_\alpha) \tensor_\Z H^\ast(E^n_\alpha, \bdry E^n_\alpha;\,\Z)
	\xto{\ \isom\ }
	H^\ast(E^n_\alpha, \bdry E^n_\alpha;\,M^\ast(X_\alpha))
\]
induced by the map
\[
	M^\ast(X_\alpha) \tensor_\Z S^\ast(E^n_\alpha, \bdry E^n_\alpha;\,\Z)
	\longto
	S^\ast(E^n_\alpha, \bdry E^n_\alpha;\,M^\ast(X_\alpha))
\]
of singular cochain complexes given by $m\tensor \phi \mapsto (z \mapsto m\phi(z))$.
By \cite[Def.~5.10]{BoardmanSS},  the spectral sequence
converges conditionally to the limit
\begin{equation}
\label{eq:sstgtlim}
	\lim_{n\to \infty} M^\ast H_\bullet(B^{(n)};\, X). 
\end{equation}
It remains to show that  under the assumption
that $X$ is bounded from below, this limit
is $M^\ast H_\bullet(B;\, X)$ and the convergence
of the spectral sequence is strong.
Working one path component of $B$ at a time,
we see that it is sufficient to consider the 
case where $B$ is path connected.
But when $B$ is path connected and $X$ is bounded from below,
for every $s$ and $t$ the target of the differential
\[
	d_r\colon E_r^{s,t} \longto E_r^{s+r,t-r+1}
\]
in the spectral sequence vanishes for large enough $r$.
Therefore \cite[Thm.~7.4 and Rk.~following Thm.~7.1]{BoardmanSS}
imply that the spectral sequence converges strongly to
the limit \eqref{eq:sstgtlim}, and also that 
the $\limOne$--term in the Milnor short exact sequence
\[
	0 
	\longto 
	\limOne_{n\;}  M^{\ast-1} H_\bullet(B^{(n)};\, X)
	\longto M^\ast  H_\bullet(B;\, X)
	\longto
	\lim_{n} M^\ast H_\bullet(B^{(n)};\, X). 
	\longto
	0
\]
vanishes, showing that this limit is $M^\ast H_\bullet(B;\, X)$.
\end{proof}

We note the following corollary demonstrating that
homology with local coefficients in the usual sense
can be recovered from $H_\bullet$--homology.

\begin{cor}
\label{cor:homologywithlocalcoefficientscomp}
Let $\calL$ be a local coefficient system of abelian groups
over $B$, and let $H\calL$ be the parametrized
$H\Z$--module given by the composite
\[\xymatrix{
	H\calL
	\colon
	\Pi_\infty (B)
	\ar[r]
	&
	\Pi_1(B)
	\ar[r]^-{\calL}
	&
	\Ab
	\ar[r]^-{H}
	&
	\Mod^{H\Z}.
}\]
Then $H_\ast H_\bullet(B;H\calL) \isom H_\ast(B;\calL)$.
\qed
\end{cor}

\begin{proof}
Make use of spectral sequence \eqref{eq:serresshomology}
with $R = M = H\Z$ and $X = H\calL$ together with the
observation that
$\calL_0(H\Z,H\calL) \isom \calL$ and
$\calL_t(H\Z,H\calL) = 0$ for $t\neq 0$.
\end{proof}

Our next goal is to investigate 
the compatibility of the spectral sequences of 
Theorem~\ref{thm:serresss}
with suspension isomorphisms.
By the proof of Theorem~\ref{thm:serresss},
the spectral sequences arise from  exact couples
\[
	EC_{M_\ast}(X)
	=
	\left(
	\vcenter{\xymatrix@!0@C=6.7em@R=10ex{
		\bigoplus_{n\geq 0} M_\ast H_\bullet (\Gamma B^{(n)};\, X)
		\ar[rr]^i
		&&
		\bigoplus_{n\geq 0} M_\ast H_\bullet (\Gamma B^{(n)};\, X)
		\ar[dl]^(0.45)j
		\\
		&
		\bigoplus_{n\geq 0} C_n(\Gamma B;\,\calL_\ast(M,X))
		\ar[ul]^(0.55)k
	}}
	\right)
\]
and
\[
	EC_{M^\ast}(X)
	=
	\left(
	\vcenter{\xymatrix@!0@C=6.7em@R=10ex{
		\bigoplus_{n\geq 0} M^\ast H_\bullet (\Gamma B^{(n)};\, X)
		\ar@{<-}[rr]^{i'}
		&&
		\bigoplus_{n\geq 0} M^\ast H_\bullet (\Gamma B^{(n)};\, X)
		\ar@{<-}[dl]^(0.45){j'}
		\\
		&
		\bigoplus_{n\geq 0} C^n(\Gamma B;\,\calL^\ast(M,X))
		\ar@{<-}[ul]^(0.55){k'}
	}}
	\right)
\]
where $\Gamma$ is the CW approximation functor 
used in the proof of Theorem~\ref{thm:serresss}
and we have continued to write $X$ for the pullback
of $X$ over $\Gamma B$ and its skeleta.
Given $u\in \Z$, a space $B$, an $R$--module $M$,
and a parametrized $R$--module
$X$ over $B$, we use $\bar{\sigma}^u$ to denote the composites
\begin{align}
	\label{eq:barsigmauhomology}
 	\bar{\sigma}^u
	\colon&
	M_\ast H_\bullet(B;\,X)
	\xto{\ \isom\ } 
	M_{\ast+u} \suspension^u H_\bullet(B;\,X)
	\xto{\ \isom\ }
	M_{\ast+u} H_\bullet (B;\,\suspension_B^u X)
	\\
	\label{eq:barsigmaucohomology}
	\bar{\sigma}^u
	\colon&
	M^\ast H_\bullet(B;\,X)
	\xto{\ \isom\ } 
	M^{\ast+u} \suspension^u H_\bullet(B;\,X)
	\xto{\ \isom\ }
	M^{\ast+u} H_\bullet (B;\,\suspension_B^u X)
\end{align}
where the first maps are the suspension isomorphisms and 
the second ones 
are induced by an instance of \eqref{eq:hbulletandfcommute}.
Moreover, when $B$ is a CW complex,
we also use $\bar{\sigma}^u$ to denote the isomorphisms
\[
	\begin{aligned}
	\bar{\sigma}^u
	\colon&
	C_n(B;\,\calL_\ast (M,X))	
	\xto{\ \isom\ }
	C_n(B;\,\calL_\ast(M,\suspension_B^u X))
	\\
	\bar{\sigma}^u
	\colon&
	C^n(B;\,\calL^\ast (M,X))	
	\xto{\ \isom\ }
	C^n(B;\,\calL^\ast(M,\suspension_B^u X))
	\end{aligned}
\]
induced by the isomorphisms
\begin{align}
\label{eq:callsuspisohomology}
	\calL_\ast(M,X)
	&\xto{\ \isom\ }
	\calL_{\ast+u}(M,\suspension_B^u X)
	\\
\label{eq:callsuspisocohomology}
	\calL^\ast(M,X)
	&\xto{\ \isom\ }
	\calL^{\ast+u}(M,\suspension_B^u X) 
\end{align}
given by the composites
\[
	\begin{aligned}
	M_\ast(X_b)&
	\xto{\ \isom\ }
	M_{\ast+u}(\suspension^u X_b)
	\xto{\ \isom\ }
	M_{\ast+u}((\suspension_B^u X)_b)
	\\
	M^\ast(X_b)&
	\xto{\ \isom\ }
	M^{\ast+u}(\suspension^u X_b)
	\xto{\ \isom\ }
	M^{\ast+u}((\suspension_B^u X)_b)
	\end{aligned}
\]
where the first maps are the suspension isomorphisms
and the second ones are 
induced by an instance of \eqref{eq:tbonfibres}.

Tracing through the construction of the exact couples
$EC_{M_\ast}$ and $EC_{M^\ast}$, we now have the following
result.
The sign in the relation
$\bar{\sigma}^u k = (-1)^u k \bar{\sigma}^u$
arises from the observation that the square
\[\xymatrix{
	M_\ast H_\bullet(\Gamma B^{(n)},\Gamma B^{(n-1)};\, X)
	\ar[r]^-\bdry
	\ar[d]_{\bar{\sigma}^u}
	&
	M_{\ast-1} H_\bullet(\Gamma B^{(n-1)};\, X)
	\ar[d]^{\bar{\sigma}^u}
	\\
	M_{\ast+u} H_\bullet(\Gamma B^{(n)},\Gamma B^{(n-1)};\,\suspension_B^u X)
	\ar[r]^-\bdry
	&
	M_{\ast+u-1} H_\bullet(\Gamma B^{(n-1)};\,\suspension_B^u X)
}\]
commutes up to the sign $(-1)^u$, and similarly 
for the relation
$\bar{\sigma}^u k' = (-1)^u k' \bar{\sigma}^u$.
Here the map $\bar{\sigma}^u$ on the left is the 
evident generalization of \eqref{eq:barsigmauhomology}
to the relative situation.
\begin{lemma}
\label{lm:ecsuspisos}
The maps $\bar{\sigma}^u$ give natural isomorphisms of exact couples
\begin{align}
	\label{iso:barsigmauechomology}
	\bar{\sigma}^u
	\colon&
	EC_{M_\ast}(X)
	\xto{\ \isom\ }
	EC_{M_\ast}(\suspension_B^u X)
	\\
	\label{iso:barsigmaueccohomology}
	\bar{\sigma}^u
	\colon&
	EC_{M^\ast}(X)
	\xto{\ \isom\ }
	EC_{M^\ast}(\suspension_B^u X)
\end{align}
such that $\bar{\sigma}^u i = i \bar{\sigma}^u$, 
$\bar{\sigma}^u j = j \bar{\sigma}^u $, 
$\bar{\sigma}^u k = (-1)^u k \bar{\sigma}^u$,
$\bar{\sigma}^u i' = i' \bar{\sigma}^u$, 
$\bar{\sigma}^u j' = j' \bar{\sigma}^u $, 
and $\bar{\sigma}^u k' = (-1)^u k' \bar{\sigma}^u$.
\qed
\end{lemma}

Let $\{F_{s,\ast}(X)\}_{s\in\Z}$
and $\{F^{s,\ast}(X)\}_{s\in\Z}$
be the filtrations of $M_\ast H_\bullet (B;\,X)$
and $M^\ast H_\bullet (B;\,X)$
associated with the 
spectral sequences of Theorem~\ref{thm:serresss},
so that 
\begin{align*}
 	F_{s,k}(X)
	&= 
	\im(M_{k} H_\bullet(\Gamma B^{(s)}; X) \to M_{k} H_\bullet(B;X))
\intertext{and}
	F^{s,k}(X)
	&=
	\ker(M^{k} H_\bullet(B;X) \to M^{k} H_\bullet(\Gamma B^{(s-1)};X)).
\end{align*}
From Lemma~\ref{lm:ecsuspisos}, we obtain 
\begin{prop}[Suspensions isomorphisms on Serre spectral sequences]
\label{prop:sssusp}
Let $M$ be an $R$--module,
let $X$ be a parametrized $R$--module over $B$,
and let $u\in \Z$.
Then \eqref{iso:barsigmauechomology} 
and \eqref{iso:barsigmaueccohomology} induce
isomorphisms of spectral sequences
\begin{align}
	\label{eq:sssuspensionisohomology}
	E^r_{s,t}(X) &\xto{\ \isom\ } E^r_{s,t+u}(\suspension_B^u X)
\intertext{and}
	\label{eq:sssuspensionisocohomology}
	E_r^{s,t}(X) &\xto{\ \isom\ } E_r^{s,t+u}(\suspension_B^u X),
\end{align}
$1\leq r \leq \infty$,
which  commute with the differential on each page
up to the sign $(-1)^u$.
The isomorphisms are natural in $B$, $M$, and $X$.
Moreover, for $r=\infty$, the isomorphism
\eqref{eq:sssuspensionisohomology}
fits into a commutative diagram
\[\xymatrix{
	E^\infty_{s,t}(X)
	\ar[r]^-\isom
	\ar[d]_\isom
	&
	F_{s,s+t}(X)/F_{s-1,s+t}(X)
	\ar[d]^\isom
	\\
    E^\infty_{s,t+u}(\suspension_B^u X)
    \ar[r]^-\isom
    &
	F_{s,s+t+u}(\suspension_B^u X)/F_{s-1,s+t+u}(\suspension_B^u X),
}\]
where 
the vertical isomorphism on the right is 
induced by
\[
    \bar{\sigma}^u
    \colon
    M_{s+t}H_\bullet(B;\,X)
    \xto{\ \isom\ } 
    M_{s+t+u} H_\bullet(B;\,\suspension_B^u X)
\]
and the horizontal isomorphisms are given by the 
strong convergence of the spectral sequences
involved.
Similarly, when $X$ is bounded from below,
for $r=\infty$,
the isomorphism
\eqref{eq:sssuspensionisocohomology}
fits into a commutative diagram
\[\xymatrix{
	E_\infty^{s,t}(X)
	\ar[r]^-\isom
	\ar[d]_\isom
	&
	F^{s,s+t}(X)/F^{s+1,s+t}(X)
	\ar[d]^\isom
	\\
    E_\infty^{s,t+u}(\suspension_B^u X)
    \ar[r]^-\isom
    &
	F^{s,s+t+u}(\suspension_B^u X)/F^{s+1,s+t+u}(\suspension_B^u X),
}\]
where
the vertical isomorphism on the right is 
induced by
\[
    \bar{\sigma}^u
    \colon
    M^{s+t}H_\bullet(B;\,X)
    \xto{\ \isom\ } 
    M^{s+t+u} H_\bullet(B;\,\suspension_B^u X)
\]
and the horizontal isomorphisms are again given by the 
strong convergence of the spectral sequences
involved. \qed
\end{prop}

Later, in Section~\ref{subsec:thomiso}, we will extend 
the category $\hpMod^R$ with morphisms of non-zero degree.
See Definitions~\ref{def:mapsofdegreek} and \ref{def:gradedmapcomp}.
In essence, a morphism $\phi \colon X\to Y$ of degree $k$ covering 
a map $f\colon A \to B$ is a map $\phi_0 \colon \suspension^k_A X \to Y$
covering $f$. See Remark~\ref{rk:mapsofdegreekalt}.
Proposition~\ref{prop:sssusp} allows us to extend the functoriality
of the Serre spectral sequences to maps of arbitrary degree.

\begin{defn}
Given a degree $k$ morphism of parametrized $R$--modules
$\phi \colon X \to Y$ covering a map $A\to B$,
we define the induced maps on Serre spectral sequences
to be the composites
\[\xymatrix{
	E^r_{s,t}(X) 
	\ar[r]^-{\isom}
	&
	E^r_{s,t+k}(\suspension^k_A X) 
	\ar[r]^-{(\phi_0)_\ast}
	&
	E^r_{s,t+k}(Y)
}\]
and
\[\xymatrix{
	E_r^{s,t+k}(Y)
	\ar[r]^-{(\phi_0)^\ast}
	&
	E_r^{s,t+k}(\suspension^k_A X)
	\ar[r]^-{\isom} 
	&
	E_r^{s,t}(X)
}\]
where $\phi_0\colon \suspension^k_A X \to Y$ is the 
underlying degree $0$ morphism of $\phi$ and 
the unlabeled isomorphisms are given by the suspension isomorphisms 
of Proposition~\ref{prop:sssusp}. 
\end{defn}

\begin{cor}
\label{cor:ssfunctorialityalldegrees}
The spectral sequences
\[
    E^2_{s,t}
	=
	H_s(B;\calL_t(M,X)) 
	\Longrightarrow 
	M_{s+t}H_\bullet(B;X).
\]
and
\[
	E_2^{s,t} 
	= 
	H^s(B; \calL^t(M,X))
	\Longrightarrow 
	M^{s+t}H_\bullet(B;X)
\]
of Theorem~\ref{thm:serresss} are functorial in $X$ 
with respect to morphisms of arbitrary degree. \qed
\end{cor}

\section{Comparison of umkehr maps}
\label{sec:comparisonofumkehrmaps}

Our aim in this section is to relate the theory of umkehr maps developed 
in the present paper to various umkehr map constructions in the 
existing literature. 
We will consider three such constructions:
first, the umkehr maps 
Ando, Blumberg and Gepner \cite[\S 4]{ABGparam}
asociate to a bundle $p\colon E\to B$ of smooth closed  manifolds;
second, the classical umkehr maps
for maps between closed smooth manifolds 
arising from Poincaré duality;
and finally, the classical ``integration along fibre'' maps
arising from the Serre spectral sequences
of fibrations whose fibres have vanishing high-dimensional
homology.

\subsection{Comparison with the umkehr maps of Ando, Blumberg and Gepner}
\label{subsec:abgcomparison}

The aim of this section is to prove Theorem~\ref{thm:abgumkehrcomparison}
showing how the twisted umkehr maps of Ando, Blumberg and Gepner \cite[\S 4]{ABGparam}
can be recovered from our theory.

\begin{proof}[Proof of Theorem~\ref{thm:abgumkehrcomparison}]
Comparing definitions, we see that the 
twisted Pontryagin--Thom transfer map $\mathrm{PT}(E,X)$
of \cite[Prop.~4.18]{ABGparam} agrees with our umkehr map
\[
	(p,\lambda)^\leftarrow 
	\colon 
	H_\bullet (B;\,X) 
	\longto 
	H_\bullet(E;\, p^\ast X \smashprod_E S^{-\tau_p})
\]
where
\[
	\lambda \colon p^\ast X \smashprod_E S^{-\tau_p} \oto X
\]
is the map obtained by applying the functor 
\[
	R\smashprod_\fw( - ) \colon \hpSpectra^\dfop \longto (\hpMod^R)^\dfop
\]
to the map
$
	\xi \colon S^{-\tau_p} \oto S_B 
$
of $\hpSpectra^\dfop$ covering $p$ defined by the
parametrized Pontryagin--Thom transfer map
\[
	\mathrm{PT}_{/B}(p) \colon S_B \longto p_! S^{-\tau_p}
\]
of \cite[Def.~4.14]{ABGparam}
and taking the $\oslash$--product with the cartesian morphism 
$p^\ast X \to X$ covering $p$.
Consequently, the twisted umkehr map of \cite[Prop.~4.18]{ABGparam}
agrees with the map
\[
	R^\ast((p,\lambda)^\leftarrow)
	\colon 
	R^\ast H_\bullet(E;\, p^\ast X \smashprod_E S^{-\tau_p})
	\longto
	R^\ast 	H_\bullet (B;\,X).
\]
Therefore all that is needed is to relate the map $\lambda$ 
to the map $\theta$ of equation \eqref{eq:abgumkehrcomparisontheta}.
By Theorem~\ref{thm:abgpthc2}, the map $\xi$ is hypercartesian,
so by Proposition~\ref{prop:hypercartmorpreservation} and 
Proposition~\ref{prop:superandhypercartmorprops}(\ref{it:hypercartimpliessupercart})
the map $\lambda$ is cartesian. By the uniqueness of cartesian morphisms,
the morphisms $\lambda$ and $\theta$ are canonically isomorphic, yielding the claim.
\end{proof}

\subsection{The Thom isomorphism}
\label{subsec:thomiso}

In preparation for the comparison of our umkehr maps to umkehr maps
arising from Poincaré duality and to integration along fibre maps,
in this subsection our aim is to show that classical Thom isomorphisms
for a virtual bundle $\xi$ in $R$--cohomology and $R$--homology can be understood 
as maps induced by a trivialization of the $R$--line bundle
$R\smashprod_\fw S^\xi$. The main result is 
Proposition~\ref{prop:thomisoreinterpretations}.

For our discussion of the Thom isomorphism, it is convenient
to augment our category $\hpMod^R$ with morphisms of nonzero degree
along with composites and external smash products of such.
This is easily done by adapting the analogous constructions 
for $\ho(\Spectra)$ \cite[pp.~188--189]{adams1974stable}
to the present setting.
In the following definitions, $e \colon S^i \smashprod S^j \homot S^{i+j}$
for $i,j\in\Z$ denote compatibly chosen equivalences of spectra
afforded by \cite[Prop.~III.4.8]{adams1974stable}.

\begin{defn}
\label{def:mapsofdegreek}
Given an integer $k$ and objects $X$ and $Y$ of $\hpMod^R$ over spaces $A$ and $B$, 
respectively, a \emph{map $\phi$ 
of degree $k$} from $X$ to $Y$ covering a map $f\colon A\to B$ 
is a map
\[
	\phi_0\colon (\underline{R\smashprod S^k}) \smashprod^R_A X \longto Y
\]
of $\hpMod^R$ covering $f$. We sometimes write
\[
	\xymatrix@1{\phi\colon X \ar[r]_-k & Y}
\]
to indicate that $\phi$ is a map of degree $k$ from $X$ to $Y$.
Via the equivalence 
$(\underline{R\smashprod S^0}) \smashprod^R_A X \homot X$
induced by the unit constraints for $\smashprod$ and $\smashprod^R_A$,
we may (and will) identify degree $0$ morphisms $X \to Y$ 
with ordinary morphisms of $X\to Y$.
\end{defn}

\begin{rem}
\label{rk:mapsofdegreekalt}
The object $(\underline{R\smashprod S^k}) \smashprod^R_A X$ is naturally 
equivalent to the object $\suspension^k_A X$, so we may alternatively 
view maps of degree $k$ from $X$ to $Y$ as maps $\suspension^k_A X \to Y$.
\end{rem}

\begin{defn}
\label{def:gradedmapcomp}
Given a degree $k$ map 
$\phi \colon X \to Y$  and a degree $\ell$ map $\psi \colon Y\to Z$
in $\hpMod^R$ covering maps $f\colon A \to B$ and $g \colon B \to C$, 
respectively, the composite $\psi \circ \phi$ is the degree $k+\ell$ map
covering $g\circ f$ given by the composite
\[
	(\psi \circ \phi)_0 
	\colon 
	(\underline{R\smashprod S^{k+\ell}}) \smashprod^R_A X
	\xto{\ \homot\ }
	(\underline{R\smashprod S^{\ell}})
	\smashprod^R_A 
	(\underline{R\smashprod S^{k}}) 
	\smashprod^R_A
	X
	\xto{\ {c \smashprod^R_\internal \phi_0}\ }
	(\underline{R\smashprod S^{\ell}})\smashprod^R_B Y
	\xto{\ \psi_0\ }
	Z
\]
where the first map is induced by the equivalence 
$e^{-1} \colon S^{k+\ell} \homot S^{\ell}\smashprod S^k$ and 
$c \colon (\underline{R\smashprod S^{\ell}}) \to (\underline{R\smashprod S^{\ell}})$
is the canonical cartesian morphism covering $f$.
\end{defn}

\begin{defn}
\label{def:gradedmapextsmashprod}
Suppose $\phi \colon X \to X'$ and $\psi\colon Y \to Y'$
are degree $k$ and $\ell$ maps of $\hpMod^R$ 
covering $f\colon A \to B$ and $g \colon C \to D$,
respectively. Then the external smash product 
$\phi \extsmashprod^R \psi \colon X\extsmashprod^R Y \to X'\extsmashprod^R Y'$
is the degree $k+\ell$ map covering $f\times g \colon A\times C \to B \times D$
given by the composite
\[\xymatrix@!0@C=5.5em{
	(\phi \extsmashprod^R \psi)_0
	\colon	
	(\underline{R\smashprod S^{k+\ell}}) \smashprod_{A\times C} (X \extsmashprod^R Y)
	\\
	\ar[r]^-\homot
	&
	*!L{\;
		(\underline{R\smashprod S^{k}})
    	\smashprod^R_{A\times C} 
    	(\underline{R\smashprod S^{\ell}}) 
    	\smashprod^R_{A\times C} 
		(X \extsmashprod^R Y)
	}
	\\
	\ar@{=}[r]
	&
	*!L{\;
		(\underline{R\smashprod S^{k}})
    	\smashprod^R_{A\times C} 
    	(\underline{R\smashprod S^{\ell}}) 
    	\smashprod^R_{A\times C} 
		(\pi^{AC}_A)^\ast X
    	\smashprod^R_{A\times C} 
		(\pi^{AC}_C)^\ast Y
	}
	\\
	\ar[r]^{1\,\extsmashprod^R \,\chi\, \extsmashprod^R\, 1}
	&
	*!L{\;
		(\underline{R\smashprod S^{k}})
    	\smashprod^R_{A\times C} 
		(\pi^{AC}_A)^\ast X
    	\smashprod^R_{A\times C} 
    	(\underline{R\smashprod S^{\ell}}) 
    	\smashprod^R_{A\times C} 
		(\pi^{AC}_C)^\ast Y
	}
	\\
	\ar[r]^\homot
	&
	*!L{\;
		(\pi^{AC}_A)^\ast ((\underline{R\smashprod S^{k}}) \smashprod^R_A X)
    	\smashprod^R_{A\times C}
		(\pi^{AC}_C)^\ast ((\underline{R\smashprod S^{\ell}}) \smashprod^R_C Y)
	}
	\\
	\ar@{=}[r]
	&
	*!L{\;
		((\underline{R\smashprod S^{k}}) \smashprod^R_A X)
    	\extsmashprod^R
		((\underline{R\smashprod S^{\ell}}) \smashprod^R_C Y)
	}
	\\
	\ar[r]^{\phi_0 \,\extsmashprod^R \,\psi_0}
	&
	*!L{\;
		X' \extsmashprod^R Y' \vphantom{{}_f}
	}
}\]
where $\pi^{AC}_A \colon A\times C \to A$ and
$\pi^{AC}_C \colon A\times C \to C$ are the projections;
the first equivalence is induced by the equivalence
$e^{-1} \colon S^{k+\ell} \homot S^{k}\smashprod S^{\ell}$;
$\chi$ is the symmetry constraint for $\extsmashprod^R$; and
the second equivalence is induced by the canonical cartesian morphisms
$\underline{R\smashprod S^{k}} \to \underline{R\smashprod S^{k}}$
over $f$ and 
$\underline{R\smashprod S^{\ell}} \to \underline{R\smashprod S^{\ell}}$
over $g$ using the fact that the internal smash product $\smashprod^R_\internal$
preserves cartesian morphisms.
\end{defn}

We note that for degree $0$ morphisms, 
Definition~\ref{def:gradedmapcomp}
and \ref{def:gradedmapextsmashprod}
agree with the previously defined composition and external smash product
of $\hpMod^R$. Moreover, 
on the fibre over the one-point space $\pt$, the definitions
recover the usual notions of morphisms of non-zero degree in $\Mod^R$
and composites and smash products of such.

The following proposition explains the compatibility between $\circ$
and $\extsmashprod^R$ for morphisms of nonzero degree. It follows 
from the properties of the equivalences $e$ by unrolling definitions.
The sign appearing in the proposition is given meaning by observing
that automorphisms of $S^k \in \ho(\Spectra)$ act on the set of 
degree $k$ morphisms of $\hpMod^R$.

\begin{prop}
\label{prop:circandextsmashcompat}
Suppose $\phi \colon X \to X'$, $\phi' \colon X' \to X''$,
$\psi \colon Y \to Y'$ and $\psi' \colon Y' \to Y''$
are morphisms of possibly nonzero degree in $\hpMod^R$.
Then 
\[
	(\phi' \extsmashprod^R \psi') \circ (\phi \extsmashprod^R \psi)
	= 
	(-1)^{\deg(\phi) \deg(\psi')} (\phi'\circ\phi) \extsmashprod^R (\psi'\circ\psi)
\]
in keeping with the Koszul sign rule.\qed
\end{prop}

The following proposition allows us to interpret an $R$--orientation
(that is, an $R$--module Thom class)
for a virtual bundle $\xi$ 
as a trivialization of the $R$--line bundle $R\smashprod_\fw S^\xi$.
Compare with \cite[Prop.~20.5.5]{MaySigurdsson}.

\begin{prop}
\label{prop:thomclassastrivialization}
Suppose $\xi$ is a virtual bundle of dimension $d$ over a space $B$. 
Then a class $u \in R^d(B^\xi)$, represented by a degree $-d$ $R$--module map
\[\xymatrix{
	u \colon R \smashprod B^\xi \ar[r]_-{-d} & R,
}\]
is an $R$--orientation for $\xi$ if and only if the right adjunct
\[\xymatrix{
	\tilde{u} \colon R \smashprod_\fw S^\xi \ar[r]_-{-d} & R_B
}\]
of $u$ under the $(r_!, r^\ast)$ adjunction for $r$ the unique map $B\to \pt$
is an equivalence.
\end{prop}
\begin{proof}
Recall that by definition, $u$ is an $R$--orientation for $\xi$ if and only if
for every point $b\in B$,
the restriction of $u$ to $R^d(\{b\}^{\xi_b})$ is a generator of $R^\ast(\{b\}^{\xi_b})$
as a free $\pi_{-\ast}(R)$--module on one generator, a condition equivalent to
requiring that the composite
\[\xymatrix{
	R \smashprod \{b\}^{\xi_b} \longto R \smashprod B^\xi \ar[r]^-{u}_-{-d} &R
}\]
where the first map is induced by the 
inclusion of the fibre $\xi_b$ over $b$ into $\xi$ 
is an equivalence for every $b \in B$.
The claim now follows
by observing that this composite agrees with the restriction of $\tilde{u}$
to fibres over $b$ and remembering that equivalences in $\ho(\Mod^R_{/B})$
are detected on fibres.
\end{proof}

The following proposition now interprets
the usual Thom isomorphisms for a virtual bundle $\xi$
as the maps induced by the trivialization of $R \smashprod_\fw S^\xi$
defined by the $R$--orientation.
\begin{prop}
\label{prop:thomisoreinterpretations}
Suppose $\xi$ is a virtual bundle of dimension $d$ over a space $B$, 
let $u \in R^d(B^\xi)$ be an $R$--orientation for $\xi$, and let 
$\tilde{u}_\bullet$ be the degree $-d$ equivalence
given by the composite 
\begin{equation}
\label{eq:tildeubullet}
\xymatrix@C+1em{
	\tilde{u}_\bullet
	\colon
	R\smashprod B^\xi 
	\homot
	H_\bullet(B; R\smashprod_\fw S^\xi) 
	\ar[r]^-{(\id,\tilde{u})_\bullet}_-{\homot}
	&
  	H_\bullet(B; R_B)
	\homot
   	R\smashprod B_+
} 
\end{equation}
where $\tilde{u}$ is the equivalence afforded by
Proposition~\ref{prop:thomclassastrivialization};
the first equivalence is given by 
the natural equivalence~\eqref{eq:hbulletandfcommute};
and the last equivalence is given by Example~\ref{ex:hbulletfortrivialcoeffs2}.
\begin{enumerate}[(i)]
\item\label{it:thomisoinrcoh}
	On $R$--cohomology, the usual Thom isomorphism map
	given by $x\mapsto (-1)^{\deg(x) d} x \cupprod u$
	agrees with the precomposition 
	map 
	\[
		(\tilde{u}_\bullet)^\sharp
		\colon 
		R^k(B)\xto{\ \isom\ } R^{k+d}(B^\xi), 
		\quad
		\left[\xymatrix@1{R\smashprod B_+ \ar[r]_-{-k}^-{x} & R}\right]
		\longmapsto
		(-1)^{kd}
		\left[\xymatrix@1{ 
			R\smashprod B^\xi 
			\ar[r]^{\tilde{u}_\bullet}_{-d} 
			& 
			R\smashprod B_+ 
			\ar[r]_-{-k}^-{x} 
			& 
			R
		}\right].
	\]
\item
	On $R$--homology, the usual Thom isomorphism map given by $z \mapsto u \capprod z$
	agrees with the postcomposition map
	\[
		(\tilde{u}_\bullet)_\sharp 
		\colon 
		R_k(B^\xi) \xto{\ \isom\ }  R_{k-d}(B),
		\quad
		\left[\xymatrix@1{R \ar[r]_-{k}^-{z} & R\smashprod B^\xi}\right]
		\longmapsto
		\left[\xymatrix@1{ 
			R
			\ar[r]^-{z}_-{k}
			&
			R\smashprod B^\xi 
			\ar[r]^{\tilde{u}_\bullet}_{-d} 
			& 
			R\smashprod B_+ 
		}\right].
	\]
\end{enumerate}
\end{prop}
In part (\ref{it:thomisoinrcoh}), the sign in the formula
$x\mapsto (-1)^{\deg(x) d} x \cupprod u$
for the Thom isomorphism is required to make the Thom isomorphism
a degree $d$ map of $R^\ast(B)$--modules, and the sign in the 
formula for the precomposition map is in keeping with the Koszul 
sign rule, which instructs us to define 
$(\tilde{u}_\bullet)^\sharp(x) = (-1)^{\deg(x) \deg(\tilde{u}_\bullet)}x \circ \tilde{u}_\bullet$.
The proof of Proposition~\ref{prop:thomisoreinterpretations}
is based on Lemma~\ref{lm:thomdiagonalcomposite} below.
Before stating the lemma, we recall the construction of 
the $R$--module Thom diagonal map.

\begin{defn}[Thom diagonal map]
\label{def:thomdiagonal}
Suppose $\xi$ is a virtual bundle over a space $B$.
Then the \emph{$R$--module Thom diagonal map of $\xi$} is 
the $R$--module map
\begin{equation}
\label{eq:thomdiagonal}
	\hat{\Delta}_\xi
	\colon 
	R\smashprod B^\xi 
	\longto 
	(R \smashprod B_+) \smashprod^R (R\smashprod B^\xi) 
\end{equation}
obtained as follows.
Let $c \colon R_B \to R$ be the canonical cartesian morphism covering
the unique map $B\to \pt$. As equivalences and external smash products
of cartesian morphisms are cartesian, the composite map
\begin{equation}
\label{eq:cextsmashidunitconstr}
	R_B \extsmashprod^R (R\smashprod_\fw S^\xi) 
	\xto{\ c \extsmashprod^R 1\ }
	R \extsmashprod^R (R\smashprod_\fw S^\xi) 
	\xto{\ \homot\ }
	R\smashprod_\fw S^\xi
\end{equation}
covering the projection $B\times B \to B$ onto the second coordinate
is cartesian. The second map in \eqref{eq:cextsmashidunitconstr}
is given by the left unit constraint of $\extsmashprod^R$.
Consequently, we may factor the identity map of $R\smashprod_\fw S^\xi$
as a composite of a unique map
\[
	\tilde{\Delta}_\xi 
	\colon 
	R\smashprod_\fw S^\xi
	\longto
	R_B \extsmashprod^R (R\smashprod_\fw S^\xi)
\]
covering the diagonal map $B\to B\times B$
and \eqref{eq:cextsmashidunitconstr}.
The Thom diagonal map $\hat{\Delta}_\xi$
is now given by the composite
\begin{multline*}
	\quad
	\hat{\Delta}_\xi
	\colon
	R\smashprod B^\xi
	\homot
	H_\bullet(B; R\smashprod_\fw S^\xi)
	\xto{\ (\Delta,\tilde{\Delta}_\xi)_\bullet\ }
	H_\bullet(B\times B; R_B \extsmashprod^R (R\smashprod_\fw S^\xi))
	\\
	\homot
	H_\bullet(B; R_B) \smashprod^R H_\bullet(B; R\smashprod_\fw S^\xi)
	\homot
	(R\smashprod B_+) \smashprod^R (R \smashprod B^\xi) 
	\quad
\end{multline*}
where $\Delta \colon B\to B\times B$ is the diagonal map; 
the equivalences between 
$R\smashprod B^\xi$ and $H_\bullet(B; R\smashprod_\fw S^\xi)$
are instances of \eqref{eq:hbulletandfcommute};
the first equivalence on the second row is 
the inverse of the monoidality constraint \eqref{eq:hbulletmonconstr};
and the equivalence between $H_\bullet(B; R_B)$ and $R\smashprod B_+$
arises from Example~\ref{ex:hbulletfortrivialcoeffs2}.
\end{defn}

\begin{lemma}
\label{lm:thomdiagonalcomposite}
Suppose $\xi$ is a virtual bundle over a space $B$.
Let $v\colon R\smashprod B^\xi \to R$ be a degree $k$ map of $R$--modules,
and let $\tilde{v} \colon R \smashprod_\fw S^\xi \to R_B$ be the right adjunct of $v$
under the $(r_!, r^\ast)$ adjunction for $r$ the unique map $B\to \pt$.
Then the composite
\begin{equation}
\label{eq:thomdiagonalandidsmashv}
\xymatrix@1@C+0.5em{
	R\smashprod B^\xi 
	\xto{\ \hat{\Delta}_\xi\ }
	(R\smashprod B_+) \smashprod^R (R\smashprod B^\xi)
	\ar[r]_-k^-{\ 1 \smashprod^R v\ }
	&
	(R\smashprod B_+) \smashprod^R R
	\xto{\ \homot\ }
	R\smashprod B_+
}
\end{equation}
where the first map is the Thom diagonal map for $\xi$ and 
the last map is a unit constraint for $\smashprod^R$
agrees with the composite 
\[\xymatrix@C+1em{
	R\smashprod B^\xi 
	\homot
	H_\bullet(B; R\smashprod_\fw S^\xi) 
	\ar[r]^-{(\id,\tilde{v})_\bullet}_-k
	&
  	H_\bullet(B; R_B)
	\homot
   	R\smashprod B_+
}\]
where the two equivalences are as in Definition~\ref{def:thomdiagonal}.
\end{lemma}

\begin{proof}
Let $c \colon R_B \to R$ be the canonical cartesian morphism, and 
consider the commutative diagram
\[\xymatrix@C+0.5em{
	R\smashprod_\fw S^\xi
	\ar[r]^-{\tilde{\Delta}_\xi}
	\ar@/_/[rdd]_\id
	&
	R_B \extsmashprod^R (R\smashprod_\fw S^\xi)
	\ar[r]^-{1\extsmashprod^R \tilde{v}}_-k
	\ar[d]_{c\extsmashprod^R 1}
	&
	R_B\extsmashprod^R R_B
	\ar[r]^-{1\extsmashprod^R c}
	\ar[d]_{c\extsmashprod^R 1}
	&
	R_B\extsmashprod^R R
	\ar[r]^-\homot
	\ar[d]_{c\extsmashprod^R 1}
	&
	R_B
	\ar[dd]^c
	\\
	&
	R \extsmashprod^R (R\smashprod_\fw S^\xi)
	\ar[r]^-{1\extsmashprod^R \tilde{v}}_-k
	\ar[d]_\homot
	&
	R\extsmashprod^R R_B
	\ar[r]^-{1\extsmashprod^R c}
	\ar[d]_\homot
	&
	R\extsmashprod^R R
	\ar[dr]^-\homot
	\\
	&
	R\smashprod_\fw S^\xi
	\ar[r]^-{\tilde{v}}
	&
	R_B
	\ar[rr]^c
	&&
	R
}\]
where $\tilde{\Delta}_\xi$ is as in Definition~\ref{def:thomdiagonal};
the vertical equivalences are given by left unit constraints;
the last map in the top row is given by the right unit constraint;
and the diagonal equivalence is given by the left unit constraint
(which in this case agrees with the right unit constraint).
As the map $c$ is cartesian, from the commutativity of the diagram
we deduce that the composite of the maps in the top row is 
equal to $\tilde{v}$. 
The claim now follows by observing that applying $H_\bullet$ to 
the top row yields \eqref{eq:thomdiagonalandidsmashv}.
\end{proof}

\begin{proof}[Proof of Proposition~\ref{prop:thomisoreinterpretations}]
Recall that $x\cupprod u$ is the class represented by
the composite
\[
	x\cupprod u
	\colon
	R\smashprod B^\xi 
	\xto{\ \hat{\Delta}_\xi\ }
	(R\smashprod B_+) \smashprod^R (R\smashprod B^\xi)
	\xto{\ x \smashprod^R u\ }
	R \smashprod^R R
	\xto{\ \homot\ }
	R
\]
while $u \capprod z$ is the class represented by the composite
\begin{multline*}
	u \capprod z
	\colon
	R
	\xto{\ z\ }
	R \smashprod B^\xi
	\xto{\ \hat{\Delta}_\xi\ }
	(R\smashprod B_+) \smashprod^R (R\smashprod B^\xi)
	\\
	\xto{\ \chi\ }
	(R\smashprod B^\xi)  \smashprod^R (R\smashprod B_+)
	\xto{\ u\smashprod^R 1\ }
	R \smashprod^R (R\smashprod B_+)
	\xto{ \homot\ }	
	R\smashprod B_+ 
\end{multline*}
where $\chi$ denotes the symmetry constraint for $\smashprod^R$.
The claim now follows easily from Lemma~\ref{lm:thomdiagonalcomposite}
with $v = u$.
\end{proof}

\subsection{Comparison with umkehr maps arising from Poincaré duality}
\label{subsec:pdcomp}
Our aim in this subsection is to prove Theorem~\ref{thm:pdumkcomp} below
generalizing Theorem~\ref{thm:classicalumkehrmapscomparison}(\ref{it:pdumkcomparison}) 
from $H\bbk$ to general commutative ring spectra $R$.

Suppose $M$ and  $N$ are closed smooth manifolds of dimensions $m$ and $n$, 
respectively, and let $f\colon M \to N$ be a continuous map.
Suppose furthermore that $R$ is a ring spectrum and $M$ and $N$ are oriented with 
respect to $R$, meaning that the tangent bundles $\tau_M$ and $\tau_N$
have been equipped with $R$--orientations.
By the proof of Corollary~\ref{cor:manifoldscart},
there exists a cartesian morphism
$\kappa \colon S^{-\tau_M} \oto S^{-\tau_N}$
in $\hpSpectra^\dfop$ 
fitting into a commutative triangle
\begin{equation}
\label{eq:mnhypercart}
	\vcenter{\xymatrix@!0@C=3em@R=6ex{
		S^{-\tau_M} 
		\ar[rr]|{\circdec}^\kappa
		\ar[dr]|{\circdec}
		&&
		S^{-\tau_N}
		\ar[dl]|{\circdec}
		\\
		&
		S
	}}
	\qquad\text{covering}\qquad
	\vcenter{\xymatrix@!0@C=3em@R=6ex{
		M
		\ar[rr]^f
		\ar[dr]
		&&
		N
		\ar[dl]
		\\
		&
		\pt
	}}
\end{equation}
where the morphisms into $S$ are hypercartesian.
From $\kappa$, we obtain a map
\[
	(f,\kappa)^\leftarrow \co H_\bullet (N; S^{-\tau_N}) \longto H_\bullet (M; S^{-\tau_M})
\]
in $\ho(\Spectra)$.
Notice that $H_\bullet (N; S^{-\tau_N})$ and $H_\bullet (M; S^{-\tau_M})$
are the Thom spectra $N^{-\tau_N}$ and $M^{-\tau_M}$, respectively.
The $R$--orientations for $\tau_M$ and $\tau_N$ induce ones for 
$-\tau_M$ and $-\tau_N$,
and combining the maps induced by $(f,\kappa)^\leftarrow$ 
with the resulting Thom isomorphisms yields umkehr maps
\begin{equation}
\label{eq:thomumkthom}
\xymatrix@!0@C=4.6em{
	R_{\ast+n} (N)
	\ar[rr]^-{\mathrm{Thom}^{-1}}_-\isom
	&&
	R_\ast (N^{-\tau_N})
	\ar[rrr]^{R_\ast((f,\kappa)^\leftarrow)}
	&&&
	R_\ast (M^{-\tau_M})
	\ar[rr]^-{\mathrm{Thom}}_-\isom
	&&
	R_{\ast+m} (M)
}
\end{equation}
and
\begin{equation}
\label{eq:thomumkthomrcoh}
\xymatrix@!0@C=4.6em{
	R^{\ast+m} (M)
	\ar[rr]^-{\mathrm{Thom}}_-\isom
	&&
	R^\ast (M^{-\tau_M})
	\ar[rrr]^{R^\ast((f,\kappa)^\leftarrow)}
	&&&
	R^\ast (N^{-\tau_N})
	\ar[rr]^-{\mathrm{Thom}^{-1}}_-\isom
	&&
	R_{\ast+n} (N)
}
\end{equation}
associated to $f$.

\begin{thrm}
\label{thm:pdumkcomp}
The umkehr maps of equations \eqref{eq:thomumkthom} and \eqref{eq:thomumkthomrcoh}
agree with the composites
\begin{equation}
\label{eq:pdindpd}
\xymatrix@C+1.1em{
	R_{\ast+n} (N)
	\ar[r]^-{(\mathrm{P.D.})^{-1}}_-\isom
	&
	R^{-\ast} (N)
	\ar[r]^{R^\ast(f)}
	&
	R^{-\ast} (M)
	\ar[r]^-{\mathrm{P.D.}}_-\isom
	&
	R_{\ast+m} (M)
}
\end{equation}
and
\begin{equation}
\label{eq:pdindpdrcoh}
\xymatrix@C+1.1em{
	R^{\ast+m} (M)
	\ar[r]^-{\mathrm{P.D.}}_-\isom
	&
	R_{-\ast} (M)
	\ar[r]^{R_\ast(f)}
	&
	R_{-\ast} (N)
	\ar[r]^-{(\mathrm{P.D.})^{-1}}_-\isom
	&
	R^{\ast+n} (N),
}
\end{equation}
respectively,
where the isomorphisms are given by Poincaré duality.
\end{thrm}

In the proof of Theorem~\ref{thm:pdumkcomp}, we will need 
the following lemma.

\begin{lemma}
\label{lm:thomandregulardiag}
Suppose $\xi$ is a virtual bundle of dimension $d$ over a space $B$
and let $u \in R^d(B^\xi)$ be an $R$--orientation of $\xi$.
Then the Thom diagonal map $\hat{\Delta}_\xi$
of Definition~\ref{def:thomdiagonal}
and the equivalence $\tilde{u}_\bullet$ of 
Proposition~\ref{prop:thomisoreinterpretations} fit 
into a commutative diagram
\begin{equation}
\label{eq:diagonalandthomdiagonal}
\vcenter{\xymatrix{
	R \smashprod B^\xi
	\ar[r]^-{\hat{\Delta}_\xi}
	\ar[d]_{\tilde{u}_\bullet}^\homot
	&
	(R \smashprod B_+) \smashprod^R (R\smashprod B^\xi) 
	\ar[d]^{1 \smashprod^R \tilde{u}_\bullet}_\homot
	\\
	R\smashprod B_+
	\ar[r]^-{\Delta}
	&
	(R\smashprod B_+) \smashprod^R (R\smashprod B_+)
}}
\end{equation}
where $\Delta$ is the map induced by the diagonal map of $B$.
\end{lemma}
\begin{proof}
Let $c \colon R_B \to R$ be the canonical cartesian morphism.
The composite map
\begin{equation}
\label{eq:rbrbtor}
	R_B \extsmashprod^R R_B
	\xto{\ c \extsmashprod^R c\ }
	R \extsmashprod^R R
	\xto{\ \homot\ }
	R 
\end{equation}
where the equivalence is the left (or what is the same, right) unit
constraint for $\extsmashprod^R$ is cartesian, and we define 
\[
	\tilde{\Delta} \colon R_B \longto R_B \extsmashprod^R R_B
\]
to be the unique map covering the diagonal map $B \to B\times B$
factoring the map $c$ through \eqref{eq:rbrbtor}. 
The map \eqref{eq:rbrbtor} agrees with the composite
\[
	R_B \extsmashprod^R R_B
	\xto{\ c \extsmashprod^R 1\ }
	R \extsmashprod^R R_B
	\xto{\ \homot\ }
	R_B
	\xto{\ c\ }
	R
\]
where the equivalence is the left unit constraint for $\extsmashprod^R$,
so the map $\tilde{\Delta}$ can be equivalently described
as the unique morphism covering the diagonal map $B \to B\times B$ 
making the rectangle on the right in the following diagram commute:
\[\xymatrix@!0@C=5em@R=10ex{
	&
    R\smashprod_\fw S^\xi
   	\ar[rrr]^{\tilde{u}}_\homot
	\ar[d]_{\tilde{\Delta}_\xi}
	\ar `l[ld] `[rddd]_{\id} [ddd]
	&&&
	R_B
	\ar[d]^{\tilde{\Delta}}
	\ar `r[rd] `[lddd]^{\id} [ddd]
	&
	\\
	&
	R_B \extsmashprod^R (R\smashprod_\fw S^\xi)
	\ar[rrr]^{1 \extsmashprod^R \tilde{u}}_\homot
	\ar[d]_{c \extsmashprod^R 1}
	&&&
	R_B \extsmashprod^R R_B
	\ar[d]^{c \extsmashprod^R 1}
	&
	\\
	&
	R \extsmashprod^R (R\smashprod_\fw S^\xi)
	\ar[rrr]^{1 \extsmashprod^R \tilde{u}}_\homot
	\ar[d]_\homot
	&&&
	R \extsmashprod^R R_B
	\ar[d]^\homot
	\\
	&
	R\smashprod_\fw S^\xi
	\ar[rrr]^{\tilde{u}}_\homot
	&&&
	R_B
}\]
Here $\tilde{\Delta}_\xi$ is as in Definition~\ref{def:thomdiagonal};
the vertical equivalences are given by unit constraints for $\extsmashprod^R$;
the rectangle on the left commutes by the construction of $\tilde{\Delta}_\xi$;
and the bottom two rectangles in the middle commute by 
Proposition~\ref{prop:circandextsmashcompat}
and the naturality of the left unit constraint.
The composite of the bottom two maps on the right hand column is cartesian,
so the commutativity of the rest of the diagram now implies that 
the square on the top commutes. Applying $H_\bullet$ to this square 
and recognizing the result as \eqref{eq:diagonalandthomdiagonal}
now yields the claim.
\end{proof}

\begin{proof}[Proof of Theorem~\ref{thm:pdumkcomp}]
Recall that given a dual pair of spectra $(X,Y)$, we have for each $k \in \Z$ 
an isomorphism 
\[
	B  = B_{(X,Y)} \colon R^k(X) \xto{\ \isom\ } R_{-k}(Y)
\]
given by sending the cohomology class represented by a degree $-k$ 
morphism $x \colon X \to R$  of spectra to 
the homology class represented by its adjunct 
$S \to R \smashprod Y$, that is, the composite
\[
	S \xto{\ \eta\ } X \smashprod Y \xto{\ x \smashprod 1\ } R \smashprod Y
\]
where $\eta$ is the unit of the duality between $X$ and $Y$.
Alternatively, using  $R$--module maps to represent $R$--cohomology and 
$R$--homology classes, $B$ is the map sending
the cohomology class represented by a degree $-k$ 
map $\tilde{x} \colon R \smashprod X \to R$ of $R$--modules
to the homology class represented by the composite
\[
	R
	\xto{\ \tilde{\eta}\ } 
	(R\smashprod X) \smashprod^R (R\smashprod Y)
	\xto{\ \tilde{x}\smashprod^R 1\ }
	R \smashprod^R (R\smashprod Y)
	\homot
	R\smashprod Y
\]
where $\tilde{\eta}$ is the $R$--module map induced by $\eta$ 
and the final equivalence is given by the left unit constraint of $\smashprod^R$.
If $(X',Y')$ is another dual pair of spectra, then for dual morphisms
$h \colon X \to X'$ and $h^{\vee} \colon Y' \to Y$ the square 
\begin{equation}
\label{eq:dualmorphismsandb} 
\vcenter{\xymatrix@C+2em{
	R^\ast(X')
	\ar[r]^{R^\ast(h)}
	\ar[d]_{B_{(X',Y')}}^{\isom}
	&
	R^\ast(X)
	\ar[d]^{B_{(X,Y)}}_{\isom}
	\\
	R_{-\ast}(Y') 
	\ar[r]^{R_\ast(h^{\vee})}
	&
	R_{-\ast}(Y)
}}
\end{equation}
commutes. 
Applying Proposition~\ref{prop:hypercartandfwduality} to 
the hypercartesian morphisms $S^{-\tau_M} \oto S$ and $S^{-\tau_N} \oto S$
of equation \eqref{eq:mnhypercart}, we recover Atiyah duality \cite{atiyah61}
stating that  $(\suspension^\infty_+ M, M^{-\tau_M})$
and $(\suspension^\infty_+ N, N^{-\tau_N})$
are dual pairs of spectra.
In view of Proposition~\ref{prop:dualmaps}, as special cases of 
\eqref{eq:dualmorphismsandb} we therefore have commutative squares
\begin{equation}
\label{eq:fandfumkandb}
\vcenter{\xymatrix@C+2.6em{
	R^\ast(N)
	\ar[r]^{R^\ast(f)}
	\ar[d]_{B
	}^\isom
	&
	R^\ast(M)
	\ar[d]^{B 
	}_\isom
	\\
	R_{-\ast}(N^{-\tau_N})
	\ar[r]^{R_\ast((f,\kappa)^\leftarrow)}
	&
	R_{-\ast}(M^{-\tau_M})
}}
\qquad\text{and}\qquad
\vcenter{\xymatrix@C+2.6em{
	R^{\ast}(M^{-\tau_M})
	\ar[r]^{R^\ast((f,\kappa)^\leftarrow)}
	\ar[d]_{B
	}^\isom
	&
	R^{\ast}(N^{-\tau_N})
	\ar[d]^{B
	}_\isom
	\\
	R_{-\ast}(M)
	\ar[r]^{R_\ast(f)}
	&
	R_{-\ast}(N)
}}
\end{equation}

It is well known that for a closed smooth $R$--oriented $p$--manifold $P$,
the Poincaré duality isomorphism $R^\ast(P) \xto{\,\isom\,} R_{p-\ast}(P)$
is given by the composite 
\[\xymatrix@C+2em{
	R^\ast(P)
	\ar[r]^-{\mathrm{Thom}}_-\isom
	&
	R^{\ast-p}(P^{-\tau_P})
	\ar[r]^-{B}_-\isom
	&
	R_{p-\ast}(P)
}\]
where the first map is the Thom isomorphism.
Thus the claim that the composites \eqref{eq:thomumkthomrcoh} 
and \eqref{eq:pdindpdrcoh} 
agree follows from the commutativity of 
the right hand square in \eqref{eq:fandfumkandb}.
Moreover, the remaining claim that 
the composites \eqref{eq:thomumkthom} and \eqref{eq:pdindpd}
agree follows from the commutativity of 
the left hand square in \eqref{eq:fandfumkandb}
once we know that the Poincaré duality isomorphism
also agrees with the composite
\[\xymatrix@C+2em{
	R^\ast(P)
	\ar[r]^-{B}_-\isom
	&
	R_{-\ast}(P^{-\tau_P})
	\ar[r]^-{\mathrm{Thom}}_-\isom
	&
	R_{p-\ast}(P),
}\]
that is, that the square 
\begin{equation}
\label{eq:bsandthomisos}
\vcenter{\xymatrix@C+2em{
	R^{\ast}(P)
	\ar[r]^{\mathrm{Thom}}_\isom
	\ar[d]_B^\isom
	&
	R^{\ast-p}(P^{-\tau_P})
	\ar[d]^B_\isom
	\\
	R_{-\ast}(P^{-\tau_P})
	\ar[r]^{\mathrm{Thom}}_\isom
	&
	R_{p-\ast}(P)
}}
\end{equation}
commutes.

Showing the commutativity of \eqref{eq:bsandthomisos} occupies the 
rest of the proof. Let 
\[
	\tilde{\eta} 
	\colon
	R 
	\longto
	(R\smashprod P_+) \smashprod^R (R\smashprod P^{-\tau_P})
\]
be the $R$--module map induced by the unit of the dual pair 
$(\suspension^\infty_+ P, P^{-\tau_P})$
and let
\[
	\tilde{u}_\bullet
	\colon
	R\smashprod P^{-\tau_P}
	\xto{\ \homot\ }
	R \smashprod P_+
\]
be the equivalence of equation \eqref{eq:tildeubullet}
for $-\tau_P$ arising from the orientation of $-\tau_P$.
Using Proposition~\ref{prop:thomisoreinterpretations} and the formula for the maps $B$,
we see that showing  the commutativity of \eqref{eq:bsandthomisos} amounts
to proving that the outer rectangle in the diagram
\begin{equation*}
\vcenter{\xymatrix@!0@C=6.3em@R=3.5ex{
	R
	\ar[rr]^-{\tilde{\eta}}
	\ar[dddd]_{\tilde{\eta}}
	&&
	(R\smashprod P_+) \smashprod^R (R\smashprod P^{-\tau_p})
	\ar[rr]^\chi
	\ar[rdd]_(0.4){1\smashprod^R \tilde{u}_\bullet}
	&&
	(R\smashprod P^{-\tau_p})\smashprod^R (R\smashprod P_+)
	\ar[dddd]^{\tilde{u}_\bullet\smashprod^R 1}
	\\ \\
	&&&
	(R\smashprod P_+) \smashprod^R (R\smashprod P_+)
	\ar[rdd]_(0.4)\chi
	\\ \\
	(R\smashprod P_+) \smashprod^R (R\smashprod P^{-\tau_p})
	\ar[rrrr]^{1\smashprod^R \tilde{u}_\bullet}
	\ar[ddd]_{x\smashprod^R 1}
	&&&&
	(R\smashprod P_+) \smashprod^R (R\smashprod P_+)
	\ar[ddd]^{x\smashprod^R 1}
	\\ \\ \\ 
	R \smashprod^R (R\smashprod P^{-\tau_p})
	\ar[rrrr]^{1\smashprod^R \tilde{u}_\bullet}
	\ar[ddd]_{\homot}
	&&&&
	R \smashprod^R (R\smashprod P_+)
	\ar[ddd]^{\homot}
	\\ \\ \\
	R\smashprod P^{-\tau_p}
	\ar[rrrr]^{\tilde{u}_\bullet}
	&&&&
	R\smashprod P_+
}}
\end{equation*}
commutes up to sign $(-1)^{\deg(x) \dim(P)}$ for every map 
of $R$--modules $x \colon R \smashprod P_+ \to R$ (possibly of
non-zero degree). Here the maps labeled $\chi$ are 
symmetry constraints  and the vertical equivalences
are unit constraints for $\smashprod^R$.
The bottom rectangle commutes by naturality, as does the 
triangle in the top right hand corner, 
while the middle rectangle commutes up to 
sign $(-1)^{\deg(x) \dim(P)}$
by Proposition~\ref{prop:circandextsmashcompat}.
It remains to show that the trapezoid in the top left hand 
corner commutes.
From Example~\ref{ex:unitfactorsthroughthomdiagonal}
it follows that the map $\tilde{\eta}$ factors through
the Thom diagonal map as a composite
\[
	\tilde{\eta} 
	\colon 
	R 
	\longto 
	R \smashprod P^{-\tau_P}
	\xto{\ \hat{\Delta}_{-\tau_P}\ }
	(R\smashprod P_+) \smashprod^R (R\smashprod P^{-\tau_p}).
\]
Therefore the claim follows from Lemma~\ref{lm:thomandregulardiag}
and the observation that $\chi \Delta = \Delta$.
\end{proof}

\begin{rem}
Theorem~\ref{thm:pdumkcomp}
in particular implies that
in the case $m = n$ and $R = H\Q$, 
\eqref{eq:thomumkthom} recovers
Hopf's original ``Umkehrungshomomorphismus'' \cite{Hopf30}.
\end{rem}

\subsection{Comparison with integration along fibre maps}
\label{subsec:iafcomp}
Our aim in this subsection is to relate our umkehr maps 
to classical integration along fibre maps. 
We will prove two theorems doing so:
Theorem~\ref{thm:integrationalongfibremfldbdlecomp}, 
which is a more precisely formulated restatement of 
Theorem~\ref{thm:classicalumkehrmapscomparison}(\ref{it:iafumkcomparison}),
and Theorem~\ref{thm:umkehrmapcomparison}, which will 
find application in \cite{stringtop}.
Throughout the subsection $\bbk$ will denote a fixed field.

\begin{thrm}
\label{thm:integrationalongfibremfldbdlecomp}
Suppose $p \colon E \to B$ is a bundle of smooth closed $d$--manifolds
for which the vertical tangent bundle $\tau_p$ is $\bbk$--oriented,
and let $\theta_p\colon S^{-\tau_p} \oto S_B$ be a cartesian
morphism of $\hpSpectra^\dfop$ covering $p$ afforded by
Theorem~\ref{thm:thetapc}(\ref{it:thetapformanifoldbundles}). Then 
the composite maps
\[\xymatrix@C+3em{
	H_\ast(B;\bbk)
	\ar[r]^-{((p,\theta_p)^\leftarrow)_\ast}
	&
	H_\ast(E^{-\tau_p};\bbk)
	\ar[r]^-{\mathrm{Thom}}_-{\isom}
	&
	H_{\ast+d}(E;\bbk)
}\]
and
\[\xymatrix@C+3em{
	H^{\ast+d}(E;\bbk)
	\ar[r]^-{\mathrm{Thom}}_-{\isom}
	&
	H^\ast(E^{-\tau_p};\bbk)
	\ar[r]^-{((p,\theta_p)^\leftarrow)^\ast}
	&
	H^\ast(B;\bbk)
}\]
agree with the integration along fibre maps 
$(p,o_p)^!$ and $(p,o^p)_!$
of Definitions~\ref{def:intalongfibrehomology}
and \ref{def:intalongfibrecohomology}
where the maps $o_p$ and $o^p$ are given by
Definition~\ref{def:mapsop}.
\end{thrm}

Before proving Theorem~\ref{thm:integrationalongfibremfldbdlecomp}, 
we define precisely the integration along fibre maps
$(p,o_p)^!$ and $(p,o^p)_!$.

\begin{defn}[Integration along fibre maps in homology]
\label{def:intalongfibrehomology}
Suppose $p\colon E\to B$ is a fibration such that  
$H_k(F;\bbk) = 0$ for $k > d$ for all fibres $F$ of $p$. 
Write $\calH_\ast(F;\bbk)$ for the local coefficient system 
of graded $\bbk$--modules over $B$ defined by
the homology of the fibres of $p$, and let 
$o\colon \bbk \to \calH_d(F;\bbk)$ be a map of local 
coefficient systems, where we have written $\bbk$
for the constant local coefficient system over $B$ 
given by $\bbk$. Then the 
\emph{integration along fibre map}
$(p,o)^!$ is defined to be the composite
\begin{equation}
\label{map:intalongfibrehomology}
	p^!
	=
	(p,o)^!
	\colon
	H_s(B;\bbk)
	\xto{\ o_\ast\ }
	H_s(B;\calH_d(F;\bbk))
	=
	E^2_{s,d}
	\longto
	E^\infty_{s,d}
	\longincl
	H_{s+d}(E;\bbk).
\end{equation}
Here $E^2$ and $E^\infty$ refer to pages in the 
Serre spectral sequence of $p$, and the last two maps
exist and are an epimorphism and a monomorphism, respectively,
by the assumption that $H_k(F;\bbk) = 0$ for $k>d$.
\end{defn}

\begin{defn}[Integration along fibre maps in cohomology]
\label{def:intalongfibrecohomology}
Suppose $p\colon E\to B$ is a fibration such that  
$H^k(F;\bbk) = 0$ for $k > d$ for all fibres $F$ of $p$. 
Write $\calH^\ast(F;\bbk)$ for the local coefficient system 
of graded $\bbk$--modules over $B$ defined by
the cohomology of the fibres of $p$, and let 
$o\colon \calH^d(F;\bbk) \to \bbk$ be a map of local 
coefficient systems. Then the 
\emph{integration along fibre map}
$(p,o)_!$ is defined to be the composite
\begin{equation}
\label{map:intalongfibrecohomology}
	p_!
	=
	(p,o)_!
	\colon
	H^{s+d}(E;\bbk)
	\longto
	E_\infty^{s,d}
	\longincl
	E_2^{s,d}
	= 
	H^s(B;\calH^d(F;\bbk))
	\xto{\ o_\ast\ }
	H^s(B;\bbk).
\end{equation}
Here $E_2$ and $E_\infty$ refer to pages in the 
Serre spectral sequence of $p$, and the first two maps
exist and are an epimorphism and a monomorphism, respectively,
by the assumption that $H^k(F;\bbk) = 0$ for $k>d$.
\end{defn}

\begin{rem}
Frequently
Definitions~\ref{def:intalongfibrehomology}
and~\ref{def:intalongfibrecohomology}
are applied in the case where the fibres $F$ of $p$
are closed connected $d$--manifolds and the map $o$ is an isomorphism
of local coefficient systems. Then $o$ amounts to a 
consistent choice of orientation for each fibre.
\end{rem}

\begin{rem}
For signs to work out correctly, in 
Definitions~\ref{def:intalongfibrehomology}
and~\ref{def:intalongfibrecohomology}
one should
regard the maps $o$ as shifting degrees by $d$. In particular,
when the local coefficient system
$\calH_d(F;\bbk)$ is trivial, the map  $o_\ast$ in 
\eqref{map:intalongfibrehomology} corresponds, under 
the isomorphisms afforded by the universal coefficient theorem,
to the map
\[
	H_s(B;\bbk)\tensor_\bbk \bbk
	\xto{\ \id\tensor_\bbk o\ }
	H_s(B;\bbk)\tensor_\bbk H_d(F;\bbk),
	\qquad
	(\id\tensor_\bbk o)(z\tensor_\bbk 1) 
	= 
	(-1)^{d\deg(z)} z \tensor_\bbk o(1).
\]
Similarly, when $\calH^d(F;\bbk)$ is trivial,
the map $o_\ast$ in 
\eqref{map:intalongfibrecohomology} corresponds, under 
the universal coefficient theorem,
to the map
\[
	H^s(B;\bbk)\tensor_\bbk H^d(F;\bbk)
	\xto{\ \id\tensor_\bbk o\ }
	H^s(B;\bbk)\tensor_\bbk \bbk,
	\qquad
	(\id\tensor_\bbk o)(x\tensor_\bbk y) 
	= (-1)^{d \deg(x)} x \tensor_\bbk o(y).
\]
Compare with \cite[Appendices~A.1 and A.2]{KM19}.
\end{rem}

\begin{defn}[The maps $o_p$ and $o^p$ of 
Theorem~\ref{thm:integrationalongfibremfldbdlecomp}]
\label{def:mapsop}
Suppose $p\colon E\to B$, $\tau_p$, and $\theta_p \colon S^{-\tau_p} \oto S_B$
are as in Theorem~\ref{thm:integrationalongfibremfldbdlecomp}.
Let $\gamma$ be the degree $d$ morphism of parametrized $H\bbk$--modules over $B$
given by the composite
\[\xymatrix@C+1.4em{
	\gamma
	\colon
	H\bbk_B
	\ar[r]^-\homot
	&
	H\bbk \smashprod_{\fw} S_B
	\ar[r]^-{H\bbk \smashprod_{\fw} \beta_p}
	&
	H\bbk \smashprod_{\fw} p_! S^{-\tau_p}
	\ar[r]^-{\homot}
	&
	p_!(H\bbk \smashprod_{\fw} S^{-\tau_p})
}\]
where $\beta_p\colon S_B \to p_!S^{-\tau_p}$ 
is the vertical morphism defined by $\theta_p$
in the sense of Definition~\ref{def:dfopmordata}.
Here the first unlabeled equivalence is given by 
the unit constraint for the functor $H\bbk\smashprod_{\fw}$
and the second one is an instance of \eqref{eq:fshriektcomm}.
Moreover, let
\begin{equation}
\label{eq:tildeumap} 
\xymatrix@C+1em{
	\tilde{u}
	\colon
	H\bbk \smashprod_{\fw} S^{-\tau_p}
	\ar[r]^-\homot_-{d}
	&
	H\bbk_E	
}
\end{equation}
be the degree $d$ equivalence 
of Proposition~\ref{prop:thomclassastrivialization}
induced by the $\bbk$--orientation
of $-\tau_p$ induced by the orientation of $\tau_p$.
Writing for brevity $\calL_\ast(X)$ and $\calL^\ast(X)$
for the local coefficient systems
$\calL_\ast(H\bbk,X)$ and $\calL^\ast(H\bbk,X)$
of Definition~\ref{def:lmxs},
we define the maps $o_p$ and $o^p$ to be the composites
\[\xymatrix@C+0.5em{
	o_p 
	\colon
	\bbk 
	\ar[r]^-{\isom}
	&
	\calL_0(H\bbk_B)
	\ar[r]^-{\gamma_\ast}
	&
	\calL_0(p_!(H\bbk \smashprod_{\fw} S^{-\tau_p}))
	\ar[r]^-{(p_!\tilde{u})_\ast}
	&
	\calL_d(p_! H\bbk_E)
	=
	\calH_d(F;\bbk)
}\]
and
\[\xymatrix@C+0.5em{
	o^p 
	\colon 
	\calH^d(F;\bbk)
	= 
	\calL^d( p_! H\bbk_E)
	\ar[r]^-{(p_!\tilde{u})^\ast}
	&
	\calL^0(p_!(H\bbk \smashprod_{\fw} S^{-\tau_p}))
	\ar[r]^-{\gamma^\ast}
	&
	\calL^0(H\bbk_B)
	\ar[r]^-{\isom}
	&
	\bbk.
}\]
Here the first isomorphism in the definition of $o_p$
and the last isomorphism in the definition of $o^p$
are
induced by the canonical isomorphisms $H\bbk_0(H\bbk)\isom\bbk$
and $H\bbk^0(H\bbk)\isom \bbk$, respectively.
\end{defn}

\begin{proof}[Proof of Theorem~\ref{thm:integrationalongfibremfldbdlecomp}]
We will prove the homological version of the statement; the
cohomological one is proven similarly.
Observe that given a space $A$, we have a natural isomorphism
\begin{equation}
\label{eq:hbbkhbulletisohast3}
H\bbk_\ast H_\bullet(A;H\bbk_A)
=
\pi_\ast (H\bbk \smashprod^{H\bbk} H_\bullet(A;H\bbk_A))
\isom
\pi_\ast  H_\bullet(A;H\bbk_A)
\isom
H_\ast(A;\bbk)
\end{equation}
where the first isomorphism is given by the unit constraint for $H\bbk$--modules
and the last one by Example~\ref{ex:hbulletfortrivialcoeffs2}.
Keeping the notation of Definition~\ref{def:mapsop},  consider the diagram
\begin{equation}
\label{diag:ssumkpf}
\vcenter{
\xymatrix@!0@C=2.3em@R=8ex{
	&&&
	H_s(B;\bbk)
	\ar `l[llldddd] `[dddd]_{(o_p)_\ast} [dddd]
	\ar[d]_\isom
	\ar@/^/[drrrrrrrrrrrr]^\isom
	\\
	&&&
	H_s(B;\calL_0(H\bbk_B))
	\ar@{=}[rrrr]
	\ar[d]_{\gamma_\ast}
	&&&&
	E^2_{s,0}(H\bbk_B)
	\ar[rrrr]^-\isom
	&&&&
	E^\infty_{s,0}(H\bbk_B)
	\ar[rrrr]^-(0.45)\isom
	&&&&
	(H\bbk)_s H_\bullet(B;H\bbk_B)
	\ar[d]^{((\id_B,\gamma)_\bullet)_\ast}
	\\
	&&&
	H_s(B;\calL_0(p_!(H\bbk \smashprod_\fw S^{-\tau_p})))
	\ar[d]_{(p_! \tilde{u})_\ast}^\isom
	&&&& &&&& &&&&
	(H\bbk)_s H_\bullet(B;p_!(H\bbk \smashprod_\fw S^{-\tau_p}))
	\ar[d]^{((\id_B,p_! \tilde{u})_\bullet)_\ast}_\isom
	\\
	&&&
	H_s(B;\calL_d(p_! H\bbk_E))
	\ar@{=}[d]
	&&&& &&&& &&&&
	(H\bbk)_{s+d} H_\bullet(B;p_! H\bbk_E)
	\ar[d]_\isom
	\\
	&&&
	H_s(B;\calH_d(F))
	\ar@{=}[rrrr]
	&&&&
	E^2_{s,d}
	\ar[rrrr]
	&&&&
	E^\infty_{s,d}
	\ar[rrrr]
	&&&&
	H_{s+d}(E;\bbk)	
}}
\end{equation}
where the unlabeled vertical isomorphism on the top left is induced 
by the first isomorphism in the definition of the map $o_p$; the maps in the bottom 
row are as in Definition~\ref{def:intalongfibrehomology};
the corresponding maps in the top row are the 
canonical ones arising from the fact that the Serre
spectral sequence of $H\bbk_B$ is concentrated on the 
line $t=0$; the vertical isomorphism in the bottom right
corner is  part of the identification of 
the spectral sequence of $p_! H\bbk_E$ with the Serre
spectral sequence of $p$, that is to say, the composite
\[
	(H\bbk)_{s+d}H_\bullet(B; p_! H\bbk_E)
	\isom
	(H\bbk)_{s+d}H_\bullet(E; H\bbk_E)
	\isom
	H_{s+d}(E;\bbk)
\]
where the first isomorphism is induced by the equivalence of 
Corollary~\ref{cor:hbulletandshriek}
and the second one is given by \eqref{eq:hbbkhbulletisohast3};
and the curved morphism on the top is given by  the inverse of 
\eqref{eq:hbbkhbulletisohast3}.
Observe that the composite map along the left hand side and the bottom row 
in the diagram is precisely $(p,o_p)^!$. Thus to prove the claim,
it suffices to show that the diagram commutes and 
that the composite map along the top and down the right hand column
agrees with the composite 
\begin{equation}
\label{eq:umkandthom}
\xymatrix@C+3em{
	H_\ast(B;\bbk)
	\ar[r]^-{((p,\theta_p)^\leftarrow)_\ast}
	&
	H_\ast(E^{-\tau_p};\bbk)
	\ar[r]^-{\mathrm{Thom}}_-{\isom}
	&
	H_{\ast+d}(E;\bbk)
}
\end{equation}

We start by arguing that the diagram commutes.
The rectangle on the left is commutative by the definition 
of the map $o_p$. Moreover, the rectangle on the right can be 
filled with rows similar to the top and bottom ones,
so the rectangle commutes by the functoriality of 
the Serre spectral sequence asserted in 
Corollary~\ref{cor:ssfunctorialityalldegrees}.
The required morphisms from the $E^2$--page to the $E^\infty$--page
and further to the target of the spectral sequence in the 
right hand column exist since in each case we are working with 
the top non-zero row of the spectral sequence.
To understand why the triangle at the top commutes,
recall from our construction of the Serre spectral sequence
that $H_s(B;\calL_0(H\bbk_B))$ was described as a cellular homology group 
of $\Gamma B$ where $\Gamma B$ is a functorial CW approximation of $B$, 
so the map $H_s(B;\bbk) \to H_s(B;\calL_0(H\bbk_B))$
should be understood as the composite of the standard isomorphism from 
$H_s(B;\bbk)$ to the corresponding cellular homology group of $\Gamma B$
and the isomorphism (in cellular homology) induced by the map of coefficients
$\bbk \to \calL_0(H\bbk_B)$. Altogether, taking into account the 
way $E^2_{s,0}(H\bbk_B)$ was identified with $H_\ast(B;\calL_0(H\bbk_B))$,
the map from $H_s(B;\bbk)$ to $E^2_{s,0}(H\bbk_B)$ in the diagram
is given by the composite
\begin{equation*}
	H_s(B;\bbk)
	\isom
	H_s(\Gamma B;\bbk)
	\isom
	H_s( C_\ast )
	\isom
	H_s( D_\ast	)
	=
	E^2_{s,0}(H\bbk_B)
\end{equation*}
where $C_\ast$ and $D_\ast$ are chain complexes
\begin{align*}
	C_\ast &= (\cdots \to H_n(\Gamma B^{(n)}, \Gamma B^{(n-1)};\bbk) \to \cdots)
	\\
	D_\ast &= (
		\cdots \to 
		H\bbk_n H_\bullet(\Gamma B^{(n)}, \Gamma B^{(n-1)};\underline{H\bbk}) 
		\to \cdots
	);
\end{align*}
the first isomorphism is induced by the weak equivalence $\Gamma B \to B$;
the second isomorphism is the isomorphism between the homology
$H_\ast(-;\bbk)$ and the cellular homology derived from it; and the 
isomorphism from $H_s(C_\ast)$ to $H_s(D_\ast)$ is induced by the chain map
given by the isomorphisms
\[
	H_n(\Gamma B^{(n)}, \Gamma B^{(n-1)};\bbk)
	\xto{\ \isom\ }
	H\bbk_n H_\bullet(\Gamma B^{(n)}, \Gamma B^{(n-1)};\underline{H\bbk}) 
\]
given by the inverse of the relative analogue of \eqref{eq:hbbkhbulletisohast3}
obtained by using Example~\ref{ex:relativehbulletfortrivialcoeffs}
in place of Example~\ref{ex:hbulletfortrivialcoeffs2}.
On the other hand, tracing through the 
workings of the spectral sequence, the composite map
$E^2_{s,0}(H\bbk_B) \to  E^\infty_{s,0}(H\bbk_B) \to  (H\bbk)_s H_\bullet(B;H\bbk_B)$ 
in the diagram is given by the composite
\[
	E^2_{s,0}(H\bbk_B) = H_s(D_\ast) 
	\isom
	H\bbk_s(\Gamma B; \underline{H\bbk})
	\isom
	H\bbk_s(B; \underline{H\bbk})
\]
where the first isomorphism is 
the isomorphism between the homology $H\bbk_\ast H_\bullet(-;\underline{H\bbk})$
and the corresponding cellular homology 
and the second isomorphism is induced by the weak equivalence $\Gamma B \to B$.
All told, the composite map from 
$H_s(B;\bbk)$ to $H\bbk_s(B; \underline{H\bbk})$
amounts to the inverse of an instance of \eqref{eq:hbbkhbulletisohast3} as claimed.

It remains to show that the composite map along the top and down the right 
hand column of  diagram \eqref{diag:ssumkpf}
agrees with the composite \eqref{eq:umkandthom}.
Consider the diagram
\begin{equation}
\label{diag:rightcolext}
\vcenter{\xymatrix{
	(H\bbk)_s H_\bullet(B;H\bbk_B)
	\ar[d]_{((\id_B,\gamma)_\bullet)_\ast}
	\ar[dr]^{((p,\theta_p^{H\bbk})^{\leftarrow})_\ast}
	\\
	(H\bbk)_s H_\bullet(B; p_!(H\bbk \smashprod_\fw S^{-\tau_p}))
	\ar[r]^\isom
	\ar[d]_{((\id_B, p_!\tilde{u})_\bullet)_\ast}
	&
	(H\bbk)_s H_\bullet(E; H\bbk \smashprod_\fw S^{-\tau_p})
	\ar[d]^{((\id_B, \tilde{u})_\bullet)_\ast}
	\\
	(H\bbk)_{s+d} H_\bullet(B;p_!H\bbk_E)
	\ar[r]^\isom
	&
	(H\bbk)_{s+d} H_\bullet(E;H\bbk_E)
}}
\end{equation}
where the maps in the left hand column are part of the right hand column
of \eqref{diag:ssumkpf}; the horizontal isomorphisms are induced by 
the equivalence of Corollary~\ref{cor:hbulletandshriek};
and  $\theta_p^{H\bbk}$ is the composite
\begin{equation}
\label{eq:thetaphk} 
\xymatrix@C+3em{
	\theta_p^{H\bbk} 
	\colon 
	H\bbk\smashprod_{\fw} S^{-\tau_p}
	\ar[r]|-{\circdec}^-{ H\bbk \smashprod_\fw \theta_p }
	&
	H\bbk\smashprod_{\fw} S_B
	\ar[r]|-{\circdec}^-{ \homot }
	&
	H\bbk_B
}
\end{equation}
where the equivalence is given by the unit constraint of $H\bbk\smashprod_{\fw}$.
The map $\gamma$ is the vertical morphism determined by $\theta_p^{H\bbk}$
in the sense of Definition~\ref{def:dfopmordata}, 
so the triangle at the top of \eqref{diag:rightcolext}
commutes by the definition of the map $(p,\theta_p^{H\bbk})^{\leftarrow}$
(Definition~\ref{def:hprime}).
On the other hand, the square at the bottom of  \eqref{diag:rightcolext}
commutes by the naturality of the equivalence of 
Corollary~\ref{cor:hbulletandshriek}.

In view of the commutativity of diagram \eqref{diag:rightcolext},
to prove the claim, it suffices to show that the composite 
of the slanted and right hand vertical maps in 
\eqref{diag:rightcolext} corresponds, under 
isomorphisms given by  \eqref{eq:hbbkhbulletisohast3},
to the composite \eqref{eq:umkandthom}.
Notice that we have a composite isomorphism
\begin{equation}
\label{eq:middleiso}
\begin{split}
	(H\bbk)_s H_\bullet(E; H\bbk \smashprod_\fw S^{-\tau_p})
	= 
	\pi_s(H\bbk \smashprod^{H\bbk} H_\bullet(E; H\bbk \smashprod_\fw S^{-\tau_p}))
	\xto{\ \isom\ }
	\pi_s(H_\bullet(E; H\bbk \smashprod_\fw S^{-\tau_p}))
	\\
	\xto{\ \isom\ }
	\pi_s(H\bbk \smashprod H_\bullet^{\Spectra}(E; S^{-\tau_p}))
	=
	\pi_s(H\bbk \smashprod E^{-\tau_p})
	=
	H_s(E^{-\tau_p};\bbk)
\end{split}
\end{equation}
where the various equalities hold by definition and the first
isomorphism is given by the left unit constraint for $\smashprod^{H\bbk}$
and the second isomorphism is induced by an instance of equivalence
\eqref{eq:hbulletandfcommute}.
Moreover, notice that the equivalence
\[
	H_\bullet(B; H\bbk_B) \xto{\ \homot\ } H\bbk \smashprod \suspension^\infty_+ B
\]
underlying Example~\ref{ex:hbulletfortrivialcoeffs2}
can be described as the unique dashed morphism making the diagram 
\[\xymatrix{
	H\bbk_B
	\ar[d]
	\ar[r]^-\homot
	&
	H\bbk \smashprod_\fw S_B
	\ar[r]^-\homot
	&
	H\bbk \smashprod_\fw (\suspension^\infty_+)_\fw (B,\id_B)
	\ar[d]
	\\
	H_\bullet(B; H\bbk_B)
	\ar@{-->}[rr]^\homot
	&&
	H\bbk \smashprod \suspension^\infty_+ B
}\]
commute where 
the vertical morphism on the left is the defining
opcartesian morphism of $H_\bullet(B; H\bbk_B)$;
the vertical morphism on the right is the opcartesian 
morphism obtained by applying 
the functor $H\bbk \smashprod_\fw (\suspension^\infty_+)_\fw(-)$
to the opcartesian morphism $(B,\id_B) \to (B,r_B)$
of parametrized spaces covering the unique map $r_B \colon B \to \pt$;
the first equivalence in the top row is given by the 
inverse of the unit constraint for the functor $(H\bbk \smashprod)_B$;
and the second equivalence in the top row is 
induced by the unit constraint for the functor $(\suspension^\infty_+)_B$.
Using the resulting description of  \eqref{eq:hbbkhbulletisohast3},
one can show that under 
isomorphims \eqref{eq:middleiso} and \eqref{eq:hbbkhbulletisohast3},
the map $((p,\theta_p^{H\bbk})^{\leftarrow})_\ast$ of diagram 
\eqref{diag:rightcolext}
corresponds to the map $((p,\theta_p)^\leftarrow)_\ast$ of
equation \eqref{eq:umkandthom}.
Moreover, making use of Proposition~\ref{prop:thomisoreinterpretations},
one sees that under \eqref{eq:middleiso} and \eqref{eq:hbbkhbulletisohast3},
the map $((\id_B, \tilde{u})_\bullet)_\ast$ of diagram \eqref{diag:rightcolext}
corresponds to the Thom isomorphism of equation \eqref{eq:umkandthom}.
Thus the claim follows.
\end{proof}

Having finished the proof of Theorem~\ref{thm:integrationalongfibremfldbdlecomp},
we now turn our sights to establishing Theorem~\ref{thm:umkehrmapcomparison}
which in \cite{stringtop} plays a pivotal role in comparing different
constructions for a string product pairing.
Like 
Theorem~\ref{thm:integrationalongfibremfldbdlecomp},
the theorem compares our umkehr maps to integration along fibre maps,
but instead of a bundle of manifolds with an oriented vertical tangent bundle, 
the theorem applies to a fibration $p\colon E \to B$ covered by 
a supercartesian morphism
$	
	\theta
	\colon
	\suspension^{n}_E H\bbk_E 
	\oto 
	\suspension^{m}_B H\bbk_B.
$
We note that in the context of Theorem~\ref{thm:integrationalongfibremfldbdlecomp},
such a morphism $\theta$ can be obtained as the composite
\[\xymatrix{
	\suspension^{-d}_E H\bbk_E
	\ar[r]^-\homot
	&
	H\bbk\smashprod_{\fw} S^{-\tau_p}
	\ar[r]|-{\circdec}^-{\theta_p^{H\bbk}}
	&
	H\bbk_B
}\]
where the equivalence is induced by the equivalence $\tilde{u}$ of 
equation~\eqref{eq:tildeumap}
and $\theta_p^{H\bbk}$ is the map of equation~\eqref{eq:thetaphk}.
Continuing to write $\calL_\ast(X)$ and $\calL^\ast(X)$
for the local coefficient systems
$\calL_\ast(H\bbk,X)$ and $\calL^\ast(H\bbk,X)$,
we begin by proving the following lemma.

\begin{lemma}
\label{lm:lcspd}
Suppose $p\colon E \to B$ is a fibration and 
\[
	\theta
	\colon
	\suspension^{n}_E H\bbk_E 
	\longoto 
	\suspension^{m}_B H\bbk_B.
\] 
is a supercartesian morphism of 
$\hpMod^{H\bbk}$
covering $p$.
Then for all $s \in \Z$, there is an isomorphism
\begin{equation}
\label{eq:lcspd2} 
	\calH^s(F;\bbk)
	=
	\calL^{s}(p_!H\bbk_E)
	\isom
	\calL_{m-n-s}(p_!H\bbk_E)
	=
	\calH_{m-n-s}(F;\bbk).
\end{equation}
In particular, the 
$\bbk$--homology and $\bbk$--cohomology 
of the fibres of $p$ vanish 
in degrees higher than $d=m-n$
as required 
for the existence of the integration along fibre maps
of Definitions~\ref{def:intalongfibrehomology}
and~\ref{def:intalongfibrecohomology}.
\end{lemma}

In the proof of Lemma~\ref{lm:lcspd}, we will make use of the following remark.

\begin{rem}
\label{rk:calldualpair}
Let $R$ be a commutative ring spectrum.
Generalizing the isomorphism 
\[
	M^\ast(X) \isom M_{-\ast}(Y)
\]
for an $R$--module $M$ and a dual pair of $R$--modules $(X,Y)$
induced by the equivalences
\[
	F^R(X,M) 
	\homot 
	F^R(X,R)\smashprod^R M 
	\homot 
	M \smashprod^R F^R(X,R)
	\homot 
	M\smashprod^R Y,
\]
given an $R$--module $M$ and a dual pair $(X,Y)$ in 
$\ho(\Mod^R_{/B})$ we have an isomorphism
\[
	\calL^\ast(M,X) \isom \calL_{-\ast}(M,Y)
\]
between the local coefficient systems of Definition~\ref{def:lmxs}.
\end{rem}

\begin{proof}[Proof of Lemma~\ref{lm:lcspd}]
The $\oslash$--product of $\theta$ with 
the canonical cartesian morphism 
$p^\ast \suspension_B^{-m} H\bbk_B\to \suspension_B^{-m}H\bbk_B$
covering $p$
is a supercartesian morphism 
$\suspension_E^{n-m} H\bbk_E \oto H\bbk_B$
covering $p$. 
See Proposition~\ref{prop:superandhypercartmorprops}(\ref{it:cartoslashhypercart}).
Now Propositions~\ref{prop:dualizingobjectsandcartesianmorphisms} 
and \ref{prop:dualizingobjectfwdual}
imply that
$p_! H\bbk_E$ and 
$p_! \suspension_E^{n-m} H\bbk_E 
\homot 
\suspension_B^{n-m} p_! H\bbk_E$
are dual to each other in $\ho(\Mod^{H\bbk}_{/B})$.
The desired isomorphism thus follows from 
Remark~\ref{rk:calldualpair} 
and isomorphism~\eqref{eq:callsuspisohomology}.
\end{proof}

We continue by defining the maps of local coefficient systems
which will feature in the integration along fibre maps of
Theorem~\ref{thm:umkehrmapcomparison}.

\begin{defn}[Maps $o_\theta$ and $o^\theta$ of local coefficient systems]
\label{def:othetat}
Let $p\colon E \to B$ and
$\theta	\colon \suspension^{n}_E H\bbk_E \oto \suspension^{m}_B H\bbk_B$
be as in Lemma~\ref{lm:lcspd}.
Let 
\[
	\beta
	\colon 
	\suspension^{m}_B H\bbk_B 
	\longto 
	p_!\suspension^{n}_E H\bbk_E
\]
be the vertical morphism determined by $\theta$ in the sense of 
Definition~\ref{def:dfopmordata}.
We define 
$o_{\theta}\colon \bbk \to \calH_{m-n}(F;\bbk)$ 
to be the map of local coefficient
systems given by the composite
\begin{multline*}
    \xymatrix@C+2em{
    	o_{\theta}
    	\colon
    	\bbk
    	\ar[r]^-\isom
		&
    	\calL_0(H\bbk_B)
		\ar[r]^-\isom
		&
		\calL_m(\suspension_B^m H\bbk_B)
    	\ar[r]^-{\calL_m(\beta)}
    	&
    	\calL_m(p_!\suspension^{n}_E H\bbk_E)
    }
    \\
    \xymatrix@C+2em{
    	\ar[r]^-\isom
		&
    	\calL_{m}(\suspension^{n}_B p_! H\bbk_E)
    	\ar[r]^-\isom
		&
		\calL_{m-n}(p_!H\bbk_E)
		= 
		\calH_{m-n}(F;\bbk).
    }
\end{multline*}
Here the first isomorphism is induced by the 
canonical isomorphism $H\bbk_0(H\bbk)\isom\bbk$,
the second one is an instance of \eqref{eq:callsuspisohomology},
the first isomorphism on the second line is
induced by an instance of \eqref{eq:fshriektcomm},
and the second isomorphism on the second line
is the inverse of an instance of \eqref{eq:callsuspisohomology}.
Similarly, we define
$o^{\theta} \colon \calH^{m-n}(F;\bbk) \to \bbk$
to be the map of local coefficient systems given by the composite
\begin{multline*}
    \xymatrix@C+2em{
    	o^{\theta}
    	\colon
		\calH^{m-n}(F;\bbk)
		=
		\calL^{m-n}(p_! H\bbk_E)
		\ar[r]^-\isom
		&
    	\calL^{m}(\suspension^{n}_B p_! H\bbk_E)
		\ar[r]^-\isom
		&
    	\calL^{m}( p_! \suspension^{n}_E H\bbk_E)
    }
    \\
    \xymatrix@C+2em{
		\ar[r]^-{\calL^m(\beta)}
		&
		\calL^m(\suspension_B^m H\bbk_B)
		\ar[r]^-\isom
		&
		\calL^0(H\bbk_B)
		\ar[r]^-\isom
		&
		\bbk.
    }
\end{multline*}
Here the first isomorphisms on both lines are induced by 
instances of 
\eqref{eq:callsuspisocohomology},
the second isomorphism on the first line is induced 
by an instance of \eqref{eq:fshriektcomm},
and the second isomorphism on the second line is
induced by the canonical isomorphism $H\bbk^0(H\bbk) \isom \bbk$.
\end{defn}

\begin{rem}
\label{rk:othetaaltdesc}
Using Proposition~\ref{prop:dualmaps},
we can alternatively describe $o_\theta$ 
as the map given on fibres over $b\in B$ 
by the composite
\[
	\bbk
	\isom
	H^0(\{b\};\bbk) 
	\xto{\quad r^\ast_{F_b}\quad}
	H^0(F_b;\bbk)
	\isom
	H_{m-n}(F_b;\bbk)
\]
where  $F_b = p^{-1}(b)$ and the second isomorphism is
given by \eqref{eq:lcspd2}.
Similarly, the map $o^\theta$ can be described as the 
map given on fibres over $b\in B$ by 
the composite
\[
	H^{m-n}(F_b;\bbk)
	\isom
	H_0(F_b;\bbk)
	\xto{\quad r^{F_b}_\ast\quad }
	H_0(\{b\};\bbk)
	\isom
	\bbk
\]
where the first isomorphism is given by \eqref{eq:lcspd2}.
Here $r_{F_{b}} = r^{F_b}$ denotes the unique map $F_b \to \pt$.
In particular, when the fibres of $p$ are connected,
the maps $o_\theta$ and $o^\theta$ are isomorphisms.
\end{rem}

Next we will define the umkehr maps which we will
compare with the integration along fibre maps.
\begin{defn}
\label{def:psharpumkehrmaps}
Let $p\colon E\to B$ 
be a continuous map and let 
$
	\theta
	\colon
	\suspension^{n}_E H\bbk_E 
	\oto 
	\suspension^{m}_B H\bbk_B
$
be a supercartesian morphism
in $(\hpMod^{H\bbk})^\dfop$ covering $p$.
We define the umkehr maps
\[
	(p,\theta)^\sharp
	 \colon H_\ast(B;\bbk) 
	\longto
	H_{\ast+m-n}(E;\bbk) 
	\qquad\text{and}\qquad
	(p,\theta)_\sharp
	\colon
	H^{\ast+m-n}(E;\bbk) 
	\longto
	H^\ast(B;\bbk) 
\]
to be the unique morphisms making the following
diagrams commutative:
\[\xymatrix@C+1.5em{
	H_{\ast-m}(B;\bbk)
	\ar@{-->}[r]^{(p,\theta)^\sharp}
	\ar[d]_\isom
	&
	H_{\ast-n}(E;\bbk)
	\ar[d]^\isom
	\\
	H_{\ast-m} H_\bullet(B;H\bbk_B)
	\ar[d]_{\bar{\sigma}^m}^\isom
	&
	H_{\ast-n} H_\bullet(E;H\bbk_E)
	\ar[d]^{\bar{\sigma}^n}_\isom
	\\
	H_\ast H_\bullet(B;	\suspension^{m}_B H\bbk_B)
	\ar[r]^{\ ((p,\theta)^{\leftarrow})_\ast\ }
	&
	H_\ast H_\bullet(E;\suspension^{n}_E H\bbk_E)
}
\quad
\xymatrix@C+1.5em{
	H^{\ast-n}(E;\bbk)
	\ar@{-->}[r]^{(p,\theta)_\sharp}
	\ar[d]_\isom
	&
	H^{\ast-m}(B;\bbk)
	\ar[d]^\isom
	\\
	H^{\ast-n} H_\bullet(E;H\bbk_E)
	\ar[d]_{\bar{\sigma}^n}^\isom
	&
	H^{\ast-m} H_\bullet(B;H\bbk_B)
	\ar[d]^{\bar{\sigma}^m}_\isom
	\\
	H^\ast H_\bullet(E;\suspension^{n}_E H\bbk_E)
	\ar[r]^{\ ((p,\theta)^{\leftarrow})^\ast\ }
	&
	H^\ast H_\bullet(B;	\suspension^{m}_B H\bbk_B)
}\]
Here the upper vertical isomorphisms
follow from Example~\ref{ex:hbulletfortrivialcoeffs2}
and the lower ones are instances of 
\eqref{eq:barsigmauhomology}
and 
\eqref{eq:barsigmaucohomology}.
\end{defn}

\begin{thrm}
\label{thm:umkehrmapcomparison}
Let $p\colon E \to B$ a fibration and let 
$\theta	\colon \suspension^{n}_E H\bbk_E \oto \suspension^{m}_B H\bbk_B$
be a supercartesian morphism of $(\hpMod^{H\bbk})^\dfop$
covering $p$. Then
\[
    (p,\theta)^\sharp 
    =
    (p,o_{\theta})^! 
    \colon 
    H_s(B;\bbk) 
    \longto 
    H_{s+m-n}(E;\bbk)
\]
and
\[
    (p,\theta)_\sharp 
    =
    (p,o^{\theta})_! 
    \colon 
    H^{s+m-n}(E;\bbk)
    \longto 
    H^s(B;\bbk) 
\]
for all $s$.
\end{thrm}
\begin{proof}%
The proof is similar to that of Theorem~\ref{thm:integrationalongfibremfldbdlecomp},
the main points being the functoriality of the spectral sequences 
with respect to the maps featuring in the definition of 
the maps $o_\theta$ and $o^\theta$
and the factorization of $(p,\theta)^{\leftarrow}$ (by Definition~\ref{def:hprime})
as the composite
\[
	(p,\theta)^{\leftarrow}
	\colon
	H_\bullet(B;\suspension^{m}_B H\bbk_B)
	\xto{\ (\id_B,\beta)_\bullet\ }
	H_\bullet(B;p_!\suspension^{n}_E H\bbk_E)
	\xto{\ \homot\ }
	H_\bullet(E; \suspension^{n}_E H\bbk_E)	
\]
where the latter map is the equivalence of 
Corollary~\ref{cor:hbulletandshriek}.
Proposition~\ref{prop:sssusp} provides the necessary compatibility 
between the spectral sequences and suspension isomorphisms.
\end{proof}

\appendix

\section{The framed bicategories \texorpdfstring{$\Ex_B(\calC)$}{ExB(C)}}
\label{app:exbc}

Our aim in this appendix is to construct the 
framed bicategories $\Ex_B(\calC)$
which provide the basis for our main
existence result for umkehr maps,
Theorem~%
\ref{thm:hypercartexistence}.
We start by outlining the definition of a framed bicategory.
For details, the reader should consult \cite{Shulman}.
The bicategories $\Ex_B(\calC)$ considered here 
are elaborations and generalizations of the bicategories
$\mathscr{E}\:\!\! x_B$ introduced by May and Sigurdsson \cite{MaySigurdsson}.

\begin{defn}
 A \emph{pseudo double category} $\bbD$ consists of 
a category $\bbD_0$, called the \emph{vertical category},
whose objects are called \emph{objects} or
\emph{$0$--cells} and whose morphisms are called 
\emph{vertical morphism};
a category $\bbD_1$ whose objects are called
\emph{horizontal $1$--cells} or just \emph{$1$--cells}
and whose morphisms are called \emph{$2$--cells};
functors $L,R\colon \bbD_1 \to \bbD_0$ sending 
$1$--cells and $2$--cells to their \emph{source} (or \emph{left frame}) 
and \emph{target} (or \emph{right frame}),
respectively; a \emph{horizontal composition} functor
$\odot \colon \bbD_1\times_{\bbD_0}\bbD_1 \to \bbD_1$
for $1$--cells and $2$--cells, where the pullback is over
$\bbD_1 \xto{\ R\ } \bbD_0 \xot{\ L\ }  \bbD_1$;
a functor $U \colon \bbD_0 \to \bbD_1$ sending a $0$--cell
to the corresponding unit $1$--cell;
and natural associativity and unitality isomorphisms for 
$\odot$ and $U$. These data are supposed to satisfy various axioms
making $\bbD$, roughly, a weak category object in the category 
categories. A \emph{framed bicategory} is a pseudo double category $\bbD$
such that the functor $(L,R) \colon \bbD_1 \to \bbD_0 \times \bbD_0$
is a bifibration, that is, a fibration and an opfibration.
\end{defn}

Given a framed bicategory $\bbD$,
we write 
\[
	A \xhto{\ M\ } B
\]
for a $1$--cell $M$ with $L(M) = A$ and $R(M) = B$, and 
\begin{equation}
\label{eq:2cell}
\vcenter{\xymatrix{
	A
	\ar@{}[dr]|{\Downarrow \alpha}
	\ar[d]_f
	\ar[r]|-@{|}^M
	&
	B
	\ar[d]^g
	\\
	C
	\ar[r]|-@{|}_(0.48)N %
	&
	D
}}
\end{equation}
for a $2$--cell $\alpha \colon M \to N$ with $L(\alpha) = f \colon A\to C$
and $R(\alpha) = g \colon B \to D$.
Following Shulman, we compose $1$--cells in the diagrammatic order, so that
\[
	(A \xhto{\ M\ } B) \odot (B\xhto{\ N\ } C)
	= 
	(A \xhto{\ M\odot N\ } C).
\]
A $2$--cell $\alpha$ is \emph{globular} if $L(\alpha)$ and $R(\alpha)$ are
both identity maps. The $0$--cells, $1$--cells, and globular $2$--cells
of a framed bicategory $\bbD$ 
form a bicategory called the \emph{horizontal bicategory} of $\bbD$.

\begin{example}
Any bicategory can be viewed as a framed bicategory 
whose vertical category is discrete.
In particular, a monoidal category 
amounts to the same thing as a framed bicategory whose
vertical category is the terminal category.
\end{example}

\begin{example}
Rings, ring homomorphisms, bimodules, and bimodule homomorphisms
assemble into a framed bicategory $\frbMod$.
The vertical category of $\frbMod$ is the category of 
rings and ring homomorphisms,
and the $1$--cells are bimodules. $2$--cells 
of the form \eqref{eq:2cell}
are group homomorphisms $\alpha\colon M\to N$ such that 
$\alpha(amb) = f(a)\alpha(m)g(b)$
for all $a\in A$, $m\in M$, and $b\in B$.
Horizontal composition of $1$--cells is given by tensor product
of bimodules: 
\[
	(A \xhto{M} B) \odot (B\xhto{N} C) = (A \xhto{M\tensor_B N} C).
\]
\end{example}

\begin{defn}[Base change objects]
\label{def:basechangeobjects}
Suppose $f\colon A\to B$ is a vertical morphism in a framed bicategory
$\bbD$. We write ${}_f B \colon A\hto B$ and $B_f \colon B \hto A$ 
for the $1$--cells characterized up to unique isomorphism by 
the existence of cartesian $2$--cells
\[
    \vcenter{\xymatrix{
    	A
    	\ar@{}[dr]|{\cart}
    	\ar[d]_f
    	\ar[r]|-@{|}^{{}_f B}
    	&
    	B
    	\ar@{=}[d]
    	\\
    	B
    	\ar[r]|-@{|}_(0.54){U_B} %
    	&
    	B
    }}
    \qquad\text{and}\qquad
    \vcenter{\xymatrix{
    	B
    	\ar@{}[dr]|{\cart}
    	\ar@{=}[d]
    	\ar[r]|-@{|}^{B_f}
    	&
    	A
    	\ar[d]^f
    	\\
    	B
    	\ar[r]|-@{|}_(0.54){U_B} %
    	&
    	B
    }}
\]
respectively.
We call ${}_f B$ and $B_f$ the \emph{base change objects}
associated to $f$.
\end{defn}

\begin{rem}
\label{rk:basechangeobjectsopcartdesc}
The $1$--cells  $B_f$ and ${}_f B$ of Definition~\ref{def:basechangeobjects}
are equivalently characterized by the existence of \emph{opcartesian}
$2$--cells
\[
   \vcenter{\xymatrix{
    	A
    	\ar@{}[dr]|{\opcart}
    	\ar[d]_f
    	\ar[r]|-@{|}^{U_A}
    	&
    	A
    	\ar@{=}[d]
    	\\
    	B
    	\ar[r]|-@{|}_(0.54){B_f} %
    	&
    	A
    }}
    \qquad\text{and}\qquad
    \vcenter{\xymatrix{
    	A
    	\ar@{}[dr]|{\opcart}
    	\ar@{=}[d]
    	\ar[r]|-@{|}^{U_A}
    	&
    	A
    	\ar[d]^f
    	\\
    	A
    	\ar[r]|-@{|}_(0.54){{}_f B} %
    	&
    	B
    }}
\]
respectively. 
See \cite[\S 4]{Shulman}.
\end{rem}

The virtue of the base change objects 
of Definition~\ref{def:basechangeobjects}
is that they allow us to express base change and cobase change
in the bifibration $(L,R)\colon \bbD_1 \to \bbD_0 \times \bbD_0$
in terms of horizontal composition.

\begin{prop}[{\cite[\S 4]{Shulman}}]
\label{prop:odotandbasechange}
Suppose $\bbD$ is a framed bicategory. Let $f\colon A\to B$ and 
$g\colon C\to D$ be vertical morphisms in $\bbD$. Then 
for every $1$--cell $M\colon B\hto D$ there exists a natural cartesian
$2$--cell
\[\xymatrix@C+2.4em{
	A
	\ar@{}[dr]|{\cart}
	\ar[d]_f
	\ar[r]|-@{|}^{{}_f B \odot M \odot D_g}
	&
	C
	\ar[d]^g
	\\
	B
	\ar[r]|-@{|}_{M} 
	&
	D
}\]
and 
for every $1$--cell $N\colon A\hto C$ there exists a natural opcartesian
$2$--cell
\[
	\pushQED{\qed} 
    \vcenter{\xymatrix@C+2.4em{
    	A
    	\ar@{}[dr]|{\opcart}
    	\ar[d]_f
    	\ar[r]|-@{|}^{N}
    	&
    	C
    	\ar[d]^g
    	\\
    	B
    	\ar[r]|-@{|}_{B_f \odot N\odot {}_g D}
    	&
    	D
    }}
	\qedhere
	\popQED
\]
\end{prop}

Let us now proceed to the construction of the 
framed double category $\Ex_B(\calC)$.
Our approach is to apply the construction $\Fr$ of
\cite[Thm.~14.2]{Shulman} to a certain fibred category $\hp_B(\calC)$
which we will now introduce. 
Let $B$ be a space,
and let $\calC$
be a presentable
symmetric monoidal $\infty$--category.
We write $(\calT/B)^\fib$ for the full subcategory of the 
overcategory $\calT/B$ 
spanned by fibrations $E\to B$,
and equip it with the monoidal structure given by 
fibred product over $B$, so that 
\[
	(E_1 \to B) \tensor (E_2\to B) = (E_1 \times_B E_2 \to B).
\]
We define
$\hp_B\calC \to (\calT/B)^\fib$ to be the symmetric monoidal
fibration obtained by applying the Grothendieck construction 
of Theorem~\ref{thm:smgrothendieckconstr}
to the pseudofunctor
\begin{equation}
\label{eq:hpbcfun} 
	((\calT/B)^\fib)^\op \longto \smCat,
	\quad
	(E\to B) \longmapsto \ho(\calC_{/E}),
	\quad
	\left(\vcenter{\xymatrix@!0@C=1.4em@R=2em{
			E_1 \ar[dr] \ar[rr]^f && E_2\ar[dl]\\ &B
	}}\right)
	\longmapsto f^\ast.
\end{equation}
If $\tensor$ denotes the symmetric monoidal product in $\calC$,
we write $\exttensor_B$ for the symmetric monoidal product in 
$\hp_B\calC$.
The underlying fibration $\hp_B\calC \to (\calT/B)^\fib$
can alternatively be described as the pullback
\[\xymatrix{
	\hp_B\calC
	\ar[d]
	\ar[r]^{\rho_B}
	&
	\hpC
	\ar[d]
	\\
	(\calT/B)^\fib 
	\ar[r]
	&
	\calT
}\]
of the fibration $\hpC \to \calT$ along the forgetful functor
$(\calT/B)^\fib \to \calT$. A morphism in $\hp_B\calC$ is 
cartesian or opcartesian iff its image in $\hpC$
under $\rho_B$ is.
\begin{lemma}
\label{lm:hpbcprops}
In the terminology of \cite{Shulman}, the symmetric monoidal fibration 
$\hp_B\calC \to (\calT/B)^\fib$ is a 
symmetric monoidal $\ast$--bifibration which is weakly BC and 
both internally and externally closed.
\end{lemma}
\begin{proof}
Write $\Phi_B$ for the symmetric monoidal fibration
$\hp_B\calC \to (\calT/B)^\fib$.
As the base change functors $f^\ast$ between the fibres of $\Phi_B$ have
both left adjoints $f_!$ and right adjoints $f_\ast$, the fibration $\Phi_B$ is 
a $\ast$--bifibration.  Moreover, since the fibres of $\Phi_B$
are closed symmetric monoidal categories and 
the base change functors $f^\ast$ are closed symmetric monoidal,
$\Phi_B$ is internally closed.
The commutation relation~\eqref{eq:commrelshriek}
implies that $\Phi_B$ is weakly BC. Now \cite[Prop.~13.15]{Shulman}
implies that $\Phi_B$ is externally closed.
By \cite[Prop.~13.17]{Shulman} it follows that 
the tensor product on $\hp_B\calC$
preserves opcartesian morphisms, so $\Phi_B$ 
is a symmetric monoidal $\ast$--bifibration.
\end{proof}

\begin{defn}[The framed bicategory $\Ex_B(\calC)$]
Writing $\Phi_B$ for the symmetric monoidal fibration
$\hp_B\calC \to (\calT/B)^\fib$,
we define $\Ex_B(\calC)$ to be the framed bicategory $\Fr(\Phi_B)$
of \cite[Thm.~14.2]{Shulman}. Notice that \cite[Thm.~14.2]{Shulman}
applies since $\Phi_B$ is frameable by Lemma~\ref{lm:hpbcprops}.
\end{defn}

Explicitly, the framed bicategory  $\Ex_B(\calC)$
can be described as follows: 
\begin{itemize}

\item $\Ex_B(\calC)_0  = (\calT/B)^\fib$, and 
	$\Ex_B(\calC)_1$ and the functors $L$, $R$
	are defined by the pullback square
	\[\xymatrix{
		\Ex_B(\calC)_1 
		\ar[r]
		\ar[d]_{(L,R)}
		&
		\hp_B\calC
		\ar[d]
		\\
		(\calT/B)^\fib \times (\calT/B)^\fib
		\ar[r]^-{\times_B}
		&
		(\calT/B)^\fib
	}\]
	In particular, an object in 
	$\Ex_B(\calC)$ is a fibration $E \to B$
	(which we will sometimes denote just by $E$, leaving the 
	projection to $B$ implicit);
	a vertical morphism $E_1 \to E_2$
	is a continuous map compatible with the projections from
	$E_1$ and $E_2$ to $B$;
	a $1$--cell $E_1 \hto E_2$ in $\Ex_B(\calC)$
	is an object in 
	$\hp_B(\calC)_{E_1 \times_B E_2} = \ho(\calC_{/E_1\times_B E_2})$;
	and a $2$--cell 
    \[\xymatrix{
    	E_1
    	\ar@{}[dr]|{\Downarrow \alpha}
    	\ar[d]_f
    	\ar[r]|-@{|}^M
    	&
    	E_2
    	\ar[d]^g
    	\\
    	E'_1
    	\ar[r]|-@{|}_{M'}
    	&
    	E'_2
    }\]
    is a morphism $\alpha \colon M \to M'$ in $\hp_B\calC$ covering 
    the map 
    $f\times_B g \colon E_1\times_B E_2\to E'_1\times_B E'_2.$
\item The horizontal composition of $1$--cells
	$M \colon E_1\hto E_2$ and $N\colon E_2\hto E_3$
	is given by
	\[
		M\odot_B N = \pi_!\delta^\ast (M\exttensor_B N)
	\]
	where 
	$\delta\colon E_1 E_2 E_3 \to E_1 E_2 E_2 E_3$
	is given by the diagonal map of $E_2$
	and 
	$\pi \colon E_1 E_2 E_3 \to E_1 E_3$
	is the projection away from $E_2$,
	and similarly for the horizontal composition of $2$--cells.
	(Here we have omitted the $\times_B$--signs.)
	More explicitly,
	taking into account the construction of $M\exttensor_B N$,
	we have a globular natural equivalence
    \begin{equation}
    \label{eq:modotbnformula}
		M\odot_B N 
		\homot 
		(\pi_{13})_! (
			\pi_{12}^\ast M \tensor_{E_1 E_2 E_3} \pi_{23}^\ast N
		)
    \end{equation}
	where $\pi_{ij}$ is the projection 
	$\pi_{ij} \colon E_1 E_2 E_3 \to E_i E_j$.
\item The horizontal unit associated to an object $E \to B$ is 
	\[
		U^B_E = \delta_! \pi^\ast I
	\]
	where $I \in \hp_B(\calC)_B$ is the monoidal unit in $\hp_B\calC$,
	$\pi \colon E\to B$ is the projection, and $\delta \colon E \to EE$
	is the diagonal map.
	More explicitly, writing $S_E$ for the monoidal unit in 
	$\ho(\calC_{/E})$, we have
    \begin{equation}
    \label{eq:uformula}
		U^B_E \homot \delta_! S_E.   
    \end{equation}
\item Given a map $f\colon D \to E$ over $B$, we have
	\[
		{}_f E \homot (\id,f)_! S_D
		\qquad\text{and}\qquad
		E_f \homot (f,\id)_! S_D		
	\]
	where
	\[
		(\id,f) \colon D \to D E
		\qquad\text{and}\qquad
		(f,\id) \colon D \to E D
	\]
	are the maps defined by $f$ and the identity map of $D$.
	See Remark~\ref{rk:basechangeobjectsopcartdesc}.
\item $\odot_B$ has adjoints
	 $\vartriangleright_B$ and $\vartriangleleft_B$,
	 so that given $1$--cells
	 $M\colon E_1 \hto E_2$, 
	 $N\colon E_2 \hto E_3$,
	 and
	 $P \colon E_1 \hto E_3$,
	  the set of globular $2$--cells $M\odot_B N \to P$
	 is in natural bijection with the set of globular $2$--cells
	 $M\to N \vartriangleright_B P$ and 
	 as well as the set of globular $2$--cells
	 $N \to P \vartriangleleft_B M$.
	 Explicitly, in view of the formula
	 \eqref{eq:modotbnformula}, we have 
     \begin{equation}
     \label{eq:ntrianglerightpformula}
     	N \vartriangleright_B  P
		\homot 
		(\pi_{12})_\ast F_{E_1 E_2 E_3}(\pi_{23}^\ast N, \pi_{13}^\ast P)
     \end{equation}
	 and
     \begin{equation}
     \label{eq:ptriangleleftmformula}
	 	P \vartriangleleft_B  M
		\homot 
		(\pi_{23})_\ast F_{E_1 E_2 E_3}(\pi_{12}^\ast M, \pi_{13}^\ast P)
     \end{equation}
	 where $F_{E_1 E_2 E_3}$ refers to the internal hom 
	 in $\ho(\calC_{/E_1 E_2 E_3})$ and the maps
	 $\pi_{ij}$ are as above.
\item $\Ex_B(\calC)$ has an involution $(-)^\op$ 
	which is the identity on the vertical category
	and which is given by pullback along 
	the coordinate interchange homeomorphisms $E_1 E_2 \to E_2 E_1$
	on $\Ex_B(\calC)_1$, so that for $M \colon E_1 \hto E_2$,
	we have $M^\op \colon E_2 \hto E_1$. There are natural equivalences
    \begin{align*}
    	(M^\op)^\op &\homot M
		\\
		(M\odot N)^\op &\homot N^\op \odot M^\op
		\\
	 	(N \vartriangleright_B  P)^\op &\homot N^\op \vartriangleleft_B  P^\op
		\\
	 	(P \vartriangleleft_B  M)^\op &\homot (M^\op \vartriangleright_B  P^\op)
		\\
		({}_f E)^\op &\homot E_f
		\\
		(E_f)^\op &\homot {}_f E.
    \end{align*}	
\end{itemize}

\begin{defn}[The framed bicategory $\Ex(\calC)$]
\label{def:exboverpt}
We write $\Ex(\calC)$ for $\Ex_\pt(\calC)$ and 
$\odot$, $\vartriangleright$, and $\vartriangleleft$
for
$\odot_\pt$, $\vartriangleright_\pt$, and $\vartriangleleft_\pt$,
respectively.
\end{defn}

\begin{rem}[$\Ex_B(\calC)$ and base change functors]
\label{rk:exbcandbasechange}
Given a fibration $E\to B$, the canonical 
homeomorphism $E\times_B B \homeom E$ allows us to
identify the category
$\ho(\calC_{/E})$ with the category of $1$--cells 
$E\hto B$ in $\Ex_B(\calC)$. Indeed, the whole fibration
$\hp_B\calC \to (\calT/B)^\fib$ can be identified with the pullback
of the fibration 
$\Ex_B(\calC)_1 \to (\calT/B)^\fib \times (\calT/B)^\fib$
along  the functor
$(\calT/B)^\fib \to (\calT/B)^\fib \times (\calT/B)^\fib$,
$(E\to B) \mapsto (E\to B, B\xto{\id} B)$.
If $E \to B$ and $E' \to B$ are fibrations
and $g\colon E\to E'$ is a map over $B$,
from
Proposition~\ref{prop:odotandbasechange} 
we now obtain natural equivalences
\begin{equation}
\label{eq:basechangeformula1}
	g_! X \homot E'_g \odot_B X
	\qquad\text{and}\qquad
	g^\ast Y \homot {}_g E' \odot_B Y
\end{equation}
for $X \in \ho(\calC_{/E})$ and $Y \in \ho(\calC_{/E'})$.
By uniqueness of right adjoints, the latter equivalence 
implies that we also have a natural equivalance
\begin{equation}
\label{eq:basechangeformula2}
	g_\ast X \homot X \vartriangleleft_B {}_g E'
\end{equation}
for $X \in \ho(\calC_{/E})$.
Compare with  \cite[Cor.~17.4.4]{MaySigurdsson}.
\end{rem}

Intuitively, everything in $\Ex_B(\calC)$ happens fibrewise
over the space $B$. Our next aim is to construct
a ``base change functor''
$f^{[\ast]} \colon \Ex_B(\calC) \to \Ex_A(\calC)$
associated to a continuous map $f\colon A\to B$.
We want $f^{[\ast]}$ to be a strong  framed functor in the 
sense of \cite[Def.~6.5]{Shulman}.
This means that we need to specify the following:
\begin{itemize}
\item Functors
	\[
		f^{[\ast]}_i \colon \Ex_B(\calC)_i \longto \Ex_A(\calC)_i,
		\quad
		i = 0,1,
	\]
	such that  $Lf^{[\ast]}_1 = f^{[\ast]}_0 L$ and 
	$Rf^{[\ast]}_1 = f^{[\ast]}_0 R$.
\item Globular natural equivalences
	\[
		f^{[\ast]}_\odot
		\colon 
		f^{[\ast]}_1 M \odot_A f^{[\ast]}_1 N 
		\xto{\ \homot\ } 
		f^{[\ast]}_1 (M \odot_B N)
		\qquad\text{and}\qquad
		U^A_{f^{[\ast]}_0 E}
		\xto{\ \homot\ }
		f^{[\ast]}_1 (U^B_E)
	\]
	satisfying the usual coherence axioms for a strong monoidal functor
	(see e.g.\ \cite[\S XI.2]{MacLane}).
\end{itemize}

The functor 
$f^{[\ast]}_0 \colon \Ex_B(\calC)_0 \to \Ex_A(\calC)_0$
is the functor 
$f^\ast \colon (\calT/B)^\fib \to	(\calT/A)^\fib$
 given by base change of fibrations along $f$.
Explicitly,
we choose for every fibration 
$E\to B$ 
a pullback
square
\[\xymatrix{
	f^\ast E 
	\ar[r]^-{f_E}
	\ar[d]%
	&
	E
	\ar[d]%
	\\
	A
	\ar[r]^f
	&
	B
}\]
in $\calT$
and define $f^{[\ast]}_0$ on objects by setting
$f^{[\ast]}_0(E\to  B) = (f^\ast E \to A)$
and on morphisms in the evident way using the universal property
of pullbacks.
The functor $f^{[\ast]}_1$
is defined on $1$--cells by setting
\[
	f^{[\ast]}_1(M)
	= 
	(f_{E_1} \times_f f_{E_2})^\ast M
	\colon 
	f^\ast E_1 \longhto f^\ast E_2,
\]
for  $M \colon E_1 \hto E_2$
where 
$f_{E_1} \times_f f_{E_2}
\colon 	
f^\ast E_1 \times_A f^\ast E_2
\to
E_1 \times_B  E_2$
is the map induced by $f_{E_1}$ and $f_{E_2}$.
To define $f^{[\ast]}_1$ on morphisms, recall that for a space $C$,
we have a forgetful functor
$\rho_C\colon \hp_C\calC \to \hpC$, $(D\to C,Z) \mapsto (D,Z)$
which preserves and reflects cartesian and opcartesian morphisms;
for a map $h\colon D \to D'$ over $C$, $\rho_C$ restricts to a 
bijection between the sets of morphisms covering $h$.
Given a $2$--cell
\[\xymatrix{
	E_1
	\ar@{}[dr]|{\Downarrow \alpha}
	\ar[d]_{g_1}
	\ar[r]|-@{|}^M
	&
	E_2
	\ar[d]^{g_2}
	\\
	E'_1
	\ar[r]|-@{|}_{M'}
	&
	E'_2
}\]
we now define $f^{[\ast]}(\alpha)$ to be the unique morphism 
making 
\begin{equation}
\label{eq:fast1mordef}
\vcenter{\xymatrix@!C=8em{
	(f_{E_1} \times_f f_{E_2})^\ast M 
	\ar[r]^-{\cart}
	\ar@{-->}[d]_{\rho_B(f^{[\ast]}_1(\alpha))}
	&
	M
	\ar[d]^{\rho_A(\alpha)}
	\\
	(f_{E'_1} \times_f f_{E'_2})^\ast M'
	\ar[r]^-{\cart}
	&
	M'
}}
\end{equation}
a commutative square in $\hpC$
covering the square
\begin{equation}
\label{eq:fast1mordefbase}
\vcenter{\xymatrix@!C=8em{
	f^\ast E_1 \times_A f^\ast E_2
	\ar[r]^-{f_{E_1} \times_f f_{E_2}}
	\ar[d]_{f^\ast(g_1) \times_A f^\ast(g_2)}
	&
	E_1\times_B E_2
	\ar[d]^{g_1\times_B g_2}
	\\
	f^\ast E'_1 \times_A f^\ast E'_2
	\ar[r]^-{f_{E'_1} \times_f f_{E'_2}}
	&
	E'_1\times_B E'_2
}}
\end{equation}
in $\calT$. Here the morphisms labeled `cart' in 
\eqref{eq:fast1mordef} are the canonical cartesian 
morphisms covering the respective morphisms in 
\eqref{eq:fast1mordefbase}.

To define $f^{[\ast]}_\odot$ and $f^{[\ast]}_U$, first notice that 
the forgetful functor
$\rho_C \colon \hp_C\calC \to \hpC$
is an oplax monoidal functor: The monoidality constraint
\[
	(\rho_C)_\tensor\colon \rho_C(X \exttensor_C Y) 
	\longto
	(\rho_C) X \exttensor (\rho_C) Y
\]
for objects $(E_1\to A, X)$ and $(E_2 \to A, Y) \in \hp_C\calC$
covers the inclusion 
$\iota_C\colon E_1 \times_C E_2 \incl E_1 \times E_2$
and is defined  by  the composite
\[\xymatrix{
	X\exttensor_C Y
	= 
	\tilde{\pi}_1^\ast X 
	\tensor_{E_1\times_C E_2}
	\tilde{\pi}_2^\ast Y
	\xto{\ \homot\ }
	\iota_C^\ast {\pi}_1^\ast X 
	\tensor_{E_1\times_C E_2}
	\iota_C^\ast {\pi}_2^\ast Y	
	\xto{\ \homot\ }
	\iota_C^\ast (
		{\pi}_1^\ast X 
		\tensor_{E_1\times E_2}
		{\pi}_2^\ast Y
	)
	=
	\iota_C^\ast (X\exttensor Y)
}\]
where $\tilde{\pi}_i \colon E_1\times_C E_2 \to E_i$
and $\pi_i \colon E_1\times E_2 \to E_i$, $i=1,2$
denote the projections, while
the unit constraint
\[
	(\rho_C)_I \colon \rho_C(S_C) \longto S_\pt
\]
is defined by the equivalence
$S_C \xto{\homot} r_C^\ast S_\pt$,
where $S_D$ for a space $D$ denotes the monoidal unit
in $\ho(\calC_{/D})$ and $r_C\colon C \to \pt$ is the unique map
onto the one-point space. Notice that $(\rho_C)_\tensor$ and
$(\rho_C)_I$ both consist of cartesian morphisms.
The map $f^{[\ast]}_\odot$
is now defined by the commutative diagram
\begin{equation}
\label{eq:fastodotdef}
    \vcenter{\xymatrix@C-1em{
    	&
		f^{[\ast]}_1 (M \odot_B N)
		\ar[dr]^{\cart}
		\\
		\ar@{-->}[ur]^{f^{[\ast]}_\odot}_\homot
    	f^{[\ast]}_1 M \odot_A f^{[\ast]}_1 N
    	\ar@{-->}[rr]^-{\cart}
    	&&
    	M\odot_B N
    	\\
    	\delta^\ast(f_1^{[\ast]} M \exttensor_A f_1^{[\ast]} N)
    	\ar[d]_{\cart}
    	\ar@{-->}[rr]^-{\cart}
    	\ar[u]^{\opcart}
    	&&
    	\delta^\ast(M \exttensor_B  N)
    	\ar[d]^{\cart}
    	\ar[u]_{\opcart}
    	\\
    	f_1^{[\ast]} M \exttensor_A f_1^{[\ast]} N
    	\ar@{-->}[rr]^-{\cart}
    	\ar[d]_{\cart}^{(\rho_A)_\tensor}
    	&&
    	M\exttensor_B N
    	\ar[d]^{\cart}_{(\rho_B)_\tensor}
    	\\
    	f_1^{[\ast]} M \exttensor f_1^{[\ast]} N
    	\ar[rr]^-{\cart\exttensor\cart}
    	&&
    	M\exttensor N
    }}
\end{equation}
in $\hpC$ 
covering the diagram
\begin{equation}
\label{eq:fastodotdefbase}
    \vcenter{\xymatrix{
    	&
    	f^\ast E_1   f^\ast E_3
    	\ar[dr]^-{f_{E_1} \times_f f_{E_3}}
		\\
		\ar@{=}[ur]
    	f^\ast E_1   f^\ast E_3
    	\ar[rr]^-{f_{E_1}\times_f f_{E_3}}
    	&&
    	E_1   E_3
    	\\
    	f^\ast E_1   f^\ast E_2   f^\ast E_3
    	\ar[rr]^-{f_{E_1}\times_f f_{E_2} \times_f f_{E_3}}
		\ar[u]^\pi
		\ar[d]_\delta
    	&&
    	E_1   E_2   E_3
		\ar[u]_\pi
		\ar[d]^\delta
    	\\
    	f^\ast E_1   f^\ast E_2  f^\ast E_2   f^\ast E_3
    	\ar[rr]^-{f_{E_1}\times_f f_{E_2}\times_f f_{E_2} \times_f f_{E_3}}
		\ar[d]_{\iota_A}
    	&&
    	E_1   E_2  E_2   E_3
		\ar[d]^{\iota_B}
    	\\
    	(f^\ast E_1   f^\ast E_2)\times (f^\ast E_2   f^\ast E_3)
    	\ar[rr]^-{(f_{E_1}\times_f f_{E_2})\times (f_{E_2} \times_f f_{E_3})}
    	&&
    	(E_1   E_2)\times (E_2   E_3)
    }}
\end{equation}
in $\calT$ 
where we have omitted  from notation the functors
$\rho_A$ and $\rho_B$  in diagram \eqref{eq:fastodotdef}
and the symbols $\times_A$ and $\times_B$ in 
diagram~\eqref{eq:fastodotdefbase}.
The bottom horizontal morphism in \eqref{eq:fastodotdef}
is the $\exttensor$--product of the canonical cartesian morphisms
$f^{[\ast]}_1 M \to M$ and $f^{[\ast]}_1 N \to N$,
and the solid diagonal arrow at the top of \eqref{eq:fastodotdef}
is the canonical cartesian morphism
$f^{[\ast]}_1 (M \odot_B N) \to M\odot_B N$.
The dashed arrows in diagram~\eqref{eq:fastodotdef} are now 
induced in succession from the bottom upwards
by the universal properties of cartesian and opcartesian 
morphisms.
The bottom horizontal arrow is cartesian because $\exttensor$ preserves
cartesian morphisms; the middle two are cartesian 
by Proposition~\ref{prop:cartmorprops}(\ref{it:compiscart}) 
and (\ref{it:cartfactor});
the top horizontal morphism is cartesian 
by Proposition~\ref{prop:commrelshriekinterpretation2};
and the dashed diagonal arrow is an equivalence 
by Proposition~\ref{prop:cartmorprops}(\ref{it:cartfactor}) 
and (\ref{it:cartoveriso}).

The equivalence
\[
	f^{[\ast]}_U 
	\colon 	
	U^A_{f^{[\ast]}_0 E}
	\xto{\ \homot\ }
	f^{[\ast]}_1 (U^B_E)
\]
can be constructed similarly, with the identity map of 
the monoidal unit $S_\pt \in \hpC$ playing the 
role of the bottom horizontal map in 
diagram~\eqref{eq:fastodotdef}.
In the construction, one needs to know that the top square in
the commutative diagram
\[\xymatrix@C+2em{
	f^\ast E 
	\ar[r]^{f_E}
	\ar[d]_\delta
	&
	E
	\ar[d]^\delta
	\\	
	f^\ast E \times_A f^\ast E
	\ar[r]^-{f_E\times_f f_E}
	\ar[d]
	&
	E\times_B E
	\ar[d]
	\\
	A
	\ar[r]^f
	&
	B
}\]
is homotopy cartesian. But this follows since 
the bottom square and the composite square are both 
homotopy cartesian, as they are pullback squares 
where the vertical map to $B$ is a fibration.

We record the following compatibility between 
$f^{[\ast]}$ and $\vartriangleright_A$ and  $\vartriangleright_B$. 
See \cite[Prop.~28]{LahtinenThesis} for a proof.
\begin{prop}
\label{prop:fastclosed}
The framed functor $f^{[\ast]}\colon \Ex_B(\calC) \to \Ex_A(\calC)$
is closed in the sense that the adjoint of the composite
\[
	f^{[\ast]} (N\vartriangleright_B P) \odot_A f^{[\ast]} N
	\xto[\homot]{\ f^{[\ast]}_\odot\ }
	f^{[\ast]} ((N\vartriangleright_B P) \odot_B N)
	\xto{\ f^{[\ast]}(\varepsilon)\ }
	f^{[\ast]} P
\]
where $\varepsilon$ denotes the counit of the $(\odot_B N, N\vartriangleright_B)$ adjunction
is an equivalence
\[
	f^{[\ast]} (N\vartriangleright_B P)
	\xto{\ \homot\ }
	f^{[\ast]} N \vartriangleright_A f^{[\ast]} P,
\]
and similarly for $\vartriangleleft_A$ and $\vartriangleleft_B$.
\qed
\end{prop}

\begin{prop}
\label{prop:exbf}
Suppose $F\colon \calC \to \calD$ is a symmetric monoidal
functor between symmetric monoidal presentable $\infty$--categories
which admits a right adjoint. Then $F$ induces a strong 
framed functor (see \cite[Def.~6.5]{Shulman})
\[
	\Ex_B(F) \colon \Ex_B(\calC) \longto \Ex_B(\calD)
\]
which is given by the identity on $0$--cells and by
\[
	\Ex_B(F)(X) = F_{E_1\times_B E_2} (X)
\]
on $1$--cells $X\colon E_1 \hto E_2$.
\end{prop}
\begin{proof}[Sketch of proof]
It is easily checked that $F$ induces a 
strong monoidal morphism of bifibrations
\cite[Def.~12.5]{Shulman}
\begin{equation}
\label{eq:hpbf}
\vcenter{\xymatrix@C-2em{
	\hp_B\calC
	\ar[rr]
	\ar[dr]
	&&
	\hp_B\calD
	\ar[dl]
	\\
	&
	(\calT/B)^\fib
}}
\end{equation}
so the claim follows from 
\cite[Thm.~14.9]{Shulman}. 
That \eqref{eq:hpbf} preserves
opcartesian morphisms as required follows from
Proposition~\ref{prop:opcartpreservation}.
\end{proof}

\section{Duality in bicategories}
\label{app:dualizability}

In this appendix, we briefly discuss
duality in the context of
bicategories.
For a more detailed discussion, 
we refer the reader to \cite[\S 16.4]{MaySigurdsson}.

\begin{defn}
\label{def:dualpair}
A \emph{dual pair} $(M,N)$ in a bicategory $\calB$
consists of $1$--cells $M\colon A\hto B$ and $N\colon B\hto A$
together with unit and counit maps:
\[
	\eta\colon U_A \longto M\odot N
	\qquad\text{and}\qquad
	\varepsilon\colon N\odot M \longto U_B
\]
such that the following composites are identity maps:
\begin{gather*}
\xymatrix{
	M 
	\isom 
	U_A \odot M
	\ar[r]^{\eta\odot 1}
	&
	M\odot N \odot M
	\ar[r]^{1\odot \varepsilon}
	&
	M \odot U_B
	\isom 
	M	
}
\\
\xymatrix{
	N 
	\isom 
	N\odot U_A
	\ar[r]^{1\odot\eta}
	&
	N\odot M\odot N
	\ar[r]^{\varepsilon \odot 1}
	&
	U_B\odot N
	\isom 
	N.
}
\end{gather*}
Here the displayed isomorphism are given by 
the left and right unit constraints, as appropriate.
If $(M,N)$ is a dual pair, 
we call $N$ the \emph{right dual} of $M$ and 
$M$ the \emph{left dual} of $N$, and say that 
$M$ is \emph{right dualizable} and that $N$ is \emph{left dualizable}.
The map $\eta$ is called the \emph{unit} or \emph{coevaluation map}
and the map $\varepsilon$ the \emph{counit} or \emph{evaluation map}
of the dual pair.
By a dual pair in a framed bicategory we mean a 
dual pair in its horizontal bicategory.
\end{defn}

\begin{example}
Definition~\ref{def:dualpair} should remind the reader of 
the definition of an adjoint pair of functors.
Indeed, a dual pair in the sense of 
Definition~\ref{def:dualpair}
in the $2$--category of categories
is precisely an adjoint pair of functors.
\end{example}

Many familiar formal properties of adjunctions generalize
to dual pairs in bicategories. For example,
if $(M,N)$ and $(M',N')$ are dual pairs,
then morphisms $M \to M'$ are in natural bijection
with morphisms $N' \to N$. Moreover, a right dual,
if it exists, is unique up to a unique isomorphism
compatible with the unit and counit maps, and similarly 
for left duals.

\begin{rem}
\label{rk:dualityadjunction}
If $M\colon A\hto B$ and $N\colon B\hto A$
form a dual pair $(M,N)$ in a bicategory $\calB$, 
then $\eta\odot$ and $\varepsilon\odot$ make 
$(N\odot, M\odot)$ into an adjoint pair of functors,
so that for all $1$--cells $P\colon B\hto C$ 
and $Q\colon A\hto C$, there is a natural bijection
\[
	\calB(N\odot P, Q) \isom \calB(P,M\odot Q).
\]
\end{rem}

\begin{prop}[{\cite[Prop.~5.3]{Shulman}}]
\label{prop:basechangeobjdualpair}
If $f\colon A\to B$ is a vertical morphism in a 
framed bicategory $\bbD$, then 
$({}_f B,B_f)$ is naturally a dual pair. \qed
\end{prop}

Dual pairs can be composed:
given dual pairs 
$(M,N)$ and $(M',N')$
with 
$M\colon A \hto B$, $N\colon B\hto A$
and $M' \colon B \hto C$, $N' \colon C \hto B$
in a bicategory $\calB$, 
we have a dual pair 
$(M\odot M', N' \odot N)$ in $\calB$.
In view of equation~\eqref{eq:basechangeformula1},
Proposition~\ref{prop:basechangeobjdualpair}
therefore has the following corollary.
\begin{cor}
\label{cor:shriekduals}
Suppose $\calC$ is a symmetric monoidal presentable 
$\infty$--category.
Let $g \colon E_1 \to E_2$ be a map over $B$,
and suppose $X,Y \in \ho(\calC_{/E_1})$,
interpreted as $1$--cells $E_1 \hto B$, form a 
dual  pair  $(X^\op,Y)$ in $\Ex_B(\calC)$.
Then $((g_! X)^\op,g_!Y)$ 
is a dual pair in $\Ex_B(\calC)$. \qed
\end{cor}

\begin{defn}
A bicategory $\calB$ is called \emph{right closed} if 
composition of $1$--cells has an adjoint on the right,
so that for all $1$--cells 
$M\colon A\hto B$,  $N\colon B\hto C$ and $P \colon A\hto C$,
there is a natural bijection
\[
	\calB(M\odot N, P) \isom \calB(M,N\vartriangleright P).
\] 
\end{defn}

\begin{rem}
\label{rk:candidaterightdual}
Every $1$--cell $M\colon A\hto B$ in a right closed bicategory
has a canonical candidate for a right dual, namely
$M\vartriangleright U_B$. If $M$ indeed has a right dual 
$N$, then there exists a unique isomorphism
$N \isom M\vartriangleright U_B$ under which 
the counit $N\odot M \hto U_B$
of the dual pair $(M,N)$
corresponds to the counit 
$(M\vartriangleright U_B)\odot M\to U_B$
of the $(\odot\, M,M\,{\vartriangleright})$ adjunction.
\end{rem}

We record the following criterion for right dualizability
in a right closed bicategory.

\begin{prop}[{\cite[Prop.~16.4.12]{MaySigurdsson}}]
\label{prop:closedbicatdualcrit}
Let $\calB$ be a right closed bicategory. Then a $1$--cell
$M\colon A\hto B$ is right dualizable if and only if the map
\[
	\mu 
	\colon 
	M \odot (M\vartriangleright U_B) 
	\longto
	M \vartriangleright M
\]
adjoint to the composite
\begin{equation}
\label{eq:muadjoint}
	M \odot (M\vartriangleright U_B) \odot M
	\xto{\ M\odot\varepsilon\ }
	M \odot U_B
 	\xto{\ \isom\ }
	M
\end{equation}
is an isomorphism.
Here $\varepsilon$ denotes the counit of the 
$(\odot\, M,M\,{\vartriangleright})$
adjunction, and the isomorphism in \eqref{eq:muadjoint}
is given by the unit constraint in $\calB$.
\end{prop}

From Proposition~\ref{prop:closedbicatdualcrit},
we obtain the following criterion for checking 
right dualizability in $\Ex_B(\calC)$.
Compare with\ \cite[Lemma~6.3]{LindMalkiewich}.

\begin{prop}
\label{prop:exbdualizabilitycriterion}
Let $\calC$ be a 
symmetric monoidal presentable $\infty$--category,
let $B$ be a space, and let $M \colon E_1\hto E_2$ be a 
$1$--cell in $\Ex_B(\calC)$. Then $M$ is right dualizable in 
$\Ex_B(\calC)$ if and only if $b^{[\ast]}(M)$ is right
dualizable in $\Ex(\calC)$ for all maps $b\colon \pt \to B$.
\end{prop}

\begin{proof}
As equivalences in
the categories $\ho(\calC_{/X})$
are detected fibrewise,
the map
\[
	\mu 
	\colon 
	M \odot_B (M\vartriangleright_B U_{E_2}) 
	\longto
	M \vartriangleright_B M
\]
of Proposition~\ref{prop:closedbicatdualcrit}, a map in 
$\ho(\calC_{/E_1 \times_B E_2})$, is an equivalence precisely when 
\[
	b^{[\ast]}(\mu) 
	\colon 
	b^{[\ast]}(
		M \odot_B (M\vartriangleright_B U_{E_2}) 
	)
	\longto
	b^{[\ast]}(
		M \vartriangleright_B M
	),
\]
a map in $\ho(\calC_{/b^\ast E_1 \times b^\ast E_2})$,
is an equivalence for every $b\colon \pt \to B$.
But since by Proposition~\ref{prop:fastclosed}
the functor $b^{[\ast]}$ is closed, the 
map $b^{[\ast]}(\mu)$ above agrees under the appropriate
coherence isomorphisms with the map
\[
	\mu
	\colon 
	b^{[\ast]}(M) \odot (b^{[\ast]}(M) \vartriangleright U_{b^\ast E_2})
	\longto
	b^{[\ast]}(M) \vartriangleright b^{[\ast]}(M).
\]
See \cite[Lemma~30]{LahtinenThesis}. The claim follows.
\end{proof}

Let us now specialize to the case of a closed symmetric monoidal 
category $\calC$ with
tensor product $\tensor$, unit object $I$, symmetry constraint $\chi$,
and internal hom $F$. 
By interpreting $\calC$ as a one-object bicategory, from
Definition~\ref{def:dualpair} we obtain the notion
of a dual pair of objects in $\calC$.
Since $\calC$ is symmetric, the notions of a left and a right
dual of an object of $\calC$ coincide, however,
and we talk of an object being dualizable or being the 
dual of another object without reference to a handedness.
Every object $X$ of $\calC$ has a canonical candidate
for a dual, namely the object 
\[
	DX = F(X,I),
\]
and if $X$ indeed has a dual $Y$, then $Y$ is
isomorphic to $DX$ via an isomorphism under
which the evaluation map $\varepsilon \co Y\tensor X \to I$ corresponds
to the usual evaluation map $\varepsilon \co DX \tensor X \to I$.

In the present context, 
Proposition~\ref{prop:closedbicatdualcrit}
specializes to the following lemma.
\begin{lemma}
\label{lm:dcrit}
Let $\calC$ be a closed symmetric monoidal category.
Then an object $X$ of $\calC$ is dualizable if and only if the map
\[
	\mu\co X \tensor DX \longto F(X,X)
\]
adjoint to the composite
\[
	X \tensor DX \tensor X \xto{\ 1\tensor \varepsilon\ } X \tensor I \xto{\ \isom\ } X
\]
is an isomorphism. \qed
\end{lemma}

Recall that an object $X$ of a symmetric monoidal category $\calC$
is called \emph{invertible} if there exists an object $Y$ of $\calC$
and an isomorphism $Y\tensor X \isom I$. In this case,
$X$ and $Y$ are duals, with the aforementioned isomorphism giving the
counit of the dual pair. In particular, we have

\begin{lemma}
\label{lm:invcrit}
Let $\calC$ be a closed symmetric monoidal category.
Then an object 
$X$ of $\calC$ is invertible if and only if the evaluation map
\[
	\varepsilon \colon DX \tensor X \longto  I
\]
is an isomorphism.\qed
\end{lemma}

Since equivalences in $\ho(\calC_{/B})$ for 
$\calC$ a presentable
symmetric monoidal $\infty$--category
are detected fibrewise,
Lemmas~\ref{lm:dcrit} and~\ref{lm:invcrit}
implies the following criterion for 
detecting dualizable and invertible objects in $\ho(\calC_{/B})$.
\begin{cor}
\label{cor:fwdinvcrit}
Let $\calC$ be a presentable
symmetric monoidal $\infty$--category,
and let $B$ be a space.
Then an object of $\ho(\calC_{/B})$
is dualizable (resp.\ invertible) if and only if all its fibres are.
\qed
\end{cor}

\bibliographystyle{alpha}
\bibliography{stringtop}
\end{document}